\documentclass[11pt]{article}
\usepackage[authoryear]{natbib}
\usepackage{hyperref,amssymb,amsmath,amsthm,bm}
\hypersetup{colorlinks=true,
	citecolor=blue,
	anchorcolor = red,
	urlcolor=blue,
}

\usepackage{algorithm}
\usepackage{algpseudocode}
\usepackage{fullpage}
\usepackage{epsfig}
\usepackage{enumerate}
\usepackage{xcolor}
\definecolor{darkred}{RGB}{208,0,32}

\usepackage{extarrows}
\usepackage{tabularx}
\usepackage{makecell}
\usepackage{multirow}
\usepackage{xr}
\usepackage{caption}
\usepackage{subcaption}

\newtheorem*{thm-supp-1}{Proposition~\ref{thm:1step-ld}}
\newtheorem*{thm-supp-2}{Theorem~\ref{thm:normality}}

\usepackage{thmtools}
\usepackage{thm-restate}

\newcommand\diff{\,{\mathrm d}}
 
\newtheorem{assumption}{Assumption}
\newtheorem{thm}{\bf Theorem}[section]
\newtheorem{lemma}[thm]{Lemma}
\newtheorem{prop}[thm]{\bf Proposition}
\newtheorem{cor}[thm]{\bf Corollary}

\newtheorem{remark}{Remark}

\newcommand{\be}{\bm{\beta}}

\newcommand{\hbe}{\widehat{\bm{\beta}}}

\newcommand{\wphi}{\widetilde{\phi}}

\newcommand{\bz}{\bm{z}}
\newcommand{\bZ}{\bm{Z}}

\newcommand{\bv}{\bm{v}}
\newcommand{\bu}{\bm{u}}
\newcommand{\R}{\mathbb{R}}
\newcommand{\E}{\mathbb{E}}
\newcommand{\I}{\mathbb{I}}
\newcommand{\Prob}{\mathbb{P}}
\renewcommand{\mid}{\,|\,}
\newcommand{\calS}{\mathcal{S}}
\newcommand{\calD}{\mathcal{D}}
\newcommand{\calH}{\mathcal{H}}

\newcommand{\abs}[1]{\left\lvert#1\right\rvert}
\newcommand{\norm}[1]{\left\lVert#1\right\rVert}

\newcommand{\paren}[1]{\left(  #1 \right)}
\newcommand{\brackets}[1]{\left[ #1 \right]}
\newcommand{\braces}[1]{\left\{ #1 \right\}}

\newcommand{\mSMSE}{\tt{(mSMSE)}}
\newcommand{\AvgSMSE}{\tt{(Avg-SMSE)}}
\newcommand{\AvgMSE}{\tt{(Avg-MSE)}}
\newcommand{\wAvgSMSE}{\tt{(wAvg-SMSE)}}
\newcommand{\wmSMSE}{\tt{(wmSMSE)}}

\newcommand{\ov}{\overline}

\DeclareMathOperator*{\argmax}{\arg\!\max}
\DeclareMathOperator*{\argmin}{\arg\!\min}

\definecolor{DSgray}{cmyk}{0,1,0,0}

\begin{document}
	\title{Distributed Estimation and Inference for Semi-parametric Binary Response Models}
	\author{Xi Chen\footnote{Stern School of Business, New York University, e-mail: xc13@stern.nyu.edu}\quad
		Wenbo Jing\footnote{Stern School of Business, New York University, e-mail: wj2093@nyu.edu}\quad
		Weidong Liu\footnote{School of Mathematical Sciences, Shanghai Jiao Tong University, e-mail: weidongl@sjtu.edu.cn}\quad
		Yichen Zhang\footnote{Krannert School of Management, Purdue University, e-mail: zhang@purdue.edu}
	}
	
	\date{}
	\maketitle
	
	\begin{abstract}
		The development of modern technology has enabled data collection of unprecedented size, which poses new challenges to many statistical estimation and inference problems. This paper studies the maximum score estimator of a semi-parametric binary choice model under a distributed computing environment without pre-specifying the noise distribution. An intuitive divide-and-conquer estimator is computationally expensive and restricted by a non-regular constraint on the number of machines, due to the highly non-smooth nature of the objective function. 
		We propose (1) a one-shot divide-and-conquer estimator after smoothing the objective to relax the constraint, and (2) a multi-round estimator to completely remove the constraint via iterative smoothing. We specify an adaptive choice of kernel smoother with a sequentially shrinking bandwidth to achieve the superlinear improvement of the optimization error over the multiple iterations. The improved statistical accuracy per iteration is derived, and a quadratic convergence up to the optimal statistical error rate is established. We further provide two generalizations to handle the heterogeneity of datasets and high-dimensional problems where the parameter of interest is sparse.
	\end{abstract}

\section{Introduction}
\label{sec:introduction}

In many statistical applications, the phenomena that practitioners would
like to explain are dichotomous: the outcome/response can take only two
values. This model, usually referred to as binary response model, is central
to a wide range of fields such as applied econometrics, pharmaceutical
studies, clinical diagnostics and political sciences. The most commonly
used parametric approaches, notably logit and probit models, assume that
the functional form of the model, typically the distribution of the response
variable conditional on the explanatory variables, is known. Nonetheless,
there is usually little justification for assuming that the functional
form of conditional probability is known in practice. Once the functional
form is misspecified, the estimate of the underlying parameter of interest
and the corresponding inference results can be highly misleading (see
\citealp{white1982maximum,Horowitz1993,horowitz2001adaptive,greene2009discrete}
for illustration). To balance between the potential model misspecification
in parametric models and the curse of dimensionality in nonparametric models,
practitioners may assume that the threshold, instead of the conditional
probability, can be approximated by some prespecified function.

In this paper, we consider the following semiparametric binary response
model:
%
\begin{equation}
\label{eq:binary-model} %
\begin{aligned} Y=\mathrm{sign} \bigl(
Y^{*} \bigr)\in \{-1,+1\},\qquad Y^{*}=X+
\bm{Z}^{\top }\bm{\beta }^{*}+\epsilon , \quad \bm{\beta
}^{*} \in \mathbb{R}^{p}, \end{aligned}
%
\end{equation}
where $Y$ is the binary response variable, $X$ and $\bm{Z}$ are covariates
and $\epsilon $ denotes a random noise that is not required to be independent
of $\bm{Z}$. Assume that
$\{  (y_{i}, x_{i}, \bm{z}_{i}, \epsilon _{i}  )\}_{i=1,2,
\dots , n}$ are \emph{i.i.d.} copies of
$  (Y, X, \bm{Z}, \epsilon   )$. Our goal is to estimate
$\bm{\beta }^{*}$ given
$\{  ( y_{i}, x_{i}, \bm{z}_{i}  )\}_{i=1,2,\dots ,n}$. If the
distribution of the noise $\epsilon $ is prespecified, the model becomes
a traditional parametric model such as the probit and logit models. However,
as misspecification of the noise distribution may cause poor estimation,
practitioners prefer to use a semiparametric model to estimate
$\bm{\beta }^{*}$. \cite{manski1975maximum} proposed the
\emph{Maximum Score Estimator} (MSE),
%
\begin{equation}
\label{eq:MSE} \widehat{\bm{\beta }}_{\mathrm{{MSE}}}=\argmax
_{\bm{\beta }\in \mathbb{R}^{p}} \frac{1}{n}\sum_{i=1}^{n}
\bigl[\mathbb{I} \bigl( y_{i}=1, x_{i}+
\bm{z}_{i}^{\top}\bm{\beta }\geq 0 \bigr)+\mathbb{I} \bigl(
y_{i}=-1, x_{i}+\bm{z}_{i}^{\top}
\bm{\beta }<0 \bigr) \bigr], 
\end{equation}
where $\mathbb{I}(\cdot )$ denotes the indicator function. Manski referred
to the objective function in \eqref{eq:MSE} as a score function, and the
estimator $\widehat{\bm{\beta }}_{\mathrm{{MSE}}}$ is obtained via maximizing
the score function. The statistical properties of MSE have been well studied
in the literature \citep{chamberlain1986asymptotic,kim1990cube}. Nonetheless,
the existing literature rarely examines the practical feasibility to achieve
the estimator in a large data set, even though it is essential to make
use of the estimator in real-world applications. Despite a macroscopical
intuition underneath the optimization problem \eqref{eq:MSE}, the objective
function is hard to optimize due to nonconvexity and nonsmoothness, especially
when the sample size is large.

Indeed, the rapid development of modern technology has enabled data collection
of unprecedented size. Such a large amount of data are usually generated
and subsequently stored in a decentralized fashion. Due to the concerns
of data privacy, data security, ownership, transmission cost and others,
the decentralized data may not be pooled \citep{zhou2018statistical}. In
distributed settings, classical statistical results, which are developed
under the assumption that the data sets across different local machines
can be pooled together, are no longer applicable. Therefore, many estimation
and inference methods need to be reinvestigated.

In a distributed setting, a common estimation strategy is the divide-and-conquer
(DC) algorithm, which estimates a local estimator on each local machine
and then aggregates the local estimators to obtain the final estimator
(please refer to Section~\ref{sec:related-work} below for a detailed discussion
of existing literature). Combining the idea of DC with the Maximum Score
Estimator (MSE) intuitively leads to an averaged divide-and-conquer MSE
\texttt{(Avg-MSE)} approach, where one can solve an MSE on each local machine
and then aggregate the obtained solutions by an averaging operation. We
use $n$ to denote the total sample size and assume that the data is stored
on $L$ machines where each machine has $m=n/L$ data points.
\cite{shi2018massive} studied the \texttt{(Avg-MSE)} approach and obtained
a convergence rate $O(L^{-\frac{1}{2}}m^{-\frac{1}{3}})$ under a restrictive
constraint on the number of local machines: $L=o(m^{\frac{1}{6}})$, ignoring
the logarithm term. Therefore, given $m$, the established analysis is only
valid as $L=n/m$ increases to $O(m^{\frac{1}{6}})$, after which the estimator
will not improve anymore. In other words, the best convergence rate that
\texttt{(Avg-MSE)} can achieve is limited to
$O ((m^{\frac{1}{6}})^{-\frac{1}{2}}m^{-\frac{1}{3}} )=O(m^{-
\frac{5}{12}})$, regardless of how large the total sample size $n$ is.\footnote{It
is possible to relax the constraint on $L$ under some special cases. For
example, \cite{shi2018massive} showed that when $p=1$, the constraint can
be relaxed to $L=o(m^{\frac{4}{9}})$. In such cases, the corresponding
best convergence rate is limited to
$O ((m^{\frac{4}{9}})^{-\frac{1}{2}}m^{-\frac{1}{3}} )=O(m^{-
\frac{5}{9}})$.} Besides the aforementioned limitation on slow convergence
rate and the constraint of $L$, the limiting distribution of MSE cannot
be given in an explicit form \citep{kim1990cube}, and hence it can be hard
in practice to apply it for inference. Additionally, an exact solver of
MSE requires solving a mixed integer programming, which is computationally
heavily demanding when the dimension $p$ is large.

In order to improve the convergence rate and relax the restriction on the
number of machines, we first consider a natural alternative, an
\emph{averaged Smoothed Maximum Score Estimator} \texttt{(Avg-SMSE)}, which
optimizes a smoothed version of the MSE objective
\citep{horowitz1992smoothed} on each local machine and then aggregates
the obtained estimators by averaging. In contrast to \texttt{(Avg-MSE)},
\texttt{(Avg-SMSE)} achieves a convergence rate
$O (n^{-\frac{\alpha}{2\alpha +1}} )$ under a constraint
$L= o (m^{\frac{2}{3}  ( \alpha -1  )} )$, assuming
$p$ is fixed and ignoring the logarithm term, where $\alpha $ is the smoothness
parameter of the kernel function used to smooth the objective function.
This implies that given $m$, the best achievable convergence rate is
$O(m^{-\frac{\alpha}{3}})$, regardless of how large $n$ is. When
$\alpha \geq 2$, \texttt{(Avg-SMSE)} achieves a convergence rate of at least
$O(m^{-\frac{2}{3}})$, faster than that of \texttt{(Avg-MSE)}, under a weakened
constraint of $L$. The detailed algorithm is described in Section~\ref{sec:method-avg-smse}, and the theoretical analysis of \texttt{(Avg-SMSE)}
is presented in Section~\ref{sec:theory-AvgSMSE}. The analysis of \texttt{(Avg-SMSE)}
is not trivially adopted from that of \texttt{(Avg-MSE)}, since it involves
handpicking an optimal bandwidth $h$ to balance the bias and variance.

The analysis reveals a fundamental bottleneck in divide-and-conquer algorithms:
the bias-variance trade-off in the mean squared error analysis. In many
statistical problems under a nondistributed environment, an asymptotic
unbiased estimator is often satisfactory as it enables theorists to establish
the asymptotic normality and design statistical inference procedures based
on the asymptotic distribution, even though the order of the asymptotic
bias depends on the sample size.
\label{disc:bias-variance-constraint}
Nonetheless, in a distributed environment, the biases across multiple local
machines cannot be reduced by aggregation. When the number of machines
$L$ is large, the bias term of the divide-and-conquer estimator is dominant
in the estimation error, which leads to an efficiency loss and limits the
improvement of the divide-and-conquer algorithms as the total sample size
$n$ increases. In an extreme scenario, as the local sample size $m$ stays
fixed and $L$ increases, the convergence rate of the divide-and-conquer
algorithms does not improve. Unfortunately, this scenario is indeed more
practically realistic as each local machine has its storage limit but the
number of machines may always be increasing.

To feature the scenario that $L$ exceeds
$m^{\frac{2}{3}(\alpha -1)}$, this paper further proposes a new approach
called \emph{multiround Smoothed Maximum Score Estimator} \texttt{(mSMSE)}
in Section~\ref{sec:method-mSMSE}, which successively refines the estimator
with multiple iterations. In iteration $t$, the algorithm implements a
Newton step to optimize the smoothed objective function governed by a smoothing
parameter $h_{t}$. Over the multiple rounds, the algorithm iteratively
smooths the score function \eqref{eq:MSE} by a diminishing sequence
$\{h_{t}\}$. The decay rate of $\{h_{t}\}$ is carefully designed by examining
the rate improvement in each iteration of the algorithm, details of which
are provided in Section~\ref{sec:theory-mSMSE}. Under this iterative smoothing
scheme, the proposed \texttt{(mSMSE)} converges to the optimal statistical
rate $O(n^{-\frac{\alpha}{2\alpha +1}})$ in $O(\log \log n)$ iterations
and is scalable in dimension and computationally efficient. More specifically,
the algorithm performs a quadratic (superlinear) convergence across iterations
until it reaches the optimal accuracy. We thereafter establish the asymptotic
normality of \texttt{(mSMSE)} in Theorem~\ref{thm:normality}, followed by
a bias-correction procedure to construct the confidence interval of \texttt{(mSMSE)}
from samples in Corollary~\ref{cor:inference} for the purpose of distributed
inference.

Lastly, we consider two important extensions of the proposed methods. First,
an important question in distributed environments is the heterogeneity
of data sets on different local machines. In Section~\ref{sec:hetero},
we consider heterogeneous data sets with a shared parameter of interest
$\bm{\beta }^{\star}$ and different distributions of covariates
$(X,\bm{Z})$ across the machines. We show that \texttt{(mSMSE)} performs
better than \texttt{(Avg-SMSE)} with a weaker condition. While handling heterogeneous
data, the performance of the divide-and-conquer algorithm relies on the
smallest sample size among the local machines, while the \texttt{(mSMSE)}
method does not rely on such conditions. Section~\ref{sec:4.2} delves
into a coefficient shift setting, which further allows different regression
coefficients $\bm{\beta }^{*}$ across some machines. We design a machine
selection method to choose machines with similar coefficients, based on
which our multiround algorithm can be extended. We show the effectiveness
of our approach by establishing an improved convergence rate of our algorithm
compared to \texttt{(Avg-SMSE)} under the coefficient shift setting.

We further consider a high-dimension extension in Section~\ref{sec:highd} where the parameter of interest
$\bm{\beta }^{*}\in \mathbb{R}^{p}$ is a sparse vector with $s$ nonzero
elements and $s<n<p$. We modify \texttt{(mSMSE)} to adapt the idea of the
Dantzig selector \citep{candes2007dantzig} to reach the convergence rate
$\sqrt{s}(\log p/n)^{\frac{\alpha}{2\alpha +1}}$ in a distributed environment,
which is very close to the minimax optimal rate
$(s\log p/n)^{\frac{\alpha}{2\alpha +1}}$ for the linear binary response
model established by \cite{feng2019nonregular} in a nondistributed environment.
Compared to the low-dimensional settings above, the algorithm in this setting
reduces the per-iteration communication cost to $p\times 1$ vectors.

We emphasize the technical challenges below and summarize the methodology
contribution and theoretical advances compared to the existing literature.
\begin{itemize}
\label{bullets}
\item We propose two algorithms for distributed estimation and inference
based on (1) a divide-and-conquer estimator \texttt{(Avg-SMSE)} that improves
the convergence rate and relaxes the constraint on $L$ of \texttt{(Avg-MSE)}
in \cite{shi2018massive}, and (2) a multiround estimator \texttt{(mSMSE)}
that further completely removes the constraint by sequentially smoothing
the objective function over iterations. \texttt{(mSMSE)} ensures a quadratic
algorithmic convergence toward optimal statistical accuracy and can be
applied even to nondistributed settings as an efficient computational algorithm
to solve MSE on the pooled data set. We further show that \texttt{(mSMSE)}
achieves the same statistical efficiency uniformly over a class of models
in a neighborhood of the given model.
\item The discontinuity in the objective function is the major technical
challenge. Existing algorithms for distributed estimation rarely handle
nonsmooth objectives, with a few exceptions that either provide statistical
results without algorithmic guarantee or heavily rely on subgradient-based
algorithms to solve the convex objective. However, the subgradient of
\eqref{eq:MSE} is almost everywhere zero. The proposed \texttt{(mSMSE)} handles
this challenge by carefully identifying a smoothing objective to approximate
\eqref{eq:MSE}, which varies over the iterations. The proposed procedure
shares the spirit of the Newton-type multiround algorithms established
for smooth objectives in \cite{jordan2016communication} and
\cite{fan2021communication}. Nonetheless, the proposed \texttt{(mSMSE)} is
not a simple application of the algorithms therein. Instead, \texttt{(mSMSE)}
maximizes different smoothed objectives over multiple iterations. A proper
decay of the smoothing parameter (bandwidth) over iterations is critical,
due to the dilemma that large bandwidths cannot achieve the optimal statistical
rate, while small bandwidths are invalid in the earlier stage of the proposed
algorithm when the initial error is poor. Furthermore, the analysis of
\cite{jordan2016communication} exploits the assumption that the Hessian
matrix is Lipschitz continuous with respect to the parameter
$\bm{\beta }$. However, the Lipschitz constant of our objective function
tends to infinity if we smooth the objective with the optimal bandwidth.
To solve this challenge, we build an analysis to establish a uniform quantification
for the improvement in each round as the bandwidth shrinks over iterations.
The uniformness is necessary, since in each iteration, the initial estimator
is the output of the last iteration, which has already utilized the complete
sample and is therefore data-dependent.
\item We consider two common settings for data heterogeneity across different
machines: covariate shift and coefficient shift. For the covariate shift
setting, we establish the asymptotic normality for the weighted algorithms
and provide a way to choose optimal weight matrices, which further highlights
that the multiround method \texttt{(wmSMSE)} is at least as efficient as
the divide-and-conquer method \texttt{(wAvg-SMSE)}. To our best knowledge,
the weighted multiround estimator and its superior efficiency over the
weighted divide-and-conquer method is not established in the literature.
\item To our best knowledge, the existing literature of low-dimensional
maximum score estimator
\citep{horowitz1992smoothed,shi2018massive,Banerjee2019Circumventing} established
the theoretical results on $p=1$ or fixed $p$, in both nondistributed and
distributed learning. Our theory formally allows diverging dimension
$p$. Additionally, we show that the constraint for the divide-and-conquer
methods is more strict for increasing $p$. Lastly, in high-dimensional
settings, neither the nonsmooth objective nor the smoothed objective is
guaranteed to be convex, which causes difficulty in applying the
$\ell _{1}$-regularization methods. Instead, we adopt the Dantzig estimator
in each iteration to make the optimization problem feasible while encouraging
the sparse structure.
\end{itemize}

The remainder of the paper is organized as follows: Section~\ref{sec:related-work} reviews the related literature on distributed
estimation and inference. Section~\ref{sec:method} describes the methodology
of the proposed \texttt{(Avg-SMSE)} and \texttt{(mSMSE)} procedures. Sections~\ref{sec:theory-AvgSMSE}--\ref{sec:theory-mSMSE} present the theoretical
results for the two estimators, respectively. Section~\ref{sec:Inference} uses the asymptotic results to facilitate distributed
inference.  In Section~\ref{sec:4}, we discuss the effect of data heterogeneity
on distributed estimation and inference. In Section~\ref{sec:highd}, we
modify \texttt{(mSMSE)} to apply to high-dimensional semiparametric binary
response models. Numerical experiments in Section~\ref{sec:simulations} lend empirical support to our theory, followed by
conclusions and future directions in Section~\ref{sec:conclusion}. Further
discussions and additional theoretical and experimental results including
all technical proofs are relegated to Appendix.
\subsection{Related works}
\label{sec:related-work}

For distributed estimation and large-scale data analysis, the divide-and-conquer
(DC) strategy has been recently adopted in many statistical estimation
problems (see, e.g.,
\citealp{zhang2012communication,li2013statistical, chen2014split, 
 lee2017communication,battey2015distributed, shi2018massive, banerjee2019divide, huang2019distributed, volgushev2017distributed,fan2019distributed, dobriban2020wonder}).
The DC strategy is utilized to handle massive data when the practitioner
observes a large-scale centralized data set, partitions it into subsamples,
performs an estimation on each subsample and finally aggregates the subsample-level
estimates to generate a (global) estimator. The divide-and-conquer principle
also fits into the scenarios where data sets are collected and stored originally
in different locations (local machines), such as sensor networks, and cannot
be centralized due to high communication costs or privacy concerns. A standard
DC approach computes a local estimator (or local statistics) on each local
machine and then transports them to the central machine in order to obtain
a global estimator by appropriate aggregation. Recently,
\cite{shi2018massive} and \cite{banerjee2019divide} studied the divide-and-conquer
principle in nonstandard problems such as isotonic regression and cube-root
M-estimation. They showed that the DC strategy improves the statistical
rate of the cubic-rate estimator, while the aggregated estimators often
entail the super-efficiency phenomenon, a circumvention of which is proposed
by \cite{Banerjee2019Circumventing} via synthesizing certain summary statistics
on local machines instead of aggregating the local estimators.

DC approaches often require a restriction that the number of machines
(subgroups) does not increase very fast compared to the smallest local
subsample size, such as to retain the optimal statistical efficiency
\citep{zhang2012communication}. Such a restriction can be stringent in
many decentralized systems with strong privacy and security concerns. In
these scenarios, \cite{shamir2014communication},
\cite{wang2016efficient} and \cite{jordan2016communication} proposed multiround
procedures for refinement, followed by
\cite{tu2021variance,dobriban2021distributed,chen2021first,duan2022heterogeneity,luo2022distributed,fan2021communication}
and others. These frameworks typically use the outputs of the preceding
iteration as the input for the succeeding iterations. After a number of
rounds, the estimator is refined to achieve the optimal statistical rate.
Such procedures, typically performed with continuous optimization algorithms,
are generally inapplicable to nonsmooth problems due to their requirement
that the loss function needs to be sufficiently smooth, although a nonsmooth
regularization term is permitted. The understanding of the multiround improvement
under the nonsmooth scenarios is still in many ways nascent, and existing
analyses mostly depend on the specific statistical model.
\cite{wang2019distributed}, \cite{chen2019quantile} and
\cite{battey2021communication} proposed remedies for specific continuous
objective functions that violate second-order differentiability but retain
root-$n$ consistency, mainly featuring linear support vector machines and
quantile regression. For a general nonsmooth objective,
\cite{chen2021first} showed that a multiround distributed estimator still
entails a slow rate of statistical convergence as well as deficiencies
in algorithmic convergence.

Moreover, for some nonstandard problems, the aforementioned restriction
of the standard (one-shot) DC procedures deteriorates in nonstandard problems
\citep{shi2018massive}, due to the slower than root-$n$ convergence rate.
It remains unexplored whether the restriction can be relaxed or removed
by multiround methods in nonstandard problems. In this paper, we spotlight
the semiparametric binary response model whose corresponding loss function
is nonconvex and discontinuous, violating the assumptions in the multiround
distributed learning literature mentioned above.

For distributed inference, existing works, for example,
\cite{jordan2016communication}, \cite{chen2021first} and others, established
asymptotic normality for their distributed estimators and yielded distributed
approaches to construct confidence regions using the sandwich-type covariance
matrices. In addition to the above,
\citeauthor{yu2020simultaneous} (\citeyear{yu2020simultaneous,yu2022distributed}) proposed distributed bootstrap
methods for simultaneous inference in generalized linear models that allow
a flexible number of local machines. \cite{wang2022reboot} proposed a bootstrap-and-refitting
procedure that improved the one-shot performance of distributed bootstrap
via refitting. We refer the readers to \cite{gao2022review} for an extended
literature review of distributed estimation and inference.

\subsection{Notation}
\label{sec1.2}

For any vector
$\bm{v}=  ( v_{1},\ldots,v_{p}  )\in \mathbb{R}^{p}$, we denote the
$\ell _{q}$-norm by
$\|\bm{v}\|_{q}:=  ( \sum_{k=1}^{p}|v_{k}|^{q}  )^{1/q}$ and
the $\ell _{\infty}$-norm by
$\|\bm{v}\|_{\infty}:=\max_{k} |v_{k}|$. Denote by
$\bm{v}_{S}=  ( v_{i_{1}},v_{i_{2}},\ldots,v_{i_{s}}  )^{\top}$ for
any given $S=\{i_{1},i_{2},\ldots,i_{s}\} \subset \{1,2,\ldots,p\}$. For any
matrix $A=  ( a_{ij}  )$, we define
$\|A\|_{1}=\max_{j}   ( \sum_{i} |a_{ij}|  )$,
$\|A\|_{2}=\max_{\|\bm{v}\|_{2}=1} \|A\bm{v}\|_{2}$,
$\|A\|_{\infty}=\max_{i}   ( \sum_{j} |a_{ij}|  )$ and
$\|A\|_{\max}:=\max_{i,j}|a_{ij}|$. Also, for a sequence of random variables
$X_{n}$ and a sequence of real numbers $a_{n}$, we let
$X_{n}=O_{\mathbb{P}}(a_{n})$ denote that
$  \{X_{n}/a_{n}  \}$ is bounded in probability and
$X_{n}=o_{\mathbb{P}}(a_{n})$ denote that
$  \{X_{n}/a_{n}  \}$ converges to zero in probability. For any
positive sequences $\{a_{n}\}$ and $\{b_{n}\}$, we write
$a_{n} \lesssim b_{n}$ if $a_{n}=O(b_{n})$, and $a_{n} \asymp b_{n}$ if
$a_{n} \lesssim b_{n}$ and $b_{n} \lesssim a_{n}$.

\section{Methodology}
\label{sec:method}

\subsection{Preliminaries}
\label{sec:Preliminaries}

To estimate $\bm{\beta }^{*}$ in model \eqref{eq:binary-model},
\cite{manski1975maximum} proposed the \emph{Maximum Score Estimator} (MSE)
in \eqref{eq:MSE}. The objective function that MSE maximizes is the a
\emph{score function}, which counts the number of correct predictions given
$\bm{\beta }$. By rewriting
\begin{align*}
&\quad \mathbb{I} \bigl( y_{i}=1, x_{i}+
\bm{z}_{i}^{\top}\bm{\beta } \geq 0 \bigr)+\mathbb{I}
\bigl( y_{i}=-1, x_{i}+\bm{z}_{i}^{\top}
\bm{\beta }<0 \bigr) =\mathbb{I} (y_{i}=-1 )+y_{i}
\mathbb{I} \bigl( x_{i}+\bm{z}_{i}^{\top}\bm{
\beta }\geq 0 \bigr),
\end{align*}
the maximum score estimator in \eqref{eq:MSE} can be simplified as the
following form:
%
\begin{equation}
\label{eq:MSE-simplified} \widehat{\bm{\beta }}_{\mathrm{{MSE}}}=\argmin
_{\bm{\beta }\in \mathbb{R}^{p}} F^{*} ( \bm{\beta } ):=\argmin _{\bm{\beta }\in \mathbb{R}^{p}}
\frac{1}{n}\sum_{i=1}^{n} (
-y_{i} )\mathbb{I} \bigl( x_{i}+
\bm{z}_{i}^{\top}\bm{\beta }\geq 0 \bigr). 
\end{equation}
For the identifiability of $\bm{\beta }^{*}$, we assume the following conditions
about the distribution of $X$, $\bm{Z}$ and $Y$ hold for the entire paper:
\begin{enumerate}[(a)]
\item[(a)] The median of the noise conditional on $X$ and $\bm{Z}$ is 0, that
is, $\mathrm{median}  (\epsilon |X, \bm{Z}  )=0$.
\item[(b)] The support of $(X, \bm{Z})$ is not contained in any proper linear
subspace of $\mathbb{R}^{p+1}$.
\item[(c)] For almost every $(X, \bm{Z})$,
$0<\mathbb{P}  ( Y=-1 |X, \bm{Z}  )<1$.
\item[(d)] The distribution of $X$ conditional on $\bm{Z}$ has positive density
almost everywhere.\vadjust{\goodbreak}
\end{enumerate}
See \cite{manski1985semiparametric,horowitz1992smoothed} for the roles
that these conditions play in ensuring model identifiability.

\begin{remark}
\label{rmk:MSE}
 The original semiparametric binary response model considered in
\cite{manski1975maximum} is defined as
$Y=\mathrm{sign}  ( \bm{W}^{\top }\bm{\theta}^{*}+\epsilon
  )$, where
$\bm{W}=(W_{1}, \dots , W_{p+1}) \in \mathbb{R}^{p+1}$ is the covariate
vector, and
$\bm{\theta}^{*}=(\theta ^{*}_{1}, \dots , \theta ^{*}_{p+1})\in
\mathbb{R}^{p+1}$ is the true underlying parameter. To guarantee the identifiability
of the model, the underlying parameter $\bm\theta ^{*}$ is normalized,
since multiplying a positive scale constant on both $\bm\theta ^{*}$ and
$\epsilon $ does not change response variable $Y$. In parametric models
such as probit model or logit model, the identifiability is generally guaranteed
by fixing the scale of the distribution of the noise $\epsilon $. However,
for the semiparametric model, which does not assume the distribution of
$\epsilon $, it is necessary to fix the scale of $\bm\theta ^{*}$ instead
by normalization. One popular and convenient normalization method suggested
in \cite{horowitz1992smoothed} is to require $|\theta ^{*}_{1}|=1$, where
$\theta _{1}$ denotes the first entry of $\bm{\theta}^{*}$. Without loss
of generality, we suppose $\theta ^{*}_{1}=1$ and let
$\bm{\theta}^{*}= (1, \bm{\beta }^{*\top} )^{\top}$ and
$\bm{W}=  (X, \bm{Z}^{\top}  )^{\top}$, which leads to model
\eqref{eq:binary-model} and estimator \eqref{eq:MSE}. The variable
$X$ is just a continuous covariate separated from the entire covariates
$\bm{W}$ such that its corresponding coefficient is positive, and the distribution
of $X$ conditional on the remaining covariates has positive density almost
everywhere. This condition ensures the identifiability and consistency
of MSE.

 Moreover, the MSE is further connected to covariate-adjusted threshold
models. The variable $X$ can be interpreted as a continuous subject-specific
random threshold, and the assumed model
$Y=\mathbb{I}(\bm{Z}^{\top}\bm{\beta }^{*}+\epsilon >-X)$ indicates that
the binary response $Y=1$ if and only if a linear predictor
$\bm{Z}^{\top}\bm{\beta }^{*}$ adding a noise is greater than the threshold
$-X$. The MSE provides an estimator for the coefficient
$\bm{\beta }^{*}$ without specifying the noise type. The binary-response
covariate-adjusted threshold model is applied in various problems, including
dose-response study \citep{brockhoff1997random}, hypertension prediction
\citep{csenturk2009covariate}, ROC-curve cut-point estimation
\citep{xu2014model}, among others.
\end{remark}

The MSE in \eqref{eq:MSE} is known to suffer from a slow rate of convergence
due to the discontinuity of the objective function. Specifically,
\cite{kim1990cube} showed that, due to the non-smoothness of the objective
function, the MSE is subject to a cubic rate
$O  ( n^{-1/3}  )$, which is slower than the parametric rate
$O  (n^{-1/2}  )$ of the maximum likelihood estimator. They also
established that the limiting distribution of the maximum score estimator
cannot be given in an explicit form, and hence it can be hard to apply
it in practice for inference.

To overcome the drawbacks of MSE, \cite{horowitz1992smoothed} proposed
a \emph{Smoothed Maximum Score Estimator} (SMSE) by replacing the objective
function $F^{*}(\bm{\beta })$ in \eqref{eq:MSE-simplified} with a sufficiently
smooth function $F_{h}  ( \bm{\beta }  )$. More specifically,
the indicator function $\mathbb{I}  (\cdot \geq 0  )$ is approximated
by a kernel smoother $H  ( \cdot /h  )$, where $h$ is the bandwidth,
and SMSE is defined as
%
\begin{equation}
\label{SMSE} \widehat{\bm{\beta }}_{\mathrm{{SMSE}}}:=\argmin
_{\bm{\beta }\in
\mathbb{R}^{p}} F_{h} ( \bm{\beta } )=\argmin _{
\bm{\beta }\in \mathbb{R}^{p}}
\frac{1}{n}\sum_{i=1}^{n} (
-y_{i} )H \biggl( \frac{x_{i}+\bm{z}_{i}^{\top}\bm{\beta }}{h} \biggr). 
\end{equation}
The asymptotic behavior of the smoothed estimator can be analyzed using
the classical non-parametric kernel regression theory. The SMSE of
$\bm{\beta }^{*}$ is consistent and has a typical non-parametric rate of
convergence $n^{-\frac{\alpha}{2\alpha +1}}$, where $\alpha $ is the order
of the kernel function (see Assumption~\ref{A1} for a formal definition).
When $\alpha =1$, this convergence rate matches
$O  (n^{-1/3}  )$ of MSE. When $\alpha \geq 2$, this rate is at
least $O  (n^{-2/5}  )$, and it can be closer to
$O  (n^{-1/2}  )$ with a larger $\alpha $. Meanwhile, in contrast
to MSE, the limiting distribution of SMSE is in an explicit form, and the
parameters in the distribution can be estimated to feature inference applications.

\subsection{Divide-and-conquer SMSE}
\label{sec:method-avg-smse}

Under the distributed environment, the data is split into $L$ equally sized
subsets (machines) $\{\mathcal{D}_{\ell}, \ell =1,2,\dots , L\}$, where
each subset $\mathcal{D}_{\ell}$ has $m=n/L$ data points and the index
is denoted by $\mathcal{H}_{\ell}$, that is,
$\mathcal{D}_{\ell}=\{y_{i},x_{i}, \bm{z}_{i},i\in \mathcal{H}_{\ell}
\}$. A natural solution to a distributed learning task is
\emph{Divide-and-Conquer} via \emph{Averaging}, which computes a local estimator
on each subset and then averages the local estimators over all subsets.
Particularly, we define the \emph{Averaged Maximum Score Estimator} \texttt{(Avg-MSE)}
as
%
\begin{equation}
\label{eq:AvgMSE} \widehat{\bm{\beta }}_{\mathrm{{{\mathtt{{(Avg-MSE)}}}}}}:=\frac{1}{L}
\sum^{L}_{
\ell =1} \biggl[\argmin
_{\bm{\beta }\in \mathbb{R}^{p}} \frac{1}{m} \sum_{i\in \mathcal{H}_{\ell}}
( -y_{i} )\mathbb{I} \bigl(x_{i}+
\bm{z}_{i}^{\top}\bm{\beta }\geq 0 \bigr) \biggr].
\end{equation}
Similarly, the \emph{Averaged Smoothed Maximum Score Estimator} \texttt{(Avg-SMSE)}
is defined by
%
\begin{equation}
\label{eq:AvgSMSE} \widehat{\bm{\beta }}_{\mathrm{{{\mathtt{{(Avg-SMSE)}}}}}} :=
\frac{1}{L}\sum_{
\ell =1}^{L}
\biggl[ \argmin _{\bm{\beta }\in \mathbb{R}^{p}}\frac{1}{m} \sum
_{i \in \mathcal{H}_{\ell}} ( -y_{i} )H \biggl(
\frac{x_{i}+\bm{z}_{i}^{\top}\bm{\beta }}{h} \biggr) \biggr]. 
\end{equation}

\cite{shi2018massive} showed that the \texttt{(Avg-MSE)} has the convergence
rate $O (L^{-1/2}m^{-1/3} )$ under the constraint that
$L=o  (m^{1/6}  )$. In Theorem~\ref{thm:asym-DC}, we will show
that \texttt{(Avg-SMSE)} achieves the optimal convergence rate
$O (n^{-\frac{\alpha}{2\alpha +1}} )$ under a constraint that
$L=o ( m^{\frac{2}{3}  ( \alpha -1  )} )$, ignoring the
logarithm term. A comparison under different specifications of
$(L,m,n)$ is provided in Remark~\ref{rem:rate_shi} after we establish the
theoretical properties of \texttt{(Avg-SMSE)} in Section~\ref{sec:theory-AvgSMSE}. In short, when $\alpha \geq 2$, the convergence
rate of \texttt{(Avg-SMSE)} is at least $O(n^{-\frac{2}{5}})$, better than
\texttt{(Avg-MSE)} under a much weakened constraint of $L$. That being said,
both methods require a constraint on the number of machines $L$. To remove
this constraint, we propose a multiround SMSE in the following section.

\subsection{Multiround SMSE}
\label{sec:method-mSMSE}

In addition to satisfying the statistical properties, the SMSE objective
in \eqref{SMSE} is twice differentiable, which enables one to use continuous
optimization algorithms to iteratively improve an inefficient estimator.
For any initial estimator $\bm{\beta }$, we adopt a Newton step,
%
\begin{equation}
\label{newton} \widehat{\bm{\beta }}= \bm{\beta }- \bigl[\nabla
^{2} F_{h} ( \bm{\beta } ) \bigr] ^{-1}
\nabla F_{h} (\bm{\beta } ), 
\end{equation}
where the gradient vector and Hessian matrix have the following analytical
form:
\begin{align*}
\nabla F_{h} ( \bm{\beta } )&=\frac{1}{nh}\sum
_{i=1}^{n} ( -y_{i}
)H^{\prime} \biggl( \frac{x_{i}+\bm{z}_{i}^{\top}\bm{\beta }}{h} \biggr)\bm{z}_{i},\\
\nabla ^{2} F_{h} ( \bm{\beta } )&=
\frac{1}{nh^{2}}\sum_{i=1}^{n} (
-y_{i} )H^{\prime \prime} \biggl( \frac{x_{i}+\bm{z}_{i}^{\top}\bm{\beta }}{h} \biggr)
\bm{z}_{i}\bm{z}_{i}^{
\top}.
\end{align*}
Newton's one-step estimator not only provides us with a fast implementation
to obtain the minimizer numerically but also fits easily into a divide-and-conquer
scheme for distributed estimation and inference. Concretely, in the
$t$th iteration, for each batch of data $\mathcal{D}_{\ell}$, we compute
the gradient and Hessian on each local machine $\ell$ by
%
\begin{equation}
{\label{eq:def:Umlt} %
\begin{aligned} U_{m,\ell}^{  ( t  )}&=
\frac{1}{mh_{t}}\sum_{i\in
\mathcal{H}_{\ell}} (
-y_{i} )H^{\prime} \biggl( \frac{ x_{i}+ \bm{z}_{i}^{\top}\widehat{\bm{\beta }}^{  ( t-1  )}}{h_{t}} \biggr)
\bm{z}_{i},\\
 V_{m,\ell}^{  ( t  )}&=
\frac{1}{mh^{2}_{t}} \sum_{i\in \mathcal{H}_{\ell}} (
-y_{i} )H^{\prime
\prime} \biggl( \frac{ x_{i}+\bm{z}_{i}^{\top}\widehat{\bm{\beta }}^{  ( t-1  )}}{h_{t}} \biggr)
\bm{z}_{i}\bm{z}_{i}^{\top}. \end{aligned}
} 
\end{equation}
We then send
$ \{U_{m,\ell}^{  ( t  )}, V_{m,\ell}^{  ( t  )}
 \}_{\ell =1}^{L}$ to a central machine and average them to obtain the
gradient and the Hessian of the entire data set. Now we formally present
the \emph{multiround Smoothed Maximum Score Estimator} \texttt{(mSMSE)} in
the distributed setting in Algorithm~\ref{alg:mSMSE}. In the following
section, we show that \texttt{(mSMSE)} achieves the optimal rate without
restrictions on the number of machines $L$.

Our algorithm \texttt{(mSMSE)} is not simply applying the Newton's method
to a smoothed objective with a fixed smoothing bandwidth, since a large
initial error prevents using a small bandwidth in the earlier stage, but
keeping bandwidth large will lead to a suboptimal rate. To deal with this
problem, \texttt{(mSMSE)} iteratively smooths the objective function with
a diminishing bandwidth sequence $\{h_{t}\}$. The decay rate of
$\{h_{t}\}$ is carefully designed by examining the rate improvement in
each iteration of the algorithm, which is in detail explained in Section~\ref{sec:theory-mSMSE}.
\label{additional-discussion}%
An alternative algorithm for Newton's method is using a first-order method
to minimize a surrogate loss, studied by
\cite{jordan2016communication}, \cite{chen2021first} and others for distributed
estimation and inference. We anticipate an analogous analysis can be done
if one replaces the second-order approach in our algorithm with the surrogate-loss
procedures, which we leave for future work.
%
\begin{algorithm}[t]
        \caption{Multiround Smoothed Maximum Score Estimator \texttt{(mSMSE)}}
        \vspace*{0.08in} {\textbf{Input:}}
        Data sets distributed on local machines $\{x_{i}, \bm{z}_{i},y_{i}\}_{i\in \mathcal{H}_{\ell}}$ $  ( \ell =1,2,\dots ,L  )$, 
        the total number of iterations $T$, a sequence of bandwidths $\{h_{t}\}$ $  ( t=1,2,\dots ,T  )$.
        \label{alg:mSMSE}
\begin{algorithmic}[1]
                \State Compute an initial estimator $\widehat{\bm{\beta }}^{  (0  )}$ by minimizing \eqref{eq:MSE} or \eqref{SMSE} on a small subset of the data;
                \For{$t=1,2,\dots ,T$} \do \\
                \State Send $\widehat{\bm{\beta }}^{   ( t-1  )}$ to each machine;
                \For {$\ell =1,2,\dots ,L$} \do\\
                \State Compute $U^{  ( t  )}_{m,\ell}$, $V^{  ( t  )}_{m,\ell}$ by \eqref{eq:def:Umlt};
                \State Send $U^{  ( t  )}_{m,\ell}$, $V^{  ( t  )}_{m,\ell}$ back to a central machine;
                \EndFor
                \State Compute $U^{  ( t  )}_{n}=\frac{1}{L}\sum_{\ell =1}^{L}U^{  ( t  )}_{m,\ell}$ and $V^{  ( t  )}_{n}=\frac{1}{L}\sum_{\ell =1}^{L}V^{  ( t  )}_{m,\ell}$;
                \State Update $\widehat{\bm{\beta }}^{   ( t  )}=\widehat{\bm{\beta }}^{   ( t-1  )}- (V^{  ( t  )}_{n} )^{-1}U^{  ( t  )}_{n}$;
                \EndFor
                \State Output $\widehat{\bm{\beta }}^{  ( T  )}$.
        \end{algorithmic}
\end{algorithm}

\section{Theoretical results}
\label{sec:theory}

In this section, we first study the consistency and asymptotic distribution
of \texttt{(Avg-SMSE)} and \texttt{(mSMSE)} under low-dimensional homogeneous
settings, where the data distribution on each local machine is identical
to each other. Extension to data heterogeneity and high-dimensional
settings are left to Section~\ref{sec:4} later.

For both methods, we need the following regularity conditions.

\begin{assumption}
\label{A1}
Assume that the function $H(x)$ is the integral of an $\alpha $-order kernel,
that is,
\begin{equation*}
\pi _{U}:=
\int _{-1}^{1} x^{\alpha }H^{\prime}
( x )\, { \mathrm{d}}x \neq 0,\quad \text{and}\quad
\int _{-1}^{1} x^{k}
H^{
\prime} ( x ) \,{\mathrm{d}}x = 0, \quad k=1,2,\dots , \alpha -1,
\end{equation*}
where $\alpha \geq 2$ is a positive integer. Further, assume that the kernel
is square-integrable, that is,
$\pi _{V}:=\int _{-1}^{1}  [ H^{\prime}  ( x  )  ]^{2}
\,{\mathrm{d}}x<\infty $, and also has a bounded and Lipschitz continuous
derivative $H^{\prime \prime}  (x  )$. Finally, assume that
$H  ( x  )=1$ when $x>1$ and $H  (x  )=0$ when
$x<-1$.
\end{assumption}

\begin{assumption}
\label{A2}
Let $\zeta :=X+\bm{Z}^{\top}\bm{\beta }^{*}$. Assume the distribution density
function of $\zeta $ conditional on $\bm{Z}$, denoted by
$\rho   ( \cdot |\bm{Z}  )$, is positive and bounded uniformly
for almost every $\bm{Z}$. Furthermore, for any integer
$1\leq k\leq \alpha $, the $k$th order derivative of
$\rho   ( \cdot |\bm{Z}  )$ exists and is bounded for almost
every $\bm{Z}$, where $\alpha $ is the order of
$H^{\prime}  (x  )$ defined in Assumption~\ref{A1}.
\end{assumption}

\begin{assumption}
\label{A3}
Let $F  ( \cdot |\bm{Z}  )$ denote the conditional cumulative
distribution function of the noise $\epsilon $ in
\eqref{eq:binary-model} given $\bm{Z}$, and assume that $\epsilon $ and
$X$ are independent given $\bm{Z}$. Furthermore, assume that, for
$1\leq k\leq \alpha +1$, the $k$th order derivative of
$F  ( \cdot |\bm{Z}  )$ exists and is bounded uniformly for
almost every $\bm{Z}$. %
\end{assumption}

\begin{assumption}
\label{A4}
Define
$V:=2\mathbb{E}  [\rho   ( 0|\bm{Z}  )F^{\prime}
  ( 0 |\bm{Z}  )\bm{Z}\bm{Z}^{\top}  ]$ and
$ V_{s}:=\pi _{V} \mathbb{E}  [ \bm{Z}\bm{Z}^{\top}\rho   ( 0
|\bm{Z}  )  ]$, where $\rho $, $F$ are defined in Assumptions
\ref{A2} and \ref{A3} and $\pi _{V}$ is defined in Assumption~\ref{A1}. Assume that there exists a constant $c_{1}>1$ such that
$c_{1}^{-1}<\Lambda _{\min}  ( V  )<\Lambda _{\max}  (V
  )<c_{1}$ and
$c_{1}^{-1}<\Lambda _{\min}  ( V_{s}  )<\Lambda _{\max}
  (V_{s}  )<c_{1}$, where $\Lambda _{\min}$ ($\Lambda _{\max}$)
denotes the minimum (maximum) eigenvalue.
\end{assumption}

\begin{assumption}
\label{A5}
Assume that the covariates are sub-Gaussian, that is, there exists
$\eta >0$ such that
$\sup_{  \lVert \bm{v}  \rVert _{2}=1} \mathbb{E}e^{\eta (
\bm{v}^{\top}\bm{Z})^{2}} < +\infty $.
\end{assumption}

 Assumptions \ref{A1}--\ref{A4} are similar to the classical assumptions
in \cite{horowitz1992smoothed} for establishing the asymptotic properties
of the smoothed maximum score estimator. Assumption~\ref{A1}, which assumes
that $H^{\prime}  (x  )$ is a continuous $\alpha $-order kernel
with Lipschitz-continuous derivative, is standard for kernel-smoothing
estimation. Note that Assumption~\ref{A1} implies
$\int _{\mathbb{R}}H^{\prime}  (x  )\,{\mathrm{d}}x=1$, which is
necessary in a standard kernel definition. The restriction that
$H^{\prime}  (x  )$ is supported on $[-1, 1]$ is only a technical
simplicity. One example is given in Section~\ref{sec:simulations} with
order $\alpha =2$. Assumptions \ref{A2} and \ref{A3} are analogous to the
smoothness assumptions made in kernel density estimation. In kernel density
estimation, it is well known that the bias of the density estimator is
$O  (h^{\alpha}  )$, using an $\alpha $-order kernel and assuming
the $\alpha $-order smoothness of the underlying density function. Similarly,
the SMSE can also achieve the bias $O  (h^{\alpha}  )$ with similar
assumptions on the density function of
$  (X, \bm{Z},\epsilon   )$. Since
$y=\mathrm{sign}  (X+\bm{Z}^{\top}\bm{\beta }^{*}+\epsilon
  )$, we separately assume the smoothness of the distribution of
$\zeta =X+\bm{Z}^{\top}\bm{\beta }^{*}$ given $\bm{Z}$ and the distribution
of $\epsilon $ given $\bm{Z}$. The independence between $X$ and
$\epsilon $ given $\bm{Z}$ is assumed for simplicity of presentation, which
can be generalized by assuming a weaker condition on $F$ (see
\citealp{horowitz1992smoothed}, Assumption~9). Assumption~\ref{A4} is a
standard assumption for the asymptotic theory, which assumes the positive-definiteness
of the population Hessian matrix and the second moment of the score function.
Assumption~\ref{A5} assumes the sub-Gaussianity of $\bm{Z}$, which is key
to guarantee the uniform convergence between the empirical and population
quantities.%

\subsection{Theoretical properties of {\texttt{(Avg-SMSE)}}}
\label{sec:theory-AvgSMSE}

Let $\widehat{\bm{\beta }}_{\mathrm{{SMSE}},\mathrm{\ell}}$ denote the smoothed maximum
score estimator on the $\ell $th machine. Now we present the asymptotic
distribution of
$\widehat{\bm{\beta }}_{\mathrm{{{\mathtt{{(Avg-SMSE)}}}}}}= \frac{1}{L}\sum_{
\ell =1}^{L} \widehat{\bm{\beta }}_{\mathrm{{SMSE}},\mathrm{\ell}}$ under the assumptions
above.

\begin{thm}%
\label{thm:asym-DC}
Assume Assumptions \ref{A1}--\ref{A5} hold and
$L=o ( m^{\frac{2}{3}  ( \alpha -1  )}/   ( {p}\log m
  )^{\frac{2\alpha +1}{3}} )$. By specifying
$h = h^{*} :=   ( 1/n  )^{\frac{1}{2\alpha +1}}$, for any
$\bm{\vartheta}\in \mathbb{R}^{p} \setminus \{\bm{0}\}$, we have, as
$n \to \infty $,
%
\begin{equation}
\frac{n^{\frac{\alpha}{2\alpha +1}}\bm{\vartheta}^{\top} ( \widehat{\bm{\beta }}_{\mathrm{{{\mathtt{{(Avg-SMSE)}}}}}}-\bm{\beta }^{*} )-\bm{\vartheta}^{\top} V^{-1}U}{\sqrt{\bm{\vartheta}^{\top}V^{-1}V_{s} V^{-1}\bm{\vartheta}}} \stackrel{d} {\longrightarrow} \mathcal{N}(0, 1),
\label{eq:asym-DC} 
\end{equation}
 where $V$ and $V_{s}$ are defined in Assumption~\ref{A4}, and $U$ is
defined below,
%
\begin{equation}
\label{eq:def:U} U:=\pi _{U}\mathbb{E} \Biggl( \sum
_{k=1}^{\alpha} \frac{2(-1)^{k+1}}{k!  ( \alpha -k  )!}F^{   ( k  )}
( 0|\bm{Z} )\rho ^{   ( \alpha -k  )} ( 0|\bm{Z} )\bm{Z} \Biggr),%
\end{equation}
with constant $\pi _{U}$ defined in Assumption~\ref{A1}.
\end{thm}

Theorem~\ref{thm:asym-DC} can be interpreted using a bias-variance decomposition.
For any $\bm{\vartheta}\in \mathbb{R}^{p} $,
\begin{align*}
\bm{\vartheta}^{\top} \bigl(\widehat{\bm{\beta }}_{\mathrm{{{\mathtt{{(Avg-SMSE)}}}}}} -
\bm{\beta }^{*} \bigr)&= \bigl(\mathbb{E} \bigl[\bm{
\vartheta}^{\top} \widehat{\bm{\beta }}_{\mathrm{{{\mathtt{{(Avg-SMSE)}}}}}} \bigr] - \bm{
\vartheta}^{\top}\bm{\beta }^{*} \bigr)\\
&\quad + \bigl(\bm{
\vartheta}^{
\top}\widehat{\bm{\beta }}_{\mathrm{{{\mathtt{{(Avg-SMSE)}}}}}} - \mathbb{E}
\bigl[\bm{\vartheta}^{\top}\widehat{\bm{\beta }}_{\mathrm{{{\mathtt{{(Avg-SMSE)}}}}}} \bigr]
\bigr),
\end{align*}
the estimation error can be decomposed into two terms, where the first
term represents the bias and shows up as
$h^{\alpha }\bm{\vartheta}^{\top}V^{-1}U$ in the asymptotic bias in
\eqref{eq:asym-DC}. The second term is zero-mean and has asymptotic variance
$\frac{1}{nh}\bm{\vartheta}^{\top}V^{-1}V_{s}V^{-1}\bm{\vartheta}$ in
a sandwich form, where $V_{s}$ and $V$ are analogous to the outer product
and the Hessian matrix in the quasi-maximum likelihood estimation
(\citeauthor{white1981consequences} (\citeyear{white1981consequences,white1982maximum})), respectively. Expressions
of $U$, $V_{s}$ and $V$ result from the Taylor's expansion of the objective
\eqref{SMSE}. For detailed derivations, see the proof of Theorem~\ref{thm:asym-DC} in Section \ref{sec:pf-AvgSMSE} of Appendix. Note that
$V_{s}$ is different from $V$ defined in Assumption~\ref{A4}. The matrix $V$ includes an additional
$F^{\prime}  ( 0 |\bm{Z}  )$ term, the derivative of the conditional
c.d.f. of noise $\epsilon $ at $0$. When $\epsilon $ is homoscedastic,
the asymptotic variance in \eqref{eq:asym-DC} can be simplified as
$V^{-1}$ multiplied by a constant.

When $\frac{p\log m}{mh^{3}} =o  (1  )$, SMSE on a single machine
with bandwidth $h$ has a bias of the order
$O  (h^{\alpha}  )$ and a standard deviation of the order
$O (\sqrt{1/mh} )$. By averaging the SMSEs on all machines, the standard
deviation can be reduced to
$O (\sqrt{1/mLh} )=O (\sqrt{1/nh} )$, while the bias
$O(h^{\alpha})$ does not change. In order to obtain the optimal rate, the
bandwidth can be chosen as
$h^{*} \asymp n^{-1/{  (2\alpha +1  )}}$ to balance the bias and
variance. As a result,
$\widehat{\bm{\beta }}_{\mathrm{{{\mathtt{{(Avg-SMSE)}}}}}}$ matches the best nonparametric
convergence rate $n^{-\alpha /  (2\alpha +1  )}$, the same as
performing SMSE on the entire data set.

\begin{remark}
\label{rmk:h><h*}
If one chooses a bandwidth $h$ that is larger than $ h^{*} $, \texttt{(Avg-SMSE)}
can still provide a consistent estimator with a slower rate. As established
in the bias-variance decomposition above, a bandwidth larger than
$h^{*}$ makes the bias dominate the standard deviation, and thus the convergence
rate decreases to $O  (h^{\alpha}  )$. The result we have in this
case is
\begin{equation*}
h^{-\alpha} \bigl( \bm{\vartheta}^{\top}\widehat{\bm{\beta
}}_{\mathrm{{{
\mathtt{{(Avg-SMSE)}}}}}} - \bm{\vartheta}^{\top}\bm{\beta }^{*}
\bigr)= \bm{\vartheta}^{\top}V^{-1}U+o_{\mathbb{P}} (1
).
\end{equation*}
For example, if one specifies $h=m^{-1/  (2\alpha +1  )}$, that
is, the local optimal bandwidth, the convergence rate will be
$m^{-\alpha /  (2\alpha +1  )}$. Conversely, if one chooses a
smaller bandwidth $h\lesssim h^{*}$, the standard deviation will be the
dominant term, and the convergence rate of
$\widehat{\bm{\beta }}_{\mathrm{{{\mathtt{{(Avg-SMSE)}}}}}}$ reduces to
$(nh)^{-1/2}$.
\end{remark}

As compared to the \texttt{(Avg-MSE)} proposed in \cite{shi2018massive},
which has the convergence rate $O  (L^{-1/2}m^{-1/3}  )$ under
the constraint that $L=o  (m^{1/6}  )$ for fixed dimension
$p$, \texttt{(Avg-SMSE)} has a faster convergence rate and a weaker restriction
on the number of machines
$L= o (m^{\frac{2}{3}  ( \alpha -1  )} )$ when
$\alpha \geq 2$, assuming $p$ is fixed and ignoring the logarithm term.
This restriction comes from the condition
$\frac{\log m}{mh^{3}} =o  (1  )$ to ensure the convergence of
the empirical Hessian to the population Hessian of the smoothed objective
in \eqref{SMSE}. The constraint on the number of machines is therefore
obtained by plugging in $h^{*}$ to
$\frac{\log m}{mh^{3}} =o  (1  )$, implying that the number of
machines can not be too large compared to $m$. Moreover, when the dimension
$p$ diverges with $n$, the constraint becomes
$L=o ( m^{\frac{2\alpha -2}{3}}/ p^{\frac{2\alpha +1}{3}} )$, which
is more stringent. We will show in Section~\ref{sec:theory-mSMSE} that
the proposed multiround estimator removes this constraint.

\begin{remark}%
\label{rem:rate_shi}
\cite{shi2018massive} claimed that the averaging method improves the convergence
rate by reducing the standard deviation through averaging. This is true
for \texttt{(Avg-MSE)}, because the standard deviation of MSE on a local
machine $O (m^{-\frac{1}{3}} )$ is larger than the bias
$O (m^{-\frac{5}{12}} )$ when the dimension $p$ is fixed. Through
averaging over $L$ machines, the standard deviation
$O (m^{-\frac{1}{3}} )$ decreases to
$O (L^{-\frac{1}{2}}m^{-\frac{1}{3}} )$ and the bias remains at
$O (m^{-\frac{5}{12}} )$. Therefore, the convergence rate of \texttt{(Avg-MSE)}
is
$O (m^{-\frac{5}{12}}\vee L^{-\frac{1}{2}}m^{-\frac{1}{3}} )$. The
constraint $L=o (m^{\frac{1}{6}} )$ is indeed placed to get the bias
dominated by the standard deviation.

Compared to MSE, SMSE on each local batch should have the same order of
bias and standard deviation
$O (m^{-\frac{\alpha}{2\alpha +1}} )$, if one used a locally optimal
bandwidth $h=m^{-\frac{1}{2\alpha +1}}$. Using such a locally optimal bandwidth,
the convergence rate is not improved by averaging and remains at
$O (m^{-\frac{\alpha}{2\alpha +1}} )$. Nonetheless, in Theorem~\ref{thm:asym-DC}, we artificially specify a ``globally optimal'' bandwidth
$h^{*}$ instead of the locally optimal bandwidth. As a consequence, \texttt{(Avg-SMSE)}
with $h^{*}$ improves the convergence rate from
$O (m^{-\frac{\alpha}{2\alpha +1}} )$ to
$O (n^{-\frac{\alpha}{2\alpha +1}} )$ under a constraint on
$L$ that is weaker than the constraint of \texttt{(Avg-MSE)}.
\end{remark}

\subsection{Theoretical results for {\texttt{(mSMSE)}}}
\label{sec:theory-mSMSE}

In this section, we present the theoretical benefit of the proposed multiround
procedure \texttt{(mSMSE)}. First, we present a quantification of the performance
improvement in one iteration of \texttt{(mSMSE)} initialized at
$\widehat{\bm{\beta }}^{  (0  )}$ with bandwidth parameter
$h_{1}>0$.
%
\begin{prop}%
\label{thm:1step-ld}
Assume Assumptions \ref{A1}--\ref{A5} hold. Further, assume that
$\|\widehat{\bm{\beta }}^{  ( 0  )}-\bm{\beta }^{*}\|_{2}= O_{
\mathbb{P}}  ( \delta _{m,0}  )$,
$h_{1}=o  ( 1  )$ and $\frac{p\log n}{nh_{1}^{3}}= o  (1
  )$. We have
%
\begin{equation}
\big\| \widehat{\bm{\beta }}^{   ( 1  )}-\bm{\beta }^{*}
\big\|_{2}=O_{\mathbb{P}} \biggl(\delta
_{m,0}^{2}+h_{1}^{\alpha}+ { \sqrt{
\frac{p}{nh_{1}}}+\delta _{m,0}\sqrt{\frac{p\log n}{nh_{1}^{3}}}}
\biggr). \label{eq:1stepbeta} 
\end{equation}
\end{prop}

Proposition~\ref{thm:1step-ld} quantifies the estimation error of the one-step
estimator $\widehat{\bm{\beta }}^{  ( 1  )}$, whose proof is provided
in Section \ref{sec:pf-mSMSE} of Appendix. The right-hand side of
\eqref{eq:1stepbeta} has four components. The first term,
$\delta ^{2}_{m,0}$, comes from the property that the Newton's one-step
estimator will reduce the estimation error to the square of it in each
step. The second term, $h_{1}^{\alpha}$, as discussed before, is the order
of the bias. The third term represents the order of the standard deviation,
and the fourth term is the higher-order deviation.

Among these four terms, the dominance depends on the choice of the bandwidth
$h_{1}$. Suppose that the initialization is good, that is,
$\delta _{m,0}$ is small enough, then the estimation error of
$\widehat{\bm{\beta }}^{  (1  )}$ is dominated by
$\max  (h_{1}^{\alpha},\sqrt{p/nh_{1}} )$. Therefore,
$\widehat{\bm{\beta }}^{  (1  )}$ nearly achieves the optimal
nonparametric rate $(p/n)^{\frac{\alpha}{2\alpha +1}}$ if one chooses
$h_{1}=(p/n)^{\frac{1}{2\alpha +1}}$. On the other hand, if the initialization
is poor, the first term $\delta _{m,0}^{2}$ dominates the others. With
a properly specified $h_{1}$ such that
$\delta _{m,0}^{2}=h_{1}^{\alpha}$, the one-step estimator improves the
estimation error from $\delta _{m,0}$ to $\delta _{m,0}^{2}$.%

The previous discussions involve only one iteration. After $t$ iterations
of the multiround algorithm, the estimation error of
$\widehat{\bm{\beta }}^{  ( t  )}$ is improved geometrically from
$\delta _{m,0}$ to $\delta _{m,0}^{2^{t}}$ until it is no longer the dominating
term, and then \texttt{(mSMSE)} reaches the optimal nonparametric rate of
$(p/n)^{\frac{\alpha}{2\alpha +1}}$. To detail the convergence rate of
the multiround estimator, we have the following theorem, whose proof is
given in Section \ref{sec:pf-mSMSE} of Appendix.

\begin{thm}%
\label{thm:tstep-ld}
Assume Assumptions \ref{A1}--\ref{A5} hold, and there exists a constant
$0<c_{2}<1$ such that $p=O(m^{c_{2}})$. Further, assume the initial estimator
satisfies
$ \|\widehat{\bm{\beta }}^{   ( 0  )}-\bm{\beta }^{*}
 \|_{2}=O_{\mathbb{P}} ( (p/m)^{\frac{1}{3}} )$. By choosing
$h_{t}=\max  \lbrace   (p/n  )^{\frac{1}{2\alpha +1}}, (p/m)^{
\frac{2^{t}}{3\alpha}} \rbrace $ at iteration $t=1,2,\dots , T$, we
have
%
\begin{equation}
\big\| \widehat{\bm{\beta }}^{   ( t  )}-\bm{\beta }^{*}
\big\| _{2}=O_{\mathbb{P}} \bigl( ( p/n)^{\frac{\alpha}{2\alpha +1}}+(p/m)^{
\frac{2^{t}}{3}}+(p/m)^{\frac{2^{t-1}}{3}}
(p/n)^{
\frac{\alpha -1}{2\alpha +1}}\sqrt{\log n} \bigr). \label{eq:tstepbeta} 
\end{equation}
\end{thm}
In Theorem~\ref{thm:tstep-ld}, we assume that the initial estimator has
a mild convergence rate
$O_{\mathbb{P}} ( (p/m)^{\frac{1}{3}} )$, which can be obtained by
applying MSE \eqref{eq:MSE} or SMSE \eqref{SMSE} to a subset of the data
with size $m$. The second and third terms in \eqref{eq:tstepbeta} indicate
that the initial condition is forgotten double-exponentially fast as the
iterations proceed, and it is easy to see that as long as
%
\begin{equation}
\label{eq:iteration_number} T \geq \log _{2} \biggl( \frac{3\alpha}{2\alpha +1}
\cdot \frac{\log (n/p)}{\log (m/p)} \biggr), 
\end{equation}
the first term in \eqref{eq:tstepbeta},
${(p/n)}^{\frac{\alpha}{2\alpha +1}}$, dominates the other two terms. Therefore,
\texttt{(mSMSE)} achieves the optimal rate
$ O ( (p/n)^{\frac{\alpha}{2\alpha +1}} )$ in
$O  (\log \log n  )$ steps. For example, when $n=10^{9}$,
$m=1000$, $p=10$ and $\alpha =2$, \texttt{(mSMSE)} only takes three steps
to converge.

\begin{remark}[Superefficiency phenomenon]
\label{rem4}
\cite{banerjee2019divide} and \cite{shi2018massive} showed that in many
cube-root problems, as compared to the estimator on the entire data set,
although the divide-and-conquer averaging estimator achieves a better convergence
rate under a fixed model, the maximal risk over a class of models in a
neighborhood of the given model diverges to infinity. This is referred
to as the superefficiency phenomenon, which often appears in nonparametric
function approximation \citep{brown1997superefficiency} and indicates a
trade-off between performance under a fixed model and performance in a
uniform sense. In contrast to the divide-and-conquer estimator \texttt{(Avg-MSE)},
the proposed multiround estimator \texttt{(mSMSE)} closely approximates the
global estimate (the estimator on the entire data set), which suggests
that \texttt{(mSMSE)} may not entail the superefficiency phenomenon. We show
in Section \ref{sec:sup-eff} of Appendix that, over a large class of
models $\Theta $, it holds that: $\forall \varepsilon >0$,
$\exists M_{\varepsilon}$, $N_{\varepsilon}$, such that
$\forall n \geq N_{\varepsilon}$,
%
\begin{equation*}
\sup_{\Theta}\mathbb{P} \bigl( \big\| \widehat{\bm{\beta
}}^{   (
T  )}-\bm{\beta }^{*} \big\| _{2} >
M_{\varepsilon}({p}/n)^{
\frac{\alpha}{2\alpha +1}} \bigr) < \varepsilon ,
\label{eq:sup-eff}
\end{equation*}
where $T$ satisfies \eqref{eq:iteration_number}.
\end{remark}

We further derive the asymptotic distribution for
$\widehat{\bm{\beta }}^{  ( T+1  )}$ and give the optimal choice
for the bandwidth $h$ based on the asymptotic mean squared error.\begin{thm}%
\label{thm:normality}
Assume the assumptions in Theorem~\ref{thm:tstep-ld} hold. Further, assume
that the local size $m > n^{c_{3}}$ for some constant $0<c_{3}<1$ and
$p=o (n^{\frac{2(\alpha -1)}{4\alpha +1}}(\log n)^{-
\frac{2\alpha +1}{4\alpha +1}} )$. When $T$ satisfies
\eqref{eq:iteration_number}, let
$h_{T+1}=(\lambda _{h}/n)^{\frac{1}{2\alpha +1}}$ for some constant
$\lambda _{h}>0$, and then for any
$\bm{\vartheta}\in \mathbb{R}^{p} \setminus \{\bm{0}\}$, we have
%
\begin{equation}
\frac{n^{\frac{\alpha}{2\alpha +1}}\bm{\vartheta}^{\top} ( \widehat{\bm{\beta }}^{  (T+1  )}-\bm{\beta }^{*} )-\lambda _{h}^{\frac{\alpha}{2\alpha +1}}\bm{\vartheta}^{\top} V^{-1}U}{\lambda _{h}^{\frac{-1}{2(2\alpha +1)}}\sqrt{\bm{\vartheta}^{\top}V^{-1}V_{s} V^{-1}\bm{\vartheta}}} \stackrel{d} {\longrightarrow} \mathcal{N}(0 , 1),
\label{eq:asym-mSMSE} 
\end{equation}
 where $V$ and $V_{s}$ are defined in Assumption~\ref{A4} and $U$ is defined
in \eqref{eq:def:U}.
\end{thm}
The proof of Theorem~\ref{thm:normality} is given in Section \ref{sec:pf-inference} of Appendix, which also provides a characterization of the asymptotic
mean squared error of
$\bm{\vartheta}^{\top}\widehat{\bm{\beta }}^{(T+1)}$. The optimal
$\lambda _{h}^{*}$ can be chosen by minimizing the asymptotic mean squared
error, as
%
\begin{equation}
\lambda _{h}^{*}:= \frac{\bm{\vartheta}^{\top} V^{-1}V_{s}V^{-1}\bm{\vartheta}}{2\alpha \bm{\vartheta}^{\top}V^{-1}UU^{\top}V^{-1}\bm{\vartheta}}.
\label{eq:lambda_opt} 
\end{equation}
We also note that the assumption $m > n^{c_{3}}$ for some
$0<c_{3}<1$ in Theorem~\ref{thm:normality} guarantees that \texttt{(mSMSE)}
converges in a finite number of iterations, that is, the lower bound in
\eqref{eq:iteration_number} is finite.

\subsection{Bias correction and statistical inference}
\label{sec:Inference}

In this section, we construct a confidence interval for
$\bm{\beta }^{*}$ using the proposed estimator
$\widehat{\bm{\beta }}^{(T+1)}$. Given a vector
$\bm{\vartheta}_{0} \in \mathbb{R}^{p} \setminus \{0\}$ and a prespecified
level $1-\xi $, Theorem~\ref{thm:normality} provides the following confidence
interval for $\bm{\vartheta}^{\top}_{0}\bm{\beta }^{*}$:
\begin{equation*}
\bm{\vartheta}^{\top}_{0} \widehat{\bm{\beta
}}^{  (T+1  )}-n^{-
\frac{\alpha}{2\alpha +1}}\lambda _{h}^{\frac{\alpha}{2\alpha +1}}
\bm{\vartheta}^{\top}_{0} V^{-1} U \pm \tau
_{1-\frac{\xi}{2}} \sqrt{n^{-
\frac{2\alpha}{2\alpha +1}}\lambda _{h}^{-\frac{1}{2\alpha +1}}
\bigl( \bm{\vartheta}^{\top}_{0} V^{-1}
V_{s} V^{-1} \bm{\vartheta}_{0} \bigr)},
\end{equation*}
where $\tau _{1-\frac{\xi}{2}}$ denotes the
$ (1-\frac{\xi}{2} )$-th quantile of a standard normal distribution.
In contrast to common confidence intervals for parametric models, it requires
a bias-correction term (i.e., the term
$-n^{-\frac{\alpha}{2\alpha +1}}\lambda _{h}^{
\frac{\alpha}{2\alpha +1}} \bm{\vartheta}^{\top}_{0} V^{-1} U$) due to
the bias of the nonparametric method.

Practically, one needs to estimate the unknown matrices $V$, $U$ and
$V_{s}$. Their empirical estimators $\widehat{V}$, $\widehat{U}$ and
$\widehat{V}_{s}$ can be naturally constructed in a distributed environment.
For example, $V$ can be estimated by aggregating $V_{m,\ell}^{(T)}$ in
\eqref{eq:def:Umlt} during the final iteration. The other two can be estimated
similarly, and we relegate the details to Section \ref{sec:pf-inference} of Appendix. With $\widehat{V}$, $\widehat{U}$ and $\widehat{V}_{s}$ on hand,
we conclude the following corollary for inference in practice.

\begin{cor}
\label{cor:inference}
Given a vector $\bm{\vartheta}_{0} \in \mathbb{R}^{p}$ and a prespecified
level $1-\xi $, we have
\begin{equation*}
\mathbb{P} \bigl\{\bm{\vartheta}_{0}^{\top}\bm{\beta
}^{*}\in \bigl[ \bm{\vartheta}_{0}^{\top}
\widehat{\bm{\beta }}^{(T+1)}+ \widehat{\mathrm{bias}}-\tau
_{1-\frac{\xi}{2}}\widehat{\mathrm{se}}, \bm{\vartheta}_{0}^{\top}
\widehat{\bm{\beta }}^{(T+1)}+ \widehat{\mathrm{bias}}+\tau
_{1-\frac{\xi}{2}}\widehat{\mathrm{se}} \bigr] \bigr\} \to 1-\xi ,
\end{equation*}
where
\begin{equation*}
\widehat{\mathrm{bias}}=-n^{-\frac{\alpha}{2\alpha +1}}\lambda _{h}^{
\frac{\alpha}{2\alpha +1}}
\bm{\vartheta}_{0}^{\top} \widehat V^{-1}
\widehat U,\qquad \widehat{\mathrm{se}}=\sqrt{n^{-\frac{2\alpha}{2\alpha +1}} \lambda
_{h}^{-\frac{1}{2\alpha +1}} \bigl( \bm{\vartheta}_{0}^{\top}
\widehat V^{-1} \widehat V_{s} \widehat
V^{-1} \bm{\vartheta}_{0} \bigr)},
\end{equation*}
and $\tau _{1-\frac{\xi}{2}}$ denotes the
$ (1-\frac{\xi}{2} )$-th quantile of a standard normal distribution.
\end{cor}

\section{Data heterogeneity}
\label{sec:4}

Until now, we assumed homogeneity across the data stored on different machines,
which means the distribution of
$  (x_{i}, \bm{z}_{i}, \epsilon _{i}  )$ are the same for all
$i$. It is of practical interest to consider the heterogeneous setting,
since data on different machines may not be identically distributed. Therefore,
we establish the theoretical results in the presence of heterogeneity in
this section. We first consider a covariate shift setting in Section~\ref{sec:hetero}, where we assume the marginal distributions of the covariates
$X$, $\bm{Z}$ are different among the machines, but the regression coefficients
$\bm{\beta }^{*}$ remain unchanged. In Section~\ref{sec:4.2}, we further
extend \texttt{(mSMSE)} to a coefficient shift setting that allows
$\bm{\beta }^{*}$ to be different across different machines.

\subsection{Covariate shift}
\label{sec:hetero}

In this section, we first remove the restriction that the sample size on
each machine is the same. Denote the number of observations on the machine
$\ell$ to be $m_{\ell}$, which satisfies
$\sum_{\ell =1}^{L}m_{\ell}=n$. Then we have to modify Assumptions
\ref{A2}--\ref{A4} for different distributions on different machines.

\begin{assumption}
\label{A2'}
For $X$ and $\bm{Z}$ in $\mathcal H_{\ell}$, define
$\zeta :=X+\bm{Z}^{\top}\bm{\beta }^{*}$, and assume that the conditional
distribution density function of $\zeta $, denoted by
$\rho _{\ell}  ( \cdot |\bm{Z}  )$, is positive and uniformly
bounded for almost every $\bm{Z}$. Further, for any integer
$1\leq k\leq \alpha $, assume that
$\rho _{\ell}^{   ( k  )}  ( \cdot |\bm{Z}  )$ exists
and is uniformly bounded for all $\ell $ and almost every $\bm{Z}$, that
is, $\exists M_{\rho , k}>0$ such that
$\sup_{\zeta , \ell}   \big\lvert \rho _{\ell }^{   ( k  )}
  ( \zeta |\bm{Z}  )  \big\rvert \leq M_{\rho , k}$.
\end{assumption}

\begin{assumption}
\label{A3'}
For $X$, $\bm{Z}$, $\epsilon $ in $\mathcal H_{\ell}$, let
$F_{\ell}  ( \cdot |\bm{Z}  )$ denote the conditional cumulative
distribution function of the noise $\epsilon $ given $\bm{Z}$, and assume
that $\epsilon $ and $X$ are independent given $\bm{Z}$. For any integer
$1\leq k\leq \alpha +1$, assume that
$F_{\ell}^{  ( k  )}  ( \cdot |\bm{Z}  )$ exists
and is uniformly bounded for all $\ell $ and almost every $\bm{Z}$, that
is, $\exists M_{F, k}>0$ such that
$\sup_{\epsilon , \ell}  \big\lvert F_{\ell }^{  ( k  )}
  ( \epsilon |\bm{Z}  )  \big\rvert \leq M_{F, k}$. Still,
we assume $\mathrm{median}  (\epsilon |\bm{Z}  )=0$ on each
machine.
\end{assumption}
%
\begin{assumption}
\label{A4'}
Assume that there exists a constant $c_{1}>1$ such that
$c_{1}^{-1}<\Lambda _{\min}  (V_{\ell}  )<\Lambda _{\max}
  (V_{\ell}  )<c_{1}$,
$c_{1}^{-1}<\Lambda _{\min}  (V_{s,\ell}  )<\Lambda _{\max}
  (V_{s,\ell}  )<c_{1}$, $\forall \ell $, where
$$ V_{\ell}:=2\mathbb{E}_{\bm{Z}\in \mathcal{H}_{\ell}} ( \rho _{
\ell}  ( 0|\bm{Z}  )\*F_{\ell}^{\prime}  ( 0 |
\bm{Z}  )\bm{Z}\bm{Z}^{\top} ),\quad V_{s,\ell}:=\pi _{V}\mathbb{E}_{\bm{Z}\in \mathcal{H}_{\ell}} (
\rho _{\ell}  ( 0|\bm{Z}  )\bm{Z}\bm{Z}^{\top} ),$$with
$\pi _{V}$ defined in Assumption~\ref{A1}.
\end{assumption}

Assumptions \ref{A2'}--\ref{A4'} are parallel to Assumptions \ref{A2}--\ref{A4},
requiring the uniform boundedness of the higher-order derivatives and the
eigenvalues of $V_{\ell}$ and $V_{s,\ell}$. As an example, suppose that
$X_{\ell} \sim \mathcal{N}(\mu _{X, \ell}, \sigma _{X,\ell}^{2})$,
$Z_{\ell} \sim (\mu _{Z,\ell}, \sigma _{Z, \ell}^{2} I_{p \times p})$ on
machine $\ell $, and $(X_{\ell}, Z_{\ell})$ are independent, and for simplicity,
suppose the noise is homogenous among the machines. Then Assumptions
\ref{A2'}--\ref{A4'} are satisfied if there exist positive constants
$\overline{\mu}$, $\underline{\sigma}$, $\overline{\sigma}$ such that
$  \lVert \mu _{X, \ell }  \rVert _{2} \leq \overline{\mu}$,
$   \lVert \mu _{Z,\ell }  \rVert _{2} \leq \overline{\mu}$ and
$\underline{\sigma} \leq \sigma _{X,\ell}$, $\sigma _{Z,\ell} \leq
\overline{\sigma}$, $\forall \ell $, that is, the mean of the covariates
is upper bounded in $\ell _{2}$ norm uniformly for all $\ell $, and the
variance of the covariates is bounded away from zero and infinity uniformly
for all $\ell $.

Similar to \eqref{eq:def:U}, we define
\begin{equation*}
U_{\ell}:=\pi _{U}\mathbb{E}_{\bm{Z}\in \mathcal{H}_{\ell}} \Biggl(
\sum_{k=1}^{\alpha}\frac{2(-1)^{k+1}}{k!  ( \alpha -k  )!}F_{
\ell}^{  ( k  )}
( 0|\bm{Z} )\rho _{\ell}^{
  ( \alpha -k  )} ( 0|\bm{Z} )\bm{Z} \Biggr),
\end{equation*}
which is related to the bias of SMSE on each machine.

Under the modified assumptions, the data on each machine are no longer
identically distributed and, therefore, it is natural to allocate a different
weight matrix $W_{\ell}$ to each machine, with
$\sum_{\ell =1}^{L} W_{\ell}=I_{p\times p}$. Formally, the
\emph{weighted-Averaged SMSE} \texttt{(wAvg-SMSE)} is defined as follows:
%
\begin{equation}
\label{eq:weighted-AvgSMSE} \widehat{\bm{\beta }}_{\mathrm{{{\mathtt{{(wAvg-SMSE)}}}}}} := \sum
_{\ell =1}^{L} W_{\ell} \widehat{\bm{
\beta }}_{\mathrm{{SMSE}},\mathrm{\ell}}, 
\end{equation}
where $\widehat{\bm{\beta }}_{\mathrm{{SMSE}},\mathrm{\ell}}$ is the SMSE on the
$\ell $th machine that minimizes the objective function
%
\begin{equation}
\label{eq:def:Fhl} F_{h,\ell} (\bm{\beta } ) := \frac{1}{m_{\ell}}
\sum_{i
\in \mathcal{H}_{\ell}} ( -y_{i} ) H \biggl(
\frac{x_{i}+\bm{z}_{i}^{\top}\bm{\beta }}{h} \biggr). 
\end{equation}
For the multiround method, we aim to apply the iterative smoothing to minimize
a weighted sum of the objective functions on each machine, that is,
$\sum_{\ell =1}^{L} W_{\ell }F_{h,\ell}  (\bm{\beta }  )$, which
leads to updating the \emph{weighted} mSMSE \texttt{(wmSMSE)} in the
$t$th iteration by
%
\begin{equation}
\widehat{\bm{\beta }}_{\mathrm{{{\mathtt{{(wmSMSE)}}}}}}^{   ( t  )}= \widehat{\bm{\beta
}}_{\mathrm{{{\mathtt{{(wmSMSE)}}}}}}^{   ( t-1  )} - \Biggl(\sum^{L}_{\ell =1}W_{\ell}
\nabla ^{2}F_{h,\ell} \bigl( \widehat{\bm{\beta
}}_{\mathrm{{{\mathtt{{(wmSMSE)}}}}}}^{   ( t-1  )} \bigr) \Biggr)^{-1} \Biggl(\sum
^{L}_{\ell =1}W_{\ell}\nabla
F_{h,\ell} \bigl( \widehat{\bm{\beta }}_{\mathrm{{{\mathtt{{(wmSMSE)}}}}}}^{   ( t-1  )}
\bigr) \Biggr). \label{eq:weighted-mSMSE} 
\end{equation}

A natural choice of weights is proportional to the local sample size, that
is, $W_{\ell}=\frac{m_{\ell}}{n}I_{p\times p}$. Using this weight, the
variances in the asymptotic distribution in \eqref{eq:asym-DC} and
\eqref{eq:asym-mSMSE} become
$\frac{1}{n}\sum_{\ell =1}^{L}m_{\ell}V_{\ell}^{-1}V_{s,\ell}V_{\ell}^{-1}$
and
$ (\frac{1}{n}\sum_{\ell =1}^{L}m_{\ell}V_{\ell} )^{-1} (
\frac{1}{n}\sum_{\ell =1}^{L}m_{\ell}V_{s,\ell} ) (\frac{1}{n}
\sum_{\ell =1}^{L}m_{\ell}V_{\ell} )^{-1}$, respectively, which can
be seen as a special case of Theorems \ref{thm:asym-DC-weighted} and
\ref{thm:asym-mSMSE-weighted} below. Nonetheless, such a choice is by no
means optimal. We can further decrease both asymptotic variances by choosing
a different weight matrix $W_{\ell}$ for each machine. To illustrate the
choices, we first derive theoretical results for general weight matrices
$W_{\ell}$ that satisfy the following restriction.

\begin{assumption}
\label{assumption:weightrestriction}
There exist constants $c_{w}, C_{W}>0$ such that
$c_{W}m_{\ell}/n\leq   \lVert W_{\ell }  \rVert _{2} \leq C_{W}m_{
\ell}/n$, with $n=\sum_{\ell =1}^{L}m_{\ell}$ and
$\sum_{\ell =1}^{L} W_{\ell}=I_{p\times p}$.
\end{assumption}

Assumption~\ref{assumption:weightrestriction} requires that the 2-norm
of $W_{\ell}$ is not too far away from $m_{\ell}/n$, violation of which
may lead to a low convergence rate. An extreme counterexample is
$W_{1}=I_{p \times p}$ and $W_{2}=\cdots=W_{L}=\bm{0}$, in which case only
the data on a single machine will be used. Following the procedures in
Sections~\ref{sec:theory-AvgSMSE} and \ref{sec:theory-mSMSE}, we establish
the asymptotic normality for \texttt{(wAvg-SMSE)} and \texttt{(wmSMSE)} in
Theorems \ref{thm:asym-DC-weighted} and \ref{thm:asym-mSMSE-weighted},
whose proof is given in Section \ref{sec:pf-weighted} of Appendix.

\begin{thm}[wAvg-SMSE]
\label{thm:asym-DC-weighted}
Suppose Assumptions \ref{A1} and \ref{A5}--\ref{assumption:weightrestriction}
hold and the sample size of the smallest local batch
$\min_{\ell} m_{\ell} \gtrsim pn^{3/(2\alpha +1)}\log n$. By taking
$h = n^{-1/(2\alpha +1)}$, we have that, for any
$\bm{\vartheta} \in \mathbb{R}^{p}\setminus \{\bm{0}\}$,
%
\begin{equation}
\label{eq:asym-DC-weighted} \frac{n^{\frac{\alpha}{2\alpha +1}}\bm{\vartheta}^{\top} (\widehat{\bm{\beta }}_{\mathrm{{{\mathtt{{(wAvg-SMSE)}}}}}}-\bm{\beta }^{*} )-\sum_{\ell =1}^{L}\bm{\vartheta}^{\top}W_{\ell}V_{\ell}^{-1}U_{\ell}}{\sqrt{\sum_{\ell =1}^{L}\frac{n}{m_{\ell}}\bm{\vartheta}^{\top}W_{\ell }V_{\ell}^{-1}V_{s,\ell}V_{\ell}^{-1}W_{\ell}^{\top}\bm{\vartheta}}} \stackrel{d} {\longrightarrow}
\mathcal{N}(0 , 1). 
\end{equation}
\end{thm}
%
\begin{thm}[wmSMSE]
\label{thm:asym-mSMSE-weighted}
Suppose Assumptions \ref{A1} and \ref{A5}--\ref{assumption:weightrestriction}
hold and
$ \|\widehat{\bm{\beta }}^{   ( 0  )}-\bm{\beta }^{*}
 \|_{2}=O_{\mathbb{P}} ( \delta _{0} )$. Further, assume that
the local size $m > n^{c_{3}}$ for some constant $0<c_{3}<1$ and $p=o
 (n^{\frac{2(\alpha -1)}{4\alpha +1}}(\log n)^{-
\frac{2\alpha +1}{4\alpha +1}} )$. By taking
$h_{t} = \max  \{(p/n)^{\frac{1}{2\alpha +1}},\delta _{0}^{2^{t}/
\alpha} \}$ at iteration $t=1,2,\dots , T$ and
$h_{T+1}=n^{-1/(2\alpha +1)}$, we have
%
\begin{equation}
\label{eq:asym-mSMSE-weighted} \frac{n^{\frac{\alpha}{2\alpha +1}}\bm{\vartheta}^{\top} ( \widehat{\bm{\beta }}_{\mathrm{{{\mathtt{{(wmSMSE)}}}}}}^{   ( T+1  )}-\bm{\beta }^{*} )-\bm{\vartheta}^{\top}\overline{V}_{W}^{-1}\overline{U}_{W}}{\sqrt{\bm{\vartheta}^{\top}\overline{V}_{W}^{-1}\sum_{\ell =1}^{L} \frac{n}{m_{\ell}}W_{\ell}V_{s,\ell}W_{\ell}^{\top}
\overline{V}_{W}^{-\top}\bm{\vartheta}}}\stackrel{d} {\longrightarrow}
\mathcal{N}(0, 1), 
\end{equation}
for any $\bm{\vartheta} \in \mathbb{R}^{p} \setminus \{\bm{0}\}$ and sufficiently
large $T$,  where
$\overline{U}_{W}:=\sum_{\ell =1}^{L} W_{\ell}U_{\ell}$,
$\overline{V}_{W}:=\sum_{\ell =1}^{L} W_{\ell}V_{\ell}$.
\end{thm}
Assumptions \ref{A4'} and \ref{assumption:weightrestriction} ensure the
variances in both \eqref{eq:asym-DC-weighted} and
\eqref{eq:asym-mSMSE-weighted} are finite. In particular, in the homogeneous
setting, $W_{\ell}=(1/L)I_{p \times p}$, then
\eqref{eq:asym-DC-weighted} and \eqref{eq:asym-mSMSE-weighted} are identical
to \eqref{eq:asym-DC} and \eqref{eq:asym-mSMSE}.

\begin{remark}
\label{rem5}
In Theorem~\ref{thm:asym-DC-weighted}, the condition
$\min_{\ell} m_{\ell} \gtrsim pn^{3/(2\alpha +1)}\log n$ is placed to
ensure $p\log m_{\ell}/  (m_{\ell}h^{3}  )=o(1)$ for all
$\ell $, which is necessary to guarantee the convergence of
$\widehat{\bm{\beta }}_{\mathrm{{SMSE}},\mathrm{ \ell}}$ to $\bm{\beta }^{*}$ on each
machine. In the homogeneous setting, this is equivalent to the constraint
$L=o ( m^{\frac{2}{3}  ( \alpha -1  )}/ {p}^{
\frac{2\alpha +1}{3}} )$ in Theorem~\ref{thm:asym-DC} (ignoring the
logarithm term). This condition requires the sample size of the smallest
batch should increase at a certain rate as $n\rightarrow \infty $. On the
other hand, for \textup{\texttt{(wmSMSE)}}, there is no restriction on the smallest
local sample size.
\end{remark}

Based on the above results, we are able to artificially choose the weight
matrices $\{W_{\ell}\}$ to minimize the variances of the two methods in
\eqref{eq:asym-DC-weighted} and \eqref{eq:asym-mSMSE-weighted}. For any
$\bm{\vartheta}$, the variance in \textup{\eqref{eq:asym-DC-weighted}} is minimized
by choosing
$ W^{*,\mathrm{{{\mathtt{{(wAvg-SMSE)}}}} }}_{\ell}= (\sum^{L}_{\ell =1}m_{
\ell}V_{\ell}V^{-1}_{s,\ell}V_{\ell} )^{-1} m_{\ell}V_{\ell}V^{-1}_{s,
\ell}V_{\ell} $, and the corresponding minimum variance is
%
\begin{equation}
\label{eq:var_opt_avgsmse} \Sigma ^{*}_{\mathrm{{{\mathtt{{(wAvg-SMSE)}}}}}}:=n\bm{
\vartheta}^{\top} \Biggl( \sum^{L}_{\ell =1}m_{\ell}V_{\ell}V^{-1}_{s,\ell}V_{\ell}
\Biggr)^{-1} \bm{\vartheta}. 
\end{equation}
 For \texttt{(wmSMSE)}, if one chooses
$ W_{\ell}^{*, \mathrm{{{\mathtt{{(wmSMSE)}}}}}}= (\sum_{\ell =1}^{L}m_{\ell}V_{
\ell}V^{-1}_{s,\ell} )^{-1}m_{\ell}V_{\ell}V^{-1}_{s,\ell} $, the asymptotic
variance in \eqref{eq:asym-mSMSE-weighted} will be the same as
$\Sigma ^{*}_{\mathrm{{{\mathtt{{(wAvg-SMSE)}}}}}}$ in \eqref{eq:var_opt_avgsmse}.
Therefore, the multiround method \texttt{(wmSMSE)} is at least as efficient
as \texttt{(wAvg-SMSE)} by specifying certain weight matrices. Note that
it is easy to verify the above optimal weights
$W^{*,\mathrm{{{\mathtt{{(wAvg-SMSE)}}}} }}_{\ell}$ and
$W_{\ell}^{*, \mathrm{{{\mathtt{{(wmSMSE)}}}}}}$ satisfy Assumption~\ref{assumption:weightrestriction}. The detailed derivation is given in
Section \ref{sec:pf-weighted} of Appendix.

\subsection{Coefficient shift}
\label{sec:4.2}

 In this section, we consider another type of data heterogeneity by allowing
the regression coefficient $\bm{\beta }^{*}$ to be different on different
machines, referred to as \emph{coefficient shift}. This setting is also
known as \emph{conditional shift} in some literature, since the distribution
of the response $Y$ conditional on the covariates $(X, \bm{Z})$ depends
on $\bm{\beta }^{*}$. Formally, we assume that
$y_{i}=\mathrm{sign}(x_{i}+\bm{z}_{i}^{\top}\bm{\beta }^{*}_{\ell}+
\epsilon _{i})$ for $i \in \mathcal{H}_{\ell}$,
$\ell =1,2,\dots , L$. Without loss of generality, our goal is to estimate
the parameter $\bm{\beta }^{*}_{1}$ on the first machine, and we assume
that there exists a nonempty set
$\mathcal{A^{*}}:= \{\ell :\bm{\beta }^{*}_{1} = \bm{\beta }^{*}_{
\ell},  \ell = 1, \dots , L \}$ with cardinality
$(1-\varepsilon ) L$, where $0< \varepsilon <1$. In other words, there
are $(1-\varepsilon ) L$ machines on which the coefficients
$\bm{\beta }^{*}_{\ell}$ are the same as that on the first machine, while
the coefficients on the remaining $\varepsilon L$ machines are shifted
away.

\begin{algorithm}[!tp]
        \caption{Multiround Maximum Score Estimator for Coefficient Shift}
        \vspace*{0.08in} {\textbf{Input:}}
        Data sets distributed on local machines $\{x_{i}, \bm{z}_{i},y_{i}\}_{i\in \mathcal{H}_{\ell}}$ $  ( \ell =1,2,\dots ,L)$, the total number of iterations $T$, bandwidth sequence $\{h_{t}\}_{t=0}^{T}$ and preselected constants $\omega $ and $C_{0}$.
        \label{alg:cond-shift}
\begin{algorithmic}[1]
                \For{$\ell =1,2,\dots ,L$} \do \\
                \State Randomly select a subset $\mathcal{H}_{\ell}^{(0)} \subset \mathcal{H}_{\ell}$, with cardinality $  \lvert \mathcal{H}_{\ell }^{(0)}  \rvert  = \lfloor \omega m\rfloor $;
                \State Compute the SMSE $\widehat{\bm{\beta }}_{\ell , \mathrm{SMSE}}^{(0)}$ using the subset $\mathcal{H}_{\ell}^{(0)}$ with bandwidth $h_{0}$, and send $\widehat{\bm{\beta }}_{\ell , \mathrm{SMSE}}^{(0)}$ back to  the first machine;
                \EndFor
                \State Compute $\mathcal{A}=\{\ell :  \big\lVert \widehat{\bm{\beta }}_{\ell ,\mathrm{SMSE}}^{(0)}-\widehat{\bm{\beta }}_{1,\mathrm{SMSE}}^{(0)}  \big\rVert _{2}\leq C_{0}  (\frac{p\log L}{\omega m}  )^{\frac{\alpha}{2\alpha +1}}  \}$, and let $\widehat{\bm{\beta }}_{1}^{(0)}=\widehat{\bm{\beta }}_{1,\mathrm{SMSE}}^{(0)}$;
                \For{$t=1,2,\dots ,T$} \do \\
                \For {$\ell \in \mathcal{A}$} \do\\
                \State Send $\widehat{\bm{\beta }}_{1}^{(t-1)}$ to machine $\ell$;
                \State Compute
\begin{align*}
&\widetilde{U}^{  ( t  )}_{m,\ell}=\frac{1}{(1-\omega )mh_{t}}\sum
_{i\in \mathcal{H}_{\ell}\setminus \mathcal{H}_{\ell}^{(0)}} (-y_{i} ) H^{\prime}
\biggl(\frac{x_{i}+\bm{z}_{i}^{\top}\widehat{\bm{\beta }}_{1}^{(t-1)}}{h_{t}} \biggr) \bm{z}_{i},
\\
&\widetilde{V}^{  ( t  )}_{m,\ell}=\frac{1}{(1-\omega )mh_{t}^{2}}\sum
_{i \in \mathcal{H}_{\ell}\setminus \mathcal{H}_{\ell}^{(0)}} (-y_{i} )H''
\biggl(\frac{x_{i}+\bm{z}_{i}^{\top}\widehat{\bm{\beta }}_{1}^{(t-1)}}{h_{t}} \biggr)\bm{z}_{i}\bm{z}_{i}^{\top};
\end{align*}
\State Send $\widetilde{U}_{m, \ell}^{(t)}$, $\widetilde{V}^{  ( t  )}_{m,\ell}$ back to the first machine;
                \EndFor
                \State Compute $\widetilde{U}^{  ( t  )}_{n}\hspace{-0.02in}=\hspace{-0.02in}\frac{1}{  \lvert \mathcal{A}  \rvert }\sum_{\ell \in \mathcal{A}}\widetilde{U}^{  ( t  )}_{m,\ell}~$ and
                $\widetilde{V}^{  ( t  )}_{n}\hspace{-0.02in}=\hspace{-0.02in}\frac{1}{  \lvert \mathcal{A}  \rvert }\sum_{\ell \in \mathcal{A}}\widetilde{V}^{  ( t  )}_{m,\ell}~$;
                \State Update $\widehat{\bm{\beta }}_{1}^{   ( t  )}=\widehat{\bm{\beta }}_{1}^{   ( t-1  )}-  (\widetilde{V}^{  ( t  )}_{n}  )^{-1}\widetilde{U}^{  (t  )}_{n}$;
                \EndFor
                \State Output $\widehat{\bm{\beta }}_{1}^{  ( T  )}$.
        \end{algorithmic}
\end{algorithm}

To clearly demonstrate the strategy we use to deal with the coefficient
shift setting, we assume there is no covariate shift in this section and
the local sample size $m$ is identical for all machines, the violation
of which is analyzed in the previous Section~\ref{sec:hetero}. We propose
Algorithm~\ref{alg:cond-shift}, which applies \texttt{(mSMSE)} after selecting
an estimate set $\mathcal{A}$ for $\mathcal{A}^{*}$ based on local SMSE
estimators computed on subsets $\{\mathcal{H}_{\ell}^{(0)}\}$, where
$\mathcal{H}_{\ell}^{(0)}$ contains a constant proportion $\omega $ of
the local data. Now we provide the convergence rates for the estimators
$ \{\widehat{\bm{\beta }}^{(t)} \}$ in Algorithm~\ref{alg:cond-shift}.
%
\begin{thm}%
\label{thm:tstep-cond-shift}
Assume the assumptions in Theorem~\ref{thm:tstep-ld} hold. Further, assume
that $\varepsilon \leq \varepsilon _{0}$ for some constant
$\varepsilon _{0}<1$. By choosing
$h_{0}= (\frac{p\log L}{m} )^{\frac{1}{2\alpha +1}}$ and
$h_{t}=\max   \{\delta ^{2^{t}/\alpha}_{m ,0},  (\frac{p}{n}
 )^{\frac{1}{2\alpha +1}}  \}$ for $t=1,2,\dots ,T$, we have that
%
\begin{equation}
\label{eq:tstepbeta-hetero} \bigl\lVert \widehat{\bm{\beta }}_{1}^{(t)}-
\bm{\beta }_{1}^{*} \bigr\rVert _{2}=O_{\mathbb{P}}
\biggl( \biggl( \frac{p}{(1-\varepsilon )n} \biggr)^{ \frac{\alpha}{2\alpha +1}}+\delta
_{m
,0}^{2^{t}}+\delta _{m ,0}^{2^{t-1}}
\biggl(\frac{p}{(1-\varepsilon )n} \biggr)^{\frac{\alpha -1}{2\alpha +1}}\sqrt{\log n}+\varepsilon
\delta _{m
,0}^{2} \biggr), 
\end{equation}
where
$\delta _{m ,0}=(\frac{p\log L}{\omega m})^{\frac{\alpha}{2\alpha +1}}$.
\end{thm}
The proof for Theorem~\ref{thm:tstep-cond-shift} is provided in Section \ref{sec:pf-weighted} of Appendix. Analogous to Theorem~\ref{thm:tstep-ld}, the second and third terms in
\eqref{eq:tstepbeta-hetero} decrease double-exponentially as the iterations
proceed and will be finally dominated by the statistical rate, that is,
the first term
$ (\frac{p}{(1-\varepsilon )n} )^{\frac{\alpha}{2\alpha +1}}$, in
$O(\log \log n)$ iterations. In addition, as compared to homogeneous settings,
there remains a bias term $\varepsilon \delta _{m ,0}^{2}$ due to the difference
among the coefficients on $\mathcal{A}\setminus \mathcal{A^{*}}$.
%
\begin{remark}
\label{rmk:coef-shift}
A one-shot averaging divide-and-conquer algorithm (such as \texttt{(Avg-MSE)}
and \texttt{(Avg-SMSE)}) can be designed similar to deal with the coefficient
shift setting. Concretely, we can first select the set $\mathcal{A}$ as
in Algorithm \textup{\ref{alg:cond-shift}} and then apply \texttt{(Avg-SMSE)} on
the machines in $\mathcal{A}$. In this case, the averaging estimator has
an error rate of the order
$O  ( (\frac{p}{(1-\varepsilon )n} )^{
\frac{\alpha}{2\alpha +1}}+\varepsilon \delta _{m, 0}   )$, under
the constraint
$L=o ( m^{\frac{2}{3}  ( \alpha -1  )}/   ( {p}\log m
  )^{\frac{2\alpha +1}{3}} )$. The proposed Algorithm~\ref{alg:cond-shift} not only removes the constraint on $L$ but also improves
the term $\varepsilon \delta _{m, 0}$ into
$\varepsilon \delta _{m, 0}^{2}$.
\end{remark}%
%

\section{High-dimensional settings}
\label{sec:highd}

In this section, we extend \texttt{(mSMSE)} to high-dimensional settings,
where the dimension $p$ is much larger than $n$. We assume that
$\bm{\beta }^{*}\in \mathbb{R}^{p}$ is a sparse vector with $s$ nonzero
elements. Recall \eqref{newton},
%
\begin{equation}
\label{eq:newton_high_d} %
\begin{aligned} \widehat{\bm{\beta
}}^{   ( t  )}&= \widehat{\bm{\beta }}^{
   ( t-1  )}- \bigl[\nabla
^{2} F_{h_{t}} \bigl( \widehat{\bm{\beta
}}^{   ( t-1  )} \bigr) \bigr] ^{-1} \nabla F_{h_{t}}
\bigl(\widehat{\bm{\beta }}^{   (t-1  )} \bigr) =\widehat{\bm{\beta
}}^{   ( t-1  )}- \bigl(V_{n}^{(t)} \bigr)
^{-1} U_{n}^{(t)}. \end{aligned}
%
\end{equation}
It is generally infeasible to compute the inverse of the Hessian matrix
$V_{n}^{(t)}\in \mathbb{R}^{p \times p}$ in the high-dimensional case.
Furthermore, it requires unacceptably high complexity to compute and communicate
$L$ high-dimensional matrices. To tackle these problems, we first note
that \eqref{eq:newton_high_d} is equivalent to solving the following quadratic
optimization problem:
%
\begin{equation}
\widehat{\bm{\beta }}^{  ( t  )}:=\argmin _{\bm{\beta }}
\frac{1}{2}\bm{\beta }^{\top}V^{  ( t  )}_{n}
\bm{\beta }-\bm{\beta }^{\top} \bigl(V^{  ( t  )}_{n}
\widehat{\bm{\beta }}^{   ( t-1  )}- U_{n}^{  (t  )}
\bigr). \label{eq:beta-opt-1} 
\end{equation}
Due to high communication complexity, we only estimate the Hessian matrix
$V_{n}^{(t)}$ using the samples on a single machine, for example,
$V^{  (t  )}_{m,1}$ on the first machine. Then
\eqref{eq:beta-opt-1} can be written as
%
\begin{equation}
\widehat{\bm{\beta }}^{  ( t  )}:=\argmin _{\bm{\beta }}
\frac{1}{2}\bm{\beta }^{\top}V^{  ( t  )}_{m,1}
\bm{\beta }-\bm{\beta }^{\top} \bigl(V^{  ( t  )}_{m,1}
\widehat{\bm{\beta }}^{   ( t-1  )}- U_{n}^{  (t  )}
\bigr). \label{eq:beta-opt-2} 
\end{equation}

We adapt the idea of the Dantzig selector proposed by
\cite{candes2007dantzig}, an $\ell _{1}$-regularization approach known
for estimating high-dimensional sparse parameters. Formally, in the
$t$th iteration, given $\widehat{\bm{\beta }}^{  ( t-1  )}$, the
bandwidth $h_{t}$ and a regularization parameter
$\lambda _{n}^{  (t  )}$, we compute
$\widehat{\bm{\beta }}^{  (t  )}$ by
%
\begin{equation}
\begin{aligned} \widehat{\bm{\beta }}^{   ( t  )}= \argmin
_{\bm{\beta }
\in \mathbb{R}^{p}} \bigl\lbrace \|\bm{\beta } \| _{1} :
\big\| V^{  ( t  )}_{m,1}\bm{\beta }- \bigl(
V^{  ( t
  )}_{m,1} \widehat{\bm{\beta }}^{   ( t-1  )}-
U_{n}^{
  ( t  )} \bigr) \big\| _{\infty} \leq
\lambda ^{  (t
  )}_{n} \bigr\rbrace . \end{aligned}
\label{eq:beta-opt-Dan} 
\end{equation}
Note that a feasible solution of \eqref{eq:beta-opt-Dan} can be obtained
by linear programming. A complete algorithm is presented in Algorithm~\ref{alg:hd}.

\begin{algorithm}[!tp]
        \caption{High-dimensional Multiround Maximum Score Estimator}
        \vspace*{0.08in} {\textbf{Input:}}
        Data sets distributed on local machines $\{x_{i}, \bm{z}_{i},y_{i}\}_{i\in \mathcal{H}_{\ell}}$ $  ( \ell =1,2,\ldots,L  )$, an initial estimator $\widehat{\bm{\beta }}^{(0)}$, the total number of iterations $T$, bandwidth sequence $\{h_{t}\}$ and parameters $\{\lambda ^{  ( t  )}_{n}\}$.
        \label{alg:hd}
\begin{algorithmic}[1]
                \For{$t=1,2,\ldots,T$} \do \\
                \State Send $\widehat{\bm{\beta }}^{   ( t-1  )}$ to each machine;
                \For {$\ell =1,2,\ldots,L$} \do\\
                \State Compute $U^{  ( t  )}_{m,\ell}$ by equation \eqref{eq:def:Umlt} and send $U^{  ( t  )}_{m,\ell}$ back to the first machine;
                \EndFor
                \State Compute $U^{  ( t  )}_{n}=\frac{1}{L}\sum_{\ell =1}^{L}U^{  ( t  )}_{m,\ell}$;
                \State Compute
                $V^{  ( t  )}_{m,1}\widehat{\bm{\beta }}^{   ( t-1  )}=\frac{1}{mh_{t}^{2}}\sum_{i\in \mathcal{H}_{1}} (-y_{i}) H^{\prime \prime} ( \frac{x_{i}+\bm{z}_{i}^{\top}\widehat{\bm{\beta }}^{   ( t-1  )}}{h_{t}} ) ( \bm{z}_{i}^{\top}\widehat{\bm{\beta }}^{   (t-1  )} )\bm{z}_{i}$;
                \State Obtain $\widehat{\bm{\beta }}^{   ( t  )}$ by solving \eqref{eq:beta-opt-Dan};
                \EndFor
                \State Output $\widehat{\bm{\beta }}^{  ( T  )}$.
        \end{algorithmic}
\end{algorithm}

Now we give the convergence rate of the estimator
$\widehat{\bm{\beta }}^{(t)}$ at iteration $t$. First, we state the one-step
improvement in the following theorem.

\begin{thm}%
\label{thm:1step-hd}
Assume Assumptions \ref{A1}--\ref{A4} hold, and the covariates are uniformly
bounded, that is, there exists $\overline{B}$ such that
$  \lVert \bm{z}_{i}  \rVert _{\infty} \leq \overline{B}$, with
finite second moment, that is,
$\sup_{  \lVert \bm{v}  \rVert _{2}=1}\mathbb{E}  [(
\bm{v}^{\top}\bm{Z})^{2}  ] < +\infty $. Further, assume that the
dimension $p=O  ( n^{\nu}  )$ for some $\nu >0$, the local sample
size $m=O(n^{c})$ for some $0<c<1$, the sparsity
$s=o (m^{1/4} )$ and the initial value
$\widehat{\bm{\beta }}^{   (0  )}$ satisfies
$ \|\widehat{\bm{\beta }}^{   ( 0  )} - \bm{\beta }^{*}
 \|_{2} = O_{\mathbb{P}}   (\delta _{m,0}  )$ and
$ \|\widehat{\bm{\beta }}^{   (0  )} - \bm{\beta }^{*}
 \|_{1} = O_{\mathbb{P}}   (\sqrt{s}\delta _{m,0}  )$ for
some $\delta _{m,0}=o(1)$. Moreover, assume that
$\sqrt{s}\delta _{m,0}=O ( h_{1}^{3/2} )$ and
$\frac{s^{2}\log m}{mh_{1}^{3}}=o  (1  )$. By specifying
\begin{equation*}
\lambda ^{  ( 1  )}_{n} = C_{0} \biggl(s\delta
_{m,0}^{2}+h_{1}^{
\alpha}+\sqrt{
\frac{\log p}{nh_{1}}}+\sqrt{\frac{s\log p}{mh_{1}^{3}}} \delta _{m,0}
\biggr),
\end{equation*}
with a sufficiently large constant $C_{0}$, it holds that
%
\begin{equation}
\big\|\widehat{\bm{\beta }}^{   ( 1  )}-\bm{\beta }^{*}
\big\|_{2} = O_{\mathbb{P}}\bigl(\sqrt{s}\lambda
_{n}^{(1)}\bigr)=O_{\mathbb{P}}
\biggl[s^{3/2}\delta _{m,0}^{2}+
\sqrt{s}h_{1}^{\alpha}+\sqrt{ \frac{s\log p}{nh_{1}}}+\sqrt{
\frac{s^{2}\log p}{mh_{1}^{3}}}\delta _{m,0} \biggr], \label{eq:1step-hd-Dan}
\end{equation}
and
$ \|\widehat{\bm{\beta }}^{   ( 1  )}-\bm{\beta }^{*}
 \|_{1} \leq 2\sqrt{s}  \|\widehat{\bm{\beta }}^{   ( 1
  )}-\bm{\beta }^{*} \|_{2}$ with probability tending to one.
\end{thm}
\label{thm:1step-hd-end}
The proof of Theorem~\ref{thm:1step-hd} is given in Section Section \ref{sec:pf-hd} of Appendix. The requirement $s=o(m^{1/4})$ is to ensure the
consistency of the estimator, which is further explained and discussed
in Remark \ref{rmk:s<m^1/4} in Section \ref{sec:supp-hd} of Appendix. An initial estimator
$\widehat{\bm{\beta }}^{(0)}$ can be obtained by applying the path-following
algorithm proposed in \cite{feng2019nonregular} to the local data on a
single machine, which satisfies the assumptions in Theorem~\ref{thm:1step-hd} with
$\delta _{m,0}=(s\log p/m)^{\alpha /(2\alpha +1)}$. In practice, the initial
can be other estimators, for example, \cite{mukherjee2019non}, or a mild
guess, for example, an estimator computed from a (potentially misspecified)
parametric model. The assumption
$\frac{s^{2}\log p}{mh_{1}^{3}}=o  (1  )$ is necessary to guarantee
the so-called restricted eigenvalue condition for $V^{(1)}_{m,1}$, which
is standard to ensure the convergence rate of the Dantzig selector in theory.
Another assumption $\sqrt{s}\delta _{m,0}=O ( h_{1}^{3/2} )$ is a
technical condition to determine the dominant term in the convergence rate.

The convergence rate in \eqref{eq:1step-hd-Dan} contains four terms. The
first and the last terms can be rewritten as
$  (\sqrt{\frac{s^{2}\log p}{mh_{1}^{3}}}+s^{3/2}\delta _{m,0}
  )\delta _{m,0}$, which is related to the initial error
$\delta _{m,0}$. If we further suppose that
$s^{3/2}\delta _{m,0}=o(1)$, then it will become $o(\delta _{m,0})$, which
can be iteratively refined in the algorithm. The remaining terms (the second
and third terms) can be minimized by specifying a bandwidth $h_{1}$ such
that $\sqrt{\frac{s\log p}{nh_{1}}} \asymp \sqrt{s}h_{1}^{\alpha}$, leading
to the rate
$\sqrt{s}  (\log p/n  )^{\frac{\alpha}{2\alpha +1}}$. Based on
these results, we are ready to give the convergence rate of the estimator
$\widehat{\bm{\beta }}^{(t)}$ at iteration $t$.

\begin{thm}%
\label{thm:tstep-hd-simplified}
Assume the assumptions in Theorem~\ref{thm:1step-hd} hold. With
proper bandwidth $h_{t}$, parameter $\lambda _{n}^{(t)}$ and kernel function
$H(\cdot )$, we can obtain that for $1\leq t\leq T$,
%
\begin{equation}
\big\|\widehat{\bm{\beta }}^{   ( t  )}-\bm{\beta }^{*}
\big\| _{2} = O_{\mathbb{P}} \biggl(\sqrt{s} \biggl(
\frac{\log p}{n} \biggr)^{
\frac{\alpha}{2\alpha +1}}+(r_{m})^{t}
\delta _{m,0} \biggr), \label{eq:tstep-hd-simplified} 
\end{equation}
and
$ \|\widehat{\bm{\beta }}^{   ( t  )}-\bm{\beta }^{*}
 \|_{1} =O_{\mathbb{P}} ( \sqrt{s} \|\widehat{\bm{\beta }}^{
   ( t  )}-\bm{\beta }^{*} \|_{2} )$, where $r_{m}$ is
an infinitesimal quantity.
\end{thm}

Theorem~\ref{thm:tstep-hd-simplified} summarizes the $\ell _{2}$ and
$\ell _{1}$ error bounds of $\widehat{\bm{\beta }}^{(t)}$ in Algorithm~\ref{alg:hd}. The detailed statement, including proper choices for $h_{t}$,
$\lambda _{n}^{(t)}$ and $H(\cdot )$ and the formal definition of
$r_{m}$, are relegated to Section \ref{sec:supp-hd} of Appendix. The
choice for parameters $\lambda _{n}^{(t)}$ depends on the sparsity
$s$, which may be unknown in practice. We provide an estimation method
to deal with unknown $s$ and additional discussions for the dependency
of $\lambda _{n}^{(t)}$ on $s$ in Section \ref{sec:Lepski} and Remark \ref{rmk:sinlambda} in Appendix.

The upper bound in \eqref{eq:tstep-hd-simplified} contains two terms. The
second term comes from the error of the initial estimator, and it decreases
exponentially as $t$ increases. As the algorithm operates, this quantity
finally gets dominated by the first term within at most $O(\log n)$ iterations.
Furthermore, the conditions on $m$, $n$ and $s$ in Theorem~\ref{thm:tstep-hd-simplified} are placed to guarantee that the algorithm
converges in finite iterations. The first term,
$\sqrt{s}  (\log p/n  )^{\frac{\alpha}{2\alpha +1}}$, represents
the statistical convergence rate of our proposed estimator, which is very
close to the optimal rate that one can obtain without a distributed environment.
\cite{feng2019nonregular} recently established the minimax optimal rate
$  (s\log p / n  )^{\frac{\alpha}{2\alpha +1}}$ of the maximum
score estimator in high-dimensional settings. Compared to that, the established
rate in Theorem~\ref{thm:tstep-hd-simplified},
$\sqrt{s}  (\log p/n  )^{\frac{\alpha}{2\alpha +1}}$, is slightly
slower due to the different designs in the algorithms, with a difference
of $s^{\frac{1}{4\alpha +2}}$. Since our proposed algorithm is designed
based on the Dantzig selector that directly controls the infinity norm
of the gradient, which is different from the path-following algorithm used
in \cite{feng2019nonregular}, their techniques cannot be directly applied
to our proposed distributed estimator. It would be a potentially interesting
future work to improve the estimator in distributed settings to match the
optimal rate.

It is also worthwhile noting that our proposed estimator intrinsically
relies on the smoothness condition of the conditional density as
\cite{feng2019nonregular} does. Instead, \cite{mukherjee2019non} studied
the high-dimensional maximum score estimator with an $\ell _{0}$ penalty
under the soft margin condition. They established a convergence rate
$(s\log p\log n/n)^{\frac{\alpha '}{\alpha '+2}}$ for their estimator,
where $\alpha ' \geq 1$ is a smoothness parameter in the soft margin condition.
Their result and ours are different due to assuming the underlying smoothness
in different ways, and it can be an interesting future direction to investigate
their strategy in distributed settings.

\section{Simulations}
\label{sec:simulations}

In this section, we use Monte Carlo simulations to verify the theoretical
properties developed in Section~\ref{sec:theory} for distributed estimation
and inference. Recall that the binary response model is
%
\begin{equation*}
\label{eq:simulation} y_{i}=\mathrm{sign} \bigl( x_{i}+
\bm{z}_{i}^{\top}\bm{\beta }^{*}+ \epsilon
_{i} \bigr),\quad  i=1,2,\dots ,n.
\end{equation*}
For each $i$, we independently generate $x_{i} \in \mathcal{N}(0, 1)$
and $\bm{z}_{i} =  ( z_{i,1},\dots ,z_{i,p}  )^{\top}$ from
$\mathcal{N}(0,\bm \Sigma _{0.5})$, where $\bm \Sigma _{0.5}$ denotes an
AR(1) matrix with correlation coefficient $0.5$, that is,
$(\bm\Sigma _{0.5})_{j_{1}, j_{2}}=0.5^{|j_{1}-j_{2}|}$. We let the true
parameter $\bm{\beta }^{*}=\bm{1}_{p}/\sqrt{p}$, where $\bm{1}_{p}$ denotes
the $p$-dimensional vector with all elements being one. We consider the
following three distributions to generate the noise $\epsilon _{i}$:
\begin{enumerate}[(3)]
\item[(1)] Homoscedastic normal distribution:
$\mathcal{N}  ( 0, \sigma ^{2}_{\epsilon}  )$;
\item[(2)] Homoscedastic uniform distribution:
$\mathcal{U}  [-\sqrt{3}\sigma _{\epsilon}, \sqrt{3}\sigma _{
\epsilon}  ]$;
\item[(3)] Heteroscedastic normal distribution:
$\mathcal{N} ( 0, \sigma ^{2}_{\epsilon} [1+ 0.5 z_{i, 1}^{2}
 ]^{2}/2.76 )$ when $p=1$ and

$\mathcal{N} ( 0, \sigma ^{2}_{\epsilon} [1+ 0.5   (z_{i,1}-z_{i,2}
  )^{2} ]^{2}/10.76 )$ when $p>1$.
\end{enumerate}

In each distribution, the constants such as $\sqrt{3}$ and $2.76$ are to
guarantee that the standard deviation of $\epsilon $ is
$\sigma _{\epsilon}$, and we set $\sigma _{\epsilon}=0.25$. The simulation
results under the homoscedastic normal noise are reported and discussed
in this section, and the remaining two settings are discussed in Section
\ref{sec:supp-non-Gaussian} of Appendix. The $n$ observations are evenly divided
into $L=\lfloor n/m \rfloor $ subsets, each with size $m$, to simulate
the different machines. We fix $m=1000$ and vary the total sample size
$n$ from $m^{1.3}$ to $m^{1.9}$. On the simulated data set, we compare
the performance of the following four algorithms:
\begin{enumerate}[4.]
\item[1.] ``\texttt{(Avg-MSE)}'': Averaged Maximum Score Estimator given by
\eqref{eq:AvgMSE};
\item[2.] ``\texttt{(Avg-SMSE)}'': Averaged Smoothed Maximum Score Estimator
given by \eqref{eq:AvgSMSE};
\item[3.] ``\texttt{(mSMSE)}'': Proposed multiround Smoothed Maximum Score Estimator
(Algorithm~\ref{alg:mSMSE});
\item[4.] ``pooled-SMSE'': The Smoothed Maximum Score Estimator using the entire
data set.
\end{enumerate}

For the smoothing kernel, we use the following function $H(x)$, which is
the integral of the second-order biweight kernel function that satisfies
Assumption~\ref{A1} with $\alpha =2$:
%
\begin{align*}
\nonumber
\label{eq:kernel} H (x ) = %
\begin{cases} 0 &
\text{if } x<-1,
\\
\frac{1}{2} + \frac{15}{16} \biggl( x - \frac{2}{3}
x^{3} + \frac{1}{5} x^{5} \biggr) &
\text{if } \lvert x \rvert \leq 1,
\\
1 & \text{if } x>1. \end{cases} %
\end{align*}
We choose the parameter dimension $p$ to be $1$ and $10$ in the main text
and relegate additional simulations for $m=2000$ and $p=20$ to Section
\ref{sec:supp-p=20} of Appendix. For the case of $p=1$,
\cite{horowitz1992smoothed} suggested finding the optimum of SMSE and MSE
by searching the optimal $\bm{\beta }$ over a discrete one-dimensional
set. However, for $p=10$, it is not computationally feasible to search
for the optimal solution via a $p$-dimensional set. Meanwhile, an exact
solver of MSE \citep{florios2008exact} involves solving a mixed integer
programming, which can be extremely slow for $m=1000$, thus relegated
in the following experiments. We would like to mention that existing literature
\citep{horowitz1992smoothed,shi2018massive} rarely examines the cases of
$p>1$ due to the computational difficulty. For SMSE, we use a first-order
gradient descent algorithm initialized in a local region containing
$\bm{\beta }^{*}$ where the population objective is convex. In our \texttt{(mSMSE)}
algorithm, we compute the initial estimator
$\widehat{\bm{\beta }}^{  ( 0  )}$ by solving an SMSE on the first
machine. For the constant $\lambda _{h}$ in the bandwidth in Theorem~\ref{thm:normality}, we set $\lambda _{h}=1$ for all of the experiments
except for Section \ref{sec:sensitivity} of Appendix, where we discuss
the effect of $\lambda _{h}$ on \texttt{(mSMSE).}

To evaluate the performance of the competitor algorithms, we record the
coverage rate of each method for estimating the projected parameter
$\bm{1}_{p} ^{\top} \bm{\beta }^{*}$, that is, the rate that the estimated
confidence interval covers the true value of
$\bm{1}_{p} ^{\top} \bm{\beta }^{*}$. By Corollary~\ref{cor:inference}, for each $t$, our \texttt{(mSMSE)} algorithm constructs
a ($1-\xi $)-level confidence interval for
$\bm{1}_{p} ^{\top} \bm{\beta }^{*}$ as follows:
%
\begin{equation}
\begin{aligned}
&\label{eq:interval} \bm{1}_{p} ^{\top} \widehat{\bm{\beta
}}^{  ( t  )}-n^{-
\alpha /(2\alpha +1)}\lambda _{h}^{\alpha /(2\alpha +1)}
\bm{1}_{p}^{
\top}\widehat V^{-1} \widehat U \\
&\quad \pm
\tau _{1-\frac{\xi}{2}} \sqrt{n^{-
2\alpha /(2\alpha +1)}\lambda _{h}^{-1/(2\alpha +1)}
\bigl( \bm{1}_{p}^{
\top} \widehat V^{-1}
\widehat V_{s}\widehat V^{-1} \bm{1}_{p}
\bigr)}, 
\end{aligned}
\end{equation}
where $\alpha =2$ and $\tau _{1-\frac{\xi}{2}}$ is the
$  (1-\frac{\xi}{2}  )$-th quantile of
$\mathcal{N}  ( 0, 1  )$. Similarly, by Theorem~\ref{thm:asym-DC} and \cite{horowitz1992smoothed}, the confidence interval
of the \texttt{(Avg-SMSE)} and the pooled-SMSE can also be given by
\eqref{eq:interval}, replacing $\widehat{\bm{\beta }}^{(t)}$ with
$\widehat{\bm{\beta }}_{\mathrm{{{\mathtt{{(Avg-SMSE)}}}}}}$ and
$\widehat{\bm{\beta }}_{\mathrm{{pooled-SMSE}}}$, respectively. As claimed before,
the asymptotic distribution of \texttt{(Avg-MSE)} cannot be given in an explicit
form, and thus we can only construct the interval using the sample standard
error, as \cite{shi2018massive} does. Concretely, the confidence interval
of \texttt{(Avg-MSE)} is computed by
\begin{equation*}
\widehat{\bm{\beta }}_{\mathrm{{{\mathtt{{(Avg-MSE)}}}}}}\pm \tau _{1-
\frac{\xi}{2}}\sqrt{
\frac{1}{L(L-1)}\sum_{\ell =1}^{L} (
\widehat{\bm{\beta }}_{\mathrm{{MSE}},\mathrm{ \ell}}-\widehat{\bm{\beta }}_{\mathrm{{{
\mathtt{{(Avg-MSE)}}}}}}
)^{2}},
\end{equation*}
and $\widehat{\bm{\beta }}_{\mathrm{{MSE}},\mathrm{ \ell}}$ denotes the MSE on machine
$\ell $. In all of the experiments, we always set $\xi =0.05$. The results
are averaged over 200 independent runs.

\begin{figure}[!t]
	\centering
	\begin{subfigure}[b]{0.45\textwidth}
		\centering
		\includegraphics[width=\textwidth]{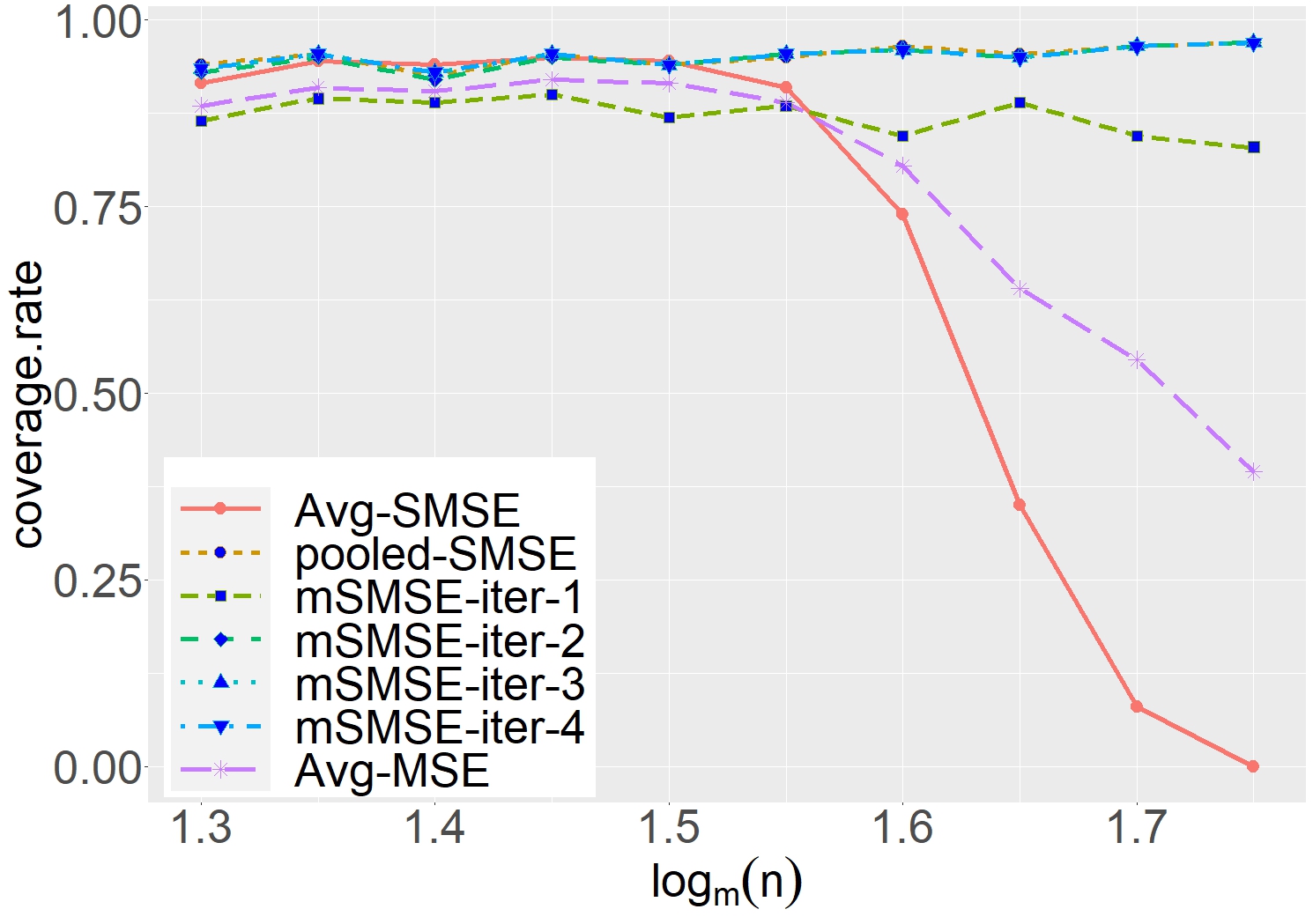}
		\caption{$p=1$}
		\label{fig:coveragerate_p=1}
	\end{subfigure}
	\hfill
	\begin{subfigure}[b]{0.45\textwidth}
		\centering
		\includegraphics[width=\textwidth]{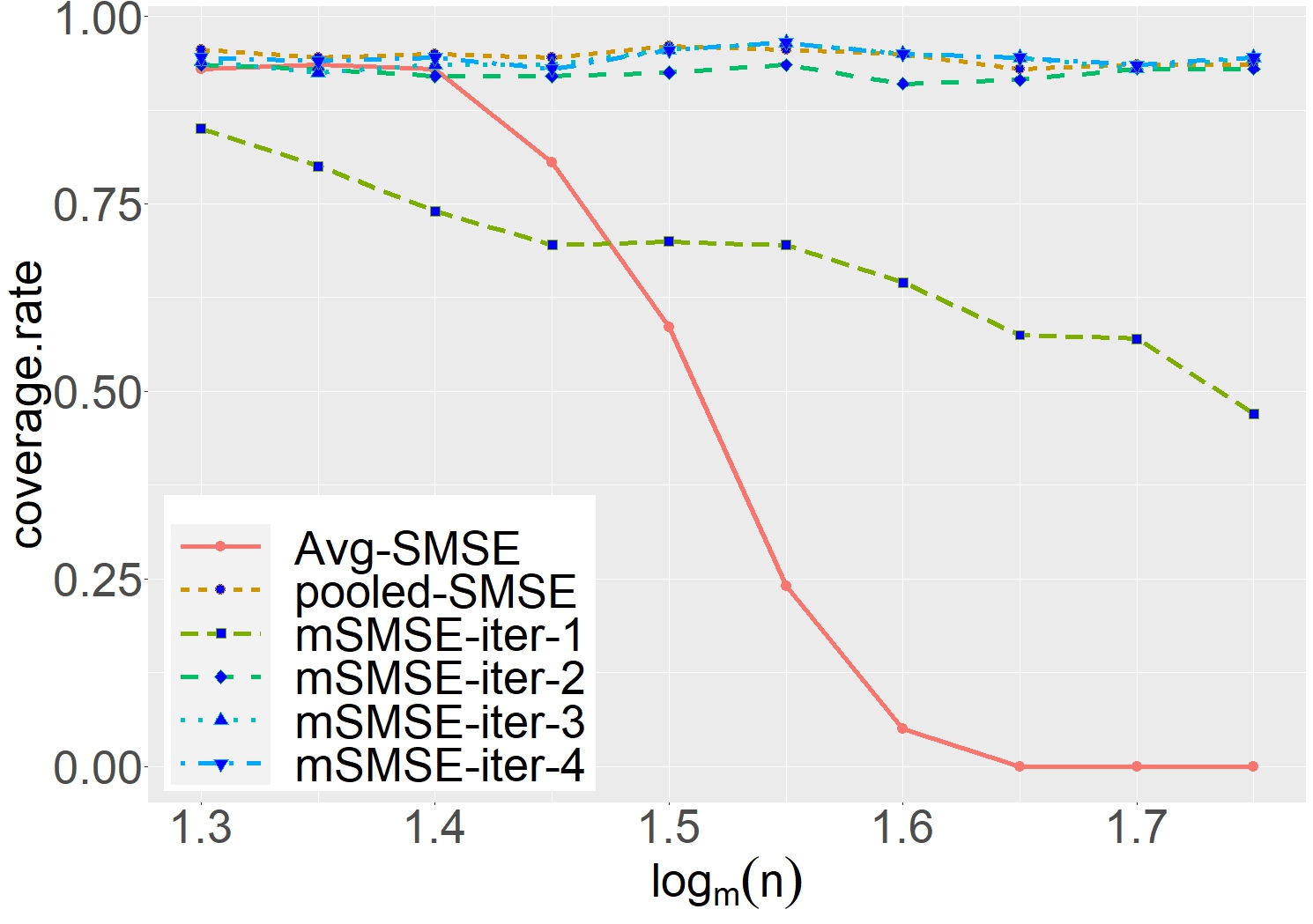}
		\caption{$p=10$}
		\label{fig:coveragerate_p=10}
	\end{subfigure}
	\caption{The coverage rates of different methods with homoscedastic normal noise}
	\label{fig:CoverageRates}
\end{figure}

Figure~\ref{fig:CoverageRates} shows the coverage rates as a function of
$\log _{m}(n)$ with $p=1$ and $10$. For both cases, our proposed \texttt{(mSMSE)},
as well as the pooled estimator, achieves a high coverage rate around 95\%
no matter how large $\log _{m}(n)$ is, while the two averaging methods
both fail when $\log _{m}(n)$ is large, as expected. This is consistent
with the findings of \cite{shi2018massive}. Note that since we fix the
local size $m$, increasing $\log _{m}(n)$ is equivalent to increasing
$L$, so the failure of averaging methods with large $\log _{m}(n)$ verifies
that there exists a restriction on the number of machines, which is illustrated
in our theoretical results. Furthermore, we also see that our \texttt{(mSMSE)}
algorithm is efficient and stable since the coverage rates of \texttt{(mSMSE)}
achieve around 95\% in only three iterations and change little in the following
fourth iteration. In fact, \texttt{(mSMSE)} converges within three iterations
in most runs, but we force it to run the fourth iteration to show stability.
Additionally, we observe that \texttt{(Avg-SMSE)} performs better than \texttt{(Avg-MSE)}
for small $\log _{m}(n)$ and $p=1$, but when the number of machines is
too large and does not satisfy the constraint in Theorem~\ref{thm:asym-DC},
\label{add}
the coverage rates of \texttt{(Avg-SMSE)} may be lower than those of \texttt{(Avg-MSE)}.
Recall that, Theorem~\ref{thm:asym-DC} has a constraint on the number of
machines
$L=o ( m^{\frac{2}{3}  ( \alpha -1  )}/   ( {p}\log m
  )^{\frac{2\alpha +1}{3}} )$. This constraint is necessary to
guarantee that $p\log m /(mh^{3}) = o(1)$, the violation of which results
in an additional bias term of the order
$O (\sqrt{\frac{p\log m}{mh^{3}}} )$. When $\log _{m}(n)$ is large,
this term diverges and invalidates the asymptotic normality established
in Theorem~\ref{thm:asym-DC}, which leads to the failure of distributed
inference. It is noteworthy to mention that the theoretical analysis for
\texttt{(Avg-SMSE)} only covers the cases when the constraint is satisfied,
and the reason that we further develop \texttt{(mSMSE)} is to completely
remove the constraint. In Section \ref{sec:bandwidth-supp} of Appendix, we
provide further discussions for the performance of \texttt{(Avg-SMSE)} when
the constraint is not satisfied. We relegate the reports of the bias and
variance, sensitivity analysis and time complexity of the above simulations
to Sections \ref{supp:simu-bias}--\ref{sec:supp_time} of Appendix.

\section{Conclusion and future works}
\label{sec:conclusion}

In this paper, we study the semiparametric binary response model in a distributed
environment. The problem has been studied in \cite{shi2018massive} with
the maximum score estimator \texttt{(Avg-MSE)} from the perspective of its
statistical properties. We provide two algorithms to achieve faster convergence
rates under relaxed constraints on the number of machines $L$: (1) the
first approach \texttt{(Avg-SMSE)} is a one-shot divide-and-conquer type
algorithm on a smoothed objective; (2) the second approach \texttt{(mSMSE)}
completely removes the constraints via iterative smoothing in a multiround
procedure. From the statistical perspective, \texttt{(mSMSE)} achieves the
optimal nonparametric rate of convergence; from the algorithmic perspective,
it achieves quadratic convergence with respect to the number of iterations,
up to the statistical error rate. We further provide two generalizations
of the problem: how to handle the heterogeneity of data sets and how to translate the problem to high-dimensional settings.

It is worthwhile noting that, in this paper, we assume that the communication
of $p\times p$ matrices is allowed in a distributed environment under low-dimensional
settings, but not under high-dimensional settings. If such communication is prohibited in the low-dimensional settings,
one may implement our high-dimensional algorithm, which entails a slower
algorithmic convergence over the iterations but achieves nearly the same statistical
accuracy when the algorithm terminates. Another option that can be considered
is applying our sequential smoothing toward the surrogate loss in
\cite{jordan2016communication} and \cite{chen2021first} to approximate
the second-order information matrix by first-order information, which we
leave for future studies.

In addition, the smoothing technique is closely related to other nonparametric
and semi-parametric methods such as the change plane problem
\citep{mukherjee2020asymptotic} and kernel density estimators
\citep{hardle2004nonparametric}. We anticipate the multiround smoothing
refinements can be adapted to other such estimators and improve the rate
of convergence without a restriction on the number of machines. It is also
worthwhile noting that recently a new smoothing method has been proposed
for quantile regression (see, e.g.,
\cite{fernandes2021smoothing}, \citeauthor{he2021smoothed} (\citeyear{he2021smoothed, he2022scalable}), \cite{battey2021communication}).
They proposed a specific form of the kernel function to ensure the loss
function in the smoothed quantile regression to be strictly convex. It
is not clear whether we can devise the choice of
$H  ( z/h  )$ to achieve the same advantage. Lastly, for the high-dimensional
settings, our proposed algorithm is based on the spirit of the Dantzig
selector. \cite{feng2019nonregular} considered a regularized
objective and obtained a minimax optimal rate for the linear binary response
model. It could be potentially interesting to apply this approach to the
maximum score estimator under the distributed setting.


\paragraph{Acknowledgements} The authors are very grateful to Professor Moulinath Banerjee and  anonymous referees, Associate Editor, and the Editor for the constructive comments that improved the quality of this paper.
\paragraph{Funding} Xi Chen's research is supported
by NSF Grant IIS-1845444. Weidong Liu's research is supported by NSFC Grant
11825104.


\bibliography{Smoothed_MSE_revised}

\begin{thebibliography}{}

\bibitem[\protect\citeauthoryear{Banerjee and Durot}{Banerjee and
  Durot}{2019}]{Banerjee2019Circumventing}
Banerjee, M. and C.~Durot (2019).
\newblock Circumventing superefficiency: {A}n effective strategy for
  distributed computing in non-standard problems.
\newblock {\em Electronic Journal of Statistics\/}~{\em 13\/}(1), 1926 -- 1977.

\bibitem[\protect\citeauthoryear{Banerjee, Durot, and Sen}{Banerjee
  et~al.}{2019}]{banerjee2019divide}
Banerjee, M., C.~Durot, and B.~Sen (2019).
\newblock Divide and conquer in non-standard problems and the super-efficiency
  phenomenon.
\newblock {\em Annals of Statistics\/}~{\em 47\/}(2), 720--757.

\bibitem[\protect\citeauthoryear{Battey, Fan, Liu, Lu, and Zhu}{Battey
  et~al.}{2018}]{battey2015distributed}
Battey, H., J.~Fan, H.~Liu, J.~Lu, and Z.~Zhu (2018).
\newblock Distributed estimation and inference with statistical guarantees.
\newblock {\em Annals of Statistics\/}~{\em 46\/}(3), 1352--1382.

\bibitem[\protect\citeauthoryear{Bousquet}{Bousquet}{2003}]{bousquet2003concentration}
Bousquet, O. (2003).
\newblock Concentration inequalities for sub-additive functions using the
  entropy method.
\newblock In {\em Stochastic {I}nequalities and {A}pplications}, pp.\
  213--247. Springer.

\bibitem[\protect\citeauthoryear{Brockhoff and M{\"u}ller}{Brockhoff and
  M{\"u}ller}{1997}]{brockhoff1997random}
Brockhoff, P.~M. and H.-G. M{\"u}ller (1997).
\newblock Random effect threshold models for dose--response relationships with
  repeated measurements.
\newblock {\em Journal of the Royal Statistical Society Series B: Statistical
  Methodology\/}~{\em 59\/}(2), 431--446.

\bibitem[\protect\citeauthoryear{Brown, Low, and Zhao}{Brown
  et~al.}{1997}]{brown1997superefficiency}
Brown, L.~D., M.~G. Low, and L.~H. Zhao (1997).
\newblock Superefficiency in nonparametric function estimation.
\newblock {\em The Annals of Statistics\/}~{\em 25\/}(6), 2607--2625.

\bibitem[\protect\citeauthoryear{Cai and Liu}{Cai and
  Liu}{2011}]{cai2011adaptive}
Cai, T. and W.~Liu (2011).
\newblock Adaptive thresholding for sparse covariance matrix estimation.
\newblock {\em Journal of the American Statistical Association\/}~{\em
  106\/}(494), 672--684.

\bibitem[\protect\citeauthoryear{Cai, Zhang, and Zhou}{Cai
  et~al.}{2010}]{cai2010optimal}
Cai, T.~T., C.-H. Zhang, and H.~H. Zhou (2010).
\newblock Optimal rates of convergence for covariance matrix estimation.
\newblock {\em Annals of Statistics\/}~{\em 38\/}(4), 2118--2144.

\bibitem[\protect\citeauthoryear{Candes and Tao}{Candes and
  Tao}{2007}]{candes2007dantzig}
Candes, E. and T.~Tao (2007).
\newblock The {D}antzig selector: {S}tatistical estimation when {$p$} is much
  larger than {$n$}.
\newblock {\em Annals of Statistics\/}~{\em 35\/}(6), 2313--2351.

\bibitem[\protect\citeauthoryear{Chamberlain}{Chamberlain}{1986}]{chamberlain1986asymptotic}
Chamberlain, G. (1986).
\newblock Asymptotic efficiency in semi-parametric models with censoring.
\newblock {\em Journal of Econometrics\/}~{\em 32\/}(2), 189--218.

\bibitem[\protect\citeauthoryear{Chen, Liu, and Zhang}{Chen
  et~al.}{2019}]{chen2019quantile}
Chen, X., W.~Liu, and Y.~Zhang (2019).
\newblock Quantile regression under memory constraint.
\newblock {\em Annals of Statistics\/}~{\em 47\/}(6), 3244--3273.

\bibitem[\protect\citeauthoryear{Chen, Liu, and Zhang}{Chen
  et~al.}{2022}]{chen2021first}
Chen, X., W.~Liu, and Y.~Zhang (2022).
\newblock First-order newton-type estimator for distributed estimation and
  inference.
\newblock {\em Journal of the American Statistical Association\/}~{\em
  117\/}(540), 1858--1874.

\bibitem[\protect\citeauthoryear{Chen and Xie}{Chen and
  Xie}{2014}]{chen2014split}
Chen, X. and M.~Xie (2014).
\newblock A split-and-conquer approach for analysis of extraordinarily large
  data.
\newblock {\em Statistica Sinica\/}~{\em 24\/}(4), 1655--1684.

\bibitem[\protect\citeauthoryear{Dobriban and Sheng}{Dobriban and
  Sheng}{2020}]{dobriban2020wonder}
Dobriban, E. and Y.~Sheng (2020).
\newblock {WONDER}: Weighted one-shot distributed ridge regression in high
  dimensions.
\newblock {\em The Journal of Machine Learning Research\/}~{\em 21\/}(1),
  2483--2534.

\bibitem[\protect\citeauthoryear{Dobriban and Sheng}{Dobriban and
  Sheng}{2021}]{dobriban2021distributed}
Dobriban, E. and Y.~Sheng (2021).
\newblock Distributed linear regression by averaging.
\newblock {\em The Annals of Statistics\/}~{\em 49\/}(2), 918--943.

\bibitem[\protect\citeauthoryear{Duan, Ning, and Chen}{Duan
  et~al.}{2022}]{duan2022heterogeneity}
Duan, R., Y.~Ning, and Y.~Chen (2022).
\newblock Heterogeneity-aware and communication-efficient distributed
  statistical inference.
\newblock {\em Biometrika\/}~{\em 109\/}(1), 67--83.

\bibitem[\protect\citeauthoryear{Fan, Guo, and Wang}{Fan
  et~al.}{2023}]{fan2021communication}
Fan, J., Y.~Guo, and K.~Wang (2023).
\newblock Communication-efficient accurate statistical estimation.
\newblock {\em Journal of the American Statistical Association\/}~{\em
  118\/}(542), 1000--1010.

\bibitem[\protect\citeauthoryear{Fan, Wang, Wang, and Zhu}{Fan
  et~al.}{2019}]{fan2019distributed}
Fan, J., D.~Wang, K.~Wang, and Z.~Zhu (2019).
\newblock Distributed estimation of principal eigenspaces.
\newblock {\em Annals of Statistics\/}~{\em 47\/}(6), 2009.

\bibitem[\protect\citeauthoryear{Feng, Ning, and Zhao}{Feng
  et~al.}{2022}]{feng2019nonregular}
Feng, H., Y.~Ning, and J.~Zhao (2022).
\newblock Nonregular and minimax estimation of individualized thresholds in
  high dimension with binary responses.
\newblock {\em Annals of Statistics\/}~{\em 50\/}(4), 2284--2305.

\bibitem[\protect\citeauthoryear{Fernandes, Guerre, and Horta}{Fernandes
  et~al.}{2021}]{fernandes2021smoothing}
Fernandes, M., E.~Guerre, and E.~Horta (2021).
\newblock Smoothing quantile regressions.
\newblock {\em Journal of Business \& Economic Statistics\/}~{\em 39\/}(1),
  338--357.

\bibitem[\protect\citeauthoryear{Florios and Skouras}{Florios and
  Skouras}{2008}]{florios2008exact}
Florios, K. and S.~Skouras (2008).
\newblock Exact computation of max weighted score estimators.
\newblock {\em Journal of Econometrics\/}~{\em 146\/}(1), 86--91.

\bibitem[\protect\citeauthoryear{Gao, Liu, Wang, Wang, Yan, and Zhang}{Gao
  et~al.}{2022}]{gao2022review}
Gao, Y., W.~Liu, H.~Wang, X.~Wang, Y.~Yan, and R.~Zhang (2022).
\newblock A review of distributed statistical inference.
\newblock {\em Statistical Theory and Related Fields\/}~{\em 6\/}(2), 89--99.

\bibitem[\protect\citeauthoryear{Greene}{Greene}{2009}]{greene2009discrete}
Greene, W. (2009).
\newblock Discrete choice modeling.
\newblock In {\em Palgrave Handbook of Econometrics}, pp.\  473--556. Springer.

\bibitem[\protect\citeauthoryear{H{\"a}rdle, M{\"u}ller, Sperlich, Werwatz,
  et~al.}{H{\"a}rdle et~al.}{2004}]{hardle2004nonparametric}
H{\"a}rdle, W., M.~M{\"u}ller, S.~Sperlich, A.~Werwatz, et~al. (2004).
\newblock {\em Nonparametric and Semiparametric Models}, Volume~1.
\newblock Springer.

\bibitem[\protect\citeauthoryear{He, Pan, Tan, and Zhou}{He
  et~al.}{2021}]{he2021smoothed}
He, X., X.~Pan, K.~M. Tan, and W.-X. Zhou (2021).
\newblock Smoothed quantile regression with large-scale inference.
\newblock {\em Journal of Econometrics\/}, To appear.

\bibitem[\protect\citeauthoryear{He, Pan, Tan, and Zhou}{He
  et~al.}{2022}]{he2022scalable}
He, X., X.~Pan, K.~M. Tan, and W.-X. Zhou (2022).
\newblock Scalable estimation and inference for censored quantile regression
  process.
\newblock {\em The Annals of Statistics\/}~{\em 50\/}(5), 2899--2924.

\bibitem[\protect\citeauthoryear{Horowitz}{Horowitz}{1992}]{horowitz1992smoothed}
Horowitz, J.~L. (1992).
\newblock A smoothed maximum score estimator for the binary response model.
\newblock {\em Econometrica\/}~{\em 60\/}(3), 505--531.

\bibitem[\protect\citeauthoryear{Horowitz}{Horowitz}{1993}]{Horowitz1993}
Horowitz, J.~L. (1993).
\newblock Semiparametric and nonparametric estimation of quantal response
  models.
\newblock In G.~Maddala, C.~Rao, and H.~Vinod (Eds.), {\em Handbook of
  Statistics}, Volume~11. North-Holland.

\bibitem[\protect\citeauthoryear{Horowitz and Spokoiny}{Horowitz and
  Spokoiny}{2001}]{horowitz2001adaptive}
Horowitz, J.~L. and V.~G. Spokoiny (2001).
\newblock An adaptive, rate-optimal test of a parametric mean-regression model
  against a nonparametric alternative.
\newblock {\em Econometrica\/}~{\em 69\/}(3), 599--631.

\bibitem[\protect\citeauthoryear{Huang and Huo}{Huang and
  Huo}{2019}]{huang2019distributed}
Huang, C. and X.~Huo (2019).
\newblock A distributed one-step estimator.
\newblock {\em Mathematical Programming\/}~{\em 174\/}(1), 41--76.

\bibitem[\protect\citeauthoryear{Jordan, Lee, and Yang}{Jordan
  et~al.}{2019}]{jordan2016communication}
Jordan, M.~I., J.~D. Lee, and Y.~Yang (2019).
\newblock Communication-efficient distributed statistical inference.
\newblock {\em Journal of the American Statistical Association\/}~{\em
  114\/}(526), 668--681.

\bibitem[\protect\citeauthoryear{Kim and Pollard}{Kim and
  Pollard}{1990}]{kim1990cube}
Kim, J. and D.~Pollard (1990).
\newblock Cube root asymptotics.
\newblock {\em Annals of Statistics\/}~{\em 18\/}(1), 191--219.

\bibitem[\protect\citeauthoryear{Lee, Liu, Sun, and Taylor}{Lee
  et~al.}{2017}]{lee2017communication}
Lee, J.~D., Q.~Liu, Y.~Sun, and J.~E. Taylor (2017).
\newblock Communication-efficient sparse regression.
\newblock {\em Journal of Machine Learning Research\/}~{\em 18\/}(5), 1--30.

\bibitem[\protect\citeauthoryear{Lepskii}{Lepskii}{1991}]{lepskii1991problem}
Lepskii, O. (1991).
\newblock On a problem of adaptive estimation in {G}aussian white noise.
\newblock {\em Theory of Probability \& Its Applications\/}~{\em 35\/}(3),
  454--466.

\bibitem[\protect\citeauthoryear{Li, Lin, and Li}{Li
  et~al.}{2013}]{li2013statistical}
Li, R., D.~K. Lin, and B.~Li (2013).
\newblock Statistical inference in massive data sets.
\newblock {\em Applied Stochastic Models in Business and Industry\/}~{\em
  29\/}(5), 399--409.

\bibitem[\protect\citeauthoryear{Luo, Sun, and Zhou}{Luo
  et~al.}{2022}]{luo2022distributed}
Luo, J., Q.~Sun, and W.-X. Zhou (2022).
\newblock Distributed adaptive {H}uber regression.
\newblock {\em Computational Statistics \& Data Analysis\/}~{\em 169}, 107419.

\bibitem[\protect\citeauthoryear{Manski}{Manski}{1975}]{manski1975maximum}
Manski, C.~F. (1975).
\newblock Maximum score estimation of the stochastic utility model of choice.
\newblock {\em Journal of Econometrics\/}~{\em 3\/}(3), 205--228.

\bibitem[\protect\citeauthoryear{Manski}{Manski}{1985}]{manski1985semiparametric}
Manski, C.~F. (1985).
\newblock Semiparametric analysis of discrete response: {A}symptotic properties
  of the maximum score estimator.
\newblock {\em Journal of Econometrics\/}~{\em 27\/}(3), 313--333.

\bibitem[\protect\citeauthoryear{Mukherjee, Banerjee, Mukherjee, and
  Ritov}{Mukherjee et~al.}{2020}]{mukherjee2020asymptotic}
Mukherjee, D., M.~Banerjee, D.~Mukherjee, and Y.~Ritov (2020).
\newblock Asymptotic normality of a linear threshold estimator in fixed
  dimension with near-optimal rate.
\newblock {\em arXiv preprint arXiv:2001.06955\/}.

\bibitem[\protect\citeauthoryear{Mukherjee, Banerjee, and Ritov}{Mukherjee
  et~al.}{2019}]{mukherjee2019non}
Mukherjee, D., M.~Banerjee, and Y.~Ritov (2019).
\newblock Non-standard asymptotics in high dimensions: {M}anski's maximum score
  estimator revisited.
\newblock {\em arXiv preprint arXiv:1903.10063\/}.

\bibitem[\protect\citeauthoryear{{\c{S}}ent{\"u}rk and
  M{\"u}ller}{{\c{S}}ent{\"u}rk and M{\"u}ller}{2009}]{csenturk2009covariate}
{\c{S}}ent{\"u}rk, D. and H.-G. M{\"u}ller (2009).
\newblock Covariate-adjusted generalized linear models.
\newblock {\em Biometrika\/}~{\em 96\/}(2), 357--370.

\bibitem[\protect\citeauthoryear{Shamir, Srebro, and Zhang}{Shamir
  et~al.}{2014}]{shamir2014communication}
Shamir, O., N.~Srebro, and T.~Zhang (2014).
\newblock Communication-efficient distributed optimization using an approximate
  {N}ewton-type method.
\newblock In {\em International Conference on Machine Learning}, pp.\
  1000--1008. PMLR.

\bibitem[\protect\citeauthoryear{Shi, Lu, and Song}{Shi
  et~al.}{2018}]{shi2018massive}
Shi, C., W.~Lu, and R.~Song (2018).
\newblock A massive data framework for {M}-estimators with cubic-rate.
\newblock {\em Journal of the American Statistical Association\/}~{\em
  113\/}(524), 1698--1709.

\bibitem[\protect\citeauthoryear{Tan, Battey, and Zhou}{Tan
  et~al.}{2022}]{battey2021communication}
Tan, K.~M., H.~Battey, and W.-X. Zhou (2022).
\newblock Communication-constrained distributed quantile regression with
  optimal statistical guarantees.
\newblock {\em The Journal of Machine Learning Research\/}~{\em 23\/}(1),
  12456--12516.

\bibitem[\protect\citeauthoryear{Tu, Liu, Mao, and Chen}{Tu
  et~al.}{2021}]{tu2021variance}
Tu, J., W.~Liu, X.~Mao, and X.~Chen (2021).
\newblock Variance reduced median-of-means estimator for {B}yzantine-robust
  distributed inference.
\newblock {\em Journal of Machine Learning Research\/}~{\em 22\/}(84), 1--67.

\bibitem[\protect\citeauthoryear{Volgushev, Chao, and Cheng}{Volgushev
  et~al.}{2019}]{volgushev2017distributed}
Volgushev, S., S.-K. Chao, and G.~Cheng (2019).
\newblock Distributed inference for quantile regression processes.
\newblock {\em Annals of Statistics\/}~{\em 47\/}(3), 1634--1662.

\bibitem[\protect\citeauthoryear{Wang, Kolar, Srebro, and Zhang}{Wang
  et~al.}{2017}]{wang2016efficient}
Wang, J., M.~Kolar, N.~Srebro, and T.~Zhang (2017).
\newblock Efficient distributed learning with sparsity.
\newblock In {\em International Conference on Machine Learning}.

\bibitem[\protect\citeauthoryear{Wang, Yang, Chen, and Liu}{Wang
  et~al.}{2019}]{wang2019distributed}
Wang, X., Z.~Yang, X.~Chen, and W.~Liu (2019).
\newblock Distributed inference for linear support vector machine.
\newblock {\em Journal of Machine Learning Research\/}~{\em 20\/}(113), 1--41.

\bibitem[\protect\citeauthoryear{Wang and Zhu}{Wang and
  Zhu}{2022}]{wang2022reboot}
Wang, Y. and Z.~Zhu (2022).
\newblock Reboot: {D}istributed statistical learning via refitting bootstrap
  samples.
\newblock {\em arXiv preprint arXiv:2207.09098\/}.

\bibitem[\protect\citeauthoryear{White}{White}{1981}]{white1981consequences}
White, H. (1981).
\newblock Consequences and detection of misspecified nonlinear regression
  models.
\newblock {\em Journal of the American Statistical Association\/}~{\em
  76\/}(374), 419--433.

\bibitem[\protect\citeauthoryear{White}{White}{1982}]{white1982maximum}
White, H. (1982).
\newblock Maximum likelihood estimation of misspecified models.
\newblock {\em Econometrica\/}~{\em 50\/}(1), 1--25.

\bibitem[\protect\citeauthoryear{Xu, Wang, and Fang}{Xu
  et~al.}{2014}]{xu2014model}
Xu, T., J.~Wang, and Y.~Fang (2014).
\newblock A model-free estimation for the covariate-adjusted {Y}ouden index and
  its associated cut-point.
\newblock {\em Statistics in Medicine\/}~{\em 33\/}(28), 4963--4974.

\bibitem[\protect\citeauthoryear{Yu, Chao, and Cheng}{Yu
  et~al.}{2020}]{yu2020simultaneous}
Yu, Y., S.-K. Chao, and G.~Cheng (2020).
\newblock Simultaneous inference for massive data: {D}istributed bootstrap.
\newblock In {\em International conference on machine learning}.

\bibitem[\protect\citeauthoryear{Yu, Chao, and Cheng}{Yu
  et~al.}{2022}]{yu2022distributed}
Yu, Y., S.-K. Chao, and G.~Cheng (2022).
\newblock Distributed bootstrap for simultaneous inference under high
  dimensionality.
\newblock {\em Journal of Machine Learning Research\/}~{\em 23\/}(195), 1--77.

\bibitem[\protect\citeauthoryear{Zhang, Wainwright, and Duchi}{Zhang
  et~al.}{2012}]{zhang2012communication}
Zhang, Y., M.~J. Wainwright, and J.~C. Duchi (2012).
\newblock Communication-efficient algorithms for statistical optimization.
\newblock In {\em Advances in Neural Information Processing Systems}.

\bibitem[\protect\citeauthoryear{Zhou, Singh, Johnson, Wahba, Initiative,
  Berkeley, and DeCArli}{Zhou et~al.}{2018}]{zhou2018statistical}
Zhou, H.~H., V.~Singh, S.~C. Johnson, G.~Wahba, A.~D.~N. Initiative, Berkeley,
  and C.~DeCArli (2018).
\newblock Statistical tests and identifiability conditions for pooling and
  analyzing multisite datasets.
\newblock {\em Proceedings of the National Academy of Sciences\/}~{\em
  115\/}(7), 1481--1486.

\end{thebibliography}

  \newpage
\appendix


\section{Theoretical Results of the High-dimensional multi-round SMSE}
\label{sec:supp-hd}

In this section, we give the complete theoretical analysis of the high-dimensional multi-round SMSE in Algorithm \ref{alg:hd}. 
First, we restate the conditions as the following three assumptions.

\begin{assumption}
	\label{B1}
	Assume that the initial value $\hbe^{\,\left(0\right)}$ satisfies
	\begin{equation}
		\left\|\hbe^{\,\left( 0\right)} - \be^{*}\right\|_2 = O_{\Prob} \left(\delta_{m,0}\right), \quad
		\left\|\hbe^{\,\left(0\right)} -  \be^{*}\right\|_1  = O_{\Prob} \left(\sqrt{s}\delta_{m,0}\right),
	\end{equation}
	for some $\delta_{m,0}=o(1)$.
\end{assumption}

\begin{assumption}
	\label{B2}
	Assume that the dimension $p=O\left( n^{\nu}\right)$ for some $\nu>0$, the local sample size $m=O(n^c)$ for some $0<c<1$, and the sparsity $s=O\left(m^r\right)$ for some $0<r<1/4$.
\end{assumption}
\begin{assumption}
	\label{B3}
	Assume that the covariates are uniformly bounded, i.e., there exists $\overline{B}$ such that $\norm{\bZ}_{\infty} \leq \overline{B}$. Further, assume that the covariates have finite second moment, i.e., $\sup_{\norm{\bv}_2=1}\E\left[(\bv^{\top}\bZ)^2\right] < +\infty$.
\end{assumption} 

Assumption \ref{B1} requires that the error of the initial value can be upper bounded in both $\ell_1$ and $\ell_2$ norm. Moreover, the bound of the $\ell_1$ error is of the same order as $\sqrt{s}$ times the bound of the $\ell_2$ error. This can be achieved by the path-following method proposed by \cite{feng2019nonregular}, with $\delta_{m,0}=\left(s\log p/m\right)^{\alpha/(2\alpha+1)}$.  Assumption \ref{B2} restricts the dimension $p$, the local size $m$, and the sparsity $s$. The first two requirements on $p$ and $m$ imply that $\log p = O \left( \log n\right)$ and $\log m=O(\log n)$, which is necessary for ensuring that the algorithm converges in finite iterations. The third assumption $S=o(m^r)$ for $0<r<1/4$ is also necessary to ensure the consistency of our estimator, which is further discussed in Remark \ref{rmk:s<m^1/4} below. Assumption \ref{B3} assumes the uniform boundedness of the predictors, which can be achieved by scaling each element of $\bZ$ into $[-1, 1]$, without
impacting $\be^*$ and $Y$. 
\begin{remark}\label{rmk:s<m^1/4}
	The assumption $S=o(m^r)$ for $0<r<1/4$ arises from the requirement that $s^{3/2}\delta_{m,0}=o(1)$, which ensures that the estimator can be iteratively refined in the algorithm. As shown in Equation \eqref{eq:1step-hd-Dan} in the main text, the convergence rate of $\hbe^{(1)}$ contains a bias term $s^{3/2}\delta_{m,0}^2$, which is an improvement from the initial error $\delta_{m,0}$ if and only if $s^{3/2}\delta_{m,0}=o(1)$. Plugging in the initial error $\delta_{m,0}=(s\log p/m)^{\alpha/(2\alpha+1)}$ yields the requirement $s=o(m^{1/4})$. 
	
	This requirement can be relaxed to $s=o(m^{1/2})$ if we assume the sub-Gaussianity of the covariate $\bZ$ in the high-dimensional settings, since the bias term is reduced to $s^{1/2}\delta_{m,0}^2$ under this assumption. The condition $s=o(m^{1/2})$ is also necessary for ensuring the restricted eigenvalue condition for the local Hessian $V_{m, 1}$.
	
	Furthermore, we note that this assumption $s=o(m^{1/4})$ is essentially $s=o\big(m_1^{1/4}\big)$, where $m_1$ is the local sample size on the first machine, since we compute the initial estimator $\hbe^{(0)}$ and the Hessian matrix $V_{m_1, 1}$ on the first machine. For the ease  of presentation, we assume the sample size $m$ is identical on each local machine, i.e., $m_1=m=n/L$. If there is a machine where its local sample size is larger, we can choose this machine as the first machine. Further, if a certain amount of the entire data is allowed to be pooled to one machine, we can use that pooled size as the $m_1$. 
\end{remark} 

Under these assumptions, we formally restate Theorem \ref{thm:tstep-hd-simplified}:

\begin{thm}\label{thm:tstep-hd}
	For $t= 1,2,\dots, T$, define  $h^*:=\big(\frac{\log p}{n}\big)^{\frac{1}{2\alpha+1}}$and $r_m:=\sqrt{\frac{s^2\log p}{m(h^*)^3}}+s^{3/2}\delta_{m,0}$.
	Assume Assumptions \ref{A1}--\ref{A4} and  \ref{B1}--\ref{B3} hold. Then there exists a constant $\alpha_0$ such that, by choosing a kernel $H^{\prime}(\cdot)$ with order $\alpha>\alpha_0$,  bandwidth $h_t \equiv h^*$, and parameters 
	\[\lambda^{\left( t\right)}_{n}=C_\lambda\left[\left(\frac{\log p}{n}\right)^{\frac{\alpha}{2\alpha+1}}+\frac{1}{\sqrt{s}}(r_m)^t\delta_{m,0}\right],\]
	with a sufficiently large constant $C_\lambda$, we have that $r_m=o(1)$, 
	\begin{equation}\label{eq:tstep-hd}
		\left\|\hbe^{\,\left( t\right)}-\be^{*}\right\|_2 = O_\Prob\paren{\sqrt{s}\lambda^{\left( t\right)}_{n}} \quad {\rm and} \quad
		\left\|\hbe^{\,\left( t\right)}-\be^{*}\right\|_1 \leq 2\sqrt{s}\left\|\hbe^{\,\left( t\right)}-\be^{*}\right\|_2,
	\end{equation}
	with probability tending to one.
\end{thm}

The proof of Theorem \ref{thm:tstep-hd} is given in Section \ref{sec:pf-hd}. The upper bound $\sqrt{s}\lambda^{\left( t\right)}_{n}$ in Theorem \ref{thm:tstep-hd} denotes the order of the $\ell_2$-error of $\hbe^{(t)}$, which contains two components:
\[\sqrt{s}\lambda^{\left( t\right)}_{n}=\sqrt{s}\left(\log p/n\right)^{\alpha/(2\alpha+1)}+(r_m)^t\delta_{m,0}.\] 
The first term $\sqrt{s}\left(\log p/n\right)^{\alpha/(2\alpha+1)}$ is the best rate our method can achieve, and the second term comes from the error of the initial estimator. As long as $r_m = o(1)$, the second term decreases exponentially as $t$ increases, which implies that $\hbe^{(t)}$ converges after at most $O(\log n)$ iterations. Concretely, the number of required iterations is
\begin{equation}\label{eq:iterations-hd}
	\frac{\frac{\alpha}{2\alpha+1}\paren{\log n-\log \log p}-\frac{1}{2}\log s+ \log \delta_{m,0}}{\left[-\log\paren{r_m}\right]},
\end{equation}
which is larger than that in the low-dimensional case. Under the Assumption \ref{B2}, it's easy to see that \eqref{eq:iterations-hd} can be upper bounded by a finite number.

The condition $r_m = o(1)$ can be ensured by choosing a kernel function with order higher than a certain constant $\alpha_0$. See Remark \ref{rmk:alpha0} for explanation.

\begin{remark}
	\label{rmk:alpha0}
	To make sure that $r_m=\frac{s^2\log p}{m(h^*)^3}+s^{3/2}\delta_{m,0}=o(1)$ and $\sqrt{s}\delta_{m,0}=O\big((h^*)^{3/2}\big)$, we need to choose a kernel $H^{\prime}(\cdot)$ with order $\alpha$ such that
	$\alpha>\alpha_0:=\max\big\{\frac{3r}{2(1-4r)}, \frac{3}{2c\,\left( 1-2r\right)}+\frac{r}{2(1-2r)} \big\}$,
	where $(\log m) / (\log n)<c<1$ and $(\log s)/(\log m)< r< 1/4$ are supposed in Assumption \ref{B2}. Here we plug in the rate $\delta_{m,0}=(s\log p /m)^{\alpha/(2\alpha+1)}$ obtained by the path-following algorithm in \cite{feng2019nonregular}. If the order of the kernel is not sufficiently high (i.e., less than $\alpha_0$), Algorithm \ref{alg:hd} still works by choosing  $h_m=m^{\frac{r-2\alpha(1-2r)}{3(2\alpha+1)} + \varepsilon}$ for some small constant $\varepsilon>0$, and the corresponding convergence rate will be $\sqrt{s}h^{\alpha}_m$.
	
\end{remark}

In Remark \ref{rmk:sinlambda}, we explain the dependency of the parameter $\lambda_{n}^{(t)}$ on $s$ and provide solutions in cases where $s$ is unknown in practice.

\begin{remark}\label{rmk:sinlambda}
	In Theorem \ref{thm:tstep-hd}, the tuning parameter $\lambda_{n}^{(t)}$ for the Dantzig Selector depends on $s$. This dependence arises from the requirement that for any iteration $t$, the parameter $\lambda_{n}^{(t)}$ has to satisfy $  \Big\|V^{\left( t\right)}_{m,1}\be^* -\paren{V^{\left( t\right)}_{m,1} \hbe^{\,\left(t-1\right)}- U_{n}^{\left( t\right)}}\Big\|_{\infty} \leq \lambda^{\left(t\right)}_{n}$, i.e., the true parameter $\be^*$ must lie in the feasible set of the Dantzig Selector \eqref{eq:beta-opt-Dan}.  This is necessary for ensuring the consistency of the Dantzig Selector. When $t=1$, as shown in the proof of Theorem \ref{thm:1step-hd} in Section \ref{sec:pf-hd}, the order of $\left\|V^{\left( 1\right)}_{m,1}\be^* -\paren{V^{\left(1\right)}_{m,1} \hbe^{\,\left(0\right)}- U_{n}^{\left( 1\right)}}\right\|_{\infty}$ is $O_{\Prob}\Big(s\delta_{m,0}^2+h^{\alpha}+\sqrt{\frac{\log p}{nh}}+\sqrt{\frac{s\log p}{mh^3}}\delta_{m,0}\Big)$, which depends on $s$ and leads to the choice of $\lambda_n^{(1)}$ in Theorem \ref{thm:1step-hd}. The dependence of  $\lambda_{n}^{(t)}$ for $t>1$ on $s$ follows from a similar reason. 
	
	We also note that, if the sparsity $s$ is unknown in practice, one can apply the Lepski's method provided in Section \ref{sec:Lepski} to estimate $s$ and obtain the same convergence rate as in Theorem \ref{thm:tstep-hd}. Furthermore, if an upper bound for  $s$ is known in practice, say $s=O(m^r)$ for some constant $0<r<1/4$, then it is possible to replace $s$ in the above equation with the upper bound, which leads to a slower algorithmic convergence (i.e., requires more rounds to converge) but does not affect the final rate $\sqrt{s}\left(\frac{\log p}{n}\right)^{\frac{\alpha}{2\alpha+1}}$ that our algorithm can achieve.  
\end{remark}

\subsection{Data Dependent Method for Unknown Parameters}
\label{sec:Lepski}

{
	In Algorithm \ref{alg:hd}, the choice of regularization parameters $\big\{\lambda_{n}^{(t)}\big\}$  and the bandwidth parameters $\{h_t\}$ depends on the sparsity $s$ and the smoothness $\alpha$, which might be unknown in practice. In this section, we provide data-adaptive estimation methods to deal with either unknown $s$ or $\alpha$ using the Lepski's approach \citep{lepskii1991problem, feng2019nonregular}. We show that the estimators obtained by the data-adaptive method achieve the same convergence rate as that in Algorithm \ref{alg:hd}.}

{\textbf{Unknown $s$.} We first consider the case where $s$ is unknown and $\alpha$ is known. Let $\calS:=\{2^q: q=0,1,\dots,\left\lfloor\log_2(p)\right\rfloor\}$, $\delta_{m,0,s'}:=\big(\frac{s'\log p}{m}\big)^{\frac{\alpha}{2\alpha+1}}$, and $h^*=\big(\frac{\log p}{n}\big)^{\frac{1}{2\alpha+1}}$. For $s' \in \calS$ and $t= 1,2,\dots, T$, define $r_{m, s'}:=\sqrt{\frac{(s')^2\log p}{m(h^*)^3}}+(s')^{3/2}\delta_{m,0,s'}$ 
	and $\lambda^{\left( t\right)}_{n,s'}=C_\lambda\Big[\big(\frac{\log p}{n}\big)^{\frac{\alpha}{2\alpha+1}}+\frac{1}{\sqrt{s'}}(r_{m. s'})^{t} \delta_{m,0}\Big]$,
	with a sufficiently large constant $C_\lambda$. In the $t$th iteration, let $\hbe^{(t)}_{s'}$ denote the estimator obtained by solving \eqref{eq:beta-opt-Dan} with parameter $\lambda^{(t)}_{n, s'}$ for any $s' \in \calS$. Then, with a sufficiently large constant $\overline{C}$, we estimate the sparsity as
	\begin{equation}
		\widehat{s}^{(t)}:=\min\Big\{\widetilde{s} \in \calS: \Big\|\hbe^{(t)}_{ \widetilde{s}}-\hbe^{(t)}_{s'}\Big\|_2 \leq \overline{C} \sqrt{s'}\lambda_{n, s'}^{(t)}, ~ \Big\|\hbe^{(t)}_{ \widetilde{s}}-\hbe^{(t)}_{s'}\Big\|_1 \leq \overline{C} s'\lambda_{n, s'}^{(t)}, ~  \forall s' > \widetilde{s}\Big\}.
	\end{equation}
	The convergence rate of $\hbe^{(t)}_{\widehat{s}^{(t)}}$ is given by the following theorem.
	\begin{thm}\label{thm:tstep-hd-unknowns}
		Assume the conditions in Theorem  \ref{thm:tstep-hd-simplified} hold, and then for $t=1,2,\dots,T$, we have
		\begin{equation}\label{eq:tstep-hd-unknowns}
			\left\|\hbe^{\left( t\right)}_{\widehat{s}^{(t)}}-\be^{*}\right\|_2 = O_{\Prob}(\delta_{m ,t}), \quad \left\|\hbe^{\left( t\right)}_{\widehat{s}^{(t)}}-\be^{*}\right\|_1 = O_{\Prob}(\sqrt{s}\delta_{m ,t})
		\end{equation}
		where $\delta_{m ,t}:=\sqrt{s}\Big(\frac{\log p}{n}\Big)^{\frac{\alpha}{2\alpha+1}}+(r_{m})^{t} \delta_{m,0}$
		is the same as the error rate in Theorem \ref{thm:tstep-hd-simplified}.
	\end{thm}
	Theorem \ref{thm:tstep-hd-unknowns} shows that the data-adaptive estimator $\hbe^{\,\left( t\right)}_{\widehat{s}^{(t)}}$ obtained by the Lepski's method achieves the same convergence rate as $\hbe^{\,\left( t\right)}$ in Algorithm \ref{alg:hd}. The proof of Theorem \ref{thm:tstep-hd-unknowns} is provided in Section \ref{sec:pf-hd}.}

{\textbf{Unknown $\alpha$.} Now we present a data adaptive method for tuning the bandwidth $h$ when the smoothing parameter $\alpha$ is unknown and the sparsity $s$ is known. In this part, we assume that the initial error bound  $\delta_{m, 0}$ does not depend on $\alpha$, for example, $\big(\frac{s\log p}{m}\big)^{\frac13}$. We then define the bandwidth set as $\mathcal{D}=\Big\{2^{-q}, q=0, 1, \dots, \min\big\{\lfloor-\frac{2}{3}\log_2(\sqrt{s}\delta_{m ,0})\rfloor, \lfloor-\frac{1}{3}\log_2(s^2\log p / m)\rfloor \big\}\Big\}$, which ensures that for all $h' \in \calD$, the assumptions $\sqrt{s}\delta_{m ,0}=O\big((h')^{\frac32}\big)$ and $\frac{s^2\log p}{m(h')^3}=o(1)$ are satisfied.  For $h' \in \calD$ and $t= 1,2,\dots, T$, define $r_{m, h'}:=\sqrt{\frac{s^2\log p}{m(h')^3}}+s^{3/2}\delta_{m,0}$ 
	and $\lambda^{\left( t\right)}_{n,h'}=C_\lambda\left[\sqrt{\frac{\log p}{nh'}}+(h')^{\alpha}+\frac{1}{\sqrt{s}}(r_{m, h'})^{t} \delta_{m,0}\right]$,
	with a sufficiently large constant $C_\lambda$. In the $t$th iteration, let $\hbe^{(t)}_{h'}$ denote the estimator obtained by solving \eqref{eq:beta-opt-Dan} with bandwidth $h'$ and parameter $\lambda^{\left( t\right)}_{n,h'}$ for any $h' \in \calD$. Then, with a sufficiently large constant $\overline{C}$, we estimate the optimal bandwidth as
	\begin{equation}
		\widehat{h}^{(t)}:=\max\left\{\widetilde{h} \in \calD: \bigl\|\hbe^{(t)}_{ \widetilde{h}}-\hbe^{(t)}_{h'}\bigr\|_2 \leq \overline{C} \sqrt{s}\lambda_{n, h'}^{(t)}, ~ \bigl\|\hbe^{(t)}_{ \widetilde{h}}-\hbe^{(t)}_{h'}\bigr\|_1 \leq \overline{C} s\lambda_{n, h'}^{(t)}, ~  \forall h' < \widetilde{h}\right\}.
	\end{equation}
	Similar to Theorem \ref{thm:tstep-hd-unknowns}, the convergence rate of $\hbe^{(t)}_{\widehat{h}^{(t)}}$ is given by the following theorem.
	\begin{thm}\label{thm:tstep-hd-unknownalpha}
		Assume the conditions in Theorem  \ref{thm:tstep-hd-simplified} hold, and further assume that $s^{3/2}\delta_{m ,0}=o(1)$.  Then for $t=1,2,\dots,T$, we have
		\begin{equation}\label{eq:tstep-hd-unknownalpha}
			\left\|\hbe^{\,\left( t\right)}_{\widehat{h}^{(t)}}-\be^{*}\right\|_2 = O_\Prob\paren{\delta_{m,t}}  ~ {\rm and} ~ \left\|\hbe^{\,\left( t\right)}_{\widehat{h}^{(t)}}-\be^{*}\right\|_1 = O_\Prob\paren{\sqrt{s}\delta_{m,t}} ,
		\end{equation}
		where $\delta_{m ,t}=\sqrt{s}\Big(\frac{\log p}{n}\Big)^{\frac{\alpha}{2\alpha+1}}+(r_{m})^{t} \delta_{m,0}$
		is the same as the error rate in Theorem \ref{thm:tstep-hd-simplified}.
	\end{thm}
	Theorem \ref{thm:tstep-hd-unknownalpha} shows that the estimator $\hbe^{\,\left( t\right)}_{\widehat{h}^{(t)}}$ obtained by the Lepski's method achieves the same convergence rate as $\hbe^{\,\left( t\right)}$ in Algorithm \ref{alg:hd}. The proof of Theorem \ref{thm:tstep-hd-unknownalpha} is provided in Section \ref{sec:pf-hd}.}

\section{Technical Proof of the Theoretical Results}\label{sec:supp-pf}

\subsection{Proof of the Results for the $\ell_2$ Error Bound of \mSMSE}
\label{sec:pf-mSMSE}

\textbf{Proof of Proposition \ref{thm:1step-ld}}

We first restate Proposition \ref{thm:1step-ld} in a detailed version. The first step of {\mSMSE} can be written as
\begin{equation}
	\hbe^{(1)}-\be^*=\Big(V_{n,h_1}\big(\hbe^{( 0)}\big)\Big)^{-1}U_{n,h_1}\big(\hbe^{\left(0\right)}\big),
	\label{eq:beta1-beta*=V-1U}
\end{equation}
where
\begin{align}\label{eq:def:Vnh}
	V_{n,h}\left( \be\right)=&\nabla^2F_h\left(\be\right)=\frac{1}{nh^2}\sum_{i=1}^{n} \left( -y_i\right) H^{\prime\prime}\left( \frac{x_i+\bz_i^{\top}\be}{h}\right)\bz_i\bz_i^{\top},\\
	\label{eq:def:Unh}
	U_{n,h}\left( \be\right)
	=&\frac{1}{nh^2}\sum_{i=1}^{n} \left( -y_i\right) H^{\prime\prime}\left( \frac{x_i+\bz_i^{\top}\be}{h}\right)\bz_i\bz_i^{\top}\left(\be-\be^*\right)-\frac{1}{nh}\sum_{i=1}^{n}\left( -y_i\right) H^{\prime}\left( \frac{x_i+\bz_i^{\top}\be}{h}\right)\bz_i.
\end{align}

\begin{thm-supp-1}\label{thm:1step-ld-supp}
	Assume Assumptions \ref{A1}--\ref{A5} hold. Further assume that $\norm{\hbe^{\left( 0\right)}-\be^*}_2= O_{\Prob}\left( \delta_{m,0}\right)$, 
	$h_1=o\left( 1\right)$ and $\frac{p\log n}{nh_1^3}= o\left(1\right)$. We then have
	\begin{equation}
		\left\| U_{n,h_1}\left( \hbe^{\,\left( 0\right)}\right)\right\|_2=O_{\mathbb{P}}\left( \delta_{m,0}^2+h_1^{\alpha}+\sqrt{\frac{p}{nh_1}}+\delta_{m,0}\sqrt{\frac{p\log n}{nh_1^3}}\right),
		\label{eq:1stepU}
	\end{equation}
	\begin{equation}
		\left\| V_{n,h_1}\left( \hbe^{\,\left( 0\right)}\right)-V\right\|_2=O_{\mathbb{P}}\left( \sqrt{\frac{p\log n}{nh_1^3}}+\delta_{m,0}+h_1^{\alpha}\right),
		\label{eq:1stepV}
	\end{equation}
	and therefore
	\begin{equation}
		\left\| \hbe^{\,\left( 1\right)}-\be^{*}\right\|_2=O_{\mathbb{P}}\left( \delta_{m,0}^2+h_1^{\alpha}+\sqrt{\frac{p}{nh_1}}+\delta_{m,0}\sqrt{\frac{p\log n}{nh_1^3}}\right).
		\label{eq:1stepbetax}
	\end{equation}
\end{thm-supp-1}

\begin{proof}[Proof of Proposition \ref{thm:1step-ld}]
	Throughout the whole proof, without loss of generality, we assume that $\norm{\hbe^{\left(0\right)}-\be^*}_2 \leq \delta_{m,0}$ with probability approaching one, i.e., we assume the constant in $O_{\Prob}(\delta_{m,0})$ to be 1. For simplicity, we replace the notation $h_1$ with $h$. 
	
	\noindent \textbf{Proof of \eqref{eq:1stepU}}
	
	We first prove \eqref{eq:1stepU}. It suffices to show that
	\begin{equation}
		\label{eq:1stepU_detailed}
		\sup_{\be: \norm{\be-\be^*}_2 \leq \delta_{m,0} }	\left\| U_{n,h}\left( \be\right)\right\|_2  = O_{\Prob}\left(\delta_{m,0}^2+h^{\alpha}+\sqrt{\frac{p}{nh}}+\delta_{m,0}\sqrt{\frac{p\log n}{nh^3}}\right),
	\end{equation}
	which implies \eqref{eq:1stepU} since $\norm{\hbe^{\left(0\right)}-\be^*}_2 \leq \delta_{m,0}$ with probability approaching one. 	
	
	By the proof of Lemma 3 in \cite{cai2010optimal}, there exists $\bv_1,...,\bv_{5^p} \in \mathbb{R}^p$, s.t. for any $\bv$ in the unit sphere $\mathbb{S}^{p-1}=\braces{\bv\in \R^{p}: \norm{\bv}_2=1}$, there exists $j_v \in [5^p]$ satisfying $\norm{\bv-\bv_{j_v}}_2 \leq 1/2$. Then we have 
	\[\norm{U_{n,h}\left(\be\right)}_2= \sup_{\bv \in \mathbb{S}^{p-1}}\abs{\bv^{\top}U_{n,h}\left(\be\right)} \leq \sup_{j_v \in [5^p]} \abs{\bv_{j_v}^{\top} U_{n,h}\left(\be\right) } + \frac{1}{2}\norm{U_{n,h}\left(\be\right)}_2,\]
	and thus 
	\[\norm{U_{n,h}\left(\be\right)}_2 \leq \sup_{{j_v} \in [5^p]} 2\abs{\bv_{j_v}^{\top} U_{n,h}\left(\be\right) }.\]
	Therefore, to show \eqref{eq:1stepU_detailed}, it suffices to show that
	\begin{equation}
		\label{eq:1stepU_detailed_supv}
		\sup_{\be:\norm{\be-\be^*}_2\leq \delta_{m,0}}\sup_{{j_v} \in [5^p]} \abs{\bv_{j_v}^{\top}U_{n,h}\left(\be\right)}= O_{\mathbb{P}}\left(\sqrt{\frac{p}{nh}}+\delta_{m,0}\sqrt{\frac{p\log n}{nh^3}}+\delta_{m,0}^2+h^{\alpha}\right).
	\end{equation}

	Recall the definition
	\begin{align*}
		&\quad U_{n,h}\left(\be\right)\\
		&=\frac{1}{nh^2}\sum_{i=1}^{n} \left(-y_i\right) H^{\prime\prime}\left(\frac{x_i+\bz_i^{\top}\be}{h}\right)\bz_i\bz_i^{\top}\left(\be-\be^*\right)-\frac{1}{nh}\sum_{i=1}^{n}\left(-y_i\right) H^{\prime}\left(\frac{x_i+\bz_i^{\top}\be}{h}\right)\bz_i\\
		&=:\frac{1}{n}\sum_{i=1}^n U_{h,i}\left(\be\right),
	\end{align*}
	where
	\begin{equation}\label{eq:def:Uhi}
		U_{h,i}\left(\be\right):=\frac{1}{h^2} \left(-y_i\right) H^{\prime\prime}\left(\frac{x_i+\bz_i^{\top}\be}{h}\right)\bz_i\bz_i^{\top}\left(\be-\be^*\right)-\frac{1}{h}\left(-y_i\right) H^{\prime}\left(\frac{x_i+\bz_i^{\top}\be}{h}\right)\bz_i.
	\end{equation} 
	For any $\bm{v}$ in $\mathbb{S}^{p-1}$, we have the following decomposition: 
	\begin{equation}
		\label{eq:vU-break-ld}
		\begin{aligned}
			&\quad \bv^{\top}U_{n,h}\left( \be\right)\\
			&= (1-\E)\left[\bv^{\top}U_{n,h}\left( \be\right)-\bv^{\top}U_{n,h}\left(\be^*\right)\right]+(1-\E)\left[\bv^{\top}U_{n,h}\left(\be^*\right)\right]+\E\left[ \bv^{\top}U_{n,h}\left( \be\right)\right]\\
			&=\frac{1}{n}\sum_{i=1}^n \phi^U_{i,\bv}(\be)+(1-\E)\left[\bv^{\top}U_{n,h}\left(\be^*\right)\right]+\E \left[\bv^{\top}U_{n,h}\left( \be\right)\right]\\
		\end{aligned}
	\end{equation}
	where $\phi^U_{i,\bv}(\be):=(1-\E)\brackets{\bv^{\top}U_{h,i}(\be)-\bv^{\top}U_{h,i}(\be^*)}$. We will separately bound the three terms in \eqref{eq:vU-break-ld}.
	through the following three steps.
	
	\noindent \hypertarget{step1}{\textbf{Step 1}}
	
	We will show that, for some sufficiently large constant $\gamma>0$, there exists $C_{\phi}>0$ such that  
	\begin{equation}\label{eq:ld-U-step1}
		\sup_{\norm{\be-\be^*}_2\leq \delta_{m,0}}\sup_{{j_v} \in [5^p]} \abs{ \frac{1}{n}\sum_{i=1}^n \phi^U_{i,\bv_{j_v}}(\be)} \leq C_{\phi} \delta_{m,0}\sqrt{\frac{p\log n}{nh^3}},
	\end{equation}
	with probability at least $1-n^{-\gamma/2}-2(5n^{-\gamma})^{p}$. 
	
	For any positive $\gamma$ and each $j \in \braces{1,\dots,p}$, divide the interval $[\beta^*_j-\delta_{m,0},\beta^*_j+\delta_{m,0}]$ into $n^{\gamma}$ small intervals, each with length $\frac{2\delta_{m,0}}{n^{\gamma}}$. This division creates $n^{\gamma}$ intervals on each dimension, and the direct product of those intervals divides the hypercube $\braces{\be: \norm{\be-\be^*}_\infty\leq \delta_{m,0}}$ into $n^{\gamma p}$ small hypercubes. By arbitrarily picking a point on each small hypercube, we can find  $\braces{\be_1,...,\be_{n^{\gamma p}}} \subset \mathbb{R}^{p}$,  such that for all $\be$ in the ball $\braces{\be: \norm{\be-\be^*}_2\leq \delta_{m,0}}$ (which is a subset of $\braces{\be: \norm{\be-\be^*}_\infty\leq \delta_{m,0}}$), there exists $j_{\beta} \in [n^{\gamma p}]$ such that $\norm{\be-\be_{j_{\beta}}}_\infty \leq \frac{2\delta_{m,0}}{n^{\gamma}}$. 
	
	Assumption \ref{A1} ensures that $H^{\prime\prime}\left(x\right)$, $H^{\prime}\left(x\right)$ are both bounded and Lipschitz continuous, and thus we have, for any $\bv$ such that $\norm{\bv}_2=1$,
	\begin{align*}
		&\quad \abs{\bv^{\top}U_{h,i}(\be)-\bv^{\top}U_{h,i}(\be_{j_{\beta}})}\\
		&\leq \frac{\abs{\bv^{\top}\bz_i}}{h^2}\abs{\bz_i^{\top}\paren{\be-\be^*}\brackets{H^{\prime\prime}\left(\frac{x_i+\bz_i^{\top}\be}{h}\right)-H^{\prime\prime}\left(\frac{x_i+\bz_i^{\top}\be_{j_{\beta}}}{h}\right)}+\bz_i^{\top}\paren{\be-\be_{j_\beta}}H^{\prime\prime}\paren{\frac{x_i+\bz_i^{\top}\be_{j_{\beta}}}{h}}}\\
		& \quad + \frac{\abs{\bv^{\top}\bz_i}}{h}\abs{H^{\prime}\left(\frac{x_i+\bz_i^{\top}\be_{j_{\beta}}}{h}\right)-H^{\prime}\left(\frac{x_i+\bz_i^{\top}\be}{h}\right)}\\
		&\leq  \frac{2C_Hp^{1/2}\delta^2_{m,0}\norm{\bz_i}_2^3}{n^{\gamma}h^3}+\frac{4C_Hp^{1/2}\delta_{m,0}\norm{\bz_i}_2^2}{n^{\gamma}h^2},
	\end{align*}
	where $C_H$ is a constant that is larger than the upper bounds and Lipschitz constants of $H^{\prime\prime}\left(x\right)$ and $H^{\prime}\left(x\right)$. Therefore,
	\begin{align*}
		&\quad \sup_{{j_v} \in [5^p]}\sup_{\norm{\be-\be^*}_2\leq \delta_{m,0}}\inf_{j_{\beta}\in [n^{\gamma p}]} \abs{\frac{1}{n}\sum_{i=1}^n\phi_{i,\bv_{j_v}}^U(\be)-\frac{1}{n}\sum_{i=1}^n\phi_{i, \bv_{j_v}}^U\left(\be_{j_{\beta}}\right)}\\
		&\leq \abs{(1-\E)\frac{1}{n}\sum\limits_{i=1}^n\left[\bv_{j_v}^{\top}U_{h,i}(\be)-\bv_{j_v}^{\top}U_{h,i}(\be_{j_{\beta}})\right]}\\
		&\leq 8C_{H}\left(\frac{p^{1/2}\delta^2_{m,0}\sum\limits_{i=1}^n\norm{\bz_i}_2^3}{n^{\gamma+1}h^3}+\frac{p^{1/2}\delta_{m,0}\sum\limits_{i=1}^n\norm{\bz_i}_2^2}{n^{\gamma+1}h^2}\right).
	\end{align*}
	The assumption $\sup_{\norm{\bv}_2 \leq 1}\E \exp(\eta (\bv^{\top}\bz_i)^2) < \infty$ implies that  $\sup_{i, j}\E |z_{i, j}|^3 < \infty$. By Markov's inequality, it holds that $\frac{1}{n}\sum_{i=1}^n \norm{\bz_i}_3^3=\frac{1}{n}\sum_{i=1}^n \sum_{j=1}^p|z_{i, j}|^3 \leq 2n^{\gamma/2}p\sup_{i, j}\E |z_{i, j}|^3 $, with probability at least $1-\frac{1}{2}n^{-\gamma/2}$. Since $\norm{\bz_i}_2^3 \leq n^{1/2}\norm{\bz_i}_3^3$, we obtain that $\frac{p^{1/2}\delta^2_{m,0}\sum\limits_{i=1}^n\norm{\bz_i}_2^3}{n^{\gamma+1}h^3} \leq 2\sup_{i, j}\E |z_{i, j}|^3\frac{p^{3/2}\delta^2_{m,0}}{n^{(\gamma-1)/2}h^3}$. Similarly, $\frac{p^{1/2}\delta_{m,0}\sum\limits_{i=1}^n\norm{\bz_i}_2^2}{n^{\gamma+1}h^2} \leq 2\sup_{i, j}\E |z_{i, j}|^2\frac{p^{3/2}\delta_{m,0}}{n^{\gamma/2}h^2}$, with probability at least $1-\frac{1}{2}n^{-\gamma/2}$. Using the assumption $p\log n/(nh^3)=o(1)$, we can choose $\gamma$ to be sufficiently large such that \[\frac{p^{3/2}\delta^2_{m,0}}{n^{(\gamma-1)/2}h^3}+\frac{p^{3/2}\delta_{m,0}}{n^{\gamma/2}h^2} = o\left(\delta_{m,0}\sqrt{\frac{p\log n}{nh^3}}\right),\]
	and then we have, with probability at least $1-n^{-\gamma/2}$, 
	\begin{equation}
		\begin{aligned}
			&\quad \sup_{\norm{\be-\be^*}_2\leq \delta_{m,0}}\sup_{{j_v} \in [5^p]} \abs{ \frac{1}{n}\sum_{i=1}^n \phi^U_{i,\bv_{j_v}}(\be)}-\sup_{j_{\beta}\in [n^{\gamma p}]}\sup_{{j_v} \in [5^p]} \abs{ \frac{1}{n}\sum_{i=1}^n \phi^U_{i,\bv_{j_v}}(\be_{j_\beta})}\\
			&=o\left(\delta_{m,0}\sqrt{\frac{p\log n}{nh^3}}\right).
		\end{aligned}
		\label{eq:ld-phi-cover-beta}
	\end{equation}

	
	Now let $\bv$ be a fixed vector in $\mathbb{S}^{p-1}$ and $\be$ be a fixed vector in $\{\be:\norm{\be-\be^*}_2\leq \delta_{m,0}\}$. 
	Recall that $\zeta=X+\bZ^{\top}\be^*$ and $\rho\left(\cdot\mid \bZ\right)$ denotes the density of $\zeta$ given $\bZ$.  Let $\mathbb{E}_{\cdot|\bZ}$ denote the expectation conditional on $\bZ$. By Assumption \ref{A2}, the density function $\rho\left(\cdot\mid \bZ\right)$ is bounded, 
	and hence,
	\begin{equation}
		\begin{aligned}
			&\quad\mathbb{E}_{.|\bZ} \brackets{H^{\prime\prime}\left(\frac{X+\bZ^{\top}\be}{h}\right)}^2=\mathbb{E}_{.|\bZ} \brackets{H^{\prime\prime}\left(\frac{\bZ^{\top}\left(\be-\be^*\right)+\zeta}{h}\right)}^2\\
			&=h\int_{-1}^1 H^{\prime\prime }\left(\xi\right)^2\rho\left(\xi h-\bZ^{\top}\left(\be-\be^*\right)\mid\bZ\right)\diff \xi=O(h),\\
		\end{aligned}\label{eq:EH2}
	\end{equation}
	where we also use the facts that $h=o(1)$ and $H^{\prime\prime }\left(x\right)$ is bounded.
	Similarly, for some $\breve{\be}$ between $\be$ and $\be^*$,
	\begin{equation}
		\begin{aligned}
			\mathbb{E}_{.|\bZ} \brackets{H^{\prime}\left(\frac{X+\bZ^{\top}\be}{h}\right)-H^{\prime}\left(\frac{X+\bZ^{\top}\be^*}{h}\right)}^2&=	\mathbb{E}_{.|\bZ} \brackets{\frac{\bZ^{\top}(\be-\be^*)}{h}H^{\prime\prime}\left(\frac{X+\bZ^{\top}\breve{\be}}{h}\right)}^2\\
			&=O\left(\frac{\abs{\bZ^{\top}\left(\be-\be^*\right)}^2}{h}\right).
		\end{aligned}\label{eq:EdH2}
	\end{equation}
	Let\begin{align*}
		&\quad \widetilde{\phi}^U_{i,\bv}(\be):=
		\frac{h^2}{\delta_{m ,0}}\phi^U_{i,\bv}(\be)\\
		&=(1-\E)(-y_i)(\bv^{\top}\bz_i)\left\{\frac{\bz_i^{\top}\left(\be-\be^*\right)}{\delta_{m ,0}}  H^{\prime\prime}\left(\frac{x_i+\bz_i^{\top}\be}{h}\right)-\frac{h}{\delta_{m ,0}} \left[H^{\prime}\left(\frac{x_i+\bz_i^{\top}\be}{h}\right)-H^{\prime}\left(\frac{x_i+\bz_i^{\top}\be^*}{h}\right)\right]\right\},
	\end{align*} 
	Then, using the inequalities $e^x\leq 1+x+x^2e^{\max(x, 0)}$, $1+x \leq e^x$, and the fact that $\E\left[\widetilde{\phi}^U_{i,\bv}(\be)\right]=0$, for any $b, t>0$, we have 
	\begin{equation}
		\begin{aligned}
			&\quad \Prob\left(\sum_{i=1}^n \wphi^U_{i,\bv}(\be)  > b\right)\\
			&= \Prob\left[\exp\left(t\sum_{i=1}^n \wphi^U_{i,\bv}(\be)\right)  > e^{tb}\right]\\
			& \leq e^{-tb} \prod_{i=1}^n \E\left[\exp\left(t\wphi^U_{i,\bv}(\be)\right)\right]\\
			&\leq  e^{-tb} \prod_{i=1}^n \E\left[1+\left(t\wphi^U_{i,\bv}(\be)\right)^2\exp\left(t\abs{\wphi^U_{i,\bv}(\be)}\right)\right]\\
			&\leq \exp\left(-tb + t^2\sum_{i=1}^n\E\left[\left(\wphi^U_{i,\bv}(\be)\right)^2\exp\left(t\abs{\wphi^U_{i,\bv}(\be)}\right)\right]\right).
		\end{aligned}
		\label{eq:phiUi-bound}
	\end{equation}
	We note that the technique used in \eqref{eq:phiUi-bound} is similar to Lemma 1 in \cite{cai2011adaptive}. By \eqref{eq:EH2} and \eqref{eq:EdH2},
	\begin{align*}
		&\quad\E\left[\left(\wphi^U_{i,\bv}(\be)\right)^2\exp\left(t\abs{\wphi^U_{i,\bv}(\be)}\right)\right] \\
		&\lesssim h\E\left[(\bv^{\top}\bz_i)^2(\bu^{\top}\bz_i)^2\exp\left(C'_{H}t\abs{\bv^{\top}\bz_i}\abs{\bu^{\top}\bz_i}\right)\right],	
	\end{align*}
	where $\bu:=(\be-\be^*)/\norm{\be-\be^*}_2$ is a unit vector, and $C'_{H}$ is a constant depending on $H(\cdot)$. By the assumption that $\sup_{\norm{\bv}_2=1}\E \left[\exp(\eta(\bv^{\top}\bz_i)^2)\right] <+\infty$ and Cauchy-Schwartz inequality, we obtain that 
	\[\E\left[\left(\wphi^U_{i,\bv}(\be)\right)^2\exp\left(t\abs{\wphi^U_{i,\bv}(\be)}\right)\right]  \leq C^2_1h,\]
	for some absolute constant $C_1$ and sufficiently small $t$. In particular, let $t=\sqrt{\frac{\gamma_1p\log n}{4C^2_1nh}}$ and  $b=C_1\sqrt{nh\gamma_1p \log n}$, where $\gamma_1$ is an arbitrary positive constant. By \eqref{eq:phiUi-bound},  
	\[\Prob\left(\sum_{i=1}^n \wphi^U_{i,\bv}(\be)  >C_1 \sqrt{nh\gamma_1p \log n}\right) \leq \exp\left[-\frac{1}{4}\gamma_1p\log n\right]=n^{-\gamma_1 p/4},\]
	which implies that 
	\[\Prob\left(\abs{\frac{1}{n}\sum_{i=1}^n \phi^U_{i,\bv}(\be)}  > C_1\sqrt{\frac{\delta_{m,0}^2\gamma_1p \log n}{nh^3}}\right) \leq 2n^{-\gamma_1 p/4}.\]
	Let $\gamma_1=8\gamma$ and $C_{\phi}=C_1\sqrt{\gamma_1}$. The above inequality is true for any $\be\in \mathbb{R}^{p}$ that satisfies $\norm{\be-\be^*}_2 \leq \delta_{m,0}$ and any $\bv \in \mathbb{S}^{p-1}$. In particular, for any $j_{\beta} \in [n^{\gamma p}]$ and $j_v \in [5^p]$, with probability at least $1-2n^{-2\gamma p}$, 
	\[\abs{\frac{1}{n}\sum_{i=1}^n\phi^U_{i,\bv_{j_v}}\left(\be_{j_{\beta}}\right)}\leq C_\phi\delta_{m,0}\sqrt{\frac{p\log n}{nh^3}},\]
	which implies that with probability at least $1-2(5n^{-\gamma})^{p}$,
	\[\sup_{j_{\beta}\in [n^{\gamma p}]} \sup_{j_v \in [5^p]}\abs{ \frac{1}{n}\sum_{i=1}^n \phi^U_{i,\bv_{j_v}}(\be_{j_\beta})} \leq C_{\phi}\delta_{m,0}\sqrt{\frac{p\log n}{nh^3}}.\]
	Combining with \eqref{eq:ld-phi-cover-beta}, we obtain \eqref{eq:ld-U-step1}.
	
	\noindent
	\hypertarget{step2}{\textbf{Step 2}} 
	We will show that, for any constant $\gamma_2>0$, there exists a constant $C^*>0$, such that
	\begin{equation}\label{eq:ld-U-step2}
		\sup_{j_v \in [5^p]}\abs{(1-\E)\bv_{j_v}^{\top}U_{n,h}\left(\be^*\right)}  \leq C^* \sqrt{\frac{p}{nh}},
	\end{equation}
	with probability at least $1-2(5e^{-\gamma_2/4})^{p}$. 
	
	Let $\bv$ be any fixed vector in $\mathbb{S}^{p-1}$. Note that  \[(1-\E)\bv^{\top}U_{n,h}\left(\be^*\right) =\frac{1}{n}\sum_{i=1}^{n}(1-\E)\brackets{\bv^{\top}U_{h,i}\left(\be^*\right)}=\frac{1}{n}\sum_{i=1}^n(1-\E)\frac{(\bv^{\top}\bz_i)y_i}{h}H^{\prime}\left(\frac{x_i+\bz_i^{\top}\be^*}{h}\right),\] and $(1-\E)\brackets{\bv^{\top}U_{h,i}\left(\be^*\right)}$ are i.i.d. among different $i$. By Assumption \ref{A2}, the density function of $\zeta=X+\bZ^{\top}\be^*$ satisfies that $\rho^{(1)}(\cdot\mid\bZ)$ is bounded uniformly for all $\bZ$, which implies that $\rho(t \mid \bZ) = \rho(0\mid \bZ)+O(t)$. The constant in $O(t)$ is the same for all $t$. Therefore,
	\begin{equation}\label{eq:EH'2}
		\begin{aligned}
			&\E_{\cdot\mid \bZ}\brackets{\bv^{\top}U_{h,i}\left(\be^*\right)}^2\\
			=&\frac{\paren{\bv^{\top}\bZ}^2}{h^2}\E_{\cdot\mid \bZ} \brackets{H^{\prime}\left(\frac{X+\bZ^{\top}\be^*}{h}\right)}^2\\
			=& \int_{-1}^1\frac{\paren{\bv^{\top}\bZ}^2}{h}\brackets{H^{\prime}\left(\xi\right)}^2\rho\left(\xi h\mid \bZ\right) \diff \xi\\
			=&\frac{\paren{\bv^{\top}\bZ}^2}{h}\int_{-1}^1\brackets{H^{\prime}\left(\xi\right)}^2 \rho\left(0\mid \bZ\right)\diff \xi + O\left((\bv^{\top}\bZ)^2\right)\\
			=& O\left(\frac{(\bv^{\top}\bZ)^2}{h}\right).
		\end{aligned}
	\end{equation}
	Let $\wphi^{U, *}_{i, \bv}=h(1-\E)\brackets{\bv^{\top}U_{h,i}\left(\be^*\right)}$. Analogous to the proof in \textbf{Step 1} for the bound of $\wphi^{U}_{i, \bv}(\be)$, we have
	\[\Prob\left(\sum_{i=1}^n \wphi^{U, *}_{i,\bv}  > b\right) \leq \exp\left(-tb + nt^2C^2_2h\right),\]
	for some absolute constant $C_2$ and small enough $t$. Letting $b=C_2\sqrt{\gamma_2nph}$ and $t=\sqrt{\frac{\gamma_2p}{4C_2^2nh}}$ leads to
	\[\Prob\left(\sum_{i=1}^n \wphi^{U, *}_{i,\bv} > C_2 \sqrt{\gamma_2nhp }\right) \leq \exp\left[-\frac{1}{4}\gamma_2p\right],\]
	which leads to
	\[\Prob\left(\frac{1}{n}\abs{\sum_{i=1}^n (1-\E)\brackets{\bv^{\top}U_{h,i}\left(\be^*\right)}} >  C_2 \sqrt{\frac{\gamma_2p}{nh} }\right) \leq 2e^{-\gamma_2 p/4}.\]
	The above inequality is true for all $\bv$ in $\mathbb{S}^{p-1}$. Therefore, with probability at least $1-2(5e^{-\gamma_2/4})^p$,
	\[\sup_{j_v \in [5^p]}\abs{(1-\E)\bv_{j_v}^{\top}U_{n,h}\left(\be^*\right)}  \leq C^* \sqrt{\frac{p}{nh}},\]
	where $C^*=C_2\sqrt{\gamma_2}$. 
	
	\noindent \hypertarget{step3}{\textbf{Step 3}}
	
	We will show that, 
	\begin{equation}\label{eq:ld-U-step3}
		\sup_{\bv\in \mathbb{S}^{p-1}}\sup_{\be: \norm{\be-\be^*}_2 \leq \delta_{m,0} }\E\brackets{\bv^{\top}U_{n,h}\left(\be\right)}  \leq C_E(\delta_{m,0}^2+h^{\alpha}),
	\end{equation}
	where $C_E$ is an absolute constant.
	
	By Assumptions \ref{A2} and \ref{A3}, for almost every $\bZ$,
	\[\rho\left(t \mid \bZ\right)=\sum_{k=0}^{\alpha-1} \frac{1}{k!}\rho^{\,\left(k\right)}\left(0\mid\bZ\right)t^k+\frac{1}{\alpha!}\rho^{\left(\alpha\right)}\left(t^{\prime} \mid \bZ\right)t^{\alpha},\]
	\[F\left(-t \mid \bZ\right)=\frac{1}{2}+\sum_{k=1}^\alpha \frac{1}{k!}F^{\,\left(k\right)}\left(0\mid\bZ\right)(-t)^k+\frac{1}{\left(\alpha+1\right)!}F^{\left(\alpha+1\right)}\left(t^{\prime\prime} \mid \bZ\right)(-t)^{\alpha+1},\]
	where $t^{\prime}, t^{\prime\prime}$ are between 0 and $t$. Therefore, 
	\begin{equation}
		\begin{aligned}
			& \left(2F\left(-t \mid \bZ\right)-1\right)\rho\left(t \mid \bZ\right)
			=\sum_{k=1}^{2\alpha+1}M_k\left(\bZ\right)t^k,
		\end{aligned}
		\label{eq:Taylor}
	\end{equation}
	where $M_{k}\left(\bZ\right)$'s are constants depending on $\rho$, $F$, $\bZ$, and $t$. Since $\rho^{\left(k\right)}\left(\cdot\mid \bZ\right)$ and  $F^{\left(k\right)}\left(\cdot\mid \bZ\right)$ are bounded for all $k$ and almost all $\bZ$, we know there exists a constant $M$ such that $\sup_{k, \bZ, t}\abs{M_k\left(\bZ\right)} \leq M$ for all $t$. (In particular, we can obtain that \[M_{1}(\bZ)=2F^{(1)}(0\mid\bZ)\rho(0\mid\bZ), \quad M_\alpha\left(\bZ\right)=\sum_{k=1}^\alpha\frac{2(-1)^{-k}}{\left(\alpha-k\right)!k!}F^{\,\left(k\right)}\left(0\mid\bZ\right)\rho^{\,\left(\alpha-k\right)}\left(0\mid\bZ\right),\] which will be used in the proof of the following theorems.)
	
	By Assumption \ref{A1}, when $x>1$ or $x<-1$,  $H^{\prime}\left(x\right)=H^{\prime\prime}\left(x\right)=0$. The kernel $H^{\prime}\left(x\right)=\int_{-1}^x H^{\prime\prime}\left(t\right) \diff t$ is bounded, satisfying $\int_{-1}^1 H^{\prime}\left(x\right)\diff x=1$ and $\int_{-1}^1 x^kH^{\prime}\left(x\right)\diff x=0$ for any $1\leq k\leq\alpha-1$. Using integration by parts, we have $\int_{-1}^1 xH^{\prime\prime}\left(x\right)\diff x=-1$ and $\int_{-1}^1 x^kH^{\prime\prime}\left(x\right)\diff x=0$ for $k=0$ and $2\leq k\leq\alpha$.
	
	Now we are ready to compute the expectation of $\bv^{\top}U_{n,h}\left(\be\right)$ for any $\bv \in \mathbb{S}^{p-1}$ and $\be \in \{\be:\norm{\be-\be^*}_2\leq \delta_{m ,0}\}$. Since $\E\left[\bv^{\top}U_{n,h}\left(\be\right)\right]=\E\left[\bv^{\top}U_{h,i}\left(\be\right)\right]$ for all $i$, we omit $i$ in the following computation of expectation. Let $\mathbb{E}_{\cdot|\bZ}$ denote the expectation conditional on $\bZ$. Define $\Delta\left(\be\right):=\be-\be^*$ and recall that $\zeta=X+\bZ^{\top}\be^*$. 
	\begin{equation}
		\begin{aligned}
			&\mathbb{E}_{\cdot|\bZ}  \left[\bm{v}^{\top}U_{n,h}\left(\be\right)\right] \\
			=&\bZ^{\top}\bm{v} \cdot \mathbb{E}_{\cdot|\bZ} \bigg[ \frac{\bZ^{\top}\Delta\left(\be\right)}{h^2} \brackets{2\I\left(X+\bZ^{\top}\be^*+\epsilon<0\right)-1} H^{\prime\prime}\left(\frac{X+\bZ^{\top}\be}{h}\right)\\
			&\quad\quad\quad\quad\quad-\frac{1}{h}\brackets{2\I\left(X+\bZ^{\top}\be^*+\epsilon<0\right)-1}H^{\prime}\left(\frac{X+\bZ^{\top}\be}{h}\right) \bigg]\\
			=&\bZ^{\top}\bm{\bm{v}} \cdot \mathbb{E}_{\cdot|\bZ} \braces{ \brackets{2\I\left(\zeta+\epsilon<0\right)-1}\brackets {\frac{\bZ^{\top}\Delta\left(\be\right)}{h^2}  H^{\prime\prime}\left(\frac{\zeta+\bZ^{\top}\Delta\left(\be\right)}{h}\right)-\frac{1}{h}H^{\prime}\left(\frac{\zeta+\bZ^{\top}\Delta\left(\be\right)}{h}\right)}}\\
			=& \paren{\bZ^{\top}\bv} \int_{-1}^{1} \brackets{2F\left(\bZ^{\top}\Delta\left(\be\right)-\xi h \mid\bZ\right)-1}\rho\left(\xi h-\bZ^{\top}\Delta\left(\be\right) \mid  \bZ\right)\\
			&\quad\cdot\left( \frac{\bZ^{\top}\Delta\left(\be\right)}{h}H^{\prime\prime}\left(\xi\right)-H^{\prime}\left(\xi\right)\right)  \diff \xi \quad \left(\mathrm{by \, changing \, variable \,}\xi=\frac{\zeta+\bZ^{\top}\Delta\left(\be\right)}{h}\right)\\
			=&\,\left(\bZ^{\top}\bv\right) \sum_{k=1}^{2\alpha+1} M_k\paren{\bZ}\int_{-1}^{1} \left(\xi h-\bZ^{\top}\Delta\left(\be\right)\right)^{k}\left( \frac{\bZ^{\top}\Delta\left(\be\right)}{h}H^{\prime\prime}\left(\xi\right)-H^{\prime}\left(\xi\right)\right) \diff \xi.
		\end{aligned}
		\label{eq:EvU-computation}
	\end{equation}
	When $1\leq k\leq\alpha-1$,
	\begin{equation}
		\begin{aligned}
			&\int_{-1}^{1} \left(\xi h-\bZ^{\top}\Delta\left(\be\right)\right)^k\left( \frac{\bZ^{\top}\Delta\left(\be\right)}{h}H^{\prime\prime}\left(\xi\right)-H^{\prime}\left(\xi\right)\right) \diff \xi\\
			=&\sum_{k^{\prime}=0}^{k} \binom{k}{k^{\prime}}h^{k^{\prime}}\left(-\bZ^{\top}\Delta\left(\be\right)\right)^{k-k^{\prime}}\left[\left(\bZ^{\top}\Delta\left(\be\right)/h\right) \int_{-1}^1\xi^{k^{\prime}}H^{\prime\prime}\left(\xi\right)\diff\xi- \int_{-1}^1\xi^{k^{\prime}}H^{\prime}\left(\xi\right)\diff\xi \right]\\ =&\,\left(k-1\right)\left(-\bZ^{\top}\Delta\left(\be\right)\right)^k,\\
			\text{When \,} k=\alpha,&\\
			&\int_{-1}^{1} \left(\xi h-\bZ^{\top}\Delta\left(\be\right)\right)^{\alpha}\left( \frac{\bZ^{\top}\Delta\left(\be\right)}{h}H^{\prime\prime}\left(\xi\right)-H^{\prime}\left(\xi\right)\right) \diff \xi\\
			=&\sum_{k^{\prime}=0}^{\alpha} \binom{\alpha}{k^{\prime}}h^{k^{\prime}}\left(-\bZ^{\top}\Delta\left(\be\right)\right)^{\alpha-k^{\prime}}\left[\left(\bZ^{\top}\Delta\left(\be\right)/h\right) \int_{-1}^1\xi^{k^{\prime}}H^{\prime\prime}\left(\xi\right)\diff\xi- \int_{-1}^1\xi^{k^{\prime}}H^{\prime}\left(\xi\right)\diff\xi \right]\\ =&\,\left(\alpha-1\right)\left(-\bZ^{\top}\Delta\left(\be\right)\right)^\alpha-\pi_Uh^{\alpha},\\
		\end{aligned}
	\end{equation}
	where $\pi_U=\int_{-1}^1\xi^{\alpha}H^{\prime}\left(\xi\right)\diff\xi$ is defined in Assumption \ref{A1}. 
	
	When $\alpha+1 \leq k\leq 2\alpha+1$, using the fact that $H', H''$ are both bounded and the inequality $a^{k'}b^{k-k'} \leq (a+b)^k \leq 2^{k-1}(a^k+b^k)$, we have
	\begin{align*}
		&\int_{-1}^{1} \left(\xi h-\bZ^{\top}\Delta\left(\be\right)\right)^k\left( \frac{\bZ^{\top}\Delta\left(\be\right)}{h}H^{\prime\prime}\left(\xi\right)-H^{\prime}\left(\xi\right)\right) \diff \xi\\
		=&\sum_{k^{\prime}=0}^{k} \binom{k}{k^{\prime}}h^{k^{\prime}}\left(-\bZ^{\top}\Delta\left(\be\right)\right)^{k-k^{\prime}}\left[\left(\bZ^{\top}\Delta\left(\be\right)/h\right) \int_{-1}^1\xi^{k^{\prime}}H^{\prime\prime}\left(\xi\right)\diff\xi- \int_{-1}^1\xi^{k^{\prime}}H^{\prime}\left(\xi\right)\diff\xi \right]\\
		= &(k-1)\left(-\bZ^{\top}\Delta\left(\be\right)\right)^{k}-\sum_{k^{\prime}=\alpha+1}^{k}\binom{k}{k^{\prime}}h^{k^{\prime}-1}\left(-\bZ^{\top}\Delta\left(\be\right)\right)^{k-k^{\prime}+1}\left[ \int_{-1}^1\xi^{k^{\prime}}H^{\prime\prime}\left(\xi\right)\diff\xi\right]\\
		&-\sum_{k^{\prime}=\alpha}^{k}\binom{k}{k^{\prime}}h^{k^{\prime}}\left(-\bZ^{\top}\Delta\left(\be\right)\right)^{k-k^{\prime}}\left[\int_{-1}^1\xi^{k^{\prime}}H^{\prime}\left(\xi\right)\diff\xi \right]\\
		= &O\brackets {h^{k}+(\bZ^{\top}\Delta(\be))^{k}}.
	\end{align*}
	
	Note that $\abs{\bZ^{\top}\Delta(\be)} \leq \delta_{m ,0} \abs{\bZ^{\top}\bu}$, where $\bu=\Delta(\be)/\norm{\Delta(\be)}_2$.
	By Assumption \ref{A5}, it holds that  $\sup_{\bv \in \mathbb{S}^{p-1}}\E(\bZ^{\top}\bv)^{2k} \leq \infty$ for all positive integer $k$. Adding that $M_k(\bZ)$'s are uniformly bounded , we finally obtain that
	\begin{equation}
		\abs{\mathbb{E} \left[\bv^{\top}U_{n,h}\left(\be\right)\right]}\leq C_E\left(\delta_{m,0}^2+h^{\alpha}\right),
		\label{eq:EvU}
	\end{equation}
	where $C_E$ is a constant not depending on $\bv$ and $\be$. This leads to \eqref{eq:ld-U-step3}.

	Combining the three steps with \eqref{eq:1stepU_detailed_supv} completes the proof of \eqref{eq:1stepU}.
	
	\noindent \textbf{Proof of \eqref{eq:1stepV}}
	
	The proof of \eqref{eq:1stepV} is similar to that of \eqref{eq:1stepU}. We use the same $\{\bv_{j_v}\}$ and $\{\be_{j_\beta}\}$ as in the proof of \eqref{eq:1stepU}. By the proof of Lemma 3 in \cite{cai2010optimal}, 
	\[\norm{A}_2\leq 4\sup_{{j_v} \in [5^p]} \abs{\bv^{\top}_{j_v}A\bv_{j_v}},\]
	for any symmetric $A \in \mathbb{R}^{p \times p}$. 
	Therefore, it suffices to bound $\sup_{{j_v} \in [5^p]} \abs{\bv^{\top}_{j_v}\brackets{V_{n,h}\left(\be^{\left(0\right)}\right)-V}\bv_{j_v}}$.
	By the choice of $\braces{\be_{j_{\beta}}}$, for all $\be$ in the ball $\braces{\be: \norm{\be-\be^*}_2\leq \delta_{m,0}}$, there exists $j_{\beta} \in [n^{\gamma p}]$ such that $\norm{\be-\be_{j_{\beta}}}_\infty \leq \frac{2\delta_{m,0}}{n^{\gamma}}$. Recall
	\[V_{n,h}\left(\be\right)=\frac{1}{nh^2}\sum_{i=1}^{n} \left(-y_i\right) H^{\prime\prime}\left(\frac{x_i+\bz_i^{\top}\be}{h}\right)\bz_i\bz_i^{\top}=:\frac{1}{n}\sum_{i=1}^n V_{h,i}\left(\be\right),\]
	where
	\begin{equation}\label{eq:def:Vhi}
		V_{h,i}\left(\be\right):=\frac{1}{h^2} \left(-y_i\right) H^{\prime\prime}\left(\frac{x_i+\bz_i^{\top}\be}{h}\right)\bz_i\bz_i^{\top}.
	\end{equation}
	
	By the Lipschitz property of $H''(x)$, we have
	\begin{align*}
		&\quad\sup_{{j_v} \in [5^p]}\sup_{\be: \norm{\be-\be^*}_2 \leq \delta_{m,0} }\inf_{j_{\beta}\in [n^{\gamma p}]}\abs{\bv_{j_v}^{\top}V_{n,h}\left(\be\right)\bv_{j_v}-\bv_{j_v}^{\top}V_{n,h}\left(\be_{j_\beta}\right)\bv_{j_v}} \\
		&\leq \sup_{{j_v} \in [5^p]}\sup_{\be: \norm{\be-\be^*}_2 \leq \delta_{m,0} }\inf_{j_{\beta}\in [n^{\gamma p}]}\frac{1}{n}\sum_{i=1}^n \frac{\paren{\bz_i^{\top}\bv_{j_v}}^2}{h^3}\abs{\bz_i^{\top}\paren{\be-\be_{j_\beta}}} \\
		&= O\left(\sum_{i=1}^n\frac{\delta_{m,0}\norm{\bz_i}_2^3}{n^{\gamma+1}h^3}\right).
	\end{align*}
	
	In the \hyperlink{step1}{\textbf{Step 1}} of the proof of \eqref{eq:1stepU}, we show that $\frac{p^{1/2}\delta_{m,0}\sum\limits_{i=1}^n\norm{\bz_i}_2^3}{n^{\gamma+1}h^3} \leq \sup_{i, j}\E |z_{i, j}|^3\frac{p^{3/2}\delta_{m,0}}{n^{(\gamma-1)/2}h^3}$, with probability and least $1-n^{-\gamma/2}$. By taking $\gamma>0$ large enough such that 
	\[\frac{p^{3/2}\delta_{m,0}}{n^{(\gamma-1)/2}h^3} = o\left(\sqrt{\frac{p\log n}{nh^3}}\right),\] 
	we obtain
	\begin{equation}
		\label{eq:1stepV-jbeta}
		\begin{aligned}
			&\quad\sup_{{j_v} \in [5^p]}\sup_{\be: \norm{\be-\be^*}_2 \leq \delta_{m,0}}\abs{\bv^{\top}_{j_v}\brackets{V_{n,h}\left(\be\right)-V}\bv_{j_v}} -\sup_{{j_v} \in [5^p]}\sup_{j_{\beta}\in [n^{\gamma p}]}\abs{\bv^{\top}_{j_v}\brackets{V_{n,h}\left(\be_{j_\beta}\right)-V}\bv_{j_v}}\\
			&=o\paren{\sqrt{\frac{p\log n}{nh^3}}}.
		\end{aligned}
	\end{equation}
	
	Note that $(1-\E)\paren{\bv^{\top}\brackets{V_{n,h}\left(\be\right)-V}\bv}=\frac{1}{n}\sum_{i=1}^n (1-\E)\bv^{\top}V_{h,i}(\be)\bv$
	and let
	\[\phi_{i,\bv}^{V}(\be):=h^2(1-\E)\bv^{\top}V_{h,i}(\be)\bv=(1-\E)(\bv^{\top}\bz_i)^2 \left(-y_i\right) H^{\prime\prime}\left(\frac{x_i+\bz_i^{\top}\be}{h}\right).\]
	By repeating the procedure in \hyperlink{step1}{\textbf{Step 1}} of the proof for \eqref{eq:1stepU}, we obtain
	\[\Prob\left(\phi_{i,\bv}^{V}(\be) > b\right) \leq \exp\left(-tb+t^2C_3^2nh\right),\]
	for some absolute constant $C_3$ and small enough $t>0$. Let $b=C_3\sqrt{\gamma_3nhp\log n}$ and $t=b/(2C_3^2nh)$, where $\gamma_3$ is an arbitrary positive constant to be determined, it holds that
	\[\Prob\left(\sum_{i=1}^{n}\phi_{i,\bv}^{V}(\be) > C_3\sqrt{\gamma_3nhp\log n}\right) \leq n^{-\gamma_3p/4},\]
	and hence
	\[\Prob\left(\abs{(1-\E)\paren{\bv^{\top}\brackets{V_{n,h}\left(\be\right)-V}\bv}} > C_3\sqrt{\frac{\gamma_3p\log n}{nh^3}}\right) \leq 2n^{-\gamma_3p/4}.\]
	The above inequality is true for all $\bv_{j_v}$ and $\be_{j_{\beta}}$, and thus
	\[\Prob\left(\sup_{{j_v} \in [5^p]}\sup_{j_{\beta}\in [n^{\gamma p}]}\abs{(1-\E)\paren{\bv_{j_v}^{\top}\brackets{V_{n,h}\left(\be_{j_{\beta}}\right)-V}\bv_{j_v}}} > C_3\sqrt{\frac{\gamma_3p\log n}{nh^3}}\right) \leq 2(5n^{\gamma-\gamma_3/4})^p.\]
	Letting $\gamma_3=8\gamma$ and  combining the above inequality with \eqref{eq:1stepV-jbeta} lead to that, for any sufficiently large $\gamma>0$, there exists a constant $C$, such that
	\begin{equation}\label{eq:V-var}
		\sup_{\bv \in \mathbb{S}^{p-1}}\sup_{\be: \norm{\be-\be^*}_2 \leq \delta_{m,0}}\abs{(1-\E)\paren{\bv^{\top}\brackets{V_{n,h}\left(\be\right)-V}\bv}} \leq C\sqrt{\frac{p\log n}{nh^3}},
	\end{equation}
	with probability at least $1-n^{-\gamma/2}-2(5n^{-\gamma})^p$.
	
	In the following proof, we first consider any fixed $\be \in \mathbb{R}^{p}, \bv\in \mathbb{R}^{p}$ that satisfy $\norm{\be-\be^*}_2 \leq \delta_{m,0}$ and $\norm{\bv}_2=1$, and then apply the result to the specific $\bv_{j_v}$ and $\be_{j_{\beta}}$. 
	The computation of $\E\brackets{\bv^{\top}\brackets{V_{n,h}\left(\be\right)-V}\bv}$ is similar to \hyperlink{step3}{\textbf{Step 3}} in the proof of \eqref{eq:1stepU}, where we obtain that
	\[ \left(2F\left(-t \mid \bZ\right)-1\right)\rho\left(t \mid \bZ\right)
	=\sum_{k=1}^{2\alpha+1}M_k\left(\bZ\right)t^k,\]
	with uniformly bounded $M_k(\bZ)$ and $M_1(\bZ)=2F'(0\mid\bZ)\rho(0\mid\bZ)$.
	Recall that $\int_{-1}^1 xH^{\prime\prime}\left(x\right)\diff x=-1$ and $\int_{-1}^1 x^kH^{\prime\prime}\left(x\right)\diff x=0$ for $k=0$ and $2\leq k\leq\alpha$.	
	\begin{equation}
		\begin{aligned}
			&\mathbb{E}_{\cdot|\bZ} \left(\bv^{\top}V_{n,h}\left(\be\right)\bv\right)\\
			=&\frac{\left(\bv^{\top}\bz\right)^2}{h^2}\mathbb{E}_{\cdot|\bZ} \brackets{2\I\left(\zeta+\epsilon<0\right)-1}H^{\prime\prime}\left(\frac{\bZ^{\top}\Delta\left(\be\right)+\zeta}{h}\right)\\
			=&\frac{\left(\bv^{\top}\bZ\right)^2}{h} \int_{-1}^1 \left(2F\left(\bZ^{\top}\Delta\left(\be\right)-\xi h\mid\bZ\right)-1\right)\rho\left(\xi h-\bZ^{\top}\Delta\left(\be\right)\mid \bZ\right)H^{\prime\prime}\left(\xi\right)\diff \xi\\
			=&-2\left(\bv^{\top}\bZ\right)^2F^{\prime}\left(0\mid \bZ\right)\rho\left(0\mid \bZ\right)\int_{-1}^1\xi H^{\prime\prime}\left(\xi\right)\diff \xi + \frac{\left(\bv^{\top}\bZ\right)^2}{h} \int_{-1}^1\sum_{k=2}^{2\alpha+1}M_k\left(\bZ\right)\left(\xi h-\bZ^{\top}\Delta\left(\be\right)\right)^kH^{\prime\prime}\left(\xi\right)\diff \xi\\
			=&2\left(\bv^{\top}\bZ\right)^2F^{\prime}\left(0\mid \bZ\right)\rho\left(0\mid \bZ\right) + \left(\bv^{\top}\bZ\right)^2O\left(h^{\alpha}+(\bZ^{\top}\bu)\delta_{m,0}\right),
		\end{aligned}
	\end{equation} 
	where $\bu=\Delta\left(\be\right)/\norm{\Delta\left(\be\right)}_2$, and the last inequality follows from that, for $k\geq 2$, 
	\begin{align*}
		&\quad\frac{\left(\bv^{\top}\bZ\right)^2}{h} \int_{-1}^1\left(\xi h-\bZ^{\top}\Delta\left(\be\right)\right)^kH^{\prime\prime}\left(\xi\right)\diff \xi\\
		&= \left(\bv^{\top}\bZ\right)^2 \sum_{k'=0}^{k}\binom{k}{k'}h^{k'-1}\left(-\bZ^{\top}\Delta\left(\be\right)\right)^{k-k'}\int_{-1}^1\xi^{k'}H^{\prime\prime}\left(\xi\right)\diff \xi\\
		&=\left(\bv^{\top}\bZ\right)^2\left[(-k)\left(-\bZ^{\top}\Delta\left(\be\right)\right)^{k-1}+\sum_{k'=\alpha+1}^{k}\binom{k}{k'}h^{k'-1}\left(-\bZ^{\top}\Delta\left(\be\right)\right)^{k-k'}\int_{-1}^1\xi^{k'}H^{\prime\prime}\left(\xi\right)\diff \xi\right]\\
		&\lesssim \left(\bv^{\top}\bZ\right)^2\left[h^{\alpha}+\abs{\bZ^{\top}\Delta\left(\be\right)}^{k-1}\right].
	\end{align*} 
	Note that the constant hidden in $O(\cdot)$ does not depend on $\bv$ or $\be$. By Assumption \ref{A5} and $V=2\mathbb{E}\left[\rho\left( 0\mid \bZ\right)F^{\prime}\left( 0 \mid \bZ\right)\bZ\bZ^{\top}\right]$, we obtain
	\begin{equation}
		\sup_{\bv \in \mathbb{S}^{p-1}}\sup_{\be: \norm{\be-\be^*}_2 \leq \delta_{m,0} }\abs{\E \paren{\bv^{\top}\brackets{V_{n,h}\left(\be\right)-V}\bv}} = O\left(\delta_{m,0} + h^{\alpha}\right),
		\label{eq:EV}
	\end{equation}
	which completes the proof for \eqref{eq:1stepV} together with \eqref{eq:V-var}.

	
	Finally, \eqref{eq:beta1-beta*=V-1U}, \eqref{eq:1stepU}, and \eqref{eq:1stepV} directly yield \eqref{eq:1stepbetax} (i.e., \eqref{eq:1stepbeta}), given the assumption that $\Lambda_{\min}\left(V\right)>c_1^{-1}$ for some $c_1>0$.
	
\end{proof}

\noindent \textbf{Proof of Theorem \ref{thm:tstep-ld}}

\begin{proof}
	Without loss of generality, we assume $\lambda_h=1$. Then we have \[h_t=\max\left\{(p/n)^{1/(2\alpha+1)}, (p/m)^{2^t/(3\alpha)}\right\} \geq (p/n)^{1/(2\alpha+1)},\] 
	which implies the assumption $\frac{p\log n}{nh^3_{t}}=o(1)$ holds for any $t$, since $\frac{p\log n}{nh^3_{t}}\leq (p/n)^{\frac{2\alpha-2}{2\alpha+1}}\log n =O\big(n^{-(1-c_2)\frac{2\alpha-2}{2\alpha+1}}\log n\big) \to 0$ for $p=O(m^{c_2})=O(n^{c_2})$ and $\alpha \geq 2$. 
	Moreover, for any $t$,  $h^{\alpha}_t=\max\left\{(p/m)^{2^t/3}, (p/n)^{\alpha/(2\alpha+1)}\right\}$ and $\sqrt{p/nh_t} \leq (p/n)^{\alpha/(2\alpha+1)}$.
	
	We first show that, for any $t$,
	\begin{equation}
		\label{eq:tstep-ld-weaker}
		\left\| \hbe^{\,\left(t\right)}-\be^{*}\right\|_2=O_{\mathbb{P}}\left((p/n)^{\alpha/(2\alpha+1)}\sqrt{\log n}+(p/m)^{2^t/3}\right).
	\end{equation}
	Recall that by \eqref{eq:1stepbeta}, if $\norm{\hbe^{(t-1)}-\be^*}_2=O_{\Prob}\left(\delta_{m, t-1}\right)$ and $h_t=o(1)$,  we have that 
	\begin{equation}
		\label{eq:1stepbeta-t}
		\left\| \hbe^{\,\left( t\right)}-\be^{*}\right\|_2=O_{\mathbb{P}}\Bigg(\delta_{m,t-1}^2+h_t^{\alpha}+ \sqrt{\frac{p}{nh_t}}+\delta_{m,t-1}\sqrt{\frac{p\log n}{nh_t^3}}\Bigg),
	\end{equation}
	
	When $t=1$, $\delta_{m,0}=(p/m)^{1/3} \leq h_1=(p/m)^{\frac{2}{3\alpha}}$, which implies $\delta_{m,0}\sqrt{\frac{p\log n}{nh_1^3}}=O\left(\sqrt{\frac{p\log n}{nh_1}}\right)$. Then \eqref{eq:tstep-ld-weaker} holds since 
	\[\left\| \hbe^{\,\left(1\right)}-\be^{*}\right\|_2=O_{\mathbb{P}}\left(\sqrt{\frac{p\log n}{nh_1}}+\left(\frac{p}{m}\right)^{2/3}+h_1^\alpha\right)=O_{\mathbb{P}}\left(\left(\frac{p}{n}\right)^{\frac{\alpha}{2\alpha+1}}\sqrt{\log n}+\left(\frac{p}{m}\right)^{2/3}\right).\]
	Assume \eqref{eq:tstep-ld-weaker} holds for $t-1$, i.e.,
	\[\left\| \hbe^{\,\left(t-1\right)}-\be^{*}\right\|_2=O_{\mathbb{P}}\left((p/n)^{\alpha/(2\alpha+1)}\sqrt{\log n}+(p/m)^{2^{t-1}/3}\right).\]
	Then $\delta_{m,t-1}=(p/n)^{\alpha/(2\alpha+1)}\sqrt{\log n}+(p/m)^{2^{t-1}/3}$. Since $(p/n)^{\alpha/(2\alpha+1)}\sqrt{\log n} \ll (p/n)^{\frac{1}{2\alpha+1}} \leq h_t$ and $(p/m)^{2^{t-1}/3} \leq (p/m)^{\frac{2^{t}}{3\alpha}} \leq h_t$, it holds that $\delta_{m,t-1}=O\left(h_t\right)$.
	Then
	\[\left\| \hbe^{\,\left(t\right)}-\be^{*}\right\|_2=O_{\mathbb{P}}\left(\sqrt{\frac{p\log n}{nh_{t}}}+\delta_{m,t-1}^2+h_t^\alpha\right)=O_{\mathbb{P}}\left(\left(\frac{p}{n}\right)^{\frac{\alpha}{2\alpha+1}}\sqrt{\log n}+(p/m)^{2^{t}/3}\right),\]
	since $h^{\alpha}_t=\max\left\{(p/m)^{2^t/3}, (p/n)^{\alpha/(2\alpha+1)}\right\}$ and $\sqrt{p/nh_t} \leq (p/n)^{\alpha/(2\alpha+1)}$. Therefore, we have proved that \eqref{eq:tstep-ld-weaker} holds for all $t$ by induction.
	
	To see \eqref{eq:tstepbeta}, note that plugging $\delta_{m,t-1}=(p/n)^{\alpha/(2\alpha+1)}\sqrt{\log n}+(p/m)^{2^{t-1}/3}$ into \eqref{eq:1stepbeta-t} yields that
	\begin{equation}
		\label{eq:tstep-bound-detailed}
		\left\| \hbe^{\,\left( t\right)}-\be^{*}\right\|_2=O_{\mathbb{P}}\Bigg(\Big(\frac{p}{n}\Big)^{\frac{2\alpha}{2\alpha+1}}\log n+\Big(\frac{p}{m}\Big)^{\frac{2^{t}}{3}}+h_t^{\alpha}+ \sqrt{\frac{p}{nh_t}}+\left(\Big(\frac{p}{n}\Big)^{\frac{\alpha}{2\alpha+1}}\sqrt{\log n}+\Big(\frac{p}{m}\Big)^{\frac{2^{t-1}}{3}}\right)\sqrt{\frac{p\log n}{nh_t^3}}\Bigg).
	\end{equation}
	Since $\alpha>1$,
	\[\Big(\frac{p}{n}\Big)^{\frac{\alpha}{2\alpha+1}}\sqrt{\log n} \cdot \sqrt{\frac{p\log n}{nh_t^3}} = \Big(\frac{p}{n}\Big)^{\frac{\alpha}{2\alpha+1}} (\log n) \cdot \sqrt{\frac{p}{nh_t^3}} \leq \Big(\frac{p}{n}\Big)^{\frac{\alpha}{2\alpha+1}}(\log n) \cdot (p/n)^{\frac{\alpha-1}{2\alpha+1}} \lesssim \Big(\frac{p}{n}\Big)^{\frac{\alpha}{2\alpha+1}}.\]
	Therefore, the rate in \eqref{eq:tstep-bound-detailed} is upper bounded by \[O\left(\Big(\frac{p}{n}\Big)^{
		\frac{\alpha}{2\alpha+1}}+\Big(\frac{p}{m}\Big)^{\frac{2^{t}}{3}}+\Big(\frac{p}{m}\Big)^{\frac{2^{t-1}}{3}}\Big(\frac{p}{n}\Big)^{\frac{\alpha-1}{2\alpha+1}}\sqrt{\log n}\right),\]
	which completes the proof.
\end{proof}

\subsection{Theoretical Results for the Inference of \mSMSE}\label{sec:pf-inference}

To show the asymptotic normality, we first prove an important Lemma about $U_{n,h}\left(\be\right)$ defined by $\eqref{eq:def:Unh}$.

\begin{lemma} \label{lem:CentralLimit}
	Under Assumptions \ref{A1}--\ref{A5}, if $h=o\left(1\right)$, $\norm{\be-\be^*}_2 \leq \delta$, $\delta=o\left(\min\{h^{\alpha/2}, h/\sqrt{p\log n}\}\right)$ then 
	\begin{equation}\label{eq:asym-EU}
		\E \left[\bv^{\top}U_{n,h}\left(\be\right)\right] = \bv^{\top}Uh^{\alpha}+o\left(h^{\alpha}\right),
	\end{equation}
	and
	\begin{equation}\label{eq:asym-U}
		\left(1-\E\right)\sqrt{nh}\left[\bv^{\top}U_{n,h}\left(\be\right)\right] / \sqrt{\bv^{\top}V_{s}\bv} \stackrel{d}{\longrightarrow} \mathcal{N}\left(0, 1\right),
	\end{equation}
	for any $\bv \in \R^p \setminus\{\bm{0}\}$.
\end{lemma}

\begin{proof}[Proof of Lemma \ref{lem:CentralLimit}]
	
	Recall that
	\[M_\alpha\left(\bZ\right)=\sum_{k=1}^\alpha\frac{2(-1)^k}{\left(\alpha-k\right)!k!}F^{\,\left(k\right)}\left(0\mid\bZ\right)\rho^{\,\left(\alpha-k\right)}\left(0\mid\bZ\right),\]
	and
	\[U:=\pi_U\mathbb{E}\left(\sum_{k=1}^{\alpha}\frac{2(-1)^{k+1}}{k!\left(\alpha-k\right)!}F^{\,\left(k\right)}\left(0\mid\bZ\right)\rho^{\,\left(\alpha-k\right)}\left(0\mid\bZ\right)\bZ\right),\]
	where $\pi_U$ is defined in Assumption \ref{A1}. For any $\bv \in \mathbb{S}^{p-1}$,
	the computation in \eqref{eq:EvU-computation} yields that 
	\begin{align*}
		\E \left[\bv^{\top}U_{n,h}\left(\be\right)\right]&=-\E\brackets{\paren{\bZ^{\top}\bv}\pi_UM_{\alpha}\left( \bZ\right)h^{\alpha}}+O\left(\delta^2+h^{\alpha+1}\right)\\
		&=\bv^{\top}U\cdot h^{\alpha}+o\left(h^{\alpha}\right),
	\end{align*}
	where $o(\cdot)$ hides a constant that does not depend on $\be$. This completes the proof of \eqref{eq:asym-EU}.
	
	To show \eqref{eq:asym-U}, further recall that in \hyperlink{step2}{\textbf{Step 2}} of the proof of \eqref{eq:1stepU}, we show that 
	\[\E_{\cdot\mid \bZ}\brackets{\bv^{\top}U_{h,i}\left(\be^*\right)}^2=\frac{\paren{\bv^{\top}\bZ}^2}{h}\int_{-1}^1\brackets{H^{\prime}\left(\xi\right)}^2 \rho\left(0\mid \bZ\right)\diff \xi + O\left((\bv^{\top}\bZ)^2\right).\]
	Since $V_{s}:=\pi_V \E  \rho\left(0\mid \bZ\right) \bZ \bZ^{\top} $, $\pi_V:=\int_{-1}^1\brackets{H^{\prime}\left(\xi\right)}^2 \diff \xi$, and $\E(\bv^{\top}\bZ)^2 < \infty$,  it holds that \[\E\brackets{\bv^{\top}U_{h,i}\left(\be^*\right)}^2=\frac{1}{h}\bv^{\top}V_{s}\bv+O(1).\]
	Therefore,
	\begin{equation}
		\label{eq:U-var}
		\begin{aligned}
			&\quad\mathrm{var}\brackets{\sqrt{h}\bv^{\top}U_{h,i}\left(\be^*\right)}\\
			&=h\braces{\E\brackets{\bv^{\top}U_{h,i}\left(\be^*\right)}^2-\paren{\E \brackets{\bv^{\top}U_{h,i}\left(\be^*\right)}}^2}\\
			&=\bv^{\top}V_{s}\bv+o\left(1\right),
		\end{aligned}
	\end{equation}
	
	By CLT and Slusky's Theorem,
	\[\sqrt{nh}\paren{1-\E}\brackets{\bv^{\top}U_{n,h}\left(\be^*\right)} / \sqrt{\bv^{\top}V_{s}\bv} \stackrel{d}{\longrightarrow} \mathcal{N}\left(0, 1 \right).\]
	Furthermore, in \hyperlink{step1}{\textbf{Step 1}}, we show that
	\[\sup_{\be:\norm{\be-\be^*}_2 \leq \delta}\abs{\left(1-\E\right)\left[\bv^{\top}U_{n,h}\left(\be\right)-\bv^{\top}U_{n,h}\left(\be^*\right) \right]}=O_{\Prob}\left(\delta\sqrt{\frac{p\log n}{nh^3}}\right),\]
	which yields \eqref{eq:asym-U} if $\delta=o\left(h/\sqrt{p\log n}\right)$.
	
\end{proof}

\noindent \textbf{Proof of Theorem \ref{thm:normality}}

We now give the proof for Theorem \ref{thm:normality} using Lemma \ref{lem:CentralLimit}.

\begin{proof}

	By Theorem \ref{thm:tstep-ld}, when 
	$T \geq  \log_2\left(\frac{3\alpha}{2\alpha+1}\cdot\frac{\log (n/p)}{\log (m/p)}\right)$,
	it holds that $h_{T}=
	\left(p/n\right)^{\frac{1}{2\alpha+1}}$ and $(p/m)^{2^{T}/3} \lesssim (p/n)^{\alpha/(2\alpha+1)}$. By taking $h_{T+1}=(\lambda_h / n)^{1/(2\alpha+1)}$, we have \[\norm{\hbe^{\left(T\right)}-\be^*}_2 = O_{\Prob}\left( (p/n)^{\frac{\alpha}{2\alpha+1}}\right)=o_{\Prob}\left(\max\{h_{T+1}^{\alpha/2}, h_{T+1}/\sqrt{p\log n}\}\right),\] where we use the assumption  $p=o\left(n^{\frac{2(\alpha-1)}{4\alpha+1}}(\log n)^{-\frac{2\alpha+1}{4\alpha+1}}\right)$. Hence, the assumptions of Lemma \ref{lem:CentralLimit} hold for $\hbe^{(T)}$ and $h_{T+1}$. 
	Since $\sqrt{nh_{T+1}}=\sqrt{\lambda_h}h_{T+1}^{-\alpha}$, we obtain by Lemma \ref{lem:CentralLimit} that
	%
	\[\frac{\sqrt{nh_{T+1}}\cdot \bm{\vartheta}^{\top}V^{-1} U_{n,h_{T+1}}\left(\hbe^{\left(T\right)}\right)-\sqrt{\lambda_h}\bm{\vartheta}^{\top}V^{-1}U}{\sqrt{\bm{\vartheta}^{\top}V^{-1}V_{s}V\bm{\vartheta}}} \stackrel{d}{\longrightarrow}\mathcal{N}(0, 1).\]
	It is easy to verify that $p\log n/(nh_{T+1}^3)=o(1)$ and $h_{T+1}=O(\delta_{m,T})$. By \eqref{eq:1stepV} and Assumption \ref{A4}, $\norm{V^{-1}_{n,h_{T+1}}\left(\hbe^{\left(T\right)}\right)-V^{-1}}_2=o_{\mathbb{P}}\left(1\right)$, which yields
	\[\frac{\sqrt{nh_{T+1}}\cdot \bm{\vartheta}^{\top}V_{n,h_{T+1}}^{-1}\left(\hbe^{\left(T\right)}\right) U_{n,h_{T+1}}\left(\hbe^{\left(T\right)}\right)-\sqrt{\lambda_h}\bm{\vartheta}^{\top}V^{-1}U}{\sqrt{\bm{\vartheta}^{\top}V^{-1}V_{s}V^{-1}\bm{\vartheta}}}\stackrel{d}{\longrightarrow} \mathcal{N}\left(0, 1 \right).\]
	Plugging in $h_{T+1}=
	\left(\frac{\lambda_h}{n}\right)^{\frac{1}{2\alpha+1}}$ and $\hbe^{\left(T+1\right)}-\be^*= V_{n,h_{T+1}}^{-1}\left(\hbe^{\left(T\right)}\right) U_{n,h_{T+1}}\left(\hbe^{\left(T\right)}\right)$ leads to $\eqref{eq:asym-mSMSE}$. Using the asymptotic bias and variance, we have
	\begin{align*}
		&\quad \E\left[\left(\bm{\vartheta}^{\top}\hbe^{\left(T+1\right)}-\bm{\vartheta}^{\top}\be^*\right)^2\right]\\
		&\asymp n^{-\frac{2\alpha}{2\alpha+1}}\brackets{\lambda_h^{-\frac{1}{2\alpha+1}}\bm{\vartheta}^{\top}V^{-1}V_{s}V^{-1}\bm{\vartheta}+\lambda_h^{\frac{2\alpha}{2\alpha+1}}U^{\top}V^{-1}\bm{\vartheta}\bm{\vartheta}^{\top}V^{-1}U},
	\end{align*}
	by minimizing which it is straightforward to obtain the optimal $\lambda_h^*$ given in \eqref{eq:lambda_opt}.
\end{proof}

\noindent \textbf{Estimators for $V$, $U$ and $V_s$}

Now we formally define the estimators for $V$, $U$ and $V_s$. Proposition \ref{thm:1step-ld-supp} in Section \ref{sec:pf-mSMSE} has already implies that, when 
$T\geq \log_2\left(\frac{3\alpha}{2\alpha+1}\cdot\frac{\log (n/p)}{\log (m/p)}\right)$, it holds that $V_{n,h_T}\left(\hbe^{\left(T\right)}\right)\stackrel{p}{\longrightarrow}V$, where $V_{n,h}(\be)$ is defined in \eqref{eq:def:Vnh}, so we can use $\widehat{V}:=V_{n,h_T}\left(\hbe^{\left(T\right)}\right)$ to estimate $V$. It remains to provide estimators for $U$ and $V_{s}$. 

\begin{thm}\label{thm:estimators-WVa}
	Assume assumptions in Theorem \ref{thm:tstep-ld} hold. Let $h_\kappa=(p/n)^{\frac{\kappa}{2\alpha+1}}$ for some $0<\kappa<1$. Define
	\begin{equation}\label{eq:est-U}
		\widehat{U}:=\frac{1}{nh_{\kappa}^{\alpha+1}}\sum_{i=1}^ny_iH^{\prime}\left( \frac{x_i+\bz_{i}^{\top}\hbe^{\left(T\right)}}{h_{\kappa}}\right) \bz_i,
	\end{equation}
	\begin{equation}\label{eq:est-Va}
		\widehat{V_{s}}:=\frac{1}{nh_T}\sum_{i=1}^n\brackets{H^{\prime}\left( \frac{x_i+\bz_i^{\top}\hbe^{\left( T\right)}}{h_T}\right)}^2\bz_i\bz_i^{\top} .
	\end{equation} When 
	$T\geq \log_2\left(\frac{3\alpha}{2\alpha+1}\cdot\frac{\log (n/p)}{\log (m/p)}\right)$,
	we have
	\begin{equation}
		\norm{\widehat{U}- U}_2 =o_{\Prob}(1), \quad \norm{\widehat{V_{s}} -V_{s}}_2=o_{\Prob}(1).
	\end{equation}
\end{thm}
\begin{proof}[Proof]
	When $T \geq  \log_2\left(\frac{3\alpha}{2\alpha+1}\cdot\frac{\log (n/p)}{\log (m/p)}\right)$, we have
	\[(p/n)^{\frac{\kappa}{2\alpha+1}}=h_\kappa\gg h_{T}=\left(p/n\right)^{\frac{1}{2\alpha+1}},\]
	and
	\[\norm{\hbe^{\left(T\right)}-\be^*}_2=O_{\Prob}(\delta_{m,T}), \quad \delta_{m,T}= (p/n)^{\frac{\alpha}{2\alpha+1}}=o\paren{h_\kappa^{\alpha}}.\]
	It is easy to verify that $p\log n/(nh_\kappa^3)=o(1)$.
	By the proof of Equation \eqref{eq:1stepU}, we have that 
	\[\norm{U_{n, h_\kappa}\left(\hbe^{\left(T\right)}\right) - U h_\kappa^{\alpha}}_2 = O_{\Prob}\left((p/n)^{\alpha/(2\alpha+1)}\right).\]
	Also, it holds that $\norm{V_{n,h_\kappa}\left(\hbe^{\left(T\right)}\right)-V}_2=o_{\Prob}(1) $ by \eqref{eq:1stepV}. Note that \[\widehat{U}=\left(h_\kappa\right)^{-\alpha}\brackets{U_{n,h_\kappa}\left(\hbe^{\left(T\right)}\right)-V_{n,h_\kappa}\left(\hbe^{\left(T\right)}\right)\paren{\hbe^{\left(T\right)}-\be^*}},\]
	and $\left(h_\kappa\right)^{-\alpha}(p/n)^{\alpha/(2\alpha+1)}=o(1)$,
	which yields 	$\norm{\widehat{U}-U}_2=o_{\Prob}(1)$.
	
	To prove $\norm{\widehat{V_{s}}-V_{s}}_2=o_{\Prob}(1)$, recall that \[\widehat{V_{s}}:=\frac{1}{nh_T}\sum_{i=1}^n\brackets{H^{\prime}\left(\frac{x_i+\bz_i^{\top}\hbe^{\left(T\right)}}{h_T}\right)}^2\bz_i\bz_i^{\top}.\]
	For any $\bv \in \mathbb{S}^{p}$,
	\begin{align*}
		\E \left[\bv^{\top}\widehat{V_{s}}\bv\right]&=\E\left(\E_{\cdot\mid\bZ} \bv^{\top}\widehat{V_{s}}\bv\right)\\
		&=\E\brackets{ \left(\bZ^{\top}\bv\right)^2\int_{-1}^1\brackets{H^{\prime}\left(\xi\right)}^2\rho\left(\xi h_T-\bZ^{\top}\left(\hbe^{\left(T\right)}-\be^*\right)\mid\bZ\right)\diff\xi}\\
		&=\E\left[\left(\bZ^{\top}\bv\right)^2\pi_V\rho\left(0\mid \bZ\right) +  O\left(h_T+\delta_{m,T}(\bZ^{\top}\bu)\right)\right]\\
		&=\bv^{\top}V_{s}\bv + o\left(1\right),
	\end{align*}
	where $ \bu=(\hbe^{\left(T\right)}-\be^*)/\norm{\hbe^{\left(T\right)}-\be^*}_2 $ and hence $\E[(\bZ^{\top}\bu)^2] < \infty$.
	Also, since $H^{\prime}\left(x\right)$ is bounded, $\mathrm{var} \left(\bv^{\top}\widehat{V_{s}}\bv\right)=O\left(\frac{1}{nh_T^2}\right)=o\left(1\right)$. Then $\norm{\widehat{V_{s}}-V_{s}}_2=o_{\Prob}(1)$ is proved by Chebyshev's inequality.
	
\end{proof}

\subsection{Proof of the Results for \AvgSMSE}
\label{sec:pf-AvgSMSE}

\noindent \textbf{Proof of Theorem \ref{thm:asym-DC}}

\begin{proof}
	Since $\hbe_{\rm{SMSE}, \ell}$ is the minimizer of $F_{h,\ell}\left(\be\right)=\sum_{i \in \mathcal{H}_\ell}\left(-y_i\right)H\left(\frac{x_i+\bz_i^{\top}\be}{h}\right)$, we have $\nabla_{\be} F_{h,\ell}\left(\hbe_{\rm{SMSE},\ell}\right)=0$. By Taylor's expansion of $\nabla_{\be} F_{h,\ell}$  at $\be^*$, we have
	\[0=\nabla_{\be} F_{h,\ell}\left(\be^*\right)+\nabla_{\be}^2 F_{h,\ell}\left(\breve\be_{\ell}\right)\left(\hbe_{\rm{SMSE},\ell}-\be^*\right),\]
	where $\breve\be_{\ell}$ is between $\hbe_{\rm{SMSE},\ell}$ and $\be^*$.

	Define \[U_{m,\ell,h}\left(\be\right):=\frac{1}{m}\sum_{i \in \mathcal{H}_\ell}U_{h, i}\left(\be\right),\quad V_{m,\ell,h}\left(\be\right):=\frac{1}{m}\sum_{i \in \mathcal{H}_\ell}V_{h, i}\left(\be\right),\] where 
	$U_{h, i}\left(\be\right)$ and $V_{h, i}\left(\be\right)$ are defined in \eqref{eq:def:Uhi} and \eqref{eq:def:Vhi}. 
	By definition, $\nabla_{\be}^2 F_{h,\ell}\left(\breve\be_{\ell}\right)=V_{m,\ell,h}\left(\breve\be_{\ell}\right)$ and $\nabla_{\be} F_{h,\ell}\left(\be^*\right)=-U_{m,\ell,h}\left(\be^*\right)$, which yields \[\hbe_{\rm{SMSE}, \ell}-\be^*=V^{-1}_{m,\ell,h}\left(\breve\be_{\ell}\right)U_{m,\ell,h}\left(\be^*\right).\] 
	Following the proof of Proposition \ref{thm:1step-ld}, if $p\log m/(mh^3)=o(1)$, we have that  for any sufficiently large $\gamma$,  there exists a constant $C$ such that
	\[\sup_{\ell}\sup_{\be:\norm{\be-\be^*}\leq \delta}\norm{V-V_{m,\ell,h}\left(\be\right)}_2\leq C\left(\sqrt{\frac{p\log m}{mh^3}}+h^{\alpha}+\delta\right),\]
	with probability at least $1-Lm^{-\gamma/2}-2L\left(5m^{-\gamma}\right)^{p}$. The assumptions $L=o\left(m^{2(\alpha-1)/3}/(p\log m)^{(2\alpha+1)/3}\right)$ and $h=(\lambda_h/n)^{1/(2\alpha+1)}$ ensure that $p\log m/(mh^3)=o(1)$ and $L(m^{-\gamma/2}-2\left(5m^{-\gamma}\right)^{p})=o(1)$ for sufficiently large $\gamma$. Moreover, Theorem 1 in \cite{horowitz1992smoothed} showed that $\norm{\hbe_{\rm{SMSE}, \ell}-\be^{*}}_2=o(1)$ almost surely for all $\ell$, and thus 
	there exists a uniform high-probability bound for $V^{-1}-V^{-1}_{m,\ell,h}\big(\breve\be_{\ell}\big)$ over all machines:
	\[\sup_{\ell}\norm{V^{-1}-V^{-1}_{m,\ell,h}\left(\breve\be_{\ell}\right)}_2=o_{\mathbb{P}}\left(1\right),\] 
	which implies that
	\[\hbe_{\rm{{\AvgSMSE}}}-\be^*=\frac{1}{L}\sum_{\ell=1}^{L}\left(\hbe_{\rm{SMSE}, \ell}-\be^*\right) = V^{-1}U_{n,h}\left(\be^*\right) + U_{n,h}\left(\be^*\right) o_{\mathbb{P}}\left(1\right),\]
	using the facts that $U_{n,h}\left(\be\right)=\frac{1}{n}\sum_{i=1}^n U_{h, i}\left(\be\right)$ and $n=mL$.

	By Lemma \ref{lem:CentralLimit}, for any $\bm{\vartheta} \in \mathbb{S}^{p-1} \setminus \{0\}$, 
	\[\E \big[\bm{\vartheta}^{\top}V^{-1} U_{n,h}\left(\be^*\right)\big] = \bm{\vartheta}^{\top} V^{-1}Uh^{\alpha}+o\left(h^\alpha\right),\]
	and
	\[\sqrt{mLh}\cdot(1-\E)\left[\bm{\vartheta}^{\top}V^{-1}U_{n,h}\left(\be^*\right)\right] / \sqrt{\bm{\vartheta}^{\top}V^{-1}V_{s}V^{-1}\bm{\vartheta}}\stackrel{d}{\longrightarrow} \mathcal{N}\left(0, 1 \right).\]

	When $h = \left(\frac{\lambda_h}{n}\right)^{\frac{1}{2\alpha+1}}$, we have $\sqrt{\lambda_h}h^{-\alpha} = \sqrt{mLh}$, and thus
	\begin{align*}
		& \quad \frac{\sqrt{mLh}\bm{\vartheta}^{\top}\left(\hbe_{\rm{{\AvgSMSE}}}-\be^*\right) - \sqrt{\lambda_h}\bm{\vartheta}^{\top}V^{-1}U}{\sqrt{\bm{\vartheta}^{\top}V^{-1}V_{s}V^{-1}\bm{\vartheta}}} \\
		& = \sqrt{mLh}(1-\E)\left[\bm{\vartheta}^{\top}V^{-1} U_{n,h}\left(\be^*\right)\right] /\sqrt{\bm{\vartheta}^{\top}V^{-1}V_{s}V^{-1}\bm{\vartheta}} + o(1)\\
		&\quad + o_{\Prob}\left(1\right)\sqrt{mLh}\left(\bm{\vartheta}^{\top}U_{n,h}\left(\be^*\right)-\E \bm{\vartheta}^{\top}U_{n,h}\left(\be^*\right)+\E \bm{\vartheta}^{\top}U_{n,h}\left(\be^*\right)\right) / \sqrt{\bm{\vartheta}^{\top}V^{-1}V_{s}V^{-1}\bm{\vartheta}}\\
		& \stackrel{d}{\longrightarrow} \mathcal{N}\left(0, 1\right),
	\end{align*}
	which proves \eqref{eq:asym-DC}. Furthermore,
	if $h \gtrsim \left(\frac{1}{mL}\right)^{\frac{1}{2\alpha+1}}$, we have  $h^{-\alpha} \lesssim \sqrt{mLh}$, and thus 
	\[\bm{\vartheta}^{\top}\left(\hbe_{\rm{{\AvgSMSE}}}-\be^*\right)-  \bm{\vartheta}^{\top}V^{-1} U = o_{\mathbb{P}}\left(h^{\alpha}\right).\]
	If $h \lesssim \left(\frac{1}{mL}\right)^{\frac{1}{2\alpha+1}}$ but still satisfies $\frac{p\log m}{mh^3}=o\left(1\right)$, we have $h^{-\alpha} \gtrsim \sqrt{mLh}$, and thus 
	\[\sqrt{mLh} \E \left[\bm{\vartheta}^{\top} U_{n,h}\left(\be^*\right)\right]=o_{\mathbb{P}}\left(1\right),\] which yields
	\[\sqrt{mLh}\bm{\vartheta}^{\top}\left(\hbe_{\rm{{\AvgSMSE}}}-\be^*\right) /\sqrt{\bm{\vartheta}^{\top}V^{-1}V_{s}V^{-1}\bm{\vartheta}} \stackrel{d}{\longrightarrow} \mathcal{N}\left(0, 1\right).\]

\end{proof}

\subsection{Proof of the Results for Data Heterogeneity}

\label{sec:pf-weighted}
\noindent \textbf{Proof of Theorem \ref{thm:asym-DC-weighted}}

\begin{proof} \label{pf:thm:asym-DC-weighted}
	
	In the proof of Theorem \ref{thm:asym-DC}, we show that $\hbe_{\rm{{SMSE},\ell}} - \be^*=V^{-1}_{m_{\ell},\ell,h}\big(\breve{\be}_{\ell}\big)U_{m_{\ell},\ell, h}(\be^*)$ for any $\ell$, where $\breve{\be}_{\ell}$ is between $\be^*$ and $\hbe_{\rm{{SMSE},\ell}}$. The proof of \eqref{eq:1stepV} shows that, if $p\log m_{\ell}/(m_{\ell}h^3)=o(1)$, then 
	for any sufficiently large $\gamma$,  there exists a constant $C_{\ell}$ such that
	\[\sup_{\be:\norm{\be-\be^*}\leq \delta}\norm{V_{\ell}-V_{m,\ell,h}\left(\be\right)}_2\leq C_{\ell}\left(\sqrt{\frac{p\log m_{\ell}}{m_{\ell}h^3}}+h^{\alpha}+\delta\right),\]
	with probability at least $1-m_{\ell}^{-\gamma/2}-2\left(5m_{\ell}^{-\gamma}\right)^{p}$. The dependence of $C_{\ell}$ on $\ell$ is due to the distributions $\rho_{\ell}$ and $F_{\ell}$ and their derivatives, which can be uniformly upper bounded over all $\ell$ by Assumptions \ref{A2'} and \ref{A3'}. Moreover, the constant $\gamma$ does not depend on $\ell$. Therefore,
	if $h=n^{-1/(2\alpha+1)}$ and $\min_{\ell}m_{\ell} \gtrsim pn^{3/(2\alpha+1)}\log n$, which satisfies  $p\log \min{m_{\ell}}/(\min{m_{\ell}}h^3)=o(1)$, then 
	for any sufficiently large $\gamma$,  there exists a constant $C$ (hidden in $o(1)$) such that
	\[\sup_{\ell}\norm{V^{-1}_{\ell}-V^{-1}_{m_{\ell},\ell,h}\big(\breve{\be}_{\ell}\big)}_2= o(1),\] 
	with probability at least $1-\sum_{\ell=1}^Lm_{\ell}^{-\gamma/2}-2\sum_{\ell=1}^L\left(5m_{\ell}^{-\gamma}\right)^{p}$. The condition $\min_{\ell}m_{\ell} \gtrsim pn^{3/(2\alpha+1)}\log n$ also ensures that $\sum_{\ell=1}^Lm_{\ell}^{-\gamma/2}+2\sum_{\ell=1}^L\left(5m_{\ell}^{-\gamma}\right)^{p}=o(1)$ for sufficiently large $\gamma$. Therefore, we obtain that
	\begin{equation}
		\hbe_{\rm{{\wAvgSMSE}}} - \be^* = \sum_{\ell=1}^L W_{\ell}V_{\ell}^{-1}U_{m_{\ell},\ell, h}(\be^*) + o_{\Prob}(1)\sum_{\ell=1}^L W_{\ell}U_{m_{\ell},\ell, h}(\be^*).
	\end{equation}
	
	By  \hyperlink{step3}{Step 3} in the proof of \eqref{eq:1stepU} and the uniformness in Assumptions \ref{A2'} and \ref{A3'}, for any $\bv$ satisfying $\norm{\bv}_2=1$, it holds that
	\begin{equation}
		\E\left[\sum_{\ell=1}^L \bv^{\top}W_{\ell}V_{\ell}^{-1}U_{m_{\ell},\ell, h}(\be^*)\right] = \sum_{\ell=1}^L \bv^{\top}W_{\ell}V_{\ell}^{-1}U_{\ell}h^{\alpha}+o\left(h^{\alpha}\right)=\bv^{\top}B_{\rm{{\wAvgSMSE}}}+o\left(h^{\alpha}\right),
	\end{equation}
	where $B_{\rm{{\wAvgSMSE}}}:=\sum_{\ell=1}^L W_{\ell}V_{\ell}^{-1}U_{\ell}h^{\alpha}$.
	
	Define $\Phi_i:=(1-\E)m^{-1}_{\ell}\bv^{\top}W_{\ell}V_{\ell}^{-1}U_{h, i}(\be^*)$ for $i \in \mathcal{H}_{\ell}$. (Notice that the definitions of $\Phi_i$ may be different in the proof of different
	theorems.) Then
	\begin{equation}
		(1-\E)\left[\sum_{\ell=1}^L \bv^{\top}W_{\ell}V_{\ell}^{-1}U_{m_{\ell},\ell, h}(\be^*) \right]  = 	(1-\E)\left[\sum_{\ell=1}^L \sum_{i \in \mathcal{H}_\ell} m^{-1}_{\ell}\bv^{\top}W_{\ell}V_{\ell}^{-1}U_{h, i}(\be^*) \right]=\sum_{i=1}^n\Phi_i.
	\end{equation}
	Recall that by \eqref{eq:U-var},  
	\[s_n^2 := \sum_{i=1}^n \mathrm{var}\left[\Phi_i\right] = h^{-1}\sum_{\ell=1}^Lm_{\ell}^{-1}\bv^{\top}W_{\ell}V_{\ell}^{-1}\left[ V_{s,\ell}+o(1)\right]V_{\ell}^{-1}W_{\ell}^{\top}\bv.\]
	Similar to the proof of \eqref{eq:U-var},  we have that
	\[\E\left[\Phi_i^4\right]\lesssim \frac{\norm{W_\ell}_2^4}{m_{\ell}^4h^3}, \quad \quad \mathrm{when \,\,} i \in \mathcal{H}_{\ell}.\]
	
	By Assumptions  \ref{A4'} and \ref{assumption:weightrestriction}, the Lindeberg's condition can be verified as follows:
	\[\frac{1}{s_n^2} \sum_{i=1}^n\E\left[\left(\Phi_i\right)^2\I\left(\abs{\Phi_i}>\varepsilon s_n\right)\right]\leq \sum_{i=1}^n\frac{\E\left[\Phi_i^4\right]}{\varepsilon^2s_n^4}\lesssim \frac{\sum_{\ell=1}^L \left[\norm{W_{\ell}}_2^4/m^3_{\ell}\right]}{h\sum_{\ell=1}^L \left[ \norm{W_{\ell}}_2^2/m_{\ell}\right]} \asymp \frac{1}{nh}\to 0.\]
	Therefore, since $s_n^2 \to \bv^{\top}\Sigma_{\rm{{\wAvgSMSE}}}\bv$, we have
	\[	\left(\bv^{\top}\Sigma_{\rm{{\wAvgSMSE}}}\bv\right)^{-1/2}\sum_{i=1}^n \Phi_i \stackrel{d}{\longrightarrow} \mathcal{N}\left(0, 1\right),\]
	where $\Sigma_{\rm{{\wAvgSMSE}}}:=h^{-1}\sum_{\ell=1}^Lm_{\ell}^{-1}W_{\ell}V_{\ell}^{-1}V_{s,\ell}V_{\ell}^{-1}W_{\ell}^{\top}$.
	Now we do the following decomposition.
	\begin{equation}
		\label{eq:asym-weighted-decomp}
		\begin{aligned}
			& \quad \left(\bv^{\top}\Sigma_{\rm{{\wAvgSMSE}}}\bv\right)^{-1/2}\bv^{\top}\left(\hbe_{\rm{{\wAvgSMSE}}} - \be^*-B_{\rm{{\wAvgSMSE}}}\right)\\
			&= \left(\bv^{\top}\Sigma_{\rm{{\wAvgSMSE}}}\bv\right)^{-1/2}\sum_{i=1}^n \Phi_i  +	\left(\bv^{\top}\Sigma_{\rm{{\wAvgSMSE}}}\bv\right)^{-1/2}o(h^{\alpha})\\
			&\quad+o_{\Prob}(1)\left(\bv^{\top}\Sigma_{\rm{{\wAvgSMSE}}}\bv\right)^{-1/2}\sum_{\ell=1}^L \bv^{\top}W_{\ell}U_{m_{\ell},\ell, h_{\ell}}(\be^*).
		\end{aligned}
	\end{equation}
	Assumptions \ref{A4'} and \ref{assumption:weightrestriction} ensures that 
	\[\left(\bv^{\top}\Sigma_{\rm{{\wAvgSMSE}}}\bv\right)^{1/2} \asymp 1/\sqrt{nh} \asymp h^{\alpha},\] with $h = n^{-1/(2\alpha+1)}$, which implies that the second term on the RHS of \eqref{eq:asym-weighted-decomp} is $o(1)$. Meanwhile, it is straightforward to show that $\left(\bv^{\top}\Sigma_{\rm{{\wAvgSMSE}}}\bv\right)^{-1/2}\sum_{\ell=1}^L W_{\ell}U_{m_{\ell},\ell, h_{\ell}}(\be^*)$ is bounded in probability by letting $V_{\ell}=I_{p\times p}$ in the previous analysis. Hence,
	\[\left(\bv^{\top}\Sigma_{\rm{{\wAvgSMSE}}}\bv\right)^{-1/2}\bv^{\top}\left(\hbe_{\rm{{\wAvgSMSE}}} - \be^*-B_{\rm{{\wAvgSMSE}}}\right) = \left(\bv^{\top}\Sigma_{\rm{{\wAvgSMSE}}}\bv\right)^{-1/2}\sum_{i=1}^n \Phi_i  + o_{\Prob}(1),\]
	which converges to $\mathcal{N}(0, 1)$ in distribution. Then rewrite
	\[\Sigma_{\rm{{\wAvgSMSE}}}=n^{-2\alpha/(2\alpha+1)}\sum_{\ell=1}^L\frac{n}{m_{\ell}}W_\ell V_{\ell}^{-1}V_{s,\ell}V_{\ell}^{-1}W_\ell^{\top}, \quad B_{\rm{{\wAvgSMSE}}}=n^{-\alpha/(2\alpha+1)}\sum_{\ell=1}^L W_{\ell}V_{\ell}^{-1}U_{\ell},\] 
	which completes the proof of Theorem \ref{thm:asym-DC-weighted}.
	
	Additionally, we show that, for any vector $\bv$,   $\bv^{\top}\Sigma_{\rm{{\wAvgSMSE}}}\bv$
	is minimized at  
	\[
	W^{*,\rm{{\wAvgSMSE}} }_{\ell}=\left(\sum^L_{\ell=1}m_{\ell}V_{\ell}V^{-1}_{s,\ell}V_{\ell}\right)^{-1} m_{\ell}V_{\ell}V^{-1}_{s,\ell}V_{\ell}.
	\] 
	We are to solve the following optimization problem:
	\[\min_{W_\ell} \sum_{\ell=1}^L\frac{n}{m_{\ell}}\bv^{\top}W_\ell V_{\ell}^{-1}V_{s,\ell}V_{\ell}^{-1}W_\ell^{\top}\bv, \, \mathrm{s.t.} \sum_{\ell=1}^L W_\ell = I_{p \times p}.\]
	The Lagrangian is
	\[ \sum_{\ell=1}^L\frac{n}{m_{\ell}}\bv^{\top}W_\ell V_{\ell}^{-1}V_{s,\ell}V_{\ell}^{-1}W_\ell^{\top}\bv+\left\langle \Lambda, \sum_{\ell=1}^L W_\ell - I_{p \times p}\right\rangle.\]
	By taking derivative w.r.t $W_{\ell}$ and letting the derivative be zero, we obtain $2\frac{n}{m_{\ell}}\bv\bv^{\top }W_\ell V_{\ell}^{-1}V_{s,\ell}V_{\ell}^{-1}+\Lambda=\bm 0$ or $2n\bv\bv^{\top}W_{\ell} =-\Lambda m_{\ell}V_{\ell}V^{-1}_{s,\ell}V_{\ell}$. By the constraint that $\sum_{\ell=1}^L W_\ell = I_{p \times p}$, we obtain that a sufficient condition for $\{W_{\ell}^*\}$ to be a minimizer is \[\bv\bv^{\top}W_{\ell}^* =\bv\bv^{\top}\big(\sum_{\ell=1}^Lm_{\ell}V_{\ell}V^{-1}_{s,\ell}V_{\ell}\big)^{-1} m_{\ell}V_{\ell}V^{-1}_{s,\ell}V_{\ell},\]
	and it is clear that the
	weight  $\left\{W^{*,\rm{{\wAvgSMSE}} }_{\ell}\right\}$ defined above satisfies this condition for all $\bv$.
	

\end{proof}

\noindent \textbf{Proof of Theorem \ref{thm:asym-mSMSE-weighted}}

\begin{proof}
	
	For {\wmSMSE}, we have
	\[\hbe_{\rm{{\wmSMSE}}}^{(t)}-\be^*=\left(\sum_{\ell=1}^L\sum_{i \in \mathcal{H}_{\ell}} m^{-1}_{\ell}W_{\ell}V_{h_{t},i}\left(\hbe^{(t-1)}\right)\right)^{-1}\left(\sum_{\ell=1}^L\sum_{i \in \mathcal{H}_{\ell}} m^{-1}_{\ell}W_{\ell}U_{h_{t},i}\left(\hbe^{(t-1)}\right)\right).\]
	By Assumption \ref{assumption:weightrestriction}, $W_{\ell}/m_{\ell}\asymp 1/n$, and hence, analogous the proof of Proposition \ref{thm:1step-ld}, we can show that, if $\norm{\hbe_{\rm{{\wmSMSE}}}^{(t-1)}-\be^*}_2=O_{\Prob}(\delta_{t-1})$, then
	\begin{equation}
		\label{eq:1step-weighted}
		\norm{\hbe_{\rm{{\wmSMSE}}}^{(t)}-\be^*}_2=O_{\Prob}\left(\sqrt{\frac{p}{nh_{t}}}+\delta_{t-1}\sqrt{\frac{p\log n}{nh_{t}^3}}+\delta_{t-1}^2+h_{t}^{\alpha}\right).
	\end{equation}
	
	Using the same analysis as in the proof of Theorem \ref{thm:tstep-ld}, by taking $h_t = \max\big\{(p/n)^{1/(2\alpha+1)},\delta_{0}^{2^{t}/\alpha}\big\}$ for $t=1,\dots, T$ and $T>\log_2\left(\frac{\alpha}{2\alpha+1}\frac{\log(n/p)}{\log(1/\delta_0)}\right)$, we have $h_T= (p/n)^{1/(2\alpha+1)}$ and  $\delta_{T}=(p/n)^{\alpha/(2\alpha+1)}$, which satisfies $\delta_{T} = o\left(\max\left\{h_{T+1}^{\alpha/2}, h_{T+1}/\sqrt{p\log n}\right\}\right)$ with $h_{T+1}=n^{-1/(2\alpha+1)}$ and $p=o\left(n^{\frac{2(\alpha-1)}{4\alpha+1}}(\log n)^{-\frac{2\alpha+1}{4\alpha+1}}\right)$. Therefore, similar to the proof of \eqref{eq:asym-EU}, for any $\bv \in \mathbb{S}^{p-1}$, we have
	\[\sup_{\norm{\be-\be^*}_2\leq \delta_{T}}\E \left(\sum_{\ell=1}^L\sum_{i \in \mathcal{H}_{\ell}} m^{-1}_{\ell}\bv^{\top}W_{\ell}U_{h_{T+1},i}\left(\be\right)\right)=\bv^{\top}\ov{U}_Wh_{T+1}^{\alpha}+o\left(h_{T+1}^{\alpha}\right),\]
	where $\ov{U}_{W}:=\sum_{\ell=1}^L W_{\ell}U_{\ell}$. By \hyperlink{step1}{\textbf{Step 1}} in the proof of \eqref{eq:1stepU} and Assumption \ref{assumption:weightrestriction}, it holds that
	\begin{align*}
		&\quad\sup_{\norm{\be-\be^*}_2\leq \delta_{T}}(1-\E) \left(\sum_{\ell=1}^L\sum_{i \in \mathcal{H}_{\ell}} m^{-1}_{\ell}\bv^{\top}W_{\ell}\left[U_{h_{T+1},i}\left(\be\right)-U_{h_{T+1},i}\left(\be^*\right)\right]\right)\\
		&= O_{\Prob}\left(\delta_{T}\sqrt{\frac{p\log n}{nh_{T+1}^3}}\right)=o_{\Prob}\left(\frac{1}{\sqrt{nh_{T+1}}}\right).
	\end{align*}
	
	Define $\Phi_i:=(1-\E)\left(m^{-1}_{\ell}\bv^{\top}W_{\ell}U_{h,i}(\be^*)\right)$ for $i \in \mathcal{H}_{\ell}$. Then
	\[s_n^2 := \sum_{i=1}^n \mathrm{var}\left[\Phi_i\right] = h^{-1}\sum_{\ell=1}^Lm_{\ell}^{-1}\bv^{\top}W_{\ell}\left[ V_{s,\ell}+o(1)\right]W_{\ell}^{\top}\bv.\]
	Similar to the proof of Theorem \ref{thm:asym-DC-weighted}, we have $\left(h^{-1}\sum_{\ell=1}^Lm_{\ell}^{-1}\bv^{\top}W_{\ell}V_{s,\ell}W_{\ell}^{\top}\bv\right)^{-1/2}\sum_{i=1}^n\Phi_i \stackrel{d}{\longrightarrow} \mathcal{N}\left(0, 1\right)$ and $h_{T+1}^{-1}\sum_{\ell=1}^Lm_{\ell}^{-1}\bv^{\top}W_{\ell}V_{s,\ell}W_{\ell}^{\top}\bv \asymp 1/\sqrt{nh_{T+1}}=h_{T+1}^{\alpha}$, and thus
	\begin{align*}
		&\quad\left(h_{T+1}^{-1}\sum_{\ell=1}^Lm_{\ell}^{-1}\bv^{\top}W_{\ell}V_{s,\ell}W_{\ell}^{\top}\bv\right)^{-1/2}\left(\sum_{\ell=1}^L\sum_{i \in \mathcal{H}_{\ell}} m^{-1}_{\ell}\bv^{\top}W_{\ell}U_{h_{T+1},i}\left(\hbe^{(T)}\right)-\bv^{\top}\ov{U}_Wh_{T+1}^{\alpha}\right)\\
		&=\left(h_{T+1}^{-1}\sum_{\ell=1}^Lm_{\ell}^{-1}\bv^{\top}W_{\ell}V_{s,\ell}W_{\ell}^{\top}\bv\right)^{-1/2}\left(\sum_{\ell=1}^L\sum_{i \in \mathcal{H}_{\ell}}(1-\E) m^{-1}_{\ell}\bv^{\top}W_{\ell}U_{h_{T+1},i}\left(\be^*\right)\right)\\
		&\quad+\left(h_{T+1}^{-1}\sum_{\ell=1}^Lm_{\ell}^{-1}\bv^{\top}W_{\ell}V_{s,\ell}W_{\ell}^{\top}\bv\right)^{-1/2}\left(\sum_{\ell=1}^L\sum_{i \in \mathcal{H}_{\ell}}(1-\E) m^{-1}_{\ell}\bv^{\top}W_{\ell}\left[U_{h_{T+1},i}\left(\hbe^{(T)}\right)-U_{h_{T+1},i}\left(\be^*\right)\right]\right)\\
		&\quad+\left(h_{T+1}^{-1}\sum_{\ell=1}^Lm_{\ell}^{-1}\bv^{\top}W_{\ell}V_{s,\ell}W_{\ell}^{\top}\bv\right)^{-1/2}\E \left(\sum_{\ell=1}^L\sum_{i \in \mathcal{H}_{\ell}} m^{-1}_{\ell}\bv^{\top}W_{\ell}U_{h_{T+1},i}\left(\hbe^{(T)}\right)\right)-\bv^{\top}\ov{U}_Wh_{T+1}^{\alpha}\\
		&=\left(h_{T+1}^{-1}\sum_{\ell=1}^Lm_{\ell}^{-1}\bv^{\top}W_{\ell}V_{s,\ell}W_{\ell}^{\top}\bv\right)^{-1/2}\left(\sum_{\ell=1}^L\sum_{i \in \mathcal{H}_{\ell}}(1-\E) m^{-1}_{\ell}\bv^{\top}W_{\ell}U_{h_{T+1},i}\left(\be^*\right)\right)+o_{\Prob}(1)\stackrel{d}{\longrightarrow} \mathcal{N}\left(0, 1\right).
	\end{align*}
	
	Note that $h_{T+1}^{-1}\sum_{\ell=1}^Lm_{\ell}^{-1}W_{\ell}V_{s,\ell}W_{\ell}^{\top}=n^{-\frac{2\alpha}{2\alpha+1}}\sum_{\ell=1}^L\frac{n}{m_{\ell}}W_{\ell}V_{s,\ell}W_{\ell}^{\top}$, $n^{\alpha/(2\alpha+1)}=h_{T+1}^{\alpha}$, and Assumption \ref{assumption:weightrestriction} ensures that $n\sum_{\ell=1}^Lm_{\ell}^{-1}W_{\ell}V_{s,\ell}W_{\ell}^{\top}$ is a finite matrix. Then we rewrite the above result as
	\[\frac{n^{\alpha/(2\alpha+1)}\sum_{\ell=1}^L\sum_{i \in \mathcal{H}_{\ell}} m^{-1}_{\ell}\bv^{\top}W_{\ell}U_{h_{T+1},i}\left(\hbe^{(T)}\right) -\bv^{\top}\ov{U}_W}{\sqrt{n\sum_{\ell=1}^Lm_{\ell}^{-1}\bv^{\top}W_{\ell}V_{s,\ell}W_{\ell}^{\top}\bv}} \stackrel{d}{\longrightarrow}\mathcal{N}\left(0, 1\right).\]
	
	Combining with
	\[\hbe_{\rm{{\wmSMSE}}}^{(T+1)}-\be^*=\left(\sum_{\ell=1}^L\sum_{i \in \mathcal{H}_{\ell}} m^{-1}_{\ell}W_{\ell}V_{h_{T+1},i}\left(\hbe^{(T)}\right)\right)^{-1}\left(\sum_{\ell=1}^L\sum_{i \in \mathcal{H}_{\ell}} m^{-1}_{\ell}W_{\ell}U_{h_{T+1},i}\left(\hbe^{(T)}\right)\right),\]
	and  
	\[\norm{\sum_{\ell=1}^L\sum_{i \in \mathcal{H}_{\ell}} m^{-1}_{\ell}W_{\ell}V_{h_{T+1},i}\left(\hbe^{(T)}\right)-\ov{V}_W}_2=O_{\Prob}\left(\sqrt{\frac{p\log n}{nh_{T+1}^3}}+\delta_{T}+h_{T+1}^{\alpha}\right),\]  by Slutsky's Theorem, \eqref{eq:asym-mSMSE-weighted} is true.
	
\end{proof}

\subsubsection{Proof for Theorem \ref{thm:tstep-cond-shift}}

We first restate Theorem \ref{thm:tstep-cond-shift} to explicitly show the dependence of the convergence rate on the constant $\omega \in (0, 1)$.

\begin{thm}[Restatement of Theorem \ref{thm:tstep-cond-shift}]\label{thm:tstep-cond-shift-1}
	Assume the assumptions in Theorem \ref{thm:tstep-ld} hold. Further assume that $\varepsilon \leq\varepsilon_0$ for some constant $\varepsilon_0<1$. By choosing $h_0=\big(\frac{p\log L}{\omega m}\big)^{\frac{1}{2\alpha+1}}$ and $h_t=\max\left\{\delta^{2^t/\alpha}_{m ,0}, \big(\frac{p}{(1-\omega)n}\big)^{\frac{1}{2\alpha+1}}\right\}$ for $t=1,2,\dots,T$, we have that
	\begin{equation}\label{eq:tstepbeta-hetero-1}
		\norm{\hbe_1^{(t)}-\be_1^*}_2=O_{\Prob}\left(\Big(\frac{p}{(1-\varepsilon)(1-\omega)n}\Big)^{
			\frac{\alpha}{2\alpha+1}}+\delta_{m ,0}^{2^{t}}+\delta_{m ,0}^{2^{t-1}}\Big(\frac{p}{(1-\varepsilon)(1-\omega)n}\Big)^{\frac{\alpha-1}{2\alpha+1}}\sqrt{\log n}+\varepsilon\delta_{m ,0}^2\right),
	\end{equation}  
	where $\delta_{m ,0}=(\frac{p\log L}{\omega m})^{\frac{\alpha}{2\alpha+1}}$.
\end{thm} 

\begin{proof}
	Throughout the whole proof, let $\abs{\calS}$ denote the cardinality for any set $\calS$. For $\ell=1,\dots, L$, denote $\calH_{\ell}^{(0)}$ to be an arbitrary subset of $\calH_{\ell}$ with size $\abs{\calH_{\ell}^{(0)}} = \omega \abs{\calH_{\ell}}=\omega m$, and let $\hbe_{\ell,\mathrm{SMSE}}^{(0)}=\argmin_{\be}\frac{1}{\omega m}\sum_{\ell\in\mathcal{H}_{\ell}^{(0)}}(-y_i)H\left(\frac{x_i+\bz_{i}^{\top}\be}{h_0}\right)$, i.e., the SMSE computed using data in $\calH_{\ell}^{(0)}$ with bandwidth $h_0=(p\log L/\omega m)^{1/(2\alpha+1)}$. For a constant $C_0$, define
	\[\mathcal{A}_{C_0}:=\left\{\ell:\norm{\hbe_{\ell,\mathrm{SMSE}}^{(0)}-\hbe_{1,\mathrm{SMSE}}^{(0)}}_2\leq C_0\delta_{m ,0}\right\}, \mathcal{A}^*:=\left\{\ell:\be_{\ell}^{*}=\be_{1}^{*}\right\},\]
	where $\delta_{m ,0}=(p\log L/\omega m)^{\alpha/(2\alpha+1)}$.
	We first show the following lemma:
	\begin{lemma}\label{lem:hetero-1}
		Under the assumptions in Theorem \ref{thm:tstep-cond-shift}, there exist constants $C$ and $C_0$ such that the following event holds with probability approaching one:
		\[E_0:=\left\{\sup_{\ell}\norm{\hbe_{\ell,\mathrm{SMSE}}^{(0)}-\be^*_{\ell}}_2\leq C\delta_{m ,0}, \quad \mathcal{A}^* \subset \mathcal{A}_{C_0}, \quad \sup_{\ell\in\mathcal{A}_{C_0}\setminus\mathcal{A}^*} \norm{\be^*_\ell-\be^*_1}_2 \leq C\delta_{m ,0}\right\}.\]
	\end{lemma}
	\begin{proof}[Proof for Lemma \ref{lem:hetero-1}]
		We first modify the definitions in the proof of Theorem \ref{thm:asym-DC} and Proposition \ref{thm:1step-ld}. For $i\in \calH_{\ell}$, define
		\[U_{h,i,\ell}\left(\be\right):=\frac{1}{h^2} \left(-y_i\right) H^{\prime\prime}\left(\frac{x_i+\bz_i^{\top}\be}{h}\right)\bz_i\bz_i^{\top}\left(\be-\be_{\ell}^*\right)-\frac{1}{h}\left(-y_i\right) H^{\prime}\left(\frac{x_i+\bz_i^{\top}\be}{h}\right)\bz_i,\]
		and 
		\[
		V_{h,i}\left(\be\right):=\frac{1}{h^2} \left(-y_i\right) H^{\prime\prime}\left(\frac{x_i+\bz_i^{\top}\be}{h}\right)\bz_i\bz_i^{\top}.
		\]
		Define \[U_{m,\ell,h}\left(\be\right)=\frac{1}{\omega m}\sum_{i \in \mathcal{H}_\ell^{(0)}}U_{h, i,\ell}\left(\be\right),\quad V_{m,\ell,h}\left(\be\right):=\frac{1}{\omega m}\sum_{i \in \mathcal{H}_\ell^{(0)}}V_{h, i}\left(\be\right).\]
		We will show that, for any sufficiently large $\gamma$, there exists a constant $C$, such that
		\begin{equation}\label{eq:sup-U*}
			\sup_{\ell}\norm{(1-\E)U_{m,\ell,h_0}(\be_{\ell}^*)}_2=O_{\Prob}\left(\sqrt{\frac{p\log L}{\omega mh_0}}\right),
		\end{equation}
		\begin{equation}\label{eq:sup-phi-V}
			\sup_{\ell}\sup_{\be: \norm{\be-\be_{\ell}^*}_2 \leq \delta}  \norm{V_{m,\ell,h_0}\left(\be\right)-V}_2 = O_{\Prob}\left(\sqrt{\frac{p\log m}{\omega mh_0^3}}+h_0^{\alpha}+\delta\right).
		\end{equation}
	
	To show \eqref{eq:sup-U*}, using the same vectors $\bv_1,...,\bv_{5^p} \in \mathbb{R}^p$ as in the proof of Proposition \ref{thm:1step-ld}, it suffices to show that, for any sufficiently large $\gamma$, there exists a constant $C^*$, such that
	\begin{equation}
		\sup_{\ell}\sup_{{j_v} \in [5^p]} \abs{\bv_{j_v}^{\top}(1-\E)U_{m,\ell,h_0}(\be_{\ell}^*)}\leq C^*\sqrt{\frac{p\log L}{\omega mh_0}}.
	\end{equation}
	with probability at least $1-2L(5L^{-\gamma})^{p}$. 
	
	Let $\bv$ be any fixed vector in $\mathbb{S}^{p-1}$. Note that  \[(1-\E)\bv^{\top}U_{m,\ell,h_0}\left(\be_{\ell}^*\right) =\frac{1}{\omega m}\sum_{i \in \mathcal{H}_{\ell}^{(0)}}(1-\E)\brackets{\bv^{\top}U_{h_0,i,\ell}\left(\be_{\ell}^*\right)}=\frac{1}{\omega m}\sum_{i \in \mathcal{H}_{\ell}^{(0)}}(1-\E)\frac{(\bv^{\top}\bz_i)y_i}{h}H^{\prime}\left(\frac{x_i+\bz_i^{\top}\be_{\ell}^*}{h}\right),\] and $(1-\E)\brackets{\bv^{\top}U_{h,i}\left(\be_{\ell}^*\right)}$ are i.i.d. among different $i \in \mathcal{H}_{\ell}^{(0)}$. By the assumptions, the density function of $\zeta_{\ell}=X+\bZ^{\top}\be_{\ell}^*$, denoted by $\rho(\cdot\mid\bZ)$, is the same for all $\ell$ and is bounded uniformly for all $\bZ$. Therefore,
	\begin{equation}\label{eq:hetero-EUbeta*2}
		\begin{aligned}
			&\E_{\cdot\mid \bZ}\brackets{\bv^{\top}U_{h_0,i,\ell}\left(\be_{\ell}^*\right)}^2\\
			=&\frac{\paren{\bv^{\top}\bZ}^2}{h_0^2}\E_{\cdot\mid \bZ} \brackets{H^{\prime}\left(\frac{X+\bZ^{\top}\be_{\ell}^*}{h_0}\right)}^2\\
			=& \int_{-1}^1\frac{\paren{\bv^{\top}\bZ}^2}{h_0}\brackets{H^{\prime}\left(\xi\right)}^2\rho\left(\xi h_0\mid \bZ\right) \diff \xi \leq C_{\rho, H}\left(\frac{(\bv^{\top}\bZ)^2}{h_0}\right),
		\end{aligned}
	\end{equation} 
	where $C_{\rho, H}$ is a constant not depending on $\ell$.
	Let $\wphi^{U, *}_{i, \bv}=h(1-\E)\brackets{\bv^{\top}U_{h_0,i,\ell}\left(\be_{\ell}^*\right)}$. By \eqref{eq:phiUi-bound}, it holds that
	\[ 
	\Prob\left(\sum_{i \in \mathcal{
			H}_{\ell}^{(0)}} \wphi^{U, *}_{i, \bv}  > b\right) \leq \exp\left(-tb + t^2\sum_{i \in \calH_{\ell} ^{(0)}}\E\left[\left(\wphi^{U, *}_{i, \bv}\right)^2\exp\left(t\abs{\wphi^{U, *}_{i, \bv}}\right)\right]\right).
	\]
	Combining with \eqref{eq:hetero-EUbeta*2} and Assumption \ref{A5}, there exists a constant $C_{\rho, H, \eta}$ that does not depend on $\ell$, such that
	\[\Prob\left(\sum_{i\in \calH_{\ell}^{(0)}} \wphi^{U, *}_{i,\bv}  > b\right) \leq \exp\left(-tb + \omega mt^2C_{\rho, H, \eta}^2h_0\right),\]
	for sufficiently small $t$. Letting $b=C_{\rho, H, \eta}\sqrt{\gamma \omega mph_0\log L}$ and $t=\sqrt{\frac{\gamma p\log L}{4C_{\rho, H, \eta}^2\omega mh_0}}$ leads to
	\[\Prob\left(\sum_{i \in \calH_{\ell}^{(0)}} \wphi^{U, *}_{i,\bv} > C_{\rho, H, \eta} \sqrt{\gamma mh_0p \log L}\right) \leq \exp\left[-\frac{1}{4}\gamma p\log L\right],\]
	which yields
	\[\Prob\left(\frac{1}{\omega m}\abs{\sum_{i \in \calH_{\ell}^{(0)}} (1-\E)\brackets{\bv^{\top}U_{h_0,i,\ell}\left(\be_{\ell}^*\right)}} >  C_{\rho, H, \eta} \sqrt{\frac{\gamma p\log L}{\omega mh_0} }\right) \leq 2L^{-\gamma p/4}.\]
	The above inequality is true for all $\bv$ in $\mathbb{S}^{p-1}$. Therefore, with probability at least $1-2L(5L^{-\gamma/4})^p$,
	\[\sup_{\ell}\sup_{j_v \in [5^p]}\abs{(1-\E)\bv_{j_v}^{\top}U_{m,\ell,h_0}\left(\be^*\right)}  \leq C^* \sqrt{\frac{p \log L}{\omega mh_0}},\]
	where $C^*=C_{\rho,H, \eta}\sqrt{\gamma}$. Since $\gamma$ can be arbitrarily large, this completes the proof of \eqref{eq:sup-U*}.

	To show \eqref{eq:sup-phi-V}, we use the same set of $\{\bv_{j_v}\}_{j_v\in[5^p]}$ as above.  By the proof of Lemma 3 in \cite{cai2010optimal}, 
	$\norm{A}_2\leq 4\sup_{{j_v} \in [5^p]} \abs{\bv^{\top}_{j_v}A\bv_{j_v}},$
	for any symmetric $A \in \mathbb{R}^{p \times p}$. 
	Therefore, it suffices to bound $\sup_{{j_v} \in [5^p]} \abs{\bv^{\top}_{j_v}\brackets{V_{m,\ell,h_0}\left(\be\right)-V}\bv_{j_v}}$. Using the same strategy as in \hyperlink{step1}{\textbf{Step 1}} of the proof of \eqref{eq:1stepU}, we can find a set $\{\be_{j_\beta}\}_{j_{\beta} \in [(\omega m)^{\gamma p}]}$ satisfying that, for all $\be$ in the ball $\braces{\be: \norm{\be-\be^*}_2\leq \delta}$, there exists $j_{\beta} \in [(\omega m)^{\gamma p}]$ such that $\norm{\be-\be_{j_{\beta}}}_\infty \leq \frac{2\delta}{(\omega m)^{\gamma}}$. By the Lipschitz property of $H''(x)$, we have
	\begin{align*}
		&\quad\sup_{{j_v} \in [5^p]}\sup_{\be: \norm{\be-\be^*}_2 \leq \delta }\inf_{j_{\beta}\in [(\omega m)^{\gamma p}]}\abs{\bv_{j_v}^{\top}V_{m,\ell,h_0}\left(\be\right)\bv_{j_v}-\bv_{j_v}^{\top}V_{m,\ell,h_0}\left(\be_{j_\beta}\right)\bv_{j_v}} \\
		&\leq \sup_{{j_v} \in [5^p]}\sup_{\be: \norm{\be-\be^*}_2 \leq \delta }\inf_{j_{\beta}\in [(\omega m)^{\gamma p}]}\frac{1}{m}\sum_{i\in \calH_{\ell}^{(0)}} \frac{\paren{\bz_i^{\top}\bv_{j_v}}^2}{h_0^3}\abs{\bz_i^{\top}\paren{\be-\be_{j_\beta}}} \\
		&= O\left(\sum_{i\in \calH_{\ell}^{(0)}}\frac{\delta\norm{\bz_i}_2^3}{(\omega m)^{\gamma+1}h_0^3}\right).
	\end{align*}
	
	Analogous to the \hyperlink{step1}{\textbf{Step 1}} of the proof of \eqref{eq:1stepU}, we can show that $\frac{1}{(\omega m)^{\gamma+1}}\sum\limits_{i \in \calH_{\ell}^{(0)}}\norm{\bz_i}_2^3 \leq \frac{p\sup_{i, j}\E |z_{i, j}|^3}{(\omega m)^{(\gamma-1)/2}}$, with probability and least $1-(\omega m)^{-\gamma/2}$. By taking $\gamma>0$ large enough such that 
	\[\frac{p^{3/2}\delta}{(\omega m)^{(\gamma-1)/2}h_0^3} = o\left(\sqrt{\frac{p\log m}{\omega m h_0^3}}\right),\] 
	we obtain
	\begin{equation}
		\begin{aligned}
			&\quad\sup_{{j_v} \in [5^p]}\sup_{\be: \norm{\be-\be^*}_2 \leq \delta}\abs{\bv^{\top}_{j_v}\brackets{V_{m,\ell,h_0}\left(\be\right)-V}\bv_{j_v}} -\sup_{{j_v} \in [5^p]}\sup_{j_{\beta}\in [m^{\gamma p}]}\abs{\bv^{\top}_{j_v}\brackets{V_{m,\ell,h_0}\left(\be_{j_\beta}\right)-V}\bv_{j_v}}\\
			&=o\paren{\sqrt{\frac{p\log m}{\omega mh_0^3}}}.
		\end{aligned}
	\end{equation}
	
	For any fixed $\be$, let
	\[\phi_{i,\bv}^{V}(\be):=h_0^2(1-\E)\bv^{\top}V_{h_0,i}(\be)\bv=(1-\E)(\bv^{\top}\bz_i)^2 \left(-y_i\right) H^{\prime\prime}\left(\frac{x_i+\bz_i^{\top}\be}{h_0}\right).\]
	By repeating the procedure in the proof for \eqref{eq:sup-U*}, we obtain
	\[\Prob\left(\sum_{i \in \calH_{\ell}^{(0)}}\phi_{i,\bv}^{V}(\be) > b\right) \leq \exp\left(-tb+t^2C_{\rho,H, \eta}^2\omega mh_0\right),\]
	for some absolute constant $C_{\rho,H, \eta}$ and sufficiently small $t>0$. Let $b=C_{\rho,H, \eta}\sqrt{8\gamma \omega mh_0p\log m}$ and $t=b/(2C_{\rho, H, \eta}^2\omega mh_0)$, where $\gamma$ is an arbitrary positive constant to be determined, it holds that
	\[\Prob\left(\phi_{i,\bv}^{V}(\be) > C_{\rho,H, \eta}\sqrt{8\gamma \omega mh_0p\log m}\right) \leq m^{-2\gamma p},\]
	and hence
	\[\Prob\left(\abs{(1-\E)\paren{\bv^{\top}\brackets{V_{m,\ell,h_0}\left(\be\right)-V}\bv}} > C_{\rho, H, \eta}\sqrt{\frac{8\gamma p\log m}{\omega  m h_0^3}}\right) \leq 2(\omega m)^{-2\gamma p}.\]
	The above inequality is true for all $\bv_{j_v}$ and $\be_{j_{\beta}}$, and thus
	\[\Prob\left(\sup_{{j_v} \in [5^p]}\sup_{j_{\beta}\in [(\omega m)^{\gamma p}]}\abs{(1-\E)\paren{\bv_{j_v}^{\top}\brackets{V_{m,\ell,h_0}\left(\be_{j_{\beta}}\right)-V}\bv_{j_v}}} > C_{\rho, H, \eta}\sqrt{\frac{8\gamma p\log m}{\omega mh_0^3}}\right) \leq 2(5(\omega m)^{-\gamma})^p.\]
	Then we obtain, for any sufficiently large $\gamma>0$, there exists a constant $C_V$, such that
	\[
	\sup_{\ell}\sup_{\bv \in \mathbb{S}^{p-1}}\sup_{\be: \norm{\be-\be^*}_2 \leq \delta}\abs{(1-\E)\paren{\bv^{\top}\brackets{V_{m,\ell,h_0}\left(\be\right)-V}\bv}} \leq C_V\sqrt{\frac{p\log m}{\omega mh_0^3}},
	\]
	with probability at least $1-L(\omega m)^{-\gamma/2}-2L(5(\omega m)^{-\gamma})^p$. By the assumption that $m>n^{c_3}$ for some $0<c_3<1$, there exists a constant $\gamma$ such that $L(\omega m)^{-\gamma/2}+2L(5(\omega m)^{-\gamma})^p=o(1)$. Therefore, it holds that 
	\[\sup_{\ell}\sup_{\bv \in \mathbb{S}^{p-1}}\sup_{\be: \norm{\be-\be^*}_2 \leq \delta}\abs{(1-\E)\paren{\bv^{\top}\brackets{V_{m,\ell,h_0}\left(\be\right)-V}\bv}} = O_{\Prob}\left(\sqrt{\frac{p\log m}{\omega mh_0^3}}\right).\]
	
	By the proof of \eqref{eq:1stepV}, we have
	\begin{equation}
		\sup_{\bv \in \mathbb{S}^{p-1}}\sup_{\be: \norm{\be-\be^*}_2 \leq \delta }\abs{\E \paren{\bv^{\top}\brackets{V_{m,\ell,h}\left(\be\right)-V}\bv}} = O\left(\delta + h^{\alpha}\right),
	\end{equation}
	where the constant hidden in $O(\cdot)$ is the same for all $\ell$. This completes the proof for \eqref{eq:sup-phi-V}.
	
	By \hyperlink{step3}{\textbf{Step 3}} in the proof of \eqref{eq:1stepU}, we have
	$\sup_{\ell}\norm{\E [U_{m,\ell,h_0}\left(\be_{\ell}^*\right)]}_2 = O(h_0^{\alpha})$. 
	Then the proof of Theorem \ref{thm:asym-DC} indicates that Equations \eqref{eq:sup-U*} and \eqref{eq:sup-phi-V} imply
	\begin{equation}\label{eq:hetero-init}
		\sup_{\ell}\norm{\hbe_{\ell,\mathrm{SMSE}}^{(0)}-\be^*_{\ell}}_2 \leq C'\left(\sqrt{\frac{p\log L}{\omega mh_0}}+h_0^\alpha\right)=2C'\left(\frac{p\log L}{\omega m}\right)^{\alpha/(2\alpha+1)}=2C'\delta_{m ,0},
	\end{equation} 
	for some absolute constant $C'$, with probability approaching one. Furthermore, the definition of $\mathcal{A}_{C_0}$  with $C_0>4C'$ ensures that $\mathcal{A}^* \subset \mathcal{A}$ if \eqref{eq:hetero-init} holds. Moreover, for any $\ell\in\mathcal{A}\setminus\mathcal{A}^*$, we have that $\norm{\be^*_\ell-\be^*_1}_2 \leq 2\sup_{\ell}\norm{\hbe_{\ell,\mathrm{SMSE}}^{(0)}-\be^*_{\ell}}_2+\sup_{\ell}\norm{\hbe_{\ell,\mathrm{SMSE}}^{(0)}-\hbe_{1,\mathrm{SMSE}}^{(0)}}_2\leq (C_0+4C')\delta_{m ,0}$  if \eqref{eq:hetero-init} holds. Therefore, the event $E_0$ holds with probability approaching one.
\end{proof}




Let $\hbe_{1}^{(0)}=\hbe_{1,\mathrm{SMSE}}^{(0)}$. 
Without loss of generality, lat $C=C_0=1$ in $E_0$. For simplicity, denote $\mathcal{A}=\mathcal{A}_{C_0}$. Then under $E_0$, it holds that
\begin{equation}\label{eq:beta-diff-1}
	\sup_{\ell \in \mathcal{A}\setminus \mathcal{A}^*}\norm{\hbe_{1}^{(0)}-\be^*_{\ell}}_2 \leq \sup_{\ell \in \mathcal{A}\setminus \mathcal{A}^*}\norm{\hbe_{1,\mathrm{SMSE}}^{(0)}-\hbe_{\ell,\mathrm{SMSE}}^{(0)}}_2+\sup_{\ell}\norm{\hbe_{\ell,\mathrm{SMSE}}^{(0)}-\be^*_{\ell}}_2 \leq 2\delta_{m ,0},
\end{equation}
and
\begin{equation}\label{eq:beta-diff-2}
	\sup_{\ell \in \mathcal{A}^* 
	}\norm{\hbe_{1}^{(0)}-\be^*_{\ell}}_2 = \norm{\hbe_{1,\mathrm{SMSE}}^{(0)}-\be^*_{1}}_2 \leq \delta_{m ,0}.
\end{equation}

Define $\widetilde{\mathcal{H}}_\ell=\mathcal{H}_\ell \setminus \mathcal{H}_\ell^{(0)}$. Recall that the algorithm is
\[\hbe_1^{(t+1)}=\hbe_1^{(t)}-\left[\frac{1}{(1-\omega)m|\mathcal{A}|}\sum_{\ell \in \mathcal{A}}\sum_{i \in \widetilde{\mathcal{H}}_\ell}V_{h, i}\left(\hbe_1^{(t)}\right)\right]^{-1}\left[\frac{1}{|\mathcal{A}|(1-\omega)mh}\sum_{\ell \in \mathcal{A}}\sum_{i\in\widetilde{\mathcal{H}}_{\ell}}\left(-y_i\right) H^{\prime}\left(\frac{x_i+\bz_{i}^{\top}\hbe_{1}^{(t)}}{h}\right) \bz_{i}\right],\]
which implies that 	
\[\hbe_1^{(t+1)}-\be_1^*=\Big[\widetilde{V}_{\mathcal{A}}\big(\hbe_1^{(t)}\big)\Big]^{-1}\left[\frac{1}{|\mathcal{A}|}\sum_{\ell\in\mathcal{A}}\widetilde{U}_{m,\ell,h}\big(\hbe_1^{(t)}\big)\right],\]
where $|\mathcal{A}|$ denote the cardinality of $\mathcal{A}$,  
\[\widetilde{V}_{\mathcal{A}}(\be):=\frac{1}{(1-\omega)m|\mathcal{A}|}\sum_{\ell \in \mathcal{A}}\sum_{i \in \widetilde{\mathcal{H}}_\ell}V_{h, i}\left(\be\right),\] 
and
\begin{align*}
	\widetilde{U}_{m,\ell,h}(\be)&:=\frac{1}{(1-\omega)m}\sum_{i \in \widetilde{\mathcal{H}}_\ell}V_{h, i}\left(\be\right)\left(\be-\be_1^*\right)-\frac{1}{(1-\omega)mh}\sum_{i\in\widetilde{\mathcal{H}}_\ell}\left(-y_i\right) H^{\prime}\left(\frac{x_i+\bz_{i}^{\top}\be}{h}\right) \bz_{i}\\
	&=\frac{1}{(1-\omega)mh^2}\sum_{i \in \widetilde{\mathcal{H}}_{\ell}} \left(-y_i\right) H^{\prime\prime}\left(\frac{x_i+\bz_i^{\top}\be}{h}\right)\bz_i\bz_i^{\top}\left(\be-\be_1^*\right)\\
	&\quad -\frac{1}{(1-\omega)mh}\sum_{i\in\widetilde{\mathcal{H}}_\ell}\left(-y_i\right) H^{\prime}\left(\frac{x_i+\bz_{i}^{\top}\be}{h}\right) \bz_{i}.
\end{align*} 

Let $\delta$ be a quantity that satisfies $\delta=o(\delta_{m ,0})$. Under $E_0$, it holds that $\norm{\be^*_1-\be^*_\ell}_2 \leq \delta_{m ,0}$ for all $\ell \in \mathcal{A}\setminus\mathcal{A^*}$. Also, the data points in $\big\{\widetilde{\mathcal{H}}_{\ell}\big\}_{\ell=1}^L$ are independent to $E_1$. Following the proof for \eqref{eq:sup-phi-V}, it is straightforward to show that, when $h=o(1)$ and $p\log n/(nh^3)=o(1)$, with probability tending to one,
\[\sup_{\norm{\be-\be^*_1}_2 \leq \delta}\norm{\widetilde{V}_{\mathcal{A}}(\be)-V}_2 \lesssim \sqrt{\frac{p\log n}{(1-\varepsilon)(1-\omega)nh^3}}+h^{\alpha}+\delta+\varepsilon\delta_{m ,0},\]
using the facts that $|\mathcal{A}|\geq (1-\varepsilon)L$, $|\mathcal{A}\setminus\mathcal{A}^*|\leq \varepsilon L$.

For the numerator, we have 
\begin{align*}
	\frac{1}{|\mathcal{A}|}\sum_{\ell \in \mathcal{A}}\widetilde{U}_{m,\ell,h}\big(\hbe_1^{(t)}\big)&=\frac{1}{|\mathcal{A}|}\sum_{\ell \in \mathcal{A}} (1-\E)\widetilde{U}_{m,\ell,h}(\be_{1}^*) + \frac{1}{|\mathcal{A}|}\sum_{\ell\in\mathcal{A}}(1-\E)\left[\widetilde{U}_{m,\ell,h}\big(\hbe_{1}^{(t)}\big)-\widetilde{U}_{m,\ell,h}(\be_{1}^*)\right]\\
	&\quad+\frac{1}{|\mathcal{A}|}\sum_{\ell\in\mathcal{A}}\E\left[\widetilde{U}_{m,\ell,h}\big(\hbe_{1}^{(t)}\big)\right].
\end{align*}


Note that by the uniformly boundedness of $\rho(\cdot\mid\bZ)$, on each machine $\ell$,  Equations \eqref{eq:EH2}, \eqref{eq:EdH2}, and \eqref{eq:EH'2} hold for all $\be$ and $\be_1^*$ such that $\norm{\be-\be_1^*}_2\leq \delta$ and $\norm{\be^*_1-\be^*_{\ell}}_2\leq \delta_{m ,0}$. Moreover, the constants hidden in the big O notations are uniform for all $\ell$, $\be$ and $\be_1^*$. Therefore, following the \hyperlink{step1}{\textbf{Step 1}} and \hyperlink{step2}{\textbf{Step 2}} in the proof of Proposition \ref{thm:1step-ld}, we have that 
\[\sup_{\be: \norm{\be-\be_1^*}_2 \leq \delta} \norm{\frac{1}{\abs{\mathcal{A}}}\sum_{\ell\in\mathcal{A}}(1-\E)\left[\widetilde{U}_{m,\ell,h}(\be)-\widetilde{U}_{m,\ell,h}(\be_{1}^*)\right]}_2=O_{\Prob}\left(\delta\sqrt{\frac{p\log n}{(1-\varepsilon)(1-\omega) nh^3}}\right),\]
and
\[\norm{\frac{1}{\abs{\mathcal{A}}}\sum_{\ell\in\mathcal{A}} (1-\E)\widetilde{U}_{m,\ell,h}(\be_{1}^*)}_2 =O_{\Prob}\left(\sqrt{\frac{p}{(1-\varepsilon)(1-\omega)nh}}\right).\]
By \hyperlink{step3}{\textbf{Step 3}} in the proof of Proposition \ref{thm:1step-ld}, we have
\[\sup_{\be: \norm{\be-\be_1^*}_2 \leq \delta}\frac{1}{|\mathcal{A}|}\sum_{\ell\in\mathcal{A}}\norm{\E\left[\widetilde{U}_{m,\ell,h}\big(\be\big)\right]}_2=O\left(h^{\alpha}+\delta^2+\varepsilon\delta_{m ,0}^2\right),\]
where again we use the facts that $|\mathcal{A}\setminus\mathcal{A}^*|\leq \varepsilon L$ and $\norm{\be^*_1-\be^*_\ell}_2 \leq \delta_{m ,0}$ for all $\ell \in \mathcal{A}\setminus\mathcal{A^*}$ under $E_0$.
The above bounds imply that
\begin{align*}
	&\quad\sup_{\be:\norm{\be_1-\be_1^*}_2\leq \delta}\norm{\frac{1}{|\mathcal{A}|}\sum_{\ell \in \mathcal{A}}\widetilde{U}_{m,\ell,h}\big(\be\big)}_2 \\ &=O_{\Prob}\left(\sqrt{\frac{p}{(1-\varepsilon)(1-\omega)nh}}+\delta\sqrt{\frac{p\log n}{(1-\varepsilon)(1-\omega)nh^3}}+\delta^2+h^\alpha+\varepsilon \delta^2_{m ,0}\right),
\end{align*}
where $\delta_{m ,0}=\left(\frac{p\log L}{\omega m}\right)^{\alpha/(2\alpha+1)}$. 
Adding that $\sup_{\norm{\be-\be^*_1}_2 \leq \delta}\norm{\widetilde{V}_{\mathcal{A}}(\be)-V}_2=o_{\Prob}(1)$, we obtain that
\begin{equation}\label{eq:1step-hetero}
	\norm{\hbe_1^{(t+1)}-\be_1^*}_2=O_{\Prob}\left(\sqrt{\frac{p}{(1-\varepsilon)(1-\omega)nh_{t+1}}}+\delta_{m ,t}\sqrt{\frac{p\log n}{(1-\varepsilon)(1-\omega)nh_{t+1}^3}}+\delta_{m, t}^2+h_{t+1}^\alpha+\varepsilon \delta^2_{m ,0}\right),
\end{equation}
if $\norm{\hbe_1^{(t)}-\be_1}_2=O_{\Prob}\left(\delta_{m, t}\right)$. Since $\omega$ is a constant, and $\varepsilon \leq \varepsilon_0$ for some constant $\varepsilon_0<1$, we omit the factor $(1-\varepsilon)(1-\omega)$ on $n$ in the sequel. 

Now we choose \[h_t=\max\left\{\delta^{2^t/\alpha}_{m ,0}, (p/n)^{1/(2\alpha+1)}\right\} \geq (p/n)^{1/(2\alpha+1)},\] 
which implies the $\frac{p\log n}{nh^3_{t}}=o(1)$ holds for any $t$, since \[\frac{p\log n}{nh^3_{t}}\leq (p/n)^{\frac{2\alpha-2}{2\alpha+1}}\log n =O\big(n^{-(1-c_2)\frac{2\alpha-2}{2\alpha+1}}\log n\big) \to 0,\] for $p=O(m^{c_2})=O(n^{c_2})$ and $\alpha \geq 2$. 
Moreover, for any $t$,  $h^{\alpha}_t=\max\left\{\delta_{m ,0}^{2^t}, (p/n)^{\alpha/(2\alpha+1)}\right\}$ and $\sqrt{p/nh_t} \leq (p/n)^{\alpha/(2\alpha+1)}$.

We now show that, for any $t$,
\begin{equation}
	\label{eq:tstep-ld-weaker-hetero}
	\left\| \hbe^{\,\left(t\right)}-\be^{*}\right\|_2=O_{\mathbb{P}}\left((p/n)^{\alpha/(2\alpha+1)}\sqrt{\log n}+\delta_{m ,0}^{2^t}+\varepsilon\delta_{m ,0}^{2}\right).
\end{equation}

When $t=1$, $\delta_{m,0} \leq h_1=\delta_{m,0}^{\frac{2}{\alpha}}$. Then \eqref{eq:tstep-ld-weaker-hetero} holds since $\delta_{m ,0}\sqrt{\frac{p\log n}{nh_{1}^3}}=O\left(\sqrt{\frac{p\log n}{nh_{1}}}\right)$ and
\[\left\| \hbe^{\,\left(1\right)}-\be^{*}\right\|_2=O_{\mathbb{P}}\left(\sqrt{\frac{p\log n}{nh_1}}+\delta_{m,0}^{2}+h_1^\alpha+\varepsilon\delta_{m ,0}^2\right)=O_{\mathbb{P}}\left(\left(\frac{p}{n}\right)^{\frac{\alpha}{2\alpha+1}}\sqrt{\log n}+\delta_{m,0}^{2}+\varepsilon\delta_{m,0}^{2}\right).\]
Assume \eqref{eq:tstep-ld-weaker-hetero} holds for $t-1$, i.e.,
\[\left\| \hbe^{\,\left(t-1\right)}-\be^{*}\right\|_2=O_{\mathbb{P}}\left((p/n)^{\alpha/(2\alpha+1)}\sqrt{\log n}+\delta_{m ,0}^{2^{t-1}}+\varepsilon\delta_{m ,0}^2\right).\]
Then $\delta_{m,t-1}=(p/n)^{\alpha/(2\alpha+1)}\sqrt{\log n}+\delta_{m ,0}^{2^{t-1}}$. Since $(p/n)^{\alpha/(2\alpha+1)}\sqrt{\log n} \ll (p/n)^{\frac{1}{2\alpha+1}} \leq h_t$ and $\delta_{m ,0}^{2^{t-1}} \leq \delta_{m ,0}^{2^{t}/\alpha} \leq h_t$, by \eqref{eq:1step-hetero},
\[\left\| \hbe^{\,\left(t\right)}-\be^{*}\right\|_2=O_{\mathbb{P}}\left(\left(\frac{p}{n}\right)^{\frac{\alpha}{2\alpha+1}}\sqrt{\log n}+\delta_{m ,0}^{2^{t}}+\varepsilon\delta_{m ,0}^2\right),\]
where we use $h^{\alpha}_t=\max\left\{\delta_{m ,0}^{2^{t}}, (p/n)^{\alpha/(2\alpha+1)}\right\}$ and $\sqrt{p/nh_t} \leq (p/n)^{\alpha/(2\alpha+1)}$. Therefore, we  prove that \eqref{eq:tstep-ld-weaker-hetero} holds for all $t$ by induction.

Plugging $\delta_{m,t-1}=(p/n)^{\alpha/(2\alpha+1)}\sqrt{\log n}+\delta_{m ,0}^{2^{t-1}}+\varepsilon\delta_{m ,0}^2$ into \eqref{eq:1step-hetero} yields that
\begin{align*}
	&\quad \left\| \hbe^{\,\left( t\right)}-\be^{*}\right\|_2\\
	&=O_{\mathbb{P}}\Bigg(\Big(\frac{p}{n}\Big)^{\frac{2\alpha}{2\alpha+1}}\log n+\delta_{m ,0}^{2^{t}}+h_t^{\alpha}+ \sqrt{\frac{p}{nh_t}}+\left(\Big(\frac{p}{n}\Big)^{\frac{\alpha}{2\alpha+1}}\sqrt{\log n}+\delta_{m ,0}^{2^{t-1}}\right)\sqrt{\frac{p\log n}{nh_t^3}}+\varepsilon\delta_{m ,0}^2\Bigg).
\end{align*}
Since $\alpha>1$,
\[\Big(\frac{p}{n}\Big)^{\frac{\alpha}{2\alpha+1}}\sqrt{\log n} \cdot \sqrt{\frac{p\log n}{nh_t^3}} = \Big(\frac{p}{n}\Big)^{\frac{\alpha}{2\alpha+1}} (\log n) \cdot \sqrt{\frac{p}{nh_t^3}} \leq \Big(\frac{p}{n}\Big)^{\frac{\alpha}{2\alpha+1}}(\log n) \cdot (p/n)^{\frac{\alpha-1}{2\alpha+1}} \lesssim \Big(\frac{p}{n}\Big)^{\frac{\alpha}{2\alpha+1}}.\]
Therefore, the rate in \eqref{eq:tstep-bound-detailed} is upper bounded by \[O\left(\Big(\frac{p}{n}\Big)^{
	\frac{\alpha}{2\alpha+1}}+\delta_{m ,0}^{2^{t}}+\delta_{m ,0}^{2^{t-1}}\Big(\frac{p}{n}\Big)^{\frac{\alpha-1}{2\alpha+1}}\sqrt{\log n}+\varepsilon\delta_{m ,0}^2\right),\]
which leads to that
\[\norm{\hbe_1^{(T)}-\be^*_1}_2 = O_{\Prob}\left((p/n)^{\alpha/(2\alpha+1)}+\varepsilon \delta^2_{m ,0}\right),\]
for sufficiently large $T$.

\end{proof}

\subsection{Proof of the Results for the High-dimensional {\mSMSE}}
\label{sec:pf-hd}

Before starting the proof, we first formalize our notation. Define 

\begin{equation}
V_{n}(\be)=\frac{1}{nh^2}\sum_{i=1}^{n} \left(-y_i\right) H^{\prime\prime}\left(\frac{x_i+\bz_i^{\top}\be}{h}\right)\bz_i\bz_i^{\top}.
\end{equation}

\begin{equation}
V_{m,\ell}(\be)=\frac{1}{mh^2}\sum_{i\in \mathcal{H}_{\ell}} \left(-y_i\right) H^{\prime\prime}\left(\frac{x_i+\bz_i^{\top}\be}{h}\right)\bz_i\bz_i^{\top}.
\end{equation}

\begin{equation}
U_{n}(\be)=\frac{1}{nh}\sum_{i=1}^{n}\left(-y_i\right) H^{\prime}\left(\frac{x_i+\bz_{i}^{\top}\be}{h}\right) \bz_{i}.
\end{equation}

As claimed before, we will not compute $V_{n}(\be)$ in the algorithm, but it is an important intermediate quantity in the theoretical analysis. Without loss of generality, we only consider $V_{m,1}(\be)$ in the sequel. We first give the convergence property of $V_{n}\left(\hbe^{(0)}\right)$, $V_{m,1}\left(\hbe^{(0)}\right)$ and $U_{n}\left(\hbe^{(0)}\right)$, which is crucial for deriving the convergence rate of $\hbe^{\left( 1\right)}$. Note that we omit the dependence of these quantities on the bandwidth $h$ in the notation.

\begin{lemma}\label{lem:D-hd}
Assume Assumptions \ref{A1}--\ref{A4}, \ref{B1}--\ref{B3} hold. Further assume that $\frac{\log m}{mh^3}=o\left( 1\right)$, $\sqrt{s}\delta_{m,0}=O\left( h^{3/2}\right)$ ($\delta_{m,0}$ is defined in Assumption \ref{B1}) and $h=o\left(1\right)$, we have the following results:
\begin{equation}
	\bigg\|(1-\E)\Big[V_{n}\big(\hbe^{(0)}\big)\Big]\bigg\|_{\max} = O_{\mathbb{P}}\left( \sqrt{\frac{\log p}{nh^3}}\right),
	\label{eq:Dn-hd}
\end{equation}	
\begin{equation}
	\sup_{\norm{\bv_1}_2=\norm{\bv_2}_2=1}\bv_1^{\top}\bigg(\E\Big[ V_{n}\big(\hbe^{(0)}\big)\Big]-V\bigg)\bv_2 =O_{\Prob}\left(\sqrt{s}\delta_{m,0}+h^{\alpha}\right),
	\label{eq:EDn-hd}
\end{equation}
\begin{equation}
	\bigg\|(1-\E)\Big[V_{m,1}\big(\hbe^{(0)}\big)\Big]\bigg\|_{\max} = O_{\mathbb{P}}\left( \sqrt{\frac{\log p}{mh^3}}\right),
	\label{eq:Dm-hd}
\end{equation}
\begin{equation}
	\sup_{\norm{\bv_1}_2=\norm{\bv_2}_2=1}\bv_1^{\top}\bigg(\E\Big[ V_{m,1}\big(\hbe^{(0)}\big)\Big]-V\bigg)\bv_2 =O_{\Prob}\left(\sqrt{s}\delta_{m,0}+h^{\alpha}\right).
	\label{eq:EDm-hd}
\end{equation}
Additionally, define \[\Psi_n\big(\hbe^{(0)}\big):=U_{n}\big(\hbe^{(0)}\big)-V_{n}\big(\hbe^{(0)}\big)\big(\hbe^{\left( 0\right)}-\be^*\big),\] and then we have
\begin{equation}
	\left\|(1-\E)\Psi_n\big(\hbe^{(0)}\big)\right\|_\infty = O_{\mathbb{P}}\left( \sqrt{\frac{\log p}{nh}}\right),
	\label{eq:z-hd}
\end{equation}
and
\begin{equation}
	\sup_{\norm{\bv}_2=1}\bigg|\bv^{\top}\E \left[\Psi_n\big(\hbe^{(0)}\big)\right]\bigg| = O_{\mathbb{P}}\left( s\delta^2_{m,0}+h^{\alpha}\right).
	\label{eq:Ez-hd}
\end{equation}

\end{lemma}

\begin{proof} \label{pf:thm:D-hd}

\noindent \textbf{Proof of \eqref{eq:Dn-hd}:}

Recall our definitions 
\[V_{n}\left(\hbe^{(0)}\right)=\frac{1}{nh^2}\sum_{i=1}^{n} \left(-y_i\right) H^{\prime\prime}\left(\frac{x_i+\bz_i^{\top}\hbe^{\left(0\right)}}{h}\right)\bz_i\bz_i^{\top},\]
and
\[V=-2\mathbb{E}\left[\rho\left(0\mid \bZ\right)F^{\prime}\left(0 \mid \bZ\right)\bZ\bZ^{\top}\right].\]
By Assumption \ref{B1}, there exists constant $C_{\ell_1}, C_{\ell_2}$ such that $\Prob\left(\hbe^{\left(0\right)} \in \Theta\right) \to 1$,
where
\[\Theta:=\braces{\be:\norm{\be-\be^*}_2 \leq C_{\ell_1}\delta_{m,0}, \norm{\be-\be^*}_1  \leq C_{\ell_2}\sqrt{s}\delta_{m,0}}.\]
Without loss of generality, we assume $C_{\ell_1}=C_{\ell_2}=1$ and $\hbe^{\left(0\right)}\in \Theta$ in the following proof.

For each $\left(j_1,j_2\right)\in \{1,..,p\} \times \{1,..,p\}$, define \[V_{n,j_1,j_2}\left(\be\right):=\frac{1}{nh^2}\sum_{i=1}^{n}\left(-y_i\right) H^{\prime\prime}\left(\frac{x_i+\bz_{i}^{\top}\be}{h}\right)z_{i,j_1}z_{i,j_2},\]
\[V_{j_1,j_2}:=-2\mathbb{E}\left[\rho\left(0\mid \bZ\right)F^{\prime}\left(0 \mid \bZ\right)Z_{j_1}Z_{j_2}\right],\]
\[\phi^V_{ij_1j_2}\left(\be\right):=\frac{-y_i}{h^2} H^{\prime\prime}\left(\frac{x_i+\bz_{i}^{\top}\be}{h}\right)z_{i,j_1}z_{i,j_2}+\frac{y_i}{h^2} H^{\prime\prime}\left(\frac{x_i+\bz_{i}^{\top}\be^*}{h}\right)z_{i,j_1}z_{i,j_2},\]
and 
\[\Phi^V_{n,j_1,j_2}:=\sup_{\be \in \Theta}\abs{\paren{1-\E}\brackets{V_{n,j_1,j_2}\left(\be\right)-V_{n,j_1,j_2}\left(\be^*\right)}}=\sup_{\be \in \Theta}\abs{\paren{1-\E}\frac{1}{n}\sum_{i=1}^n \phi^V_{ij_1j_2}\left(\be\right)}.\]
Since
\begin{equation}
	\begin{aligned}
		&\quad\left\|(1-\E)\left[V_{n}\left(\hbe^{(0)}\right)\right]\right\|_{\max}\\
		&= \sup_{j_1,j_2}\abs{(1-\E)\left[V_{n,j_1,j_2}\left(\hbe^{\left(0\right)}\right)\right]}\\
		&\leq\sup_{j_1,j_2}\Phi^V_{n,j_1,j_2} + \sup_{j_1,j_2}\abs{\left(1-\E\right)V_{n,j_1,j_2}\left(\be^*\right)},
	\end{aligned}
	\label{eq:breakVnh-V}
\end{equation}
we break the proof of \eqref{eq:Dn-hd} into two steps, separately controlling the two terms in \eqref{eq:breakVnh-V}. 

\textbf{Step 1}:
\begin{equation}
	\sup_{j_1,j_2}\Phi^V_{n,j_1,j_2}=O_{\mathbb{P}}\left(\delta_{m,0}\sqrt{\frac{s\log p}{nh^6}}\right)\xlongequal{\sqrt{s}\delta_{m,0}=O\left(h^{3/2}\right)}O_{\Prob}\left(\sqrt{\frac{\log p}{nh^3}}\right).
	\label{eq:supPhi}
\end{equation}
The proof in this step is analogous to the proof of Lemma B.1 in \cite{luo2022distributed}.
Since $\abs{z_{i,j}}$ is upper bounded by $\overline{B}$, for any $i$ and $\be \in \Theta$, we have $\bz_i^{\top}\paren{\be-\be^*}\leq \overline{B}\norm{\be-\be^*}_1\leq \overline{B}\sqrt{s}\delta_{m,0}$.
Since $H^{\prime\prime}\left(x\right)$ is Lipschitz, we have
\begin{equation}
	\overline{\phi}:=\sup_{i,j_1,j_2}\sup_{\be \in \Theta} \abs{\phi^V_{ij_1j_2}\left(\be\right)}=O\left(\frac{\sqrt{s}\delta_{m,0}}{h^3}\right).
	\label{eq:phiabs}
\end{equation} 
Since $\rho\left(\cdot \mid \bZ\right)$ is bounded, we also have
\begin{equation}
	\begin{aligned}
		&\quad \sup_{\be \in \Theta}\E_{.|\bZ} \left[H^{\prime\prime}\left(\frac{X+\bZ^{\top}\be}{h}\right)-H^{\prime\prime}\left(\frac{X+\bZ^{\top}\be^*}{h}\right)\right]^2\\
		&=\sup_{\be \in \Theta} h\int_{-1}^1 \left[H^{\prime\prime }\left(\xi+\frac{\bZ^{\top}\paren{\be-\be^*}}{h}\right)-H^{\prime\prime }\left(\xi\right)\right]^2\rho\left(h\xi\mid\bZ\right)\diff \xi\\
		&=O\left(s\delta_{m,0}^2/h\right),
	\end{aligned}
	\label{eq:EH22-hd}
\end{equation}
which implies that
\begin{equation}
	\sup_{i,j_1,j_2}\sup_{\be \in \Theta}\E \left[\abs{\phi^V_{ij_1j_2}\left(\be\right)}^2 \right]= O\left(s\delta_{m,0}^2/h^5\right).
	\label{eq:Ephi2}
\end{equation}  
Define $\sigma_1,...,\sigma_{n}$ to be independent Rademacher variables, i.e., binary variables that are uniformly distributed on $\braces{-1,+1}$. By Rademacher symmetrization,
\[\E\Phi^V_{n,j_1,j_2}\leq 2\E\sup_{\be \in \Theta}\abs{\frac{1}{n}\sum_{i=1}^{n}\sigma_i\phi^V_{ij_1j_2}\left(\be\right)}.\]
Further, as
\[\phi^V_{ij_1j_2}\left(\be\right):=\frac{-y_i}{h^2} H^{\prime\prime}\left(\frac{x_i+\bz_{i}^{\top}\be^*+\bz_i^{\top}\left(\be-\be^*\right)}{h}\right)z_{i,j_1}z_{i,j_2}+\frac{y_i}{h^2} H^{\prime\prime}\left(\frac{x_i+\bz_{i}^{\top}\be^*}{h}\right)z_{i,j_1}z_{i,j_2},\]
we can be view $\phi^V_{ij_1j_2}\left(\be\right)$ as a function of $\bz_i^{\top}\left(\be-\be^*\right)$ with Lipschitz constant $\asymp\overline{B}^2/h^3$. By Talagrand’s Lemma,
\[\E_\sigma \sup_{\be \in \Theta}\abs{\frac{1}{n}\sum_{i=1}^{n}\sigma_i\phi^V_{ij_1j_2}\left(\be\right)}\lesssim \frac{1}{h^3}\E_\sigma  \sup_{\be \in \Theta}\abs{\frac{1}{n}\sum_{i=1}^{n}\sigma_i\bz_i^{\top}\paren{\be-\be^*}}\lesssim \left(\frac{\sqrt{s}\delta_{m,0}}{nh^3}\right)\E_\sigma \norm{\sum_{i=1}^{n}\sigma_i\bz_i}_{\infty}.\]
Since $\sigma_iz_{i,j} \in \brackets{-\overline{B},+\overline{B}}$ for all $j \in \braces{1,...,p}$, by Hoeffding's inequality,
\[\mathbb{P}\paren{\abs{\sum_{i=1}^{n}\sigma_iz_{i,j}}\geq \sqrt{4\overline{B}^2n\log \max\braces{n,p}}} \leq 2\exp\paren{-\frac{4n\overline{B}^2\log \max\braces{n,p}}{2n\overline{B}^2}}=\frac{2}{\max\braces{n,p}^{2}},\]
which implies that with probability at least $1-\frac{2}{\max\braces{n,p}}$,
\[\norm{\sum_{i=1}^{n}\sigma_i\bz_i}_{\infty} \leq 2\overline{B}\sqrt{n\log \max\braces{n,p}}.\]
Assumption \ref{B2} supposes that $\log p = O\left(\log n\right)$, and thus we obtain
\[\E_\sigma \norm{\sum_{i=1}^{n}\sigma_i\bz_i}_{\infty} \leq 2\overline{B} \sqrt{n\log \max\braces{n,p}}\left(1-\frac{2}{\max\braces{n,p}}\right)+\frac{2\overline{B}}{\max\braces{n,p}}=O\left(\sqrt{n\log p}\right),\]
and hence
\begin{equation}
	\E\Phi^V_{n,j_1,j_2}=O\left(\delta_{m,0}\sqrt{\frac{s\log p}{nh^6}}\right).
	\label{eq:EPhi}
\end{equation}
To show \eqref{eq:supPhi}, we use Theorem 7.3 in \cite{bousquet2003concentration}, which is restated as the following Lemma \ref{lem:Bousquet}.
\begin{lemma}[\cite{bousquet2003concentration}]
	Assume $\{\bz_i\}_{i=1}^n$ are identically distributed random variables. Let $\mathcal{F}$ be a set of countable real-value functions such that all functions $f \in \mathcal{F}$ are measurable, square-integrable and satisfy $\E f\left(\bz_i\right)=0$. Assume $\sup_{f,\bz} f\left(\bz\right) \leq 1$. Define 
	\[\Upsilon:=\sup_{f \in \mathcal{F}}\sum_{i=1}^n f\left(\bz_i\right).\]
	If $\sum_{i=1}^n \sup_{f \in \mathcal{F}} \E f^2\left(\bz_i\right) \leq n\sigma^2$, then for all $x>0$, we have
	\[\mathbb{P}\paren{\Upsilon>\E \Upsilon + \sqrt{2x\paren{n\sigma^2+2\E \Upsilon}}+\frac{x}{3}}< e^{-x}.\]
	
	\label{lem:Bousquet}
\end{lemma}
Note that Lemma \ref{lem:Bousquet} requires $\mathcal{F}$ to be countable. We first apply Lemma \ref{lem:Bousquet} to prove  \eqref{eq:supPhi} on rational $\be$, i.e.,
\[\sup_{j_1, j_2}\sup_{\be \in \Theta \cap \mathbb{Q}^p}\abs{\paren{1-\E}\frac{1}{n}\sum_{i=1}^n \phi^V_{ij_1j_2}\left(\be\right)}=O_{\mathbb{P}}\left(\delta_{m,0}\sqrt{\frac{s\log n}{nh^6}}\right).\]
Fix $j_1, j_2$ and take 
\[\mathcal{F}:=\braces{f_{\be}\left(\bz_i\right)=\frac{\paren{1-\E} \phi^V_{ij_1j_2}\left(\be\right)}{2\overline{\phi}}:\be \in \Theta \cap \mathbb{Q}^p}.\]

By \eqref{eq:phiabs}, \eqref{eq:Ephi2} and \eqref{eq:EPhi}, we have
$f\left(\bz_i\right)\leq 1$, 
\[\sum_{i=1}^n \sup_{f_{\be} \in \mathcal{F}} \E f_{\be}^2\left(\bz_i\right) =O\left(\frac{ns\delta_{m,0}^2}{\overline{\phi}^2h^5}\right),\] 
and 
\[\E \Upsilon=O\left(\frac{\delta_{m,0}}{\overline{\phi}h^3}\sqrt{ns\log p}\right).\]

By Lemma \ref{lem:Bousquet}, for all $x>0$, with probability $1-e^{-x}$, 
\[ \frac{1}{n}\sup_{\be \in \Theta\cap\mathbb{Q}^p}\paren{1-\E}\sum_{i=1}^n \phi^V_{ij_1j_2}\left(\be\right) =O\left(\frac{\delta_{m,0}}{h^3}\sqrt{\frac{s\log p}{n}}+\sqrt{2x\frac{s\delta_{m,0}^2}{nh^5}+4x\frac{\overline{\phi}\delta_{m,0}\sqrt{s\log p}}{n^{3/2}h^3}}+\frac{\overline{\phi}x}{3n}\right).\]
By taking $x=3\log \max\{n,p\}$, plugging \eqref{eq:phiabs} in and using that $\log p=O\left(\log n\right)$ again, the above bound can be written as
\[O\brackets{\delta_{m,0}\paren{\sqrt{\frac{s\log p}{nh^6}}+\sqrt{\frac{s\log p}{nh^5}+\frac{s\log p}{n^{3/2}h^6}}+\frac{\sqrt{s}\log p}{nh^3}}}=O\left(\delta_{m,0}\sqrt{\frac{s\log p}{nh^6}}\right).\]
For the same reason, with probability $1-1/\max\{n,p\}^3$,
\[\frac{1}{n}\sup_{\be \in \Theta\cap\mathbb{Q}^p }\paren{1-\E}\sum_{i=1}^n \brackets{-\phi^V_{ij_1j_2}\left(\be\right)} = O\left(\delta_{m,0}\sqrt{\frac{s\log p}{nh^6}}\right).\]
Therefore, with probability $1-2/\max\{n,p\}^3$, 
\[\sup_{\be \in \Theta \cap \mathbb{Q}^p}\abs{\paren{1-\E}\frac{1}{n}\sum_{i=1}^n \phi^V_{ij_1j_2}\left(\be\right)}=O\left(\delta_{m,0}\sqrt{\frac{s\log p}{nh^6}}\right).\]
By the continuity of $\phi^V_{ij_1j_2}\left(\be\right)$,
\[\sup_{\be \in \Theta }\abs{\paren{1-\E}\frac{1}{n}\sum_{i=1}^n \phi^V_{ij_1j_2}\left(\be\right)}=O\left(\delta_{m,0}\sqrt{\frac{s\log p}{nh^6}}\right),\]
with the same probability. This is true for any $j_1, j_2$,
hence 
\[\sup_{j_1,j_2}\Phi^V_{n,j_1,j_2} = O\left(\delta_{m,0}\sqrt{\frac{s\log p}{nh^6}}\right)\xlongequal{\sqrt{s}\delta_{m,0}=O\left(h^{3/2}\right)}O\left(\sqrt{\frac{\log p}{nh^3}}\right),\]
with probability at least $1-2/\max\{n,p\}$, which completes the proof of \eqref{eq:supPhi}.

\textbf{Step 2}:
\begin{equation}
	\sup_{j_1,j_2}\left|\left(1-\E\right)V_{n,j_1,j_2}\left(\be^*\right)\right|=O_{\mathbb{P}}\left(\sqrt{\frac{\log p}{nh^3}}\right).
	\label{eq:sup1-EVbeta*}
\end{equation}
Recall that 
\[V_{n,j_1,j_2}\left(\be^*\right):=\frac{1}{nh^2}\sum_{i=1}^{n}\left(-y_i\right) H^{\prime\prime}\left(\frac{x_i+\bz_{i}^{\top}\be^*}{h}\right)z_{i,j_1}z_{i,j_2}.\]
We have
\[\sup_{i} \abs{\frac{-y_i}{h^2} H^{\prime\prime}\left(\frac{x_i+\bz_{i}^{\top}\be^*}{h}\right)z_{i,j_1}z_{i,j_2}} = O\left(1/h^2\right),\]
and
\[\sup_{i} \E \abs{\frac{-y_i}{h^2} H^{\prime\prime}\left(\frac{x_i+\bz_{i}^{\top}\be^*}{h}\right)z_{i,j_1}z_{i,j_2}}^2 = O\left(1/h^3\right),\]
as $H^{\prime\prime}\left(x\right)$, $\abs{z_{i,j}}$ is bounded and 
\begin{equation}
	\E_{\cdot\mid\bZ}\left[H^{\prime\prime}\left(\frac{X+\bZ^{\top}\be^*}{h}\right)\right]^2 = h \int_{-1}^1\left[H^{\prime\prime}\left(\xi\right)\right]^2\rho\left(\xi h \mid \bZ\right) \diff\xi =O\left(h\right).
	\label{eq:EH22*-hd}
\end{equation}
By Bernstein's inequality, there exists a constant $C>0$,
\[\Prob\paren{\abs{\paren{1-\E}V_{n,j_1,j_2}\left(\be^*\right)}\geq \sqrt{\frac{C_2\log \max\{n,p\}}{nh^3}}} \leq 2\exp\paren{\frac{-\frac{C}{2}\log \max\{n,p\}}{1+\frac{1}{3}\sqrt{C\log \max\{n,p\}/ \left(nh\right)}}}.\]
Our assumptions $\log m/mh^3=o(1)$, $m>n^{c}$ and $p=O(n^{\gamma})$ ensure that $\log \max\{n,p\}/ \left(nh\right)=o(1)$, and hence we can take a sufficiently large $C$ to make
\[2\exp\paren{\frac{-\frac{C}{2}\log \max\{n,p\}}{1+\frac{1}{3}\sqrt{C\log \max\{n,p\}/ \left(nh\right)}}} \leq \frac{2}{\max\{n,p\}^3}.\] 
This implies that
$$\Prob\paren{\sup_{j_1,j_2}\abs{\paren{1-\E}V_{n,j_1,j_2}\left(\be^*\right)}\leq  \sqrt{\frac{C\log \max\{n,p\}}{nh^3}}} \geq 1-\frac{2}{\max\{n,p\}},$$
which proves \eqref{eq:sup1-EVbeta*}. Together with \eqref{eq:breakVnh-V} and \eqref{eq:supPhi}, we conclude the proof of \eqref{eq:Dn-hd}.

\textbf{Proof of \eqref{eq:EDn-hd}:}


Recall that, by Equation \eqref{eq:Taylor}, for almost every $\bZ$,
\begin{equation}
	\begin{aligned}
		&\quad \left(2F\left(-t \mid \bZ\right)-1\right)\rho\left(t \mid \bZ\right)\\
		& =2F^{\,\left(1\right)}\left(0\mid\bZ\right)\rho\left(0 \mid \bZ\right)t+\sum_{k=2}^{2\alpha+1}M_k\left(\bZ\right)t^k,
	\end{aligned}
\end{equation}
where $M_{k}\left(\bZ\right)$ is a constant depending on $\bZ$, $t^{\prime}$ and $t^{\prime\prime}$. Since  $\rho^{\left(k\right)}\left(\cdot\mid \bZ\right)$ and  $F^{\left(k\right)}\left(\cdot\mid \bZ\right)$ are bounded around 0 for all $k$, we know there exists a constant $M$ such that $\sup_{k}\abs{M_k\left(\bZ\right)} \leq M$ for all $\bZ,t^{\prime},t^{\prime\prime}$. In the following computation, we let $t=\xi h-\bZ^{\top}\Delta(\be)$, where $\Delta\left(\be\right):=\be-\be^*$.

Recall that when $x>1$ or $x<-1$,  $H^{\prime}\left(x\right)=H^{\prime\prime}\left(x\right)=0$. The kernel $H^{\prime}\left(x\right)$ is bounded, $\int_{-1}^1 H^{\prime}\left(x\right)\diff x=1$, and $\int_{-1}^1 x^kH^{\prime}\left(x\right)\diff x=0$ for all $1\leq k\leq\alpha-1$. Moreover, $\int_{-1}^1 xH^{\prime\prime}\left(x\right)\diff x=-1$ and $\int_{-1}^1 x^kH^{\prime\prime}\left(x\right)\diff x=0$ for $k=0$ and $2\leq k\leq\alpha$. Also, recall that $\zeta=X+\bZ^{\top}\be^*$ and $-y=-\mathrm{sign}\left(y^*\right)=-\mathrm{sign}\left(\bZ+\epsilon\right)=2\I\left(\bZ+\epsilon<0\right)-1$.

For all $\bv_1, \bv_2 $ that satisfies $\norm{\bv_1}_2=\norm{\bv_2}_2=1$ and $\be \in \Theta$, 
\begin{equation}
	\begin{aligned}
		&\mathbb{E}_{\cdot|\bZ} \left[\frac{\left(\bv_1^{\top}\bZ\right)\left(\bv_2^{\top}\bZ\right)}{h^2}\left(-Y\right) H^{\prime\prime}\left(\frac{{X+\bZ}^{\top}\be}{h}\right)\right]\\
		=&\frac{\left(\bv_1^{\top}\bZ\right)\left(\bv_2^{\top}\bZ\right)}{h^2}\mathbb{E}_{\cdot\mid\bZ} \left[2\I\left(\bZ+\epsilon<0\right)-1\right]H^{\prime\prime}\left(\frac{\bZ^{\top}\Delta\left(\be\right)+\zeta}{h}\right)\\
		=&\frac{\left(\bv_1^{\top}\bZ\right)\left(\bv_2^{\top}\bZ\right)}{h} \int_{-1}^1 \brackets{2F\left(\bZ^{\top}\Delta\left(\be\right)-\xi h\mid\bZ\right)-1}\rho\left(\xi h-\bZ^{\top}\Delta\left(\be\right)\mid\bZ\right)H^{\prime\prime}\left(\xi\right)\diff \xi\\
		=&\,\frac{\left(\bv_1^{\top}\bZ\right)\left(\bv_2^{\top}\bZ\right)}{h}2F^{\,\left(1\right)}\left(0\mid\bZ\right)\rho\left(0 \mid \bZ\right)\int_{-1}^1\left(\xi h-\bZ^{\top}\Delta\left(\be\right)\right)H^{\prime\prime}\left(\xi\right)\diff \xi\\ 
		&+ \frac{\left(\bv_1^{\top}\bZ\right)\left(\bv_2^{\top}\bZ\right)}{h}\cdot  \sum_{k=2}^{2\alpha+1} M_k\left(\bZ\right) \int_{-1}^{1} \left(\xi h-\bZ^{\top}\Delta\left(\be\right)\right)^{k} H^{\prime\prime}\left(\xi\right) \diff \xi\\
	\end{aligned}
	\label{eq:EV-hd}
\end{equation}

For $1\leq k \leq \alpha$,
\[\int_{-1}^{1} \left(\xi h-\bZ^{\top}\Delta\left(\be\right)\right)^{k} H^{\prime\prime}\left(\xi\right) \diff \xi=\sum_{k^{\prime}=0}^k h^{k^{\prime}} \left(\bZ^{\top}\Delta\left(\be\right)\right)^{k-k^{\prime}}\int_{-1}^{1}  \xi^{k^{\prime}} H^{\prime\prime}\left(\xi\right) \diff \xi=h\left(\bZ^{\top}\Delta\left(\be\right)\right)^{k-1}.\]

For $\alpha+1\leq k \leq 2\alpha+1$, since $H^{\prime\prime}\left(x\right)$ is bounded,
\begin{align*}
	\abs{\int_{-1}^{1} \left(\xi h-\bZ^{\top}\Delta\left(\be\right)\right)^{k} H^{\prime\prime}\left(\xi\right) \diff \xi}&\leq \int_{-1}^{1}2^{k-1}\left(\abs{\xi h}^k+\abs{\bZ^{\top}\Delta\left(\be\right)}^k\right)\abs{ H^{\prime\prime}\left(\xi\right)} \diff \xi\\
	&\leq 2^{2\alpha+1}\sup_x \abs{H^{\prime\prime}\left(x\right)}\brackets{h^{\alpha+1}+\abs{\bZ^{\top}\Delta\left(\be\right)}^k}\\
	&\leq2^{2\alpha+2}\left(1+\overline{B}^k\right)\sup_x \abs{H^{\prime\prime}\left(x\right)}h^{\alpha+1} .
\end{align*}
The last inequality holds because $\abs{Z_{j}}\leq \overline{B}$, $\bZ^{\top}\Delta\left(\be\right) \leq \overline{B}\sqrt{s}\delta_{m,0} $, $\sqrt{s}\delta_{m,0}=O\left(h^{3/2}\right)$ and $h=o\left(1\right)$. Hence,
\begin{align*}
	&\quad \mathbb{E}_{\cdot|\bZ} \left[\frac{\left(\bv_1^{\top}\bZ\right)\left(\bv_2^{\top}\bZ\right)}{h}\left(-y\right) H^{\prime\prime}\left(\frac{X+{\bZ}^{\top}\be}{h}\right)\right]-\frac{\left(\bv_1^{\top}\bZ\right)\left(\bv_2^{\top}\bZ\right)}{h}2F^{\,\left(1\right)}\left(0\mid\bZ\right)\rho\left(0 \mid \bZ\right)\\
	&\leq \abs{\frac{\left(\bv_1^{\top}\bZ\right)\left(\bv_2^{\top}\bZ\right)}{h}\sum_{k=2}^{\alpha} M_k\left(\bZ\right)h\left(\bZ^{\top}\Delta\left(\be\right)\right)^{k-1}}\\
	&+\abs{\frac{\left(\bv_1^{\top}\bZ\right)\left(\bv_2^{\top}\bZ\right)}{h}2^{2\alpha+2}\sup_x \abs{H^{\prime\prime}\left(x\right)}\sum_{k=\alpha+1}^{2\alpha+1}\left(1+\overline{B}^k\right) M_k\left(\bZ\right)h^{\alpha+1}}\\
	&\leq  C\left(\bv_1^{\top}\bZ\right)\left(\bv_2^{\top}\bZ\right)\left(\sqrt{s}\delta_{m ,0}+h^{\alpha}\right), 
\end{align*}
where $C$ 
is a constant not depending on $\be$ and $\bZ$. Therefore, by the assumption that $\bZ$ has finite second moment and Cauchy-Schwarz inequality, we obtain that
\begin{equation}
	\begin{aligned}
		&\quad\sup_{\be\in\Theta} \sup_{\norm{\bv_1}_2=\norm{\bv_2}_2=1}\bv_1^{\top}\left(\E\left[ V_{n}\left(\be\right)\right]-V\right)\bv_2 \\
		&=\sup_{\be\in\Theta} \sup_{\norm{\bv_1}_2=\norm{\bv_2}_2=1}\mathbb{E}\brackets{  \frac{\left(\bv_1^{\top}\bZ\right)\left(\bv_2^{\top}\bZ\right)}{h}\left(-y\right) H^{\prime\prime}\left(\frac{{\bZ}^{\top}\be}{h}\right)-\frac{\left(\bv_1^{\top}\bZ\right)\left(\bv_2^{\top}\bZ\right)}{h}2F^{\,\left(1\right)}\left(0\mid\bZ\right)\rho\left(0 \mid \bZ\right)}\\
		& \lesssim \sqrt{s}\delta_{m,0}+h^\alpha,
	\end{aligned}
\end{equation}
which completes the proof of $\eqref{eq:EDn-hd}$.

\textbf{Proof of  \eqref{eq:Dm-hd} and \eqref{eq:EDm-hd}:}

Equation \eqref{eq:Dm-hd} and \eqref{eq:EDm-hd} can be shown in the same way as above by replacing all the $n$ with $m$.

\textbf{Proof of \eqref{eq:z-hd}:}

The proof of \eqref{eq:z-hd} is analogous to that of $\eqref{eq:Dn-hd}$. We will omit some details since they are the same. For each $j \in \{1,...,p\}$, define
\begin{align*}
	U_{n,h,j}\left(\be\right)&:=\frac{1}{nh^2}\sum_{i=1}^{n} \left(-y_i\right) H^{\prime\prime}\left(\frac{x_i+\bz_{i}^{\top}\be}{h}\right)\bz_{i}^{\top}\left(\be-\be^*\right) z_{i,j}\\
	&-\frac{1}{nh}\sum_{i=1}^{n}\left(-y_i\right) H^{\prime}\left(\frac{x_i+\bz_{i}^{\top}\be}{h}\right) z_{i,j}.		
\end{align*}
Then we have
\begin{align*}
	&\quad \left \|(1-\E)\Psi_n\big(\hbe^{(0)}\big) \right\|_{\infty}\\
	&= \left\|(1-\E)\bigg[\frac{1}{nh^2}\sum_{i=1}^{n} \left(-y_i\right) H^{\prime\prime}\Big(\frac{x_i+\bz_{i}^{\top}\hbe^{\,\left(0\right)}}{h}\Big)\bz_{i}^{\top}\left(\hbe^{\,\left(0\right)}-\be^*\right) \bz_{i}
	-\frac{1}{nh}\sum_{i=1}^{n}\left(-y_i\right) H^{\prime}\Big(\frac{x_i+\bz_{i}^{\top}\be^{\,\left(0\right)}}{h}\Big) \bz_{i} \bigg]\right\|_{\infty}\\
	&\leq\sup_{j}\sup_{\be \in \Theta}\left|\paren{1-\E}\paren{ U_{n,h,j}\left(\be\right)-U_{n,h,j}\left(\be^{*}\right)}\right| + \sup_{j} \abs{\paren{1-\E} U_{n,h,j}\left(\be^{*}\right)}.
\end{align*}

Define
\[\phi^U_{i,j}\left(\be\right):=\left|\frac{-y_i}{h^2}H^{\prime\prime}\left(\frac{x_i+\bz_{i}^{\top}\be}{h}\right)\bz_{i}^{\top}\left(\be-\be^*\right) z_{i,j}-\frac{-y_i}{h} \brackets{H^{\prime}\left(\frac{x_i+\bz_{i}^{\top}\be}{h}\right)-H^{\prime}\left(\frac{x_i+\bz_{i}^{\top}\be^*}{h}\right)}z_{i,j}\right|.\]

\[\Phi^U_{n,h,j}:=\sup_{\be \in \Theta}\left|\paren{1-\E}\paren{ U_{n,h,j}\left(\be\right)-U_{n,h,j}\left(\be^{*}\right)}\right|=\sup_{\be \in \Theta}\abs{\paren{1-\E}\frac{1}{n}\sum_{i=1}^n\phi^U_{i,j}\left(\be\right)}.\]

Similar to the analysis of $\phi^V_{i,j}$, we have
\[\sup_{i,j}\sup_{\be \in \Theta}\abs{\phi^U_{i,j}\left(\be\right)}=O\left(\frac{\sqrt{s}\delta_{m,0}}{h^2}\right),\]
and 
\[\sup_{i,j}\sup_{\be \in \Theta}\E\brackets{\phi^U_{i,j}\left(\be\right)}^2=O\left(\frac{s\delta_{m,0}^2}{h^3}\right).\]

By Rademacher symmetrization, Talagrand's concentration principle and Hoeffding's inequality, 
\[ \E \Phi^U_{n,h,j} \leq 2\E \sup_{\be \in \Theta}\abs{\frac{1}{n}\sum_{i=1}^n\sigma_i\phi^U_{i,j}\left(\be\right)} \lesssim  \paren{\frac{\sqrt{s}\delta_{m,0}}{nh^2}}\cdot \E \paren{\E_\sigma \norm{\sum_{i=1}^n\sigma_i\bz_i^{\top}}_\infty}\lesssim\delta_{m,0}\sqrt{\frac{s\log p}{nh^4}}.\]
(Details are the same as the proof of \eqref{eq:EPhi}, while the only difference is that $\phi^U_{ij}\left(\be\right)$ is a Lipschitz function of $\bz_i^{\top}\paren{\be-\be^*}$ with Lipschitz constant $\asymp 1/h^2$, instead of $1/h^3$ for $\phi^V_{ij}\left(\be\right)$.)

Using Lemma \ref{lem:Bousquet} again,
we can show that
\[\sup_j \Phi^{U}_{n,h,j} = O_{\Prob}\left(\delta_{m,0}\sqrt{\frac{s\log p}{nh^4}}\right)\xlongequal{\sqrt{s}\delta_{m,0}=O\left(h^{3/2}\right)}O\left(\sqrt{\frac{\log p}{nh}}\right).\]

Similar to the proof of \eqref{eq:sup1-EVbeta*},
we have
\[\sup_{\be \in \Theta}\sup_{i,j} \abs{\frac{-y_i}{h} H^{\prime}\left(\frac{x_i+\bz_{i}^{\top}\be^*}{h}\right) x_{i,j}}=O\left(1/h\right),\]
and
\[\sup_{\be \in \Theta}\sup_{i,j} \E \abs{\frac{-y_i}{h} H^{\prime}\left(\frac{x_i+\bz_{i}^{\top}\be^*}{h}\right) x_{i,j}}^2=O\left(1/h\right).\]
By Bernstein's inequality,
\[
\sup_j \abs{\paren{1-\E}U_{n,h,j}\left(\be^*\right)} = O_{\Prob}\left(\sqrt{\frac{\log p}{nh}}\right).
\]

\textbf{Proof of \eqref{eq:Ez-hd}:}

Recall equation \eqref{eq:Taylor} and $\pi_U=\int_{-1}^1x^{\alpha}H^{\prime}\left(x\right)\diff x \neq 0$. For any $\bv \in \R^{p}$ that satisfies $\norm{\bv}_2=1$ and $\be \in \Theta$, we have
\begin{equation}
	\begin{aligned}
		&\mathbb{E}_{\cdot|\bZ} \bv^{\top}\Psi_{n}\left(\be\right) \\
		=&(\bv^{\top}\bZ) \cdot \mathbb{E}_{\cdot|\bZ} \left[ \frac{\bZ^{\top}\Delta\left(\be\right)}{h^2} [2\I\left(\zeta+\epsilon<0\right)-1] H^{\prime\prime}\left(\frac{X+\bZ^{\top}\be}{h}\right)-\frac{1}{h}[2\I\left(\zeta+\epsilon<0\right)-1]H^{\prime}\left(\frac{X+\bZ^{\top}\be}{h}\right) \right]\\
		=&(\bv^{\top}\bZ) \int_{-1}^{1} \brackets{2F\left(\bZ^{\top}\Delta\left(\be\right)-\xi h \mid \bZ\right)-1}\rho\left(\xi h-\bZ^{\top}\Delta\left(\be\right) \mid \bZ\right)\\
		&\quad\cdot\left( \frac{\bZ^{\top}\Delta\left(\be\right)}{h}H^{\prime\prime}\left(\xi\right)-H^{\prime}\left(\xi\right)\right)  \diff \xi\\
		=&\,(\bv^{\top}\bZ)\sum_{k=1}^{2\alpha+1} M_k\left(\bZ\right)\int_{-1}^{1} \left(\xi h-\bZ^{\top}\Delta\left(\be\right)\right)^{k}\left( \frac{\bZ^{\top}\Delta\left(\be\right)}{h}H^{\prime\prime}\left(\xi\right)-H^{\prime}\left(\xi\right)\right) \diff \xi 
	\end{aligned}\label{EUnhj}
\end{equation}
For $1\leq k\leq\alpha-1$,
\begin{equation*}
	\begin{aligned}
		&\sup_{\be\in \Theta}\abs{\int_{-1}^{1} \left(\xi h-\bZ^{\top}\Delta\left(\be\right)\right)^k\left( \frac{\bZ^{\top}\Delta\left(\be\right)}{h}H^{\prime\prime}\left(\xi\right)-H^{\prime}\left(\xi\right)\right) \diff \xi}\\
		=&\sup_{\be\in \Theta}\abs{\sum_{k^{\prime}=0}^{k} \binom{k}{k^{\prime}}h^{k^{\prime}}\left(-\bZ^{\top}\Delta\left(\be\right)\right)^{k-k^{\prime}}\left[\left(\bZ^{\top}\Delta\left(\be\right)/h\right) \int_{-1}^1\xi^{k^{\prime}}H^{\prime\prime}\left(\xi\right)\diff\xi- \int_{-1}^1\xi^{k^{\prime}}H^{\prime}\left(\xi\right)\diff\xi \right]}\\ =&\,\left(k-1\right)\abs{-\bZ^{\top}\Delta\left(\be\right)}^k = O \left(s\delta_{m,0}^2\right).\\
	\end{aligned}
\end{equation*}
For $k=\alpha$,
\begin{equation*}
	\begin{aligned}
		&\sup_{\be\in \Theta}\abs{\int_{-1}^{1} \left(\xi h-\bZ^{\top}\Delta\left(\be\right)\right)^\alpha\left( \frac{\bZ^{\top}\Delta\left(\be\right)}{h}H^{\prime\prime}\left(\xi\right)-H^{\prime}\left(\xi\right)\right) \diff \xi}\\
		=&\sup_{\be\in \Theta}\abs{\sum_{k=0}^{\alpha} \binom{\alpha}{k}h^k\left(-\bZ^{\top}\Delta\left(\be\right)\right)^{\alpha-k}\left[\left(\bZ^{\top}\Delta\left(\be\right)/h\right) \int_{-1}^1\xi^kH^{\prime\prime}\left(\xi\right)\diff\xi- \int_{-1}^1\xi^kH^{\prime}\left(\xi\right)\diff\xi \right]}\\ 
		\leq&\,\left(\alpha-1\right)\abs{\bZ^{\top}\Delta\left(\be\right)}^\alpha+\abs{\pi_Uh^{\alpha}}=O\brackets {h^{\alpha}+\paren{\sqrt{s}\delta_{m,0}}^{\alpha}}.\\			
	\end{aligned}
\end{equation*}
For $\alpha+1 \leq k\leq 2\alpha+1$,
\[\sup_{\be\in \Theta}\abs{\int_{-1}^{1} \left(\xi h-\bZ^{\top}\Delta\left(\be\right)\right)^{k}\left( \frac{\bZ^{\top}\Delta\left(\be\right)}{h}H^{\prime\prime}\left(\xi\right)-H^{\prime}\left(\xi\right)\right) \diff \xi} 
= O\brackets {h^{\alpha+1}+\left(\sqrt{s}\delta_{m,0}\right)^{\alpha+1}}.\]

Therefore, by the assumption that $\bZ$ has finite second moment and Cauchy-Schwarz inequality, we obtain that
\[\E\left[\bv^{\top} \Psi_n\left(\be\right)\right] \lesssim \mathbb{E}\left[(\bv^{\top}\bZ)\left(s\delta_{m,0}^2+h^{\alpha}\right)\right]\lesssim s\delta_{m,0}^2+h^{\alpha},\] 
which completes the proof of \eqref{eq:Ez-hd}.

\end{proof}

\noindent \textbf{Proof of Theorem \ref{thm:1step-hd}}

Now we are ready to prove the 1-step error for $\hbe^{(1)}$. 

\begin{proof}\label{pf:thm:1step-hd}
For simplicity, we replace $V_{m,1}\left(\hbe^{(0)}\right)$, $V_{n}\left(\hbe^{(0)}\right)$, $U_{n}\left(\hbe^{(0)}\right)$, and $\lambda^{\left(1\right)}_{n}$ by $V_{m,1}$, $V_{n}$, $U_{n}$, and $\lambda_{n}$, respectively. Then, by Algorithm \ref{alg:hd}, 
\[\hbe^{(1)}=\argmin_{\be\in \R^{p}} \left\lbrace \left\|\be\right\|_1 :   \left\|V_{m,1}\be -\paren{V_{m,1} \hbe^{\,\left( 0\right)}- U_{n}}\right\|_{\infty} \leq \lambda_{n}\right\rbrace.\]
Hence, 
\begin{equation}
	\left\|V_{m,1}\hbe^{(1)}-\paren{V_{m,1}\hbe^{\left(0\right)}-U_{n}} \right\|_\infty \leq \lambda_{n}.
	\label{eq:beta1leqlambda}
\end{equation}

Using Lemma \ref{lem:D-hd}, with probability tending to one, we have
\begin{equation}
	\begin{aligned}
		&\quad \left\|(1-\E)\left[V_{m,1}\be^{*}-\paren{V_{m,1}\hbe^{\left(0\right)}-U_{n}} \right]\right\|_\infty\\
		& \leq  \left\|(1-\E)\left[U_{n} - V_n\paren{\hbe^{\left(0\right)}-\be^*}\right]\right\|_\infty + \left\| (1-\E) \left(V_{m,1}-V_n\right)\left(\hbe^{\,\left(0\right)}-\be^*\right) \right\|_\infty\\
		&\leq\left\|(1-\E)\left[U_{n} - V_n\paren{\hbe^{\left(0\right)}-\be^*}\right]\right\|_\infty + \norm{ (1-\E)(V_{m,1}-V_n)}_{\max}\norm{\hbe^{\,\left(0\right)}-\be^*}_{1}\\
		&\leq C_{\lambda}\left(\sqrt{\frac{\log p}{nh}}+\sqrt{\frac{s\log p}{mh^3}}\delta_{m,0}\right).
	\end{aligned}
	\label{eq:beta*leqlambda-var}
\end{equation}
and
\begin{equation}
	\begin{aligned}
		&\quad \left\|\E\left[V_{m,1}\be^{*}-\paren{V_{m,1}\hbe^{\left(0\right)}-U_{n}} \right]\right\|_2\\
		& \leq  \left\|\E\left[U_{n} - V_n\paren{\hbe^{\left(0\right)}-\be^*}\right]\right\|_2 + \left\| \E \left(V_{m,1}-V_n\right)\left(\hbe^{\,\left(0\right)}-\be^*\right) \right\|_2\\
		&\leq C_{\lambda}\left(s\delta_{m ,0}^2+h^{\alpha}\right).
	\end{aligned}
	\label{eq:beta*leqlambda-bias}
\end{equation}
for some large enough constant $C_{\lambda}$. By letting $\lambda_{n}=C_{\lambda}\left[\sqrt{\frac{\log p}{nh}}+\sqrt{\frac{s\log p}{mh^3}}\delta_{m,0}+s\delta_{m ,0}^2+h^{\alpha}\right]$, Equations \eqref{eq:beta*leqlambda-var} and \eqref{eq:beta*leqlambda-bias} implies that 	
\begin{equation}
	\label{eq:beta*leqlambda}
	\left\|V_{m,1}\be^{*}-\paren{V_{m,1}\hbe^{\left(0\right)}-U_{n}} \right\|_\infty \leq \lambda_{n}.
\end{equation}

Combining it with \eqref{eq:beta1leqlambda}, we obtain that, with probability tending to 1,
\begin{equation}
	\left\|V_{m,1}\left(\be^*-\hbe^{\left(1\right)}\right)\right\|_\infty \leq 2\lambda_{n}.
	\label{eq:Dbeta}
\end{equation} 
Moreover, by the optimality of $\hbe^{\left(1\right)}$, it holds that
$\left\|\hbe^{\left(1\right)}\right\|_1 \leq \left\|\be^{*}\right\|_1$, which implies
\[\left\|\hbe_{S}^{\,\left(1\right)}\right\|_1 +\left\|\hbe_{S^c}^{\,\left(1\right)}\right\|_1  =\left\|\hbe^{\left(1\right)}\right\|_1 \leq \left\|\be^*\right\|_1 = \left\|\be^*_{ S}\right\|_1. \]
Therefore,
\[\left\|\left(\be^*-\hbe^{\left(1\right)}\right)_{S^c}\right\|_1= \left\|\hbe_{S^c}^{\,\left(1\right)}\right\|_1 \leq  \left\|\be^*_{ S}-\hbe_{S}^{\,\left(1\right)} \right\|_1=\left\|\left(\be^*-\hbe^{\left(1\right)}\right)_S\right\|_1,\] 
and hence, \[\left\|\be^*-\hbe^{\left(1\right)}\right\|_1\leq 2\left\|\left(\be^*-\hbe^{\left(1\right)}\right)_S\right\|_1 \leq 2\sqrt{s}\left\|\left(\be^*-\hbe^{\left(1\right)}\right)_S\right\|_2\leq 2\sqrt{s}\left\|\be^*-\hbe^{\left(1\right)}\right\|_2.\]

Let $\bm{\delta}:=\hbe^{\left(1\right)}-\be^{*}$. So far we have shown that, with probability tending to one, $\|\bm{\delta}\|_1 \leq2 \sqrt{s} \left\|\bm{\delta}\right\|_2$ and $\|V_{m,1}\bm{\delta}\|_{\infty}\leq 2\lambda_{m,0}$. Therefore,
\begin{equation}
	\begin{aligned}
		\bm{\delta}^{\top}V_{m,1}\bm{\delta}&=\bm{\delta}^{\top}V\bm{\delta}+\bm{\delta}^{\top}\left(V_{m,1}-\E\left[V_{m,1}\right]\right)\bm{\delta}+\bm{\delta}^{\top}\left(\E\left[V_{m,1}\right]-V\right)\bm{\delta}\\
		&\geq \Lambda_{\min}\left(V\right)\|\bm{\delta}\|_2^2 - \left\|(1-\E)V_{m,1}\right\|_{\max}\|\bm{\delta}\|_1^2+\bm{\delta}^{\top}\left(\E\left[V_{m,1}\right]-V\right)\bm{\delta}\\
		&\geq\Lambda_{\min}\left(V\right)\|\bm{\delta}\|_2^2 - s\left\|(1-\E)V_{m,1}\right\|_{\max}\|\bm{\delta}\|_2^2-\abs{\bm{\delta}^{\top}\left(\E\left[V_{m,1}\right]-V\right)\bm{\delta}}.
	\end{aligned}
	\label{eq:REdeltaL}
\end{equation} 
By \eqref{eq:Dm-hd},
\[s\left\|(1-\E)V_{m,1}\right\|_{\max}=O_{\Prob}\paren{\sqrt{\frac{s^2\log p}{mh^3}}}=o_{\mathbb{P}}\left(1\right).\] 
By \eqref{eq:EDm-hd}, 
\[\abs{\bm{\delta}^{\top}\left(\E\left[V_{m,1}\right]-V\right)\bm{\delta}}\lesssim \left(\sqrt{s}\delta_{m,0}+h^{\alpha}\right)\norm{\bm \delta}_2^2=o(\norm{\bm \delta}_2^2).\]
Therefore, \eqref{eq:REdeltaL} leads to 
\[\bm{\delta}^{\top}V_{m,1}\bm{\delta} \geq \Lambda_{\min}\left(V\right)\|\bm{\delta}\|_2^2-o_{\Prob}\left(\|\bm{\delta}\|_2^2\right) \geq  \left(\Lambda_{\min}\left(V\right)/2\right)\|\bm{\delta}\|_2^2,\]
with probability tending to one.

On the other hand, combining \eqref{eq:beta1leqlambda}, \eqref{eq:beta*leqlambda} and \eqref{eq:beta*leqlambda-bias} yields that
\begin{equation}
	\begin{aligned}
		&\quad \bm{\delta}^{\top}V_{m,1}\bm{\delta} \\
		&=  \bm{\delta}^{\top}\left\{V_{m,1}\hbe^{(1)} -\paren{V_{m,1} \hbe^{\left(0\right)}- U_{n}}-(1-\E)\left[V_{m,1}\be^{*}-\paren{V_{m,1}\hbe^{\left(0\right)}-U_{n}}\right]\right\}\\
		&\quad -\bm{\delta}^{\top}\E\left[V_{m,1}\be^{*}-\paren{V_{m,1}\hbe^{\left(0\right)}-U_{n}}\right]\\
		&\leq 	C_{\lambda} \left(\sqrt{\frac{\log p}{nh}}+\sqrt{\frac{s\log p}{mh^3}}\delta_{m,0}\right)\|\bm{\delta}\|_1+ C_{\lambda}(s\delta_{m ,0}^2+h^{\alpha})\|\bm{\delta}\|_2\\
		&\leq C_{\lambda} \left(\sqrt{\frac{s\log p}{nh}}+\sqrt{\frac{s^2\log p}{mh^3}}\delta_{m,0}+s\delta_{m ,0}^2+h^{\alpha}\right)\|\bm{\delta}\|_2.
	\end{aligned}
\end{equation}
Combining the two inequalities above completes the proof.
\end{proof}

\noindent\textbf{Proof of Theorem \ref{thm:tstep-hd}}

\begin{proof}
\label{pf:thm:tstep-hd}
First note that $\sqrt{\frac{s\log p}{nh}}+\sqrt{s}h^{\alpha}\asymp\sqrt{s}\paren{\frac{\log p}{n}}^{\frac{\alpha}{2\alpha+1}}$ when $h=h^*=\paren{\frac{s\log p}{n}}^{\frac{1}{2\alpha+1}}$. Let $\alpha_0=\max\left\{\frac{3}{2c\,\left( 1-2r\right)}+\frac{r}{2(1-2r)}, \frac{3r}{2(1-4r)}\right\}$ and $\delta_{m,0}=(s\log p/m)^{\alpha/(2\alpha+1)}$.
Since $s=O\left(m^r\right)$ and $n=O\left(m^{1/c}\right)$ for some $0<c<1$ and $0<r<1/4$, the condition $\alpha>\alpha_0$ guarantees $\frac{ s^2 \log p}{m(h^*)^3}=o\left(1\right)$, $s^{3/2}\delta_{m,0}=o(1)$, (which implies $r_m=o(1)$),  $s\delta_{m,0}^2=O((h^*)^{3})$, and $s(h^*)^{\alpha}=o\left(1\right)$. Adding the requirement of $\delta_{m,0}$, the assumptions in Theorem \ref{thm:1step-hd} hold, which proves Theorem \ref{thm:tstep-hd} when $t=1$.  

Now we show Theorem \ref{thm:tstep-hd} by induction. Define $\delta_{m ,t}=\sqrt{s}\left(\frac{\log p}{n}\right)^{\frac{\alpha}{2\alpha+1}}+(r_m)^{t}\delta_{m ,0}=\sqrt{s}\lambda_{n}^{(t)}$. Assume that $\norm{\hbe^{(t)}-\be^*}_2=O_{\Prob}(\delta_{m ,t})$ and $\norm{\hbe^{(t)}-\be^*}_1=O_{\Prob}(\sqrt{s}\delta_{m ,t})$ is true for $t$. Note that 
$\sqrt{s}\paren{\frac{\log p}{n}}^{\frac{\alpha}{2\alpha+1}}<\delta_{m,0}$, which implies $\delta_{m,t}<\delta_{m,0}$, and thus $s^{3/2}\delta_{m,t}=o\left(1\right)$ and $\sqrt{s}\delta_{m,t}=O\left((h^*)^{3/2}\right)$. Then by Theorem \ref{thm:1step-hd}, when
\[\lambda^{\left( t+1\right)}_{n}=C_\lambda\left(\sqrt{ \frac{\log p}{nh^{*}}}+\delta_{m,t}\sqrt{\frac{s\log p}{m(h^*)^3}}+ s\delta^2_{m,t}+(h^{*})^{\alpha}\right),\]
we have
\begin{align*}
	\left\|\hbe^{\,\left(t+1\right)}-\be^{*}\right\|_2 &  = O_{\mathbb{P}}\left[\sqrt{s}\left(\frac{\log p}{n}\right)^{\frac{\alpha}{2\alpha+1}}+\left(\sqrt{\frac{s^2\log p}{m(h^*)^3}}+s^{3/2}\delta_{m,t}\right)\delta_{m,t}\right]\\
	& = O_{\mathbb{P}}\left[\sqrt{s}\left(\frac{\log p}{n}\right)^{\frac{\alpha}{2\alpha+1}}+\left(\sqrt{\frac{s^2\log p}{m(h^*)^3}}+s^{3/2}\delta_{m,0}\right)^{t+1}\delta_{m,0}\right],
\end{align*}
and
\[\left\|\hbe^{\,\left(t+1\right)}-\be^{*}\right\|_1\leq 2\sqrt{s} \left\|\hbe^{\,\left(t+1\right)}-\be^{*}\right\|_2,\]
with probability tending to 1. This completes the proof.
\end{proof}

\noindent\textbf{Proof of Data Adaptive Methods for Unknown Parameters}

\begin{proof}[Proof of Theorem \ref{thm:tstep-hd-unknowns}]
First we prove the theorem for $t=1$.  Define $s^*=2^{\lfloor\log_2(s)\rfloor+1} \in \calS$ and define a good event
\begin{align*}
	E_{m, n}:=&\bigg\{\norm{V_{m,1}\be^{*}-\paren{V_{m,1}\hbe^{\left(0\right)}-U_{n}}}_{\infty} \leq \lambda_{n,  s^*}^{(1)} \bigg\} \cap\\
	&  \bigg\{\bm{\delta}^{\top}V_{m,1}\bm{\delta} \geq  \frac{\Lambda_{\min}\left(V\right)}{2}\|\bm{\delta}\|_2^2, \quad \forall \, \bm{\delta} \mathrm{\, such \, that \, }  \norm{\bm{\delta}}_1 \leq 2\sqrt{s} \norm{\bm{\delta}}_2 \bigg\}.
\end{align*}Since $s< s^* \leq 2s$, Lemma \ref{lem:D-hd} and Equations  \eqref{eq:beta*leqlambda-var},  \eqref{eq:beta*leqlambda-bias}, and \eqref{eq:REdeltaL} ensure that $\Prob(E_{m, n}) \to 1$. Note that $\lambda_{n, s'}^{(t)}$  increases in $s'$, which implies that, under $E_{m, n}$, 
\[\norm{V_{m,1}\be^{*}-\paren{V_{m,1}\hbe^{\left(0\right)}-U_{n}}}_{\infty} \leq \lambda_{n,  s'}^{(1)},\] for any $\lambda_{n, s'}^{(1)}$ with $s' \geq s^*$. Therefore, following the proof of Theorem \ref{thm:1step-hd}, we have that, under $E_{m, n}$,
\[\norm{\hbe^{(1)}_{s'}-\be^*}_2 \lesssim \sqrt{s'}\lambda_{n,  s'}^{(1)} \mathrm{\quad and\quad} \norm{\hbe^{(1)}_{s'}-\be^*}_1 \lesssim s'\lambda_{n,  s'}^{(1)},\]
for all $s' \geq  s^*$, which further implies
\[\norm{\hbe^{(1)}_{s^*}-\hbe^{(1)}_{s'}}_2 \leq \norm{\hbe^{(1)}_{s^*}-\be^*}_2+\norm{\hbe^{(1)}_{s'}-\be^*}_2 \lesssim \sqrt{s'}\lambda_{n, s'}^{(1)},\]
and
\[\norm{\hbe^{(1)}_{s^*}-\hbe^{(1)}_{s'}}_1 \leq \norm{\hbe^{(1)}_{s^*}-\be^*}_1+\norm{\hbe^{(1)}_{s'}-\be^*}_1 \lesssim s'\lambda_{n, s'}^{(1)}.\]
By the definition of $\widehat{s}^{(1)}$, we obtain that $\widehat{s}^{(1)} \leq s^*$, and hence \[\norm{\hbe^{(1)}_{\widehat{s}^{(1)}}-\hbe^{(1)}_{s^*}}_2 \lesssim \sqrt{s^*}\lambda_{n,  s^*}^{(1)} \asymp \delta_{m,1} \mathrm{\quad and\quad} \norm{\hbe^{(1)}_{\widehat{s}^{(1)}}-\hbe^{(1)}_{s^*}}_1 \lesssim s^*\lambda_{n,  s^*}^{(1)} \asymp \sqrt{s}\delta_{m,1},\] 
since $s^* \asymp s$. Adding that  $\norm{\hbe^{(1)}_{s^*}-\be^*}_2 \lesssim \sqrt{s^*}\lambda_{n,  s^*}^{(1)}$ and $\norm{\hbe^{(1)}_{s^*}-\be^*}_1 \lesssim s^*\lambda_{n,  s^*}^{(1)}$, we complete the proof of \eqref{eq:tstep-hd-unknowns} for $t=1$. The proof for $t \geq 2$ is straightforward by combining the proof for Theorem \ref{thm:tstep-hd} and that for $t=1$.
\end{proof}

\begin{proof}[Proof of Theorem \ref{thm:tstep-hd-unknownalpha}]
First we prove the theorem for $t=1$. Recall that $h^*=(\log p / n)^{1/(2\alpha+1)}$ and define $h^{**}=2^{\lfloor\log_2(h^*)\rfloor-1} \in \calD$. Then $h^* /2 < h^{**} \leq h^*$. For any $h'\in \calD$, by the proof of Lemma \ref{lem:D-hd}, Equations \eqref{eq:Dn-hd}--\eqref{eq:Ez-hd} hold for $h'$ with probability at least $1-12/p$, which, by the proof of Theorem \ref{thm:1step-hd}, further implies that 
\begin{equation}
	\label{eq:unknown-h}
	\norm{\hbe^{(1)}_{h'}-\be^*}_2\lesssim \sqrt{s}\lambda_{n,  h'}^{(1)} \mathrm{\quad and\quad} \norm{\hbe^{(1)}_{h'}-\be^*}_1 \lesssim s\lambda_{n,  h'}^{(1)}.
\end{equation}
Therefore, since the cardinality of $\calD$ is less than $\log_2(m)$, we have that, with probability at least $1-12\log_2(m)/p$, 	Equation \eqref{eq:unknown-h} holds for all $h' \in \calD$. In particular, when $h' \leq h^{**}$, we have that
\[\norm{\hbe^{(1)}_{h^{**}}-\hbe^{(1)}_{h'}}_2 \leq \norm{\hbe^{(1)}_{h^{**}}-\be^*}_2+\norm{\hbe^{(1)}_{h'}-\be^*}_2 \lesssim \sqrt{s}\lambda_{n, h'}^{(1)},\]
and
\[\norm{\hbe^{(1)}_{h^{**}}-\hbe^{(1)}_{h'}}_1 \leq \norm{\hbe^{(1)}_{h^{**}}-\be^*}_1+\norm{\hbe^{(1)}_{h'}-\be^*}_1 \lesssim s\lambda_{n, h'}^{(1)},\]
where we use the fact that $\sqrt{\frac{\log p}{nh'}} \geq (h')^{\alpha}$ for all $h' \leq h^{**} \leq h^*$, and thus \begin{align*}
	\lambda^{(t)}_{n,  h'}&\asymp \left(\sqrt{\frac{\log p}{nh'}}+(h')^{\alpha}+\frac{1}{\sqrt{s}}(r_{m, h'})^{t} \delta_{m,0}\right)\\
	&\asymp \left(\sqrt{\frac{\log p}{nh'}}+\frac{1}{\sqrt{s}}(r_{m, h'})^{t} \delta_{m,0}\right),
\end{align*} 
which decreases in $h'$. By the definition of $\widehat{h}^{(1)}$, we obtain that $\widehat{h}^{(1)} \geq h^{**}$ and \[\norm{\hbe^{(1)}_{\widehat{h}^{(1)}}-\hbe^{(1)}_{h^{**}}}_2 \lesssim \sqrt{s}\lambda^{(1)}_{n,h^{**}} \asymp \delta_{m, 1},  \norm{\hbe^{(1)}_{\widehat{h}^{(1)}}-\hbe^{(1)}_{h^{**}}}_1  \lesssim s\lambda^{(1)}_{n,h^{**}} \asymp \sqrt{s}\delta_{m, 1},\]
since $h^{**} \asymp h^*$. Together with $\norm{\hbe^{(1)}_{h^{**}}-\be^*}_2 \lesssim \sqrt{s}\lambda_{n,  h^{**}}^{(1)}$ and $\norm{\hbe^{(1)}_{h^{**}}-\be^*}_1 \lesssim s\lambda_{n,  h^{**}}^{(1)}$, we complete the proof of \eqref{eq:tstep-hd-unknownalpha} for $t=1$. The proof for $t \geq 2$ is straightforward by combining the proof for Theorem \ref{thm:tstep-hd} and that for $t=1$.
\end{proof}

\section{Discussions on the Super-Efficiency Phenomenon}
\label{sec:sup-eff}
In this section, we show that our estimator $\hbe^{(T)}$ achieves the same asymptotic performance over a class of underlying distributions under certain uniform assumptions. In model \eqref{eq:binary-model}, for any $\be^*$, denote the density function of $\zeta:=X+\bZ^{\top}\be^*$ conditional on $\bZ$ by $\rho\left( \cdot\mid \bZ\right)$ and the cumulative distribution function of $\epsilon$ conditional on $\bZ$ by $F(\cdot\mid\bZ)$. We define the distribution class $\Theta$ to be the set of tuples $(\be^*, \rho, F)$ that satisfy the following assumptions:
\begin{assumption}
Assume that 
for all $(\be^*, \rho, F) \in \Theta$ and all integers $1\leq k\leq\alpha$, the $k$-th order derivative of $\rho\left( \cdot\mid\bZ\right)$ exists 
for almost every $\bZ$. Furthermore, there exists a constant $C_{\Theta, 1}>0$ such that $\sup_{\zeta, \bZ, k} \left|\rho^{(k)}\left( \zeta\mid\bZ\right)\right| < C_{\Theta, 1}$. 
\label{A2''}
\end{assumption}

\begin{assumption}
Assume that $\epsilon$ and $X$ are independent given $\bZ$, and 
for all $(\be^*, \rho, F) \in \Theta$ and all integers $1\leq k\leq\alpha+1$, the $k$-th order derivative of $F\left( \cdot \mid \bZ \right)$ exists 
for almost every $\bZ$. Furthermore, there exists a constant $C_{\Theta, 2}>0$ such that $\sup_{\epsilon, \bZ, k} \left|F^{(k)}\left( \epsilon\mid\bZ\right)\right| < C_{\Theta, 2}$. 
\label{A3''}
\end{assumption}

\begin{assumption}
Assume that there exists a constant $c_\Theta>1$ such that, for all $(\be^*, \rho, F) \in \Theta$, the matrices $V=2\mathbb{E}\left[\rho\left( 0\mid \bZ\right)F^{\prime}\left( 0 \mid \bZ\right)\bZ\bZ^{\top}\right]$ and $V_s=\pi_V\mathbb{E}\left[\rho\left( 0\mid \bZ\right)\bZ\bZ^{\top}\right]$ satisfy $c_\Theta^{-1}<\Lambda_{\min}\left( V\right)<\Lambda_{\max}\left(V\right)<c_\Theta$, $c_\Theta^{-1}<\Lambda_{\min}\left( V_s\right)<\Lambda_{\max}\left(V_s\right)<c_\Theta$, where $\Lambda_{\min}$ ($\Lambda_{\max}$) denotes the minimum (maximum) eigenvalue.
\label{A4''}
\end{assumption}

Assumptions \ref{A2''}--\ref{A4''} for the distribution class $\Theta$ are parallel to Assumptions \ref{A2}--\ref{A4} for a fixed distribution. Assumptions \ref{A2''}--\ref{A3''} require $\alpha$-order smoothness for all $\rho$ and $F$, which ensures that the Taylor's expansion in the technical proof always hold when $n$ is sufficiently large. Furthermore, the constants $C_{\Theta, 1}$ and $C_{\Theta,2}$ provide uniform upper bounds for the derivatives of $\rho$ and $F$ over $\Theta$. Similarly, Assumption \ref{A4''} ensures that the population Hessian matrix is always positive semi-definite with eigenvalues uniformly bounded away from 0 and $\infty$. Under these assumptions, replicating the analysis of {\mSMSE} in Section \ref{sec:theory} leads to the following result:
\begin{thm}\label{thm:sup-eff}
Assume Assumptions \ref{A1}, \ref{A5}, \ref{A2''}, \ref{A3''} and \ref{A4''} hold, and there exists a constant $0<c_2<1$ such that $p=O(m^{c_2})$. Further assume that $\sup_{(\be^*, \rho, F) \in \Theta}\big\|\hbe^{(0)}-\be^*\big\|_2=O_{\Prob}((p/m)^{1/3})$. When $T$ satisfies \eqref{eq:iteration_number}, by choosing $h_t=\max\big\lbrace (p/n)^{\frac{1}{2\alpha+1}},  (p/m)^{\frac{2^{t}}{3\alpha}}\big\rbrace$ at iteration $t=1,2,\dots, T$, we have: $\forall \varepsilon>0$, $\exists M_{\varepsilon}, N_{\varepsilon}$, such that $\forall n \geq N_{\varepsilon}$, it holds that 
\begin{equation}
	\sup_{(\be^*, \rho, F) \in \Theta}\Prob\left(\big\| \hbe^{\,\left( T\right)}-\be^{*}\big\|_2 > M_{\varepsilon}(p/n)^{\frac{\alpha}{2\alpha+1}}\right) < \varepsilon.
	\label{eq:sup-eff-supp}
\end{equation}
\end{thm}
Note that \eqref{eq:sup-eff-supp} is equivalent to $\big\| \hbe^{\,\left( T\right)}-\be^{*}\big\|_2=O_{\Prob}( (p/n)^{\frac{\alpha}{2\alpha+1}})$ if $\Theta$ only contains a single distribution. The proof of Theorem \ref{thm:sup-eff} is almost the same as the proof of Proposition \ref{thm:1step-ld-supp} and Theorem \ref{thm:tstep-ld}, by noting that, for all distributions in $\Theta$, the Taylor's expansion in \eqref{eq:Taylor} and the computation related to the bias in \eqref{eq:EvU-computation} are always correct. Moreover, the constant in the big O notation in $\eqref{eq:EvU}$ is uniform for all distributions in $\Theta$, which is guaranteed by Assumptions \ref{A2''} and \ref{A3''}.

\section{Additional Results in Simulations}\label{sec:supp-simulation}

\begin{table}[!t]
\caption{The bias, variance and coverage rates of {\mSMSE} ($t=1, 2, 3$), {\AvgMSE}, {\AvgSMSE} and pooled-SMSE, with $p=1$, $\log_m(n)$ from 1.5 to 1.9 and homoscedastic normal noise. }
\resizebox{\textwidth}{!}{

	\begin{tabular}{c|ccc|ccc|ccc}
		\hline
		$\log_m(n)$ & \thead{ \normalsize Bias \\ \normalsize($\times 10^{-2}$)} & \thead{\normalsize Variance\\ \normalsize ($\times 10^{-4}$)} & \thead{\normalsize Coverage \\ \normalsize Rate}   & \thead{\normalsize Bias \\ \normalsize($\times 10^{-2}$)} & \thead{\normalsize Variance\\ \normalsize ($\times 10^{-4}$)} & \thead{\normalsize Coverage \\ \normalsize Rate} & \thead{\normalsize Bias \\ \normalsize($\times 10^{-2}$)} & \thead{\normalsize Variance\\ \normalsize ($\times 10^{-4}$)} & \thead{\normalsize Coverage \\ \normalsize Rate}
		\\
		\hline
		&  \multicolumn{3}{c|}{{\mSMSE} $t=1$} &  \multicolumn{3}{c|}{{\mSMSE} $t=3$} & \multicolumn{3}{c}{{\AvgSMSE}}\\
		\hline
		
		1.5 & $-0.19$ & 2.52 & 0.92 & 0.21 & 1.29 & 0.93 & $-0.20$ & 1.18 & 0.95 \\ 
		1.6 & $-0.47$ & 3.92 & 0.85 & 0.07 & 0.88 & 0.92 & $-0.57$ & 0.71 & 0.86 \\ 
		1.7 & $-0.45$ & 2.11 & 0.84 & 0.03 & 0.38 & 0.95 & $-1.02$ & 0.32 & 0.60 \\ 
		1.8 & $-0.23$ & 1.45 & 0.86 & 0.12 & 0.27 & 0.96 & $-1.46$ & 0.21 & 0.13 \\ 
		1.9 & $-0.32$ & 1.91 & 0.83 & 0.04 & 0.12 & 0.96 & $-1.95$ & 0.08 & 0.00 \\ 
		\hline
		& \multicolumn{3}{c|}{{\mSMSE} $t=2$}	&  \multicolumn{3}{c|}{{\AvgMSE}} &  \multicolumn{3}{c}{pooled-SMSE} \\
		\hline
		1.5 & 0.19 & 1.32 & 0.92 & $-1.32$ & 1.70 & 0.84 & 0.20 & 1.29 & 0.93 \\ 
		1.6 & 0.10 & 0.97 & 0.85 & $-1.39$ & 1.04 & 0.69 & 0.07 & 0.87 & 0.93 \\ 
		1.7 & 0.03 & 0.41 & 0.84 & $-1.35$ & 0.47 & 0.48 & 0.02 & 0.38 & 0.95 \\ 
		1.8 & 0.11 & 0.30 & 0.86 & $-1.26$ & 0.25 & 0.28 & 0.11 & 0.27 & 0.96 \\ 
		1.9 & 0.02 & 0.13 & 0.83 & $-1.32$ & 0.09 & 0.01 & 0.03 & 0.12 & 0.96 \\ 
		\hline
	\end{tabular}
	
}

\label{tab:p=1}

\end{table}

\begin{table}[!t]
\centering

\caption{The bias, variance and coverage rates of {\mSMSE} ($t=1,2,3,4$), {\AvgSMSE} and pooled-SMSE, with $p=10$, $\log_m(n)$ from $1.5$ to $1.9$ and homoscedastic normal noise. }
\resizebox{1.0\textwidth}{!}{
	\begin{tabular}{c|ccc|ccc|ccc}
		\hline
		$\log_m(n)$ & \thead{ \normalsize Bias \\ \normalsize($\times 10^{-2}$)} & \thead{\normalsize Variance\\ \normalsize ($\times 10^{-4}$)} & \thead{\normalsize Coverage \\ \normalsize Rate}   & \thead{\normalsize Bias \\ \normalsize($\times 10^{-2}$)} & \thead{\normalsize Variance\\ \normalsize ($\times 10^{-4}$)} & \thead{\normalsize Coverage \\ \normalsize Rate} & \thead{ \normalsize Bias \\ \normalsize($\times 10^{-2}$)} & \thead{\normalsize Variance\\ \normalsize ($\times 10^{-4}$)} & \thead{\normalsize Coverage \\ \normalsize Rate}  
		\\
		\hline
		&  \multicolumn{3}{c|}{{\mSMSE} $t=1$} &  \multicolumn{3}{c|}{{\mSMSE} $t=3$} & \multicolumn{3}{c}{{\AvgSMSE}}\\ 
		\hline
		
		1.5 & $-7.00$ & 114.33 & 0.66 & 0.54 & 14.50 & 0.93 & $-0.39$ & 7.94 & 0.95 \\ 
		1.6 & $-9.33$ & 119.20 & 0.44 & 0.69 & 8.43 & 0.93 & $-1.77$ & 4.43 & 0.89 \\ 
		1.7 & $-9.68$ & 130.22 & 0.34 & 0.39 & 5.97 & 0.95 & $-2.94$ & 2.08 & 0.57 \\ 
		1.8 & $-10.71$ & 140.50 & 0.23 & 0.32 & 3.63 & 0.90 & $-3.97$ & 1.19 & 0.10 \\ 
		1.9 & $-11.56$ & 151.23 & 0.12 & 0.12 & 1.39 & 0.94 & $-4.73$ & 0.65 & 0.01 \\ 
		\hline
		&  \multicolumn{3}{c|}{{\mSMSE} $t=2$} &  \multicolumn{3}{c|}{{\mSMSE} $t=4$}  & \multicolumn{3}{c}{pooled-SMSE} \\
		\hline
		1.5 & 0.64 & 29.11 & 0.86 & 1.22 & 9.83 & 0.95 & 1.21 & 9.74 & 0.95 \\ 
		1.6 & 0.65 & 16.98 & 0.84 & 0.94 & 5.55 & 0.94 & 0.97 & 5.51 & 0.95 \\ 
		1.7 & 0.22 & 12.94 & 0.80 & 0.76 & 2.56 & 0.98 & 0.77 & 2.56 & 0.98 \\ 
		1.8 & 0.47 & 7.67 & 0.80 & 0.60 & 1.92 & 0.93 & 0.60 & 1.91 & 0.93 \\ 
		1.9 & $-0.11$ & 4.69 & 0.74 & 0.27 & 0.89 & 0.97 & 0.24 & 0.88 & 0.97 \\ 
		\hline
	\end{tabular} 
}

\label{tab:p=10}

\end{table}

\subsection{Bias and Variance}
\label{supp:simu-bias}

In this section, we report the bias and the variance of the estimators with different $\log_m(n)$ in Tables \ref{tab:p=1} and \ref{tab:p=10}. In the tables, when $t\geq 3$, both the bias and the variance of our proposed multiround method generally decrease as $n$ increases, and they are close to the bias and the variance of the pooled-SMSE.  This is consistent with our theoretical analysis in Section \ref{sec:theory-mSMSE}, where we establish that the bias and the variance of {\mSMSE} are both of the rate $n^{-\alpha/(2\alpha+1)}$. On the contrary, the biases of the averaging methods are much larger than the bias of our method and stay large as $n$  increases, as the bias cannot be reduced by averaging in a distributed environment. Note that the bias of {\AvgSMSE} is high since the necessary condition $L=\big(m^{\frac{2(\alpha-1)}{3}}/(p\log m)^{\frac{2\alpha+1}{3}}\big)$ in Theorem \ref{thm:asym-DC} is violated. While the bias stays large, the variance decreases as $n$ increases, and therefore we observe the failure of inference when $n$ is large for {\AvgMSE} and {\AvgSMSE}. 

\begin{table}[!t]
\centering
\caption{The coverage rates (nominal 95\%) of {\mSMSE} with different values of $\lambda_h$ and $\log_m(n)$. The noise is homoscedastic normal.}
\resizebox{0.6\textwidth}{!}{
	\begin{tabular}{cccc|cccc}
		\hline
		$\lambda_h$ &$\log_m(n)$ & $p=1$ & $p=10$ & $\lambda_h$ & $\log_m(n)$ & $p=1$ & $p=10$ \\
		\hline
		
		\multirow{3}{*}{1}
		&1.5       &    0.94     &    0.93      & \multirow{3}{*}{10}&1.5      &    0.95         &  0.92  \\
		&1.7     &  0.93& 0.92     &        &	1.7  & 0.94 & 0.94 \\
		&	1.9         &   0.96      &      0.94     &         &  1.9 & 0.96	&	0.92   \\
		\hline
		\multirow{3}{*}{30}
		&1.5       &      0.93  &    0.91     &	\multirow{3}{*}{$\widehat{\lambda_h^*}$}	&1.5      &  0.94         &  0.93 \\
		&1.7     &       0.95      &  0.95   &         &    1.7    & 0.94	& 0.94 \\
		&	1.9         &    0.94     &     0.95     &    &  1.9  & 0.96	&	0.92  \\

		\hline
	\end{tabular}
}

\label{tab:sensitivity}

\end{table} 

	%
	%
	%
%
%
%

\subsection{Sensitivity Analysis}
\label{sec:sensitivity}

In this section, we use numerical experiments to show the sensitivity of the constant $\lambda_h$ in the bandwidth $h_t$ in Theorem \ref{thm:normality}. An expression of the optimal value $\lambda_h^*$ is given in \eqref{eq:lambda_opt} by minimizing the asymptotic mean squared error. We estimate $\lambda_h^*$ using $\widehat{U}$, $\widehat{V}$, and $\widehat{V}_s$. Under our experiment settings, the estimated constant $\widehat{\lambda^*_h}$ ranges from 10 to 25 in practice. To study the effect of $\lambda_h$ on the validity of inference, we choose a wider range for $\lambda_h$, from 1 to 30, and report the coverage rates of {\mSMSE} in Table \ref{tab:sensitivity} for $p=1$ and $10$, with different $\lambda_h$ and $\log_m(n)$. 
In summary, {\mSMSE} generally allows arbitrary choices of $\lambda_h$ in a wide range, which suggests that our proposed {\mSMSE} algorithm is robust with respect to $\lambda_h$.

\begin{table}[!t]
\centering
\caption{The CPU time (in seconds) that different methods take to compute the estimator, with $p=10$, $\log_m(n)$ from 1.5 to 1.9, and homoscedastic normal noise.}

\resizebox{0.7\textwidth}{!}{
	\begin{tabular}{c|cccc}
		\hline
		$\log_m(n)$ & {\mSMSE} $t=2$ & {\mSMSE} $t=3$ & {\AvgSMSE} & pooled-SMSE \\
		\hline

		1.5 & 0.086 & 0.123 & 0.121 & 0.264 \\ 
		1.6 & 0.092 & 0.128 & 0.127 & 0.511 \\ 
		1.7 & 0.106 & 0.152 & 0.146 & 1.025 \\ 
		1.8 & 0.139 & 0.200 & 0.148 & 1.847 \\ 
		1.9 & 0.195 & 0.279 & 0.156 & 3.782 \\ 
		\hline
	\end{tabular}
}

\label{tab:time}

\end{table}

\subsection{Time Complexity}\label{sec:supp_time}
In this section, we compare the computational complexity of each method. The average CPU time that each method takes when $p=10$ is reported in Table \ref{tab:time}. The computation time is recorded in a simulated distributed environment on a RedHat Enterprise Linux cluster containing 524 Lenovo SD650 nodes interconnected by high-speed networks. On each computer node, two Intel Xeon Platinum 8268 24C 205W 2.9GHz Processors are equipped with 48 processing cores.  

In Table \ref{tab:time}, we first notice that the speed of {\mSMSE} is much faster than the pooled estimator, and the discrepancy greatly increases when $n$ gets larger. Second, the computation time of {\mSMSE} is comparable to {\AvgSMSE}. This result may seem counterintuitive since {\mSMSE} still requires running an SMSE on the first machine for the initial estimator. However, since the computation time of {\AvgSMSE} is mainly determined by the maximum computation time of the $L$ local machines, {\AvgSMSE} greatly suffers from the computational performance of the ``worst'' machine, especially when the number of machines is large. On the other hand, {\mSMSE} only runs SMSE on one machine and therefore achieves comparable computation time in the experiments. 

{Now we focus on our {\mSMSE} and demonstrate the computational complexity as we increase the local sample size $m$, number of machines $L$, and dimension $p$. The computational cost of Algorithm \ref{alg:mSMSE} is composed of four parts: (a) computing the initial estimator in step 1, whose time complexity depends on what initial is used, (b) computing the local gradient and Hessian in step 5, whose time complexity is $O(mp^2)$ per iteration, (c) computing global gradient and Hessian in step 8, whose time complexity is $O(Lp^2)$ per iteration, and (d) the Newton's update in step 9, whose time complexity is $O(p^3)$ per iteration. We note that $p$ is smaller than $m$, and therefore the computation cost of {\mSMSE} is dominated by (b) and (c), which is scalable with $m$, $L$, and $p^2$. 
To verify the computational scalability, we run numerical simulations for increasing $m$, $L$, and $p$, and present the computation time of {\mSMSE} with one, three, and five iterations in Figure \ref{fig:scale}, which verifies that the computation time increases linearly in $m$ and $L$ and superlinearly in $p$.}

\begin{figure}[!t]
\centering
\begin{subfigure}[b]{0.32\textwidth}
	\centering
	\includegraphics[width=\textwidth]{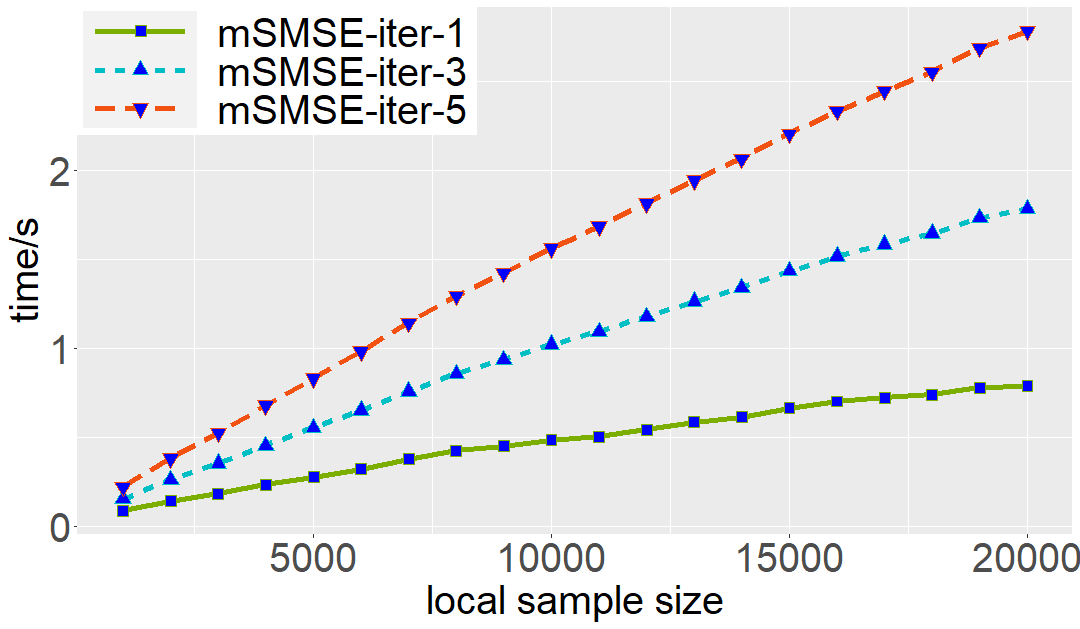}
	\caption{$p=10$, $L=50$}
	\label{fig:scale_m}
\end{subfigure}
\begin{subfigure}[b]{0.32\textwidth}
	\centering
	\includegraphics[width=\textwidth]{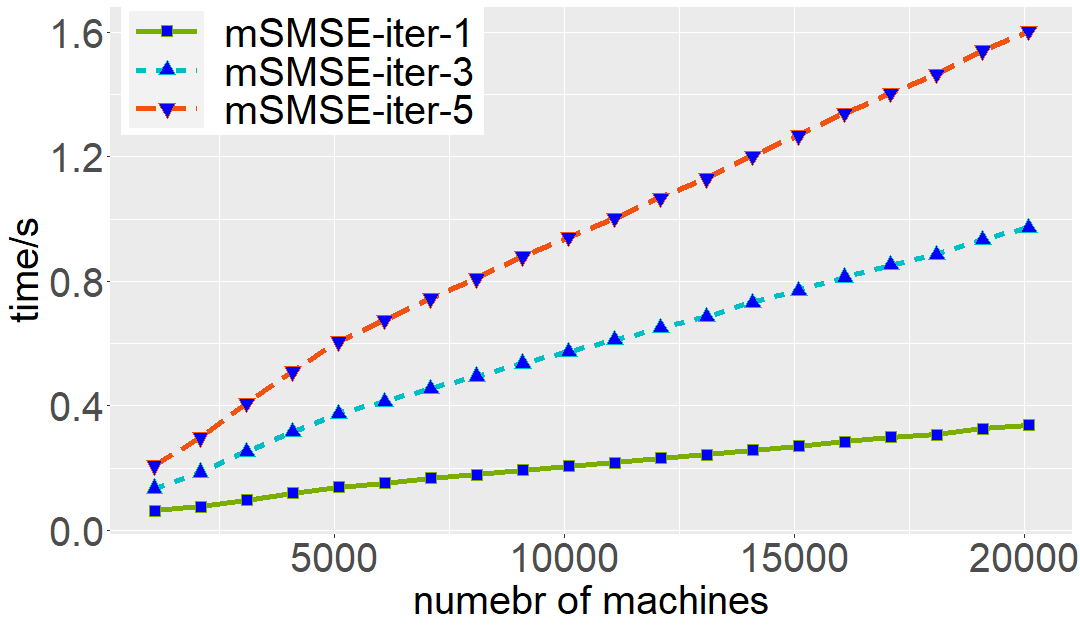}
	\caption{$p=10$, $m=100$}
	\label{fig:scale_L}
\end{subfigure}
\begin{subfigure}[b]{0.32\textwidth}
	\centering
	\includegraphics[width=\textwidth]{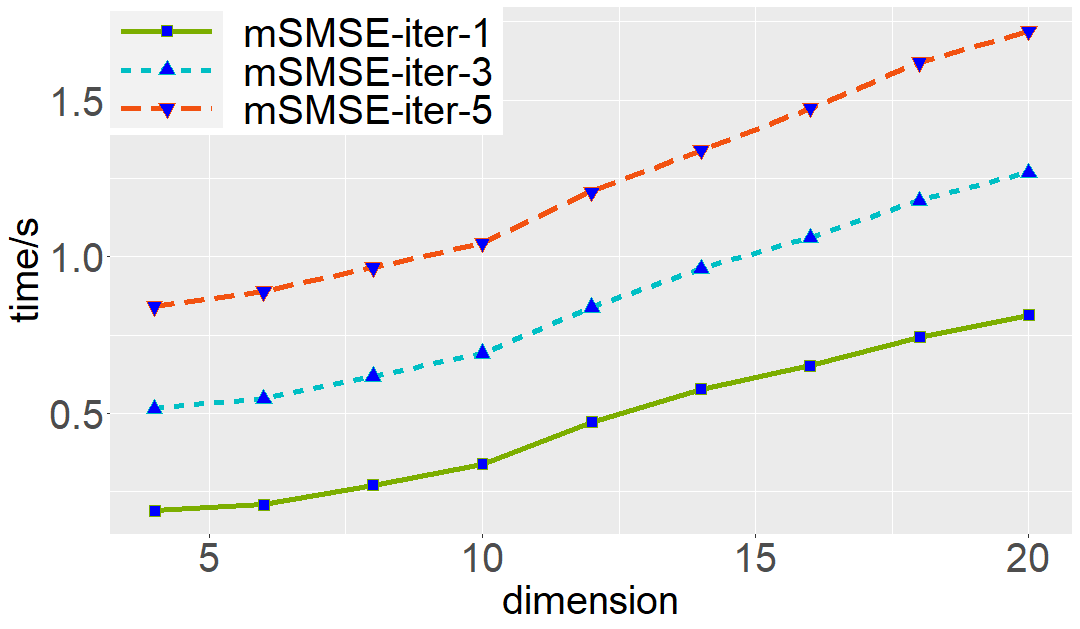}
	\caption{$m=5000$, $L=500$}
	\label{fig:scale_p}
\end{subfigure}
\caption{The CPU time (in seconds) of {\mSMSE} as we increase $m$, $L$, $p$ in subfigures (a), (b), (c), respectively.}
\label{fig:scale}
\end{figure}

\subsection{Results for Non-Gaussian Noises}\label{sec:supp-non-Gaussian}
In this section, we report the bias, variance and coverage rates in Tables \ref{tab:p=1-unif}--\ref{tab:p=10-hetero} for the other two noise types, i.e., the homoscedastic uniform and heteroscedastic normal noise, with $p=1$ and $10$. From these tables, we can still see the failure of inference of {\AvgMSE} and {\AvgSMSE} when $\log_m(n)$ is large, while the {\mSMSE} method with $t\geq 3$ achieves near-nominal coverage rates no matter how large $\log_m(n)$ is. In addition, we also report the time cost of each method in  Table \ref{tab:supp-time}, which shows that the computational time of {\mSMSE} is comparable to {\AvgSMSE}. These findings are all consistent with the results for the homoscedastic normal noise.

\begin{table}[!htp]
\centering
\caption{The bias, variance and coverage rates of {\mSMSE} ($t=1, 2, 3$), {\AvgMSE}, {\AvgSMSE} and pooled-SMSE, with $p=1$, $\log_m(n)$ from 1.5 to 1.9 and homoscedastic uniform noise. }
\resizebox{0.9\textwidth}{!}{

	\begin{tabular}{c|ccc|ccc|ccc}
		\hline
		$\log_m(n)$ & \thead{ \normalsize Bias \\ \normalsize($\times 10^{-2}$)} & \thead{\normalsize Variance\\ \normalsize ($\times 10^{-4}$)} & \thead{\normalsize Coverage \\ \normalsize Rate}   & \thead{\normalsize Bias \\ \normalsize($\times 10^{-2}$)} & \thead{\normalsize Variance\\ \normalsize ($\times 10^{-4}$)} & \thead{\normalsize Coverage \\ \normalsize Rate} & \thead{\normalsize Bias \\ \normalsize($\times 10^{-2}$)} & \thead{\normalsize Variance\\ \normalsize ($\times 10^{-4}$)} & \thead{\normalsize Coverage \\ \normalsize Rate}
		\\
		\hline
		&  \multicolumn{3}{c|}{{\mSMSE} $t=1$} &  \multicolumn{3}{c|}{{\mSMSE} $t=3$} & \multicolumn{3}{c}{{\AvgSMSE}}\\
		\hline
		
		1.5 & $-0.17$ & 5.16 & 0.90 & 0.17 & 2.37 & 0.93 & $-0.89$ & 1.53 & 0.95 \\ 
		1.6 & $-0.40$ & 5.91 & 0.81 & 0.21 & 1.42 & 0.90 & $-1.38$ & 1.02 & 0.74 \\ 
		1.7 & $-0.69$ & 7.06 & 0.84 & 0.09 & 0.80 & 0.94 & $-2.11$ & 0.48 & 0.23 \\ 
		1.8 & $-0.27$ & 2.13 & 0.86 & 0.15 & 0.41 & 0.94 & $-2.78$ & 0.23 & 0.00 \\ 
		1.9 & $-0.14$ & 0.83 & 0.86 & 0.08 & 0.22 & 0.97 & $-3.42$ & 0.13 & 0.00 \\

		\hline
		&  \multicolumn{3}{c|}{{\mSMSE} $t=2$} &  \multicolumn{3}{c|}{{\AvgMSE}} & \multicolumn{3}{c}{pooled-SMSE}\\
		\hline
		
		1.5 & 0.15 & 2.63 & 0.90 & $-1.37$ & 2.70 & 0.81 & 0.15 & 2.35 & 0.93 \\ 
		1.6 & 0.19 & 1.68 & 0.81 & $-1.26$ & 1.27 & 0.78 & 0.20 & 1.41 & 0.90 \\ 
		1.7 & 0.10 & 1.08 & 0.84 & $-1.40$ & 0.59 & 0.57 & 0.08 & 0.80 & 0.94 \\ 
		1.8 & 0.13 & 0.43 & 0.86 & $-1.25$ & 0.33 & 0.41 & 0.13 & 0.41 & 0.94 \\ 
		1.9 & 0.09 & 0.23 & 0.86 & $-1.29$ & 0.21 & 0.16 & 0.07 & 0.22 & 0.97 \\ 
		
		\hline
	\end{tabular}
	
}

\label{tab:p=1-unif}

\end{table}

\begin{table}[!htp]
\centering
\caption{The bias, variance and coverage rates of {\mSMSE} ($t=1,2,3,4$), {\AvgSMSE} and pooled-SMSE, with $p=10$, $\log_m(n)$ from 1.5 to 1.9 and homoscedastic uniform noise. }
\resizebox{0.9\textwidth}{!}{
	\begin{tabular}{c|ccc|ccc|ccc}
		\hline
		$\log_m(n)$ & \thead{ \normalsize Bias \\ \normalsize($\times 10^{-2}$)} & \thead{\normalsize Variance\\ \normalsize ($\times 10^{-4}$)} & \thead{\normalsize Coverage \\ \normalsize Rate}   & \thead{\normalsize Bias \\ \normalsize($\times 10^{-2}$)} & \thead{\normalsize Variance\\ \normalsize ($\times 10^{-4}$)} & \thead{\normalsize Coverage \\ \normalsize Rate} & \thead{ \normalsize Bias \\ \normalsize($\times 10^{-2}$)} & \thead{\normalsize Variance\\ \normalsize ($\times 10^{-4}$)} & \thead{\normalsize Coverage \\ \normalsize Rate}  
		\\
		\hline
		&  \multicolumn{3}{c|}{{\mSMSE} $t=1$} &  \multicolumn{3}{c|}{{\mSMSE} $t=3$} & \multicolumn{3}{c}{{\AvgSMSE}}\\ 
		\hline
		
		1.5 & $-7.75$ & 129.45 & 0.65 & 0.52 & 23.40 & 0.91 & $-2.00$ & 12.00 & 0.97 \\ 
		1.6 & $-10.50$ & 211.98 & 0.49 & 0.12 & 15.64 & 0.91 & $-3.84$ & 4.68 & 0.89 \\ 
		1.7 & $-8.79$ & 120.42 & 0.43 & 0.43 & 8.29 & 0.92 & $-5.34$ & 2.79 & 0.37 \\ 
		1.8 & $-10.10$ & 196.01 & 0.29 & 0.32 & 3.91 & 0.92 & $-6.38$ & 1.15 & 0.02 \\ 
		1.9 & $-8.48$ & 96.22 & 0.23 & 0.39 & 2.69 & 0.92 & $-7.18$ & 0.71 & 0.00 \\ 
		
		\hline
		&  \multicolumn{3}{c|}{{\mSMSE} $t=2$} &  \multicolumn{3}{c|}{{\mSMSE} $t=4$}  & \multicolumn{3}{c}{pooled-SMSE} \\
		\hline
		
		1.5 & 0.02 & 44.75 & 0.82 & 1.07 & 19.06 & 0.91 & 1.03 & 18.97 & 0.92 \\ 
		1.6 & $-1.73$ & 36.80 & 0.77 & 0.50 & 9.83 & 0.94 & 0.52 & 9.75 & 0.95 \\ 
		1.7 & $-1.08$ & 21.49 & 0.78 & 0.58 & 5.94 & 0.95 & 0.54 & 5.89 & 0.95 \\ 
		1.8 & $-0.56$ & 12.31 & 0.77 & 0.53 & 3.37 & 0.92 & 0.43 & 3.37 & 0.93 \\ 
		1.9 & $-0.29$ & 8.00 & 0.80 & 0.47 & 2.01 & 0.94 & 0.39 & 2.00 & 0.95 \\

		\hline
	\end{tabular} 
}

\label{tab:p=10-unif}

\end{table}

\begin{table}[!htp]
\centering
\caption{The bias, variance and coverage rates of {\mSMSE} ($t=1, 2, 3$), {\AvgMSE}, {\AvgSMSE} and pooled-SMSE, with $p=1$, $\log_m(n)$ from 1.5 to 1.9 and heteroscedastic normal noise. }
\resizebox{0.9\textwidth}{!}{

	\begin{tabular}{c|ccc|ccc|ccc}
		\hline
		$\log_m(n)$ & \thead{ \normalsize Bias \\ \normalsize($\times 10^{-2}$)} & \thead{\normalsize Variance\\ \normalsize ($\times 10^{-4}$)} & \thead{\normalsize Coverage \\ \normalsize Rate}   & \thead{\normalsize Bias \\ \normalsize($\times 10^{-2}$)} & \thead{\normalsize Variance\\ \normalsize ($\times 10^{-4}$)} & \thead{\normalsize Coverage \\ \normalsize Rate} & \thead{\normalsize Bias \\ \normalsize($\times 10^{-2}$)} & \thead{\normalsize Variance\\ \normalsize ($\times 10^{-4}$)} & \thead{\normalsize Coverage \\ \normalsize Rate}
		\\
		\hline
		&  \multicolumn{3}{c|}{{\mSMSE} $t=1$} &  \multicolumn{3}{c|}{{\mSMSE} $t=3$} & \multicolumn{3}{c}{{\AvgSMSE}}\\

		\hline
		
		1.5 & 0.02 & 2.11 & 0.90 & 0.35 & 1.41 & 0.92 & $-0.08$ & 1.25 & 0.92 \\ 
		1.6 & $-0.25$ & 3.23 & 0.86 & 0.19 & 0.70 & 0.94 & $-0.54$ & 0.75 & 0.84 \\ 
		1.7 & $-0.25$ & 1.90 & 0.84 & 0.07 & 0.39 & 0.95 & $-1.01$ & 0.30 & 0.56 \\ 
		1.8 & $-0.06$ & 0.58 & 0.85 & 0.11 & 0.27 & 0.90 & $-1.46$ & 0.16 & 0.08 \\ 
		1.9 & $-0.22$ & 0.70 & 0.78 & 0.06 & 0.15 & 0.92 & $-1.80$ & 0.08 & 0.00 \\ 
		
		\hline
		&  \multicolumn{3}{c|}{{\mSMSE} $t=2$} &  \multicolumn{3}{c|}{{\AvgMSE}} & \multicolumn{3}{c}{pooled-SMSE}\\
		\hline
		1.5 & 0.34 & 1.42 & 0.90 & $-1.25$ & 1.95 & 0.80 & 0.34 & 1.41 & 0.92 \\ 
		1.6 & 0.19 & 0.80 & 0.86 & $-1.41$ & 1.09 & 0.64 & 0.19 & 0.69 & 0.94 \\ 
		1.7 & 0.05 & 0.41 & 0.84 &$-1.44$ & 0.43 & 0.42 & 0.06 & 0.39 & 0.95 \\ 
		1.8 & 0.11 & 0.27 & 0.85 & $-1.46$ & 0.23 & 0.15 & 0.10 & 0.27 & 0.90 \\ 
		1.9 & 0.06 & 0.16 & 0.78 & $-1.40$ & 0.11 & 0.00 & 0.05 & 0.15 & 0.93 \\ 
		
		\hline
	\end{tabular}
	
}

\label{tab:p=1-hetero}

\end{table}

\begin{table}[!htp]
\centering
\caption{The bias, variance and coverage rates of {\mSMSE} ($t=1,2,3,4$), {\AvgSMSE} and pooled-SMSE, with $p=10$, $\log_m(n)$ from 1.5 to 1.9 and heteroscedastic normal noise. }
\resizebox{0.9\textwidth}{!}{
	\begin{tabular}{c|ccc|ccc|ccc}
		\hline
		$\log_m(n)$ & \thead{ \normalsize Bias \\ \normalsize($\times 10^{-2}$)} & \thead{\normalsize Variance\\ \normalsize ($\times 10^{-4}$)} & \thead{\normalsize Coverage \\ \normalsize Rate}   & \thead{\normalsize Bias \\ \normalsize($\times 10^{-2}$)} & \thead{\normalsize Variance\\ \normalsize ($\times 10^{-4}$)} & \thead{\normalsize Coverage \\ \normalsize Rate} & \thead{ \normalsize Bias \\ \normalsize($\times 10^{-2}$)} & \thead{\normalsize Variance\\ \normalsize ($\times 10^{-4}$)} & \thead{\normalsize Coverage \\ \normalsize Rate}  
		\\
		\hline
		&  \multicolumn{3}{c|}{{\mSMSE} $t=1$} &  \multicolumn{3}{c|}{{\mSMSE} $t=3$} & \multicolumn{3}{c}{{\AvgSMSE}}\\ 
		\hline
		
		1.5 & $-5.32$ & 51.33 & 0.54 & 0.67 & 6.80 & 0.91 & 1.35 & 4.11 & 0.94 \\ 
		1.6 & $-5.42$ & 54.03 & 0.42 & 0.42 & 5.15 & 0.92 & 0.35 & 2.25 & 0.94 \\ 
		1.7 & $-4.95$ & 41.12 & 0.38 & 0.34 & 2.90 & 0.90 & $-0.34$ & 1.40 & 0.90 \\ 
		1.8 & $-5.17$ & 61.55 & 0.34 & $-0.07$ & 5.18 & 0.88 & $-0.92$ & 0.61 & 0.74 \\ 
		1.9 & $-6.65$ & 75.74 & 0.22 & $-0.27$ & 3.92 & 0.87 & $-1.35$ & 0.30 & 0.28 \\ 
		\hline
		&  \multicolumn{3}{c|}{{\mSMSE} $t=2$} &  \multicolumn{3}{c|}{{\mSMSE} $t=4$}  & \multicolumn{3}{c}{pooled-SMSE} \\
		\hline
		1.5 & 1.67 & 9.03 & 0.87 & 1.12 & 3.53 & 0.94 & 1.12 & 3.53 & 0.94 \\ 
		1.6 & 1.17 & 6.30 & 0.83 & 0.82 & 1.92 & 0.94 & 0.82 & 1.92 & 0.94 \\ 
		1.7 & 1.12 & 7.13 & 0.74 & 0.70 & 1.17 & 0.93 & 0.69 & 1.17 & 0.94 \\ 
		1.8 & 1.21 & 8.72 & 0.77 & 0.50 & 0.65 & 0.92 & 0.45 & 0.65 & 0.94 \\ 
		1.9 & 1.28 & 8.44 & 0.68 & 0.29 & 0.36 & 0.94 & 0.28 & 0.36 & 0.95 \\ 
		\hline
	\end{tabular} 
}

\label{tab:p=10-hetero}

\end{table}

\begin{table}[!htp]
\centering
\caption{The cpu times (in seconds) that different methods take to compute the estimator, with $p=10$, $\log_m(n)$ from 1.5 to 1.9 and two types of noise.}
\resizebox{0.77\textwidth}{!}{
	\begin{tabular}{cc|cccc}
		\hline
		Noise Type &	$\log_m(n)$ & {\mSMSE} $t=2$ & {\mSMSE} $t=3$ & {\AvgSMSE} & pooled-SMSE \\
		\hline
		
		\multirow{5}{*}{\thead{\normalsize  Homoscedastic \\ \normalsize Uniform}}
		&1.5 & 0.091 & 0.126 & 0.111 & 0.323 \\ 
		&1.6 & 0.094 & 0.133 & 0.125 & 0.634 \\ 
		&1.7 & 0.109 & 0.154 & 0.145 & 1.276 \\ 
		&1.8 & 0.141 & 0.202 & 0.148 & 2.359 \\ 
		&1.9 & 0.196 & 0.282 & 0.158 & 4.784 \\ 
		
		\hline
		\multirow{5}{*}{\thead{\normalsize  Heteroscedastic \\ \normalsize Normal}}     
		
		&1.5 & 0.088 & 0.122 & 0.113 & 0.984 \\ 
		&1.6 & 0.090 & 0.126 & 0.132 & 2.157 \\ 
		&1.7 & 0.106 & 0.156 & 0.144 & 4.768 \\ 
		&1.8 & 0.139 & 0.198 & 0.145 & 10.688 \\ 
		&1.9 & 0.190 & 0.276 & 0.156 & 22.036 \\ 
		\hline 
	\end{tabular}
	
}

\label{tab:supp-time}

\end{table}

\begin{figure}[!t]
\centering
\includegraphics[width=0.6\textwidth]{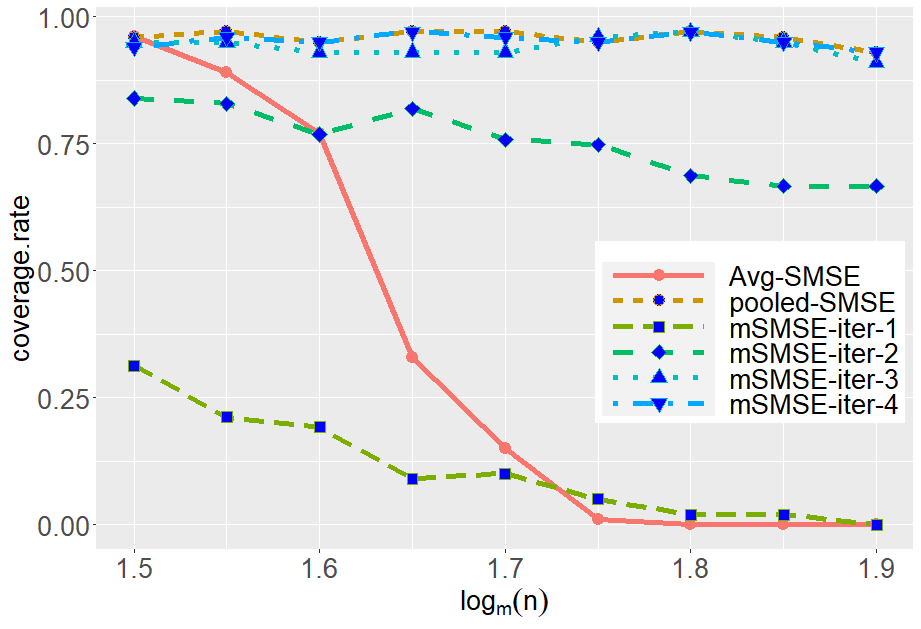}
\caption{Coverage rates for different methods with $p=20$ and homoscedastic normal noise.}
\label{fig:p=20}
\end{figure}

\begin{table}[!htp]
\centering
\caption{The bias, variance and coverage rates of {\mSMSE} ($t=1,2,3,4$), {\AvgSMSE} and pooled-SMSE, with $p=20$, $\log_m(n)$ from 1.5 to 1.9 and homoscedastic normal noise. }
\resizebox{0.9\textwidth}{!}{
	\begin{tabular}{c|ccc|ccc|ccc}
		\hline
		$\log_m(n)$ & \thead{ \normalsize Bias \\ \normalsize($\times 10^{-2}$)} & \thead{\normalsize Variance\\ \normalsize ($\times 10^{-4}$)} & \thead{\normalsize Coverage \\ \normalsize Rate}   & \thead{\normalsize Bias \\ \normalsize($\times 10^{-2}$)} & \thead{\normalsize Variance\\ \normalsize ($\times 10^{-4}$)} & \thead{\normalsize Coverage \\ \normalsize Rate} & \thead{ \normalsize Bias \\ \normalsize($\times 10^{-2}$)} & \thead{\normalsize Variance\\ \normalsize ($\times 10^{-4}$)} & \thead{\normalsize Coverage \\ \normalsize Rate}  
		\\
		\hline
		&  \multicolumn{3}{c|}{{\mSMSE} $t=1$} &  \multicolumn{3}{c|}{{\mSMSE} $t=3$} & \multicolumn{3}{c}{{\AvgSMSE}}\\ 
		\hline
		
		1.5 & $-14.78$ & 215.96 & 0.31 & 0.97 & 15.58 & 0.96 & $-0.58$ & 5.49 & 0.96 \\ 
		1.6 & $-13.39$ & 151.88 & 0.20 & 0.71 & 4.83 & 0.93 & $-2.41$ & 2.42 & 0.77 \\ 
		1.7 & $-14.47$ & 138.67 & 0.10 & 0.45 & 5.73 & 0.93 & $-3.83$ & 1.38 & 0.15 \\ 
		1.8 & $-15.57$ & 142.76 & 0.02 & 0.49 & 1.08 & 0.97 & $-5.13$ & 0.72 & 0.00 \\ 
		1.9 & $-16.12$ & 125.10 & 0.00 & 0.45 & 1.33 & 0.91 & $-5.71$ & 0.30 & 0.00 \\ 
		\hline
		&  \multicolumn{3}{c|}{{\mSMSE} $t=2$} &  \multicolumn{3}{c|}{{\mSMSE} $t=4$}  & \multicolumn{3}{c}{pooled-SMSE} \\
		\hline
		1.5 & 1.22 & 18.83 & 0.84 & 1.77 & 6.55 & 0.95 & 1.79 & 6.49 & 0.96 \\ 
		1.6 & 0.73 & 17.80 & 0.77 & 1.11 & 3.53 & 0.95 & 1.08 & 3.52 & 0.95 \\ 
		1.7 & $-0.07$ & 8.09 & 0.76 & 0.92 & 1.97 & 0.96 & 0.86 & 1.94 & 0.97 \\ 
		1.8 & $-0.53$ & 3.63 & 0.68 & 0.61 & 0.95 & 0.97 & 0.54 & 0.95 & 0.97 \\ 
		1.9 & $-0.25$ & 3.80 & 0.67 & 0.58 & 0.76 & 0.93 & 0.50 & 0.76 & 0.93 \\ 
		\hline
	\end{tabular} 
}
\label{tab:p=20}
\end{table}

\subsection{Results for $p=20$}
\label{sec:supp-p=20}
In this section, we report the performance of {\AvgSMSE} and {\mSMSE} under local size $m=2000$ and dimension $p=20$. 	Figure \ref{fig:p=20} presents the coverage rates as a function of $\log_m(n)$ with $p=20$. Our proposed {\mSMSE}, as well as the pooled estimator, achieves a high coverage rate around 95\% no matter how large $\log_m(n)$ is, while the averaging methods both fail when $\log_m(n)$ is large. Table \ref{tab:p=20} reports the bias and the variance of the estimators with different $\log_m(n)$. When $t=4$, both the bias and the variance of our proposed multiround method generally decrease as $n$ increases, and they are close to the bias and the variance of the pooled-SMSE. These findings are all consistent with the results for $p=1$ and $p=10$. 

\subsection{Results for High-Dimensional Simulations}
\label{sec:simulations-hd}
In this section, we present the performance of Algorithm \ref{alg:hd} using high-dimensional simulations. Following the settings in \cite{feng2019nonregular}, we first generate the covariates $(x_i, \bz_i) \sim \mathcal{N}(0, \Sigma_{0.5})$, where $\Sigma_{0.5}$ is an AR(1) covariance matrix with correlation coefficient $0.5$. The parameter of interest $\be^* \in \R^p$ is set to be $(\underbrace{1/\sqrt{s},\dots, 1/\sqrt{s},}_{s \, \mathrm{entries}} 0, \dots, 0)^{\top}$, where we fix the dimension $p=500$ and the sparsity $s=10$. The responses are generated by $y_i=\mathrm{sign}\big(x_i+\bz_i^{\top}\be^*+\epsilon_i\big)$ for $i=1,2,\dots,n$, with $\epsilon_i \sim \mathcal{N}(0, (0.5)^2)$. The $n$ observations are then evenly divided into $L=n/m$ subsets, where we choose the local sample size $m \in \{400, 800\}$, and vary the total sample size $n \in [2400, 8800]$. On the simulated datasets, we compare the $L_2$ estimation errors of the following three algorithms: 
\begin{enumerate}
\item[(1)] ``mSMSE-hd'': our proposed high-dimensional mSMSE algorithm in Algorithm \ref{alg:hd};
\item[(2)] ``Avg-pf'': the Averaged Divide-and-Conquer path-following algorithm, which applies the path-following algorithm proposed in \cite{feng2019nonregular} on each subset and aggregates the local estimators by averaging;
\item[(3)] ``pooled-pf'': the path-following algorithm using the entire dataset. 
\end{enumerate}
For the smoothing kernel in Algorithm \ref{alg:hd}, we use a higher-order biweight kernel with $\alpha=6$:
\[H'(x)= \frac{4725}{2048}(1-x^2)^2 \Big(1-\frac{22}{3}x^2+\frac{143}{15}x^4\Big)\mathbb{I}(x\leq 1).\]
Moreover, in each iteration, the bandwidth $h_t$ and the penalty parameter $\lambda_n^{(t)}$ are chosen through  cross-validation. Concretely, we use the score function defined in \eqref{eq:MSE} as a measure and,  among a grid of possible values, select the combination $(h_t, \lambda_n^{(t)})$ that achieves the highest cross-validated score. The same strategy is also used for tuning the other two algorithms.  
\begin{figure}[!t]
\centering
\begin{subfigure}[b]{0.45\textwidth}
	\centering
	\includegraphics[width=\textwidth]{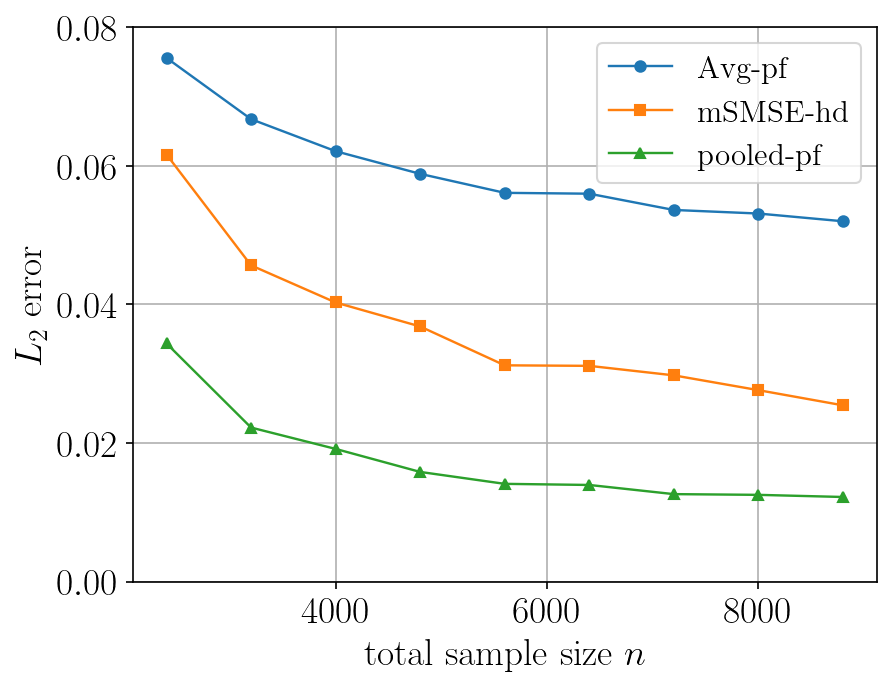}
	\caption{$m=400$}
	\label{fig:EE_m=400}
\end{subfigure}
\hfill
\begin{subfigure}[b]{0.45\textwidth}
	\centering
	\includegraphics[width=\textwidth]{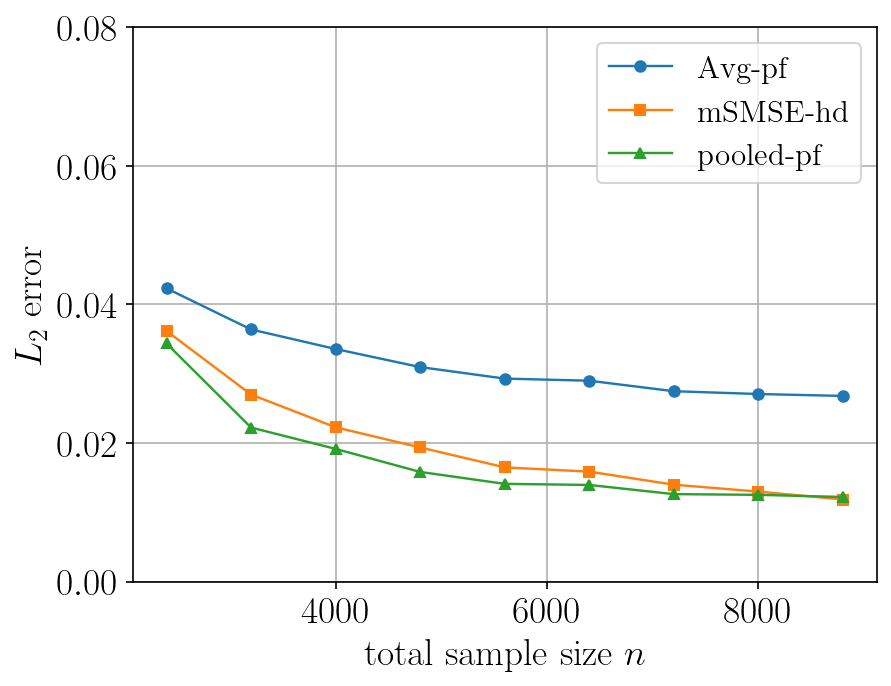}
	\caption{$m=800$}
	\label{fig:EE_m=800}
\end{subfigure}
\caption{The $L_2$ estimation errors of different methods in the high-dimensional setting.}
\label{fig:EE-hd}
\end{figure}

Figure \ref{fig:EE-hd} presents the $L_2$ estimation errors of the three algorithms averaged over 500 independent runs. Clearly, the proposed high-dimensional mSMSE algorithm (``mSMSE-hd'') always achieves lower estimation error compared to the Averaged Divide-and-Conquer algorithm (``Avg-pf''), and their difference increases as $n$ grows. In Figure \ref{fig:EE_m=800}, we also see that the $L_2$ errors of ``mSMSE-hd'' get very close to those of the oracle, the path-following algorithm applied to the pooled data (``pooled-pf'').

\subsection{More discussions on choosing bandwidth for {\AvgSMSE}}
\label{sec:bandwidth-supp}
\begin{table}[!t]
\centering
\caption{The bias, variance and coverage rates of {\AvgSMSE}, with bandwidth chosen by 5-fold cross-validation, $p=10$, $\log_m(n)$ from $1.4$ to $2.0$, and homoscedastic normal noise. }
\resizebox{0.6\textwidth}{!}{
	\begin{tabular}{c|ccc}
		\hline
		$\log_m(n)$ & bias ($\times 10^{-2}$)  & variance ($\times 10^{-4}$) & coverage rate \\
		\hline
		
		1.4 & $-2.80$  & 34.72 & 0.95 \\ 
		1.5 &  $-2.89$ & 38.53 &   0.89 \\ 
		1.6 &  $-3.25$ & 29.01 &  0.76 \\ 
		1.7 &  $-3.94$ & 15.31& 0.57 \\ 
		1.8 &  $-3.18$ & 16.20 &  0.53 \\ 
		1.9 &  $-1.97$&18.70&0.52\\
		2.0 & $-0.59$ & 21.21 & 0.45\\
		\hline 
	\end{tabular}
}
\label{tab:cv-avg}

\end{table}

In this section, we provide more discussions on the choice of bandwidth $h$ for the {\AvgSMSE} algorithm.
Theorem \ref{thm:asym-DC} for {\AvgSMSE} shows that, under a constraint $L=o\big( m^{\frac{2}{3}\left( \alpha-1\right)}/ \left( p\log m\right)^{\frac{2\alpha+1}{3}}\big)$,	{\AvgSMSE} achieves the optimal convergence rate $(p/n)^{\alpha/(2\alpha+1)}$. This constraint comes from a condition $\frac{p\log m}{mh^3}=o(1)$, which is necessary to ensure the convergence of the empirical Hessian to the population Hessian of the smoothed objective. Aiming to obtain the optimal convergence rate $(p/n)^{\alpha/(2\alpha+1)}$, one would like to choose the bandwidth $h \asymp h^* = (p/n)^{\frac{1}{2\alpha+1}}$, and the constraint $L=o\big( m^{\frac{2}{3}\left( \alpha-1\right)}/ \left( p\log m\right)^{\frac{2\alpha+1}{3}}\big)$ follows from plugging $h^* = (p/n)^{\frac{1}{2\alpha+1}}$ into $\frac{p\log m}{m(h^*)^3}=o(1)$. This constraint translates to the total sample size $n=mL=o(m^{\frac{2\alpha+1}{3}}/ \left( p\log m\right)^{\frac{2\alpha+1}{3}})$. 

However, in the simulation, we increase the total sample size $n$ beyond this constraint, and therefore the constraint is violated. If we look at the theoretical analysis in this case, there would be an additional bias term of the order \[O\Big(\sqrt{\frac{p\log m}{m(h^*)^3}}\Big)=O\Big(\sqrt{\frac{n^{\frac{3}{2\alpha+1}}p^{\frac{2\alpha-2}{2\alpha+1}}\log m}{m}}\Big),\] 
which increases in $n$ and invalidates the asymptotic normality established in \eqref{eq:asym-DC}. This is reflected in Tables \ref{tab:p=1} and \ref{tab:p=10} in Section \ref{supp:simu-bias}, where we can see that, as $\log_m(n)$ increases, the bias of {\AvgSMSE} becomes larger accompanied by a reduction in the coverage rate. This phenomenon calls for the need to propose a multiround procedure {\mSMSE} which can completely remove this constraint on $L$.

Additionally, as discussed in Remark \ref{rmk:h><h*}, {\AvgSMSE} still works for bandwidth $h>h^*$ at the sacrifice of the convergence rate. In particular, when the constraint $\frac{p\log m}{m(h^*)^3}=o(1)$ is violated, we can alternatively select a bandwidth $h>h^*$ such that $\frac{p\log m}{mh^3}=o(1)$, which results in a slower (sub-optimal) convergence rate. In addition, no asymptotic distribution is provided in theory for this sub-optimal bandwidth since the bias term will dominate the error.

In addition to Tables \ref{tab:p=1} and \ref{tab:p=10}, we provide additional simulations to confirm this phenomenon numerically in Table \ref{tab:cv-avg}, where we use a scale constant $c_h$ to fine-tune the bandwidth $h=c_h (p/n)^{1/(2\alpha+1)}$. The constant $c_h$ is determined by a five-fold cross-validation (CV). As compared to the results for fixed $c_h$ in Table \ref{tab:p=10}, when $\log_m(n)$ is small, the performance of {\AvgSMSE} with fine-tuned $c_h$ (via CV) is comparable to that with fixed $c_h$. As $\log_m(n)$ goes larger, the fine-tuned $c_h$ leads to smaller bias than the fixed $c_h$. Nonetheless, the convergence rates still fall down as $\log_m(n)$ goes larger, and moreover, the variances are significantly larger than those with fixed $c_h$. Therefore, there is no perfect way to choose bandwidth when the constraint $\frac{p\log m}{m(h^*)^3}=o(1)$ is violated, which necessitates the development of {\mSMSE}.


\end{document}